\documentclass[numbers=enddot,12pt,final,onecolumn,notitlepage]{scrartcl}%
\usepackage[headsepline,footsepline,manualmark]{scrlayer-scrpage}
\usepackage[all,cmtip]{xy}
\usepackage{amsfonts}
\usepackage{amssymb}
\usepackage{framed}
\usepackage{amsmath}
\usepackage{comment}
\usepackage{color}
\usepackage[breaklinks=true]{hyperref}
\usepackage[sc]{mathpazo}
\usepackage[T1]{fontenc}
\usepackage{amsthm}
\usepackage{needspace}
\providecommand{\U}[1]{\protect\rule{.1in}{.1in}}
\theoremstyle{definition}
\newtheorem{theo}{Theorem}[section]
\newenvironment{theorem}[1][]
{\begin{theo}[#1]\begin{leftbar}}
{\end{leftbar}\end{theo}}
\newtheorem{lem}[theo]{Lemma}
\newenvironment{lemma}[1][]
{\begin{lem}[#1]\begin{leftbar}}
{\end{leftbar}\end{lem}}
\newtheorem{prop}[theo]{Proposition}
\newenvironment{proposition}[1][]
{\begin{prop}[#1]\begin{leftbar}}
{\end{leftbar}\end{prop}}
\newtheorem{defi}[theo]{Definition}
\newenvironment{definition}[1][]
{\begin{defi}[#1]\begin{leftbar}}
{\end{leftbar}\end{defi}}
\newtheorem{remk}[theo]{Remark}
\newenvironment{remark}[1][]
{\begin{remk}[#1]\begin{leftbar}}
{\end{leftbar}\end{remk}}
\newtheorem{coro}[theo]{Corollary}
\newenvironment{corollary}[1][]
{\begin{coro}[#1]\begin{leftbar}}
{\end{leftbar}\end{coro}}
\newtheorem{conv}[theo]{Convention}
\newenvironment{convention}[1][]
{\begin{conv}[#1]\begin{leftbar}}
{\end{leftbar}\end{conv}}
\newtheorem{quest}[theo]{Question}

\newtheorem{warn}[theo]{Warning}

\newtheorem{conj}[theo]{Conjecture}
\newenvironment{conjecture}[1][]
{\begin{conj}[#1]\begin{leftbar}}
{\end{leftbar}\end{conj}}
\newtheorem{exmp}[theo]{Example}
\newenvironment{example}[1][]
{\begin{exmp}[#1]\begin{leftbar}}
{\end{leftbar}\end{exmp}}
\newtheorem{qust}[theo]{Question}
\newenvironment{question}[1][]
{\begin{qust}[#1]\begin{leftbar}}
{\end{leftbar}\end{qust}}
\newenvironment{statement}{\begin{quote}}{\end{quote}}
\newenvironment{verlong}{}{}

\newenvironment{noncompile}{}{}
\excludecomment{verlong}
\includecomment{vershort}
\excludecomment{noncompile}

\let\sumnonlimits\sum
\let\prodnonlimits\prod
\renewcommand{\sum}{\sumnonlimits\limits}
\renewcommand{\prod}{\prodnonlimits\limits}
\ihead{A basis for a quotient of symmetric polynomials (draft)}
\ohead{\today}
\begin{document}

\title{A basis for a quotient of symmetric polynomials (draft)}
\author{Darij Grinberg}
\date{
\today
}
\maketitle

\begin{abstract}
\textbf{Abstract.} Consider the ring $\mathcal{S}$ of symmetric polynomials in
$k$ variables over an arbitrary base ring $\mathbf{k}$. Fix $k$ scalars
$a_{1},a_{2},\ldots,a_{k}\in\mathbf{k}$.

Let $I$ be the ideal of $\mathcal{S}$ generated by $h_{n-k+1}-a_{1}%
,h_{n-k+2}-a_{2},\ldots,h_{n}-a_{k}$, where $h_{i}$ is the $i$-th complete
homogeneous symmetric polynomial.

The quotient ring $\mathcal{S}/I$ generalizes both the usual and the quantum
cohomology of the Grassmannian.

We show that $\mathcal{S}/I$ has a $\mathbf{k}$-module basis consisting of
(residue classes of) Schur polynomials fitting into an $\left(  n-k\right)
\times k$-rectangle; and that its multiplicative structure constants satisfy
the same $S_{3}$-symmetry as those of the Grassmannian cohomology. We prove a
Pieri rule and a \textquotedblleft rim hook algorithm\textquotedblright, and
conjecture a positivity property generalizing that of Gromov-Witten
invariants. We construct two further bases of $\mathcal{S}/I$ as well.

We also study the quotient of the whole polynomial ring (not just the
symmetric polynomials) by the ideal generated by the same $k$ polynomials as
$I$.

\end{abstract}
\tableofcontents

\section{Introduction}

\textbf{This is still a draft }-- proofs are at various levels of detail, and
the order of the results reflects the order in which I found them more than
the order in which they are most reasonable to read. This draft will probably
be split into several smaller papers for publication. \textbf{I recommend
\cite{Grinbe19} as a quick survey of the main results proved here}.

This work is devoted to a certain construction that generalizes both the
regular and the quantum cohomology ring of the Grassmannian \cite{Postni05}.
This construction is purely algebraic -- we do not know any geometric meaning
for it at this point -- but shares some basic properties with quantum
cohomology, such as an $S_{3}$-symmetry of its structure constants
(generalizing the $S_{3}$-symmetry for Littlewood-Richardson coefficients and
Gromov-Witten invariants) and conjecturally a positivity as well. All our
arguments are algebraic and combinatorial.

\subsection{Acknowledgments}

DG thanks Dongkwan Kim, Alex Postnikov, Victor Reiner, Mark Shimozono, Josh
Swanson, Kaisa Taipale, and Anders Thorup for enlightening conversations, and
the Mathematisches Forschungsinstitut Oberwolfach for its hospitality during
part of the writing process. The SageMath computer algebra system
\cite{SageMath} has been used for experimentation leading up to some of the
results below.

\section{The basis theorems}

\subsection{Definitions and notations}

Let $\mathbb{N}$ denote the set $\left\{  0,1,2,\ldots\right\}  $.

Let $\mathbf{k}$ be a commutative ring. Let $k\in\mathbb{N}$.

Let $\mathcal{P}$ denote the polynomial ring $\mathbf{k}\left[  x_{1}%
,x_{2},\ldots,x_{k}\right]  $. This is a graded ring, where the grading is by
total degree (so $\deg x_{i}=1$ for each $i\in\left\{  1,2,\ldots,k\right\}  $).

For each $\alpha\in\mathbb{Z}^{k}$ and each $i\in\left\{  1,2,\ldots
,k\right\}  $, we denote the $i$-th entry of $\alpha$ by $\alpha_{i}$ (so that
$\alpha=\left(  \alpha_{1},\alpha_{2},\ldots,\alpha_{k}\right)  $). For each
$\alpha\in\mathbb{N}^{k}$, we define a monomial $x^{\alpha}$ by $x^{\alpha
}=x_{1}^{\alpha_{1}}x_{2}^{\alpha_{2}}\cdots x_{k}^{\alpha_{k}}$.

Let $\mathcal{S}$ denote the ring of symmetric polynomials in $\mathcal{P}$;
in other words, $\mathcal{S}$ is the ring of invariants of the symmetric group
$S_{k}$ acting on $\mathcal{P}$. (The action here is the one you would expect:
A permutation $\sigma\in S_{k}$ sends a monomial $x_{i_{1}}x_{i_{2}}\cdots
x_{i_{m}}$ to $x_{\sigma\left(  i_{1}\right)  }x_{\sigma\left(  i_{2}\right)
}\cdots x_{\sigma\left(  i_{m}\right)  }$.)

The following fact is well-known (going back to Emil Artin):

\begin{proposition}
\label{prop.artin}The $\mathcal{S}$-module $\mathcal{P}$ is free with basis
$\left(  x^{\alpha}\right)  _{\alpha\in\mathbb{N}^{k};\ \alpha_{i}<i\text{ for
each }i}$.
\end{proposition}

Proofs of Proposition \ref{prop.artin} can be found in \cite[(DIFF.1.3)]%
{LLPT95}, in \cite[Chapter IV, \S 6, no. 1, Theorem 1 c)]{Bourba03} and in
\cite[(5.1)]{Macdon91}\footnote{Strictly speaking, \cite[(5.1)]{Macdon91} is
only the particular case of Proposition \ref{prop.artin} for $\mathbf{k}%
=\mathbb{Z}$. However, with some minor modifications, the proof given in
\cite{Macdon91} works for any $\mathbf{k}$.}. The particular case when
$\mathbf{k}$ is a field is also proved in \cite[result shown at the end of
\S II.G]{Artin71}\footnote{To be more precise, Artin proves in \cite[\S II.G,
Example 2]{Artin71} that (when $\mathbf{k}$ is a field)
\par
\begin{itemize}
\item the monomials $x^{\alpha}$ with $\alpha\in\mathbb{N}^{k}$ satisfying
$\alpha_{i}<i$ for each $i$ are linearly independent over the field
$\mathcal{S}_{\operatorname*{rat}}$ of symmetric rational functions in
$x_{1},x_{2},\ldots,x_{k}$ over $\mathbf{k}$ (and therefore also linearly
independent over the ring $\mathcal{S}$ of symmetric polynomials), and
\par
\item each polynomial $g\in\mathcal{P}$ can be represented as a polynomial in
$x_{1},x_{2},\ldots,x_{k}$ with coefficients in $\mathcal{S}$ and having
degree $<i$ in each $x_{i}$ (that is, as an $\mathcal{S}$-linear combination
of the monomials $x^{\alpha}$ with $\alpha\in\mathbb{N}^{k}$ satisfying
$\alpha_{i}<i$).
\end{itemize}
\par
Combining these two facts yields Proposition \ref{prop.artin} (when
$\mathbf{k}$ is a field).}. The particular case of Proposition
\ref{prop.artin} when $\mathbf{k}=\mathbb{Q}$ also appears in \cite[Remark
3.2]{Garsia02}. A related result is proven in \cite[Proposition 3.4]{FoGePo97}
(for $\mathbf{k}=\mathbb{Z}$, but the proof applies equally over any
$\mathbf{k}$).

Now, fix an integer $n\geq k$. For each $i\in\left\{  1,2,\ldots,k\right\}  $,
let $a_{i}$ be an element of $\mathcal{P}$ with degree $<n-k+i$. (This is
clearly satisfied when $a_{1},a_{2},\ldots,a_{k}$ are constants in
$\mathbf{k}$, but also in some other cases. Note that the $a_{i}$ do not have
to be homogeneous.)

For each $\alpha\in\mathbb{Z}^{k}$, we let $\left\vert \alpha\right\vert $
denote the sum of the entries of the $k$-tuple $\alpha$ (that is, $\left\vert
\alpha\right\vert =\alpha_{1}+\alpha_{2}+\cdots+\alpha_{k}$).

For each $m\in\mathbb{Z}$, we let $h_{m}$ denote the $m$-th complete
homogeneous symmetric polynomial; this is the element of $\mathcal{S}$ defined
by%
\begin{equation}
h_{m}=\sum_{1\leq i_{1}\leq i_{2}\leq\cdots\leq i_{m}\leq k}x_{i_{1}}x_{i_{2}%
}\cdots x_{i_{m}}=\sum_{\substack{\alpha\in\mathbb{N}^{k};\\\left\vert
\alpha\right\vert =m}}x^{\alpha}. \label{eq.hm}%
\end{equation}
(Thus, $h_{0}=1$, and $h_{m}=0$ when $m<0$.)

Let $J$ be the ideal of $\mathcal{P}$ generated by the $k$ differences%
\begin{equation}
h_{n-k+1}-a_{1},h_{n-k+2}-a_{2},\ldots,h_{n}-a_{k}. \label{eq.gensJ}%
\end{equation}

If $M$ is a $\mathbf{k}$-module and $N$ is a submodule of $M$, then the
projection of any $m\in M$ onto the quotient $M/N$ (that is, the congruence
class of $m$ modulo $N$) will be denoted by $\overline{m}$.

\subsection{The basis theorem for $\mathcal{P}/J$}

The following is our first result:

\begin{theorem}
\label{thm.P/J}The $\mathbf{k}$-module $\mathcal{P}/J$ is free with basis
$\left(  \overline{x^{\alpha}}\right)  _{\alpha\in\mathbb{N}^{k};\ \alpha
_{i}<n-k+i\text{ for each }i}$.
\end{theorem}

\begin{example}
Let $n=5$ and $k=2$. Then, $\mathcal{P}=\mathbf{k}\left[  x_{1},x_{2}\right]
$, and $J$ is the ideal of $\mathcal{P}$ generated by the $2$ differences
\begin{align*}
h_{4}-a_{1}  &  =\left(  x_{1}^{4}+x_{1}^{3}x_{2}+x_{1}^{2}x_{2}^{2}%
+x_{1}x_{2}^{3}+x_{2}^{4}\right)  -a_{1}\ \ \ \ \ \ \ \ \ \ \text{and}\\
h_{5}-a_{2}  &  =\left(  x_{1}^{5}+x_{1}^{4}x_{2}+x_{1}^{3}x_{2}^{2}+x_{1}%
^{2}x_{2}^{3}+x_{1}x_{2}^{4}+x_{2}^{5}\right)  -a_{2}.
\end{align*}
Theorem \ref{thm.P/J} yields that the $\mathbf{k}$-module $\mathcal{P}/J$ is
free with basis $\left(  \overline{x^{\alpha}}\right)  _{\alpha\in
\mathbb{N}^{2};\ \alpha_{i}<3+i\text{ for each }i}$; this basis can also be
rewritten as $\left(  \overline{x_{1}^{\alpha_{1}}x_{2}^{\alpha_{2}}}\right)
_{\alpha_{1}\in\left\{  0,1,2,3\right\}  ;\ \alpha_{2}\in\left\{
0,1,2,3,4\right\}  }$. As a consequence, any $\overline{x_{1}^{\beta_{1}}%
x_{2}^{\beta_{2}}}\in\mathcal{P}/J$ can be written as a linear combination of
elements of this basis. For example,%
\begin{align*}
\overline{x_{1}^{4}}  &  =a_{1}-\overline{x_{1}^{3}x_{2}}-\overline{x_{1}%
^{2}x_{2}^{2}}-\overline{x_{1}x_{2}^{3}}-\overline{x_{2}^{4}}%
\ \ \ \ \ \ \ \ \ \ \text{and}\\
\overline{x_{2}^{5}}  &  =a_{2}-a_{1}\overline{x_{1}}.
\end{align*}
These expressions will become more complicated for higher values of $n$ and
$k$.
\end{example}

Theorem \ref{thm.P/J} is related to the second part of \cite[Proposition
2.9]{CoKrWa09} (and our proof below can be viewed as an elaboration of the
argument sketched in the last paragraph of \cite[proof of Proposition
2.9]{CoKrWa09}).

\subsection{The basis theorem for $\mathcal{S}/I$}

To state our next result, we need some more notations.

\begin{definition}
\label{def.partitions}\textbf{(a)} We define the concept of
\textit{partitions} (of an integer) as in \cite[Chapter 2]{GriRei18}. Thus, a
partition is a weakly decreasing infinite sequence $\left(  \lambda
_{1},\lambda_{2},\lambda_{3},\ldots\right)  $ of nonnegative integers such
that all but finitely many $i$ satisfy $\lambda_{i}=0$. We identify each
partition $\left(  \lambda_{1},\lambda_{2},\lambda_{3},\ldots\right)  $ with
the finite list $\left(  \lambda_{1},\lambda_{2},\ldots,\lambda_{p}\right)  $
whenever $p\in\mathbb{N}$ has the property that $\left(  \lambda_{i}=0\text{
for all }i>p\right)  $. For example, the partition $\left(
3,1,1,\underbrace{0,0,\ldots}_{\text{zeroes}}\right)  $ is identified with
$\left(  3,1,1,0\right)  $ and with $\left(  3,1,1\right)  $.

\textbf{(b)} A \textit{part} of a partition $\lambda$ means a nonzero entry of
$\lambda$.

\textbf{(c)} Let $P_{k,n}$ denote the set of all partitions that have at most
$k$ parts and have the property that each of their parts is $\leq n-k$.
(Visually speaking, $P_{k,n}$ is the set of all partitions whose Young diagram
fits into a $k\times\left(  n-k\right)  $-rectangle.)

\textbf{(d)} We let $\varnothing$ denote the empty partition $\left(
{}\right)  $.
\end{definition}

\begin{example}
If $n=4$ and $k=2$, then%
\[
P_{k,n}=P_{2,4}=\left\{  \varnothing,\left(  1\right)  ,\left(  2\right)
,\left(  1,1\right)  ,\left(  2,1\right)  ,\left(  2,2\right)  \right\}  .
\]

If $n=5$ and $k=2$, then%
\[
P_{k,n}=P_{2,5}=\left\{  \varnothing,\left(  1\right)  ,\left(  2\right)
,\left(  3\right)  ,\left(  1,1\right)  ,\left(  2,1\right)  ,\left(
3,1\right)  ,\left(  2,2\right)  ,\left(  3,2\right)  ,\left(  3,3\right)
\right\}  .
\]

\end{example}

It is well-known (and easy to see) that $P_{k,n}$ is a finite set of size
$\dbinom{n}{k}$. (Indeed, the map%
\begin{align*}
P_{k,n}  &  \rightarrow\left\{  \left(  a_{1},a_{2},\ldots,a_{k}\right)
\in\left\{  1,2,\ldots,n\right\}  ^{k}\ \mid\ a_{1}>a_{2}>\cdots
>a_{k}\right\}  ,\\
\lambda &  \mapsto\left(  \lambda_{1}+k,\lambda_{2}+k-1,\ldots,\lambda
_{k}+1\right)
\end{align*}
is easily seen to be well-defined and to be a bijection; but the set
\newline$\left\{  \left(  a_{1},a_{2},\ldots,a_{k}\right)  \in\left\{
1,2,\ldots,n\right\}  ^{k}\ \mid\ a_{1}>a_{2}>\cdots>a_{k}\right\}  $ has size
$\dbinom{n}{k}$.)

\begin{definition}
\label{def.slam}For any partition $\lambda$, we let $s_{\lambda}$ denote the
Schur polynomial in $x_{1},x_{2},\ldots,x_{k}$ corresponding to the partition
$\lambda$. This Schur polynomial is what is called $s_{\lambda}\left(
x_{1},x_{2},\ldots,x_{k}\right)  $ in \cite[Chapter 2]{GriRei18}. Note that%
\begin{equation}
s_{\lambda}=0\ \ \ \ \ \ \ \ \ \ \text{if }\lambda\text{ has more than
}k\text{ parts.} \label{eq.slam=0-too-long}%
\end{equation}

\end{definition}

If $\lambda$ is any partition, then the Schur polynomial $s_{\lambda
}=s_{\lambda}\left(  x_{1},x_{2},\ldots,x_{k}\right)  $ is symmetric and thus
belongs to $\mathcal{S}$.

We now state our next fundamental fact:

\begin{theorem}
\label{thm.S/J}Assume that $a_{1},a_{2},\ldots,a_{k}$ belong to $\mathcal{S}$.
Let $I$ be the ideal of $\mathcal{S}$ generated by the $k$ differences
(\ref{eq.gensJ}). Then, the $\mathbf{k}$-module $\mathcal{S}/I$ is free with
basis $\left(  \overline{s_{\lambda}}\right)  _{\lambda\in P_{k,n}}$.
\end{theorem}

We will prove Theorem \ref{thm.S/J} below; a different proof has been given by
Weinfeld in \cite[Corollary 6.2]{Weinfe19}.

The $\mathbf{k}$-algebra $\mathcal{S}/I$ generalizes several constructions in
the literature:

\begin{itemize}
\item If $\mathbf{k}=\mathbb{Z}$ and $a_{1}=a_{2}=\cdots=a_{k}=0$, then
$\mathcal{S}/I$ becomes the cohomology ring of the Grassmannian of
$k$-dimensional subspaces in an $n$-dimensional space (see, e.g.,
\cite[\S 9.4]{Fulton97} or \cite[Exercise 3.2.12]{Manive01}); the elements of
the basis $\left(  \overline{s_{\lambda}}\right)  _{\lambda\in P_{k,n}}$
correspond to the Schubert classes.

\item If $\mathbf{k}=\mathbb{Z}\left[  q\right]  $ and $a_{1}=a_{2}%
=\cdots=a_{k-1}=0$ and $a_{k}=-\left(  -1\right)  ^{k}q$, then $\mathcal{S}/I$
becomes isomorphic to the quantum cohomology ring of the same Grassmannian
(see \cite{Postni05}). Indeed, our ideal $I$ becomes the $J_{kn}^{q}$ of
\cite[(6)]{Postni05} in this case, and Theorem \ref{thm.S/J} generalizes the
fact that the quotient $\left(  \Lambda_{k}\otimes\mathbb{Z}\left[  q\right]
\right)  /J_{kn}^{q}$ in \cite[(6)]{Postni05} has basis $\left(  s_{\lambda
}\right)  _{\lambda\in P_{kn}}$.
\end{itemize}

One goal of this paper is to provide a purely algebraic foundation for the
study of the standard and quantum cohomology rings of the Grassmannian,
without having to resort to geometry for proofs of the basic properties of
these rings. In particular, Theorem \ref{thm.S/J} shows that the
\textquotedblleft abstract Schubert classes\textquotedblright\ $\overline
{s_{\lambda}}$ (with $\lambda\in P_{k,n}$) form a basis of the $\mathbf{k}%
$-module $\mathcal{S}/I$, whereas Corollary \ref{cor.S3sym} further below
shows that the structure constants of the $\mathbf{k}$-algebra $\mathcal{S}/I$
with respect to this basis (we may call them \textquotedblleft generalized
Gromov-Witten invariants\textquotedblright) satisfy an $S_{3}$-symmetry. These
two properties are two of the facts for whose proofs \cite{Postni05} relies on
algebro-geometric literature; thus, our paper helps provide an alternative
footing for \cite{Postni05} using only combinatorics and
algebra\footnote{This, of course, presumes that one is willing to forget the
cohomological definition of the ring $\operatorname*{QH}\nolimits^{\ast
}\left(  \operatorname*{Gr}\nolimits_{kn}\right)  $, and instead to define it
algebraically as the quotient ring $\left(  \Lambda_{k}\otimes\mathbb{Z}%
\left[  q\right]  \right)  /J_{kn}^{q}$, using the notations of
\cite{Postni05}.}.

\begin{remark}
The $\mathbf{k}$-algebra $\mathcal{P}/J$ somewhat resembles the
\textquotedblleft splitting algebra\textquotedblright\ $\operatorname*{Split}%
\nolimits_{A}^{d}\left(  p\right)  $ from \cite[\S 1.3]{LakTho12}; further
analogies between these concepts can be made as we study the former. For
example, the basis we give in Theorem \ref{thm.P/J} is like the basis in
\cite[(1.5)]{LakTho12}. It is not currently clear to us whether there is more
than analogies.
\end{remark}

\section{A fundamental identity}

Let us use the notations $h_{m}$ and $e_{m}$ for complete homogeneous
symmetric polynomials and elementary symmetric polynomials in general. Thus,
for any $m\in\mathbb{Z}$ and any $p$ elements $y_{1},y_{2},\ldots,y_{p}$ of a
commutative ring, we set%
\begin{align}
h_{m}\left(  y_{1},y_{2},\ldots,y_{p}\right)   &  =\sum_{1\leq i_{1}\leq
i_{2}\leq\cdots\leq i_{m}\leq p}y_{i_{1}}y_{i_{2}}\cdots y_{i_{m}%
}\ \ \ \ \ \ \ \ \ \ \text{and}\label{eq.def.hm}\\
e_{m}\left(  y_{1},y_{2},\ldots,y_{p}\right)   &  =\sum_{1\leq i_{1}%
<i_{2}<\cdots<i_{m}\leq p}y_{i_{1}}y_{i_{2}}\cdots y_{i_{m}}.
\label{eq.def.em}%
\end{align}
(Thus, $h_{0}\left(  y_{1},y_{2},\ldots,y_{p}\right)  =1$ and $e_{0}\left(
y_{1},y_{2},\ldots,y_{p}\right)  =1$. Also, $e_{m}\left(  y_{1},y_{2}%
,\ldots,y_{p}\right)  =0$ for all $m>p$. Also, for any $m<0$, we have
$h_{m}\left(  y_{1},y_{2},\ldots,y_{p}\right)  =0$ and $e_{m}\left(
y_{1},y_{2},\ldots,y_{p}\right)  =0$. Finally, what we have previously called
$h_{m}$ without any arguments can now be rewritten as $h_{m}\left(
x_{1},x_{2},\ldots,x_{k}\right)  $. Similarly, we shall occasionally
abbreviate $e_{m}\left(  x_{1},x_{2},\ldots,x_{k}\right)  $ as $e_{m}$.)

\begin{lemma}
\label{lem.heh-id}Let $i\in\left\{  1,2,\ldots,k+1\right\}  $ and
$p\in\mathbb{N}$. Then,%
\[
h_{p}\left(  x_{i},x_{i+1},\ldots,x_{k}\right)  =\sum_{t=0}^{i-1}\left(
-1\right)  ^{t}e_{t}\left(  x_{1},x_{2},\ldots,x_{i-1}\right)  h_{p-t}\left(
x_{1},x_{2},\ldots,x_{k}\right)  .
\]

\end{lemma}

Notice that if $i=k+1$, then the term $h_{p}\left(  x_{i},x_{i+1},\ldots
,x_{k}\right)  $ on the left hand side of Lemma \ref{lem.heh-id} is understood
to be $h_{p}$ of an empty list of vectors; this is $1$ when $p=0$ and $0$ otherwise.

Lemma \ref{lem.heh-id} is actually a particular case of \cite[detailed
version, Theorem 3.15]{dimcr} (applied to $a=x_{i}\in\mathbf{k}\left[  \left[
x_{1},x_{2},x_{3},\ldots\right]  \right]  $ and $b=h_{p}\left(  x_{1}%
,x_{2},x_{3},\ldots\right)  \in\operatorname*{QSym}$)\ \ \ \ \footnote{Here,
we are using the ring $\mathbf{k}\left[  \left[  x_{1},x_{2},x_{3}%
,\ldots\right]  \right]  $ of formal power series in \textbf{infinitely} many
variables $x_{1},x_{2},x_{3},\ldots$, and its subring $\operatorname*{QSym}$
of quasisymmetric functions. See \cite{dimcr} for a brief introduction to both
of these. Note that the symmetric function $h_{p}\left(  x_{1},x_{2}%
,x_{3},\ldots\right)  $ is called $h_{p}$ in \cite{dimcr}.}. However, we shall
give a more elementary proof of it here. This proof relies on the following
two basic identities:

\begin{lemma}
\label{lem.he-powser}Let $A$ be a commutative ring. Let $y_{1},y_{2}%
,\ldots,y_{p}$ be some elements of $A$. Consider the ring $A\left[  \left[
u\right]  \right]  $ of formal power series in one indeterminate $u$ over $A$.
Then, in this ring, we have%
\begin{equation}
\sum_{q\in\mathbb{N}}h_{q}\left(  y_{1},y_{2},\ldots,y_{p}\right)  u^{q}%
=\prod_{j=1}^{p}\dfrac{1}{1-y_{j}u} \label{eq.lem.he-powser.h}%
\end{equation}
and%
\begin{equation}
\sum_{q\in\mathbb{N}}\left(  -1\right)  ^{q}e_{q}\left(  y_{1},y_{2}%
,\ldots,y_{p}\right)  u^{q}=\prod_{j=1}^{p}\left(  1-y_{j}u\right)  .
\label{eq.lem.he-powser.e}%
\end{equation}

\end{lemma}

\begin{proof}
[Proof of Lemma \ref{lem.he-powser}.]The identity (\ref{eq.lem.he-powser.h})
can be obtained from the identities \cite[(2.2.18)]{GriRei18} by substituting
$y_{1},y_{2},\ldots,y_{p},0,0,0,\ldots$ for the indeterminates $x_{1}%
,x_{2},x_{3},\ldots$ and substituting $u$ for $t$. The identity
(\ref{eq.lem.he-powser.e}) can be obtained from the identities \cite[(2.2.19)]%
{GriRei18} by substituting $y_{1},y_{2},\ldots,y_{p},0,0,0,\ldots$ for the
indeterminates $x_{1},x_{2},x_{3},\ldots$ and substituting $-u$ for $t$. Thus,
Lemma \ref{lem.he-powser} is proven.
\end{proof}

\begin{proof}
[Proof of Lemma \ref{lem.heh-id}.]Consider the ring $\mathcal{P}\left[
\left[  u\right]  \right]  $ of formal power series in one indeterminate $u$
over $\mathcal{P}$. Applying (\ref{eq.lem.he-powser.h}) to $\mathcal{P}$ and
$\left(  x_{i},x_{i+1},\ldots,x_{k}\right)  $ instead of $A$ and $\left(
y_{1},y_{2},\ldots,y_{p}\right)  $, we obtain%
\[
\sum_{q\in\mathbb{N}}h_{q}\left(  x_{i},x_{i+1},\ldots,x_{k}\right)
u^{q}=\prod_{j=1}^{k-i+1}\dfrac{1}{1-x_{i+j-1}u}=\prod_{j=i}^{k}\dfrac
{1}{1-x_{j}u}%
\]
(here, we have substituted $j$ for $i+j-1$ in the product). Applying
(\ref{eq.lem.he-powser.e}) to $\mathcal{P}$ and $\left(  x_{1},x_{2}%
,\ldots,x_{i-1}\right)  $ instead of $A$ and $\left(  y_{1},y_{2},\ldots
,y_{p}\right)  $, we obtain%
\begin{equation}
\sum_{q\in\mathbb{N}}\left(  -1\right)  ^{q}e_{q}\left(  x_{1},x_{2}%
,\ldots,x_{i-1}\right)  u^{q}=\prod_{j=1}^{i-1}\left(  1-x_{j}u\right)  .
\label{pf.lem.heh-id.2}%
\end{equation}
Applying (\ref{eq.lem.he-powser.h}) to $\mathcal{P}$ and $\left(  x_{1}%
,x_{2},\ldots,x_{k}\right)  $ instead of $A$ and $\left(  y_{1},y_{2}%
,\ldots,y_{p}\right)  $, we obtain
\begin{equation}
\sum_{q\in\mathbb{N}}h_{q}\left(  x_{1},x_{2},\ldots,x_{k}\right)  u^{q}%
=\prod_{j=1}^{k}\dfrac{1}{1-x_{j}u}. \label{pf.lem.heh-id.3}%
\end{equation}
Thus,
\begin{align*}
&  \sum_{q\in\mathbb{N}}h_{q}\left(  x_{i},x_{i+1},\ldots,x_{k}\right)
u^{q}\\
&  =\prod_{j=i}^{k}\dfrac{1}{1-x_{j}u}=\left(  \prod_{j=1}^{k}\dfrac
{1}{1-x_{j}u}\right)  /\left(  \prod_{j=1}^{i-1}\dfrac{1}{1-x_{j}u}\right) \\
&  =\underbrace{\left(  \prod_{j=1}^{i-1}\left(  1-x_{j}u\right)  \right)
}_{\substack{=\sum_{q\in\mathbb{N}}\left(  -1\right)  ^{q}e_{q}\left(
x_{1},x_{2},\ldots,x_{i-1}\right)  u^{q}\\\text{(by (\ref{pf.lem.heh-id.2}))}%
}}\ \ \underbrace{\left(  \prod_{j=1}^{k}\dfrac{1}{1-x_{j}u}\right)
}_{\substack{=\sum_{q\in\mathbb{N}}h_{q}\left(  x_{1},x_{2},\ldots
,x_{k}\right)  u^{q}\\\text{(by (\ref{pf.lem.heh-id.3}))}}}\\
&  =\left(  \sum_{q\in\mathbb{N}}\left(  -1\right)  ^{q}e_{q}\left(
x_{1},x_{2},\ldots,x_{i-1}\right)  u^{q}\right)  \left(  \sum_{q\in\mathbb{N}%
}h_{q}\left(  x_{1},x_{2},\ldots,x_{k}\right)  u^{q}\right)  .
\end{align*}
Comparing the coefficient before $u^{p}$ in this equality of power series, we
obtain%
\begin{align*}
h_{p}\left(  x_{i},x_{i+1},\ldots,x_{k}\right)   &  =\sum_{t=0}^{p}\left(
-1\right)  ^{t}e_{t}\left(  x_{1},x_{2},\ldots,x_{i-1}\right)  h_{p-t}\left(
x_{1},x_{2},\ldots,x_{k}\right) \\
&  =\sum_{t=0}^{\infty}\left(  -1\right)  ^{t}e_{t}\left(  x_{1},x_{2}%
,\ldots,x_{i-1}\right)  h_{p-t}\left(  x_{1},x_{2},\ldots,x_{k}\right) \\
&  \ \ \ \ \ \ \ \ \ \ \left(  \text{since }h_{p-t}\left(  x_{1},x_{2}%
,\ldots,x_{k}\right)  =0\text{ for all }t>p\right) \\
&  =\sum_{t=0}^{i-1}\left(  -1\right)  ^{t}e_{t}\left(  x_{1},x_{2}%
,\ldots,x_{i-1}\right)  h_{p-t}\left(  x_{1},x_{2},\ldots,x_{k}\right) \\
&  \ \ \ \ \ \ \ \ \ \ \left(  \text{since }e_{t}\left(  x_{1},x_{2}%
,\ldots,x_{i-1}\right)  =0\text{ for all }t>i-1\right)  .
\end{align*}
This proves Lemma \ref{lem.heh-id}.
\end{proof}

\begin{corollary}
\label{cor.heh-id.0}Let $p$ be a positive integer. Then,%
\[
h_{p}=-\sum_{t=1}^{k}\left(  -1\right)  ^{t}e_{t}h_{p-t}.
\]

\end{corollary}

\begin{proof}
[Proof of Corollary \ref{cor.heh-id.0}.]Lemma \ref{lem.heh-id} (applied to
$i=k+1$) yields%
\begin{align*}
h_{p}\left(  x_{k+1},x_{k+2},\ldots,x_{k}\right)   &  =\sum_{t=0}^{k}\left(
-1\right)  ^{t}\underbrace{e_{t}\left(  x_{1},x_{2},\ldots,x_{k}\right)
}_{=e_{t}}\underbrace{h_{p-t}\left(  x_{1},x_{2},\ldots,x_{k}\right)
}_{=h_{p-t}}\\
&  =\sum_{t=0}^{k}\left(  -1\right)  ^{t}e_{t}h_{p-t}.
\end{align*}
Comparing this with%
\[
h_{p}\left(  x_{k+1},x_{k+2},\ldots,x_{k}\right)  =h_{p}\left(  \text{an empty
list of variables}\right)  =0\ \ \ \ \ \ \ \ \ \ \left(  \text{since
}p>0\right)  ,
\]
we obtain%
\[
0=\sum_{t=0}^{k}\left(  -1\right)  ^{t}e_{t}h_{p-t}=\underbrace{\left(
-1\right)  ^{0}}_{=1}\underbrace{e_{0}}_{=1}\underbrace{h_{p-0}}_{=h_{p}}%
+\sum_{t=1}^{k}\left(  -1\right)  ^{t}e_{t}h_{p-t}=h_{p}+\sum_{t=1}^{k}\left(
-1\right)  ^{t}e_{t}h_{p-t}.
\]
Hence,%
\[
h_{p}=-\sum_{t=1}^{k}\left(  -1\right)  ^{t}e_{t}h_{p-t}.
\]
This proves Corollary \ref{cor.heh-id.0}.
\end{proof}

\section{Proof of Theorem \ref{thm.P/J}}

We shall next prove Theorem \ref{thm.P/J} using Gr\"{o}bner bases. For the
concept of Gr\"{o}bner bases over a commutative ring, see \cite[detailed
version, \S 3]{Grinbe17}.

We define a degree-lexicographic term order on the monomials in $\mathcal{P}$,
where the variables are ordered by $x_{1}>x_{2}>\cdots>x_{k}$. Explicitly,
this term order is the total order on the set of monomials in $x_{1}%
,x_{2},\ldots,x_{k}$ defined as follows: Two monomials $x_{1}^{\alpha_{1}%
}x_{2}^{\alpha_{2}}\cdots x_{k}^{\alpha_{k}}$ and $x_{1}^{\beta_{1}}%
x_{2}^{\beta_{2}}\cdots x_{k}^{\beta_{k}}$ satisfy $x_{1}^{\alpha_{1}}%
x_{2}^{\alpha_{2}}\cdots x_{k}^{\alpha_{k}}>x_{1}^{\beta_{1}}x_{2}^{\beta_{2}%
}\cdots x_{k}^{\beta_{k}}$ if and only if

\begin{itemize}
\item \textbf{either} $\alpha_{1}+\alpha_{2}+\cdots+\alpha_{k}>\beta_{1}%
+\beta_{2}+\cdots+\beta_{k}$,

\item \textbf{or} $\alpha_{1}+\alpha_{2}+\cdots+\alpha_{k}=\beta_{1}+\beta
_{2}+\cdots+\beta_{k}$ and there exists some $i\in\left\{  1,2,\ldots
,k\right\}  $ such that $\alpha_{i}>\beta_{i}$ and $\left(  \alpha_{j}%
=\beta_{j}\text{ for all }j<i\right)  $.
\end{itemize}

This total order is a term order (in the sense of \cite[detailed version,
Definition 3.5]{Grinbe17}). Fix this term order; thus it makes sense to speak
of Gr\"{o}bner bases of ideals.

\begin{proposition}
\label{prop.J.grobas}The family
\[
\left(  h_{n-k+i}\left(  x_{i},x_{i+1},\ldots,x_{k}\right)  -\sum_{t=0}%
^{i-1}\left(  -1\right)  ^{t}e_{t}\left(  x_{1},x_{2},\ldots,x_{i-1}\right)
a_{i-t}\right)  _{i\in\left\{  1,2,\ldots,k\right\}  }%
\]
is a Gr\"{o}bner basis of the ideal $J$. (Recall that we are using the
notations from (\ref{eq.def.hm}) and (\ref{eq.def.em}).)
\end{proposition}

Proposition \ref{prop.J.grobas} is somewhat similar to \cite[Theorem
1.2.7]{Sturmf08} (or, equivalently, \cite[\S 7.1, Proposition 5]{CoLiOs15}),
but not the same.\footnote{For example, our $a_{1},a_{2},\ldots,a_{k}$ are
elements of $\mathbf{k}$ rather than indeterminates (although they
\textbf{can} be indeterminates if $\mathbf{k}$ itself is a polynomial ring),
and our term order is degree-lexicographic rather than lexicographic. Thus, it
should not be surprising that the families are different.} It is also similar
to \cite[comment at the end of \S III.4]{LomQui21}. Our proof of it relies on
the following elementary fact:

\begin{lemma}
\label{lem.submodules.same-sum}Let $A$ be a commutative ring. Let $b_{1}%
,b_{2},\ldots,b_{k}\in A$ and $c_{1},c_{2},\ldots,c_{k}\in A$. Assume that
\begin{equation}
b_{i}\in c_{i}+\sum_{t=1}^{i-1}c_{i-t}A \label{eq.lem.submodules.same-sum.ass}%
\end{equation}
for each $i\in\left\{  1,2,\ldots,k\right\}  $. Then, $b_{1}A+b_{2}%
A+\cdots+b_{k}A=c_{1}A+c_{2}A+\cdots+c_{k}A$ (as ideals of $A$).
\end{lemma}

\begin{proof}
[Proof of Lemma \ref{lem.submodules.same-sum}.]We claim that%
\begin{equation}
\sum_{p=1}^{j}b_{p}A=\sum_{p=1}^{j}c_{p}A\ \ \ \ \ \ \ \ \ \ \text{for each
}j\in\left\{  0,1,\ldots,k\right\}  . \label{pf.lem.submodules.same-sum.main}%
\end{equation}

[\textit{Proof of (\ref{pf.lem.submodules.same-sum.main}):} We shall prove
(\ref{pf.lem.submodules.same-sum.main}) by induction on $j$:

\textit{Induction base:} For $j=0$, both sides of the equality
(\ref{pf.lem.submodules.same-sum.main}) are the zero ideal of $A$ (since they
are empty sums of ideals of $A$). Thus, (\ref{pf.lem.submodules.same-sum.main}%
) holds for $j=0$. This completes the induction base.

\textit{Induction step:} Let $i\in\left\{  1,2,\ldots,k\right\}  $. Assume
that (\ref{pf.lem.submodules.same-sum.main}) holds for $j=i-1$. We must prove
that (\ref{pf.lem.submodules.same-sum.main}) holds for $j=i$.

We have assumed that (\ref{pf.lem.submodules.same-sum.main}) holds for
$j=i-1$. In other words, we have $\sum_{p=1}^{i-1}b_{p}A=\sum_{p=1}^{i-1}%
c_{p}A$. But (\ref{eq.lem.submodules.same-sum.ass}) yields $b_{i}\in
c_{i}+\sum_{t=1}^{i-1}c_{i-t}A=c_{i}+\sum_{p=1}^{i-1}c_{p}A$ (here, we have
substituted $p$ for $i-t$ in the sum). Thus,
\[
c_{i}\in b_{i}-\sum_{p=1}^{i-1}c_{p}A=b_{i}+\sum_{p=1}^{i-1}c_{p}A,
\]
so that%
\[
c_{i}A\subseteq\left(  b_{i}+\sum_{p=1}^{i-1}c_{p}A\right)  A\subseteq
b_{i}A+\sum_{p=1}^{i-1}c_{p}A.
\]

But from $b_{i}\in c_{i}+\sum_{p=1}^{i-1}c_{p}A$, we obtain%
\[
b_{i}A\subseteq\left(  c_{i}+\sum_{p=1}^{i-1}c_{p}A\right)  A\subseteq
c_{i}A+\sum_{p=1}^{i-1}c_{p}A=\sum_{p=1}^{i}c_{p}A.
\]
Now,%
\[
\sum_{p=1}^{i}b_{p}A=\underbrace{\sum_{p=1}^{i-1}b_{p}A}_{\substack{=\sum
_{p=1}^{i-1}c_{p}A\subseteq\sum_{p=1}^{i}c_{p}A\\\text{(since }i-1\leq
i\text{)}}}+\underbrace{b_{i}A}_{\subseteq\sum_{p=1}^{i}c_{p}A}\subseteq
\sum_{p=1}^{i}c_{p}A+\sum_{p=1}^{i}c_{p}A=\sum_{p=1}^{i}c_{p}A.
\]
Combining this inclusion with%
\begin{align*}
\sum_{p=1}^{i}c_{p}A  &  =\sum_{p=1}^{i-1}c_{p}A+\underbrace{c_{i}%
A}_{\subseteq b_{i}A+\sum_{p=1}^{i-1}c_{p}A}\subseteq\sum_{p=1}^{i-1}%
c_{p}A+b_{i}A+\sum_{p=1}^{i-1}c_{p}A\\
&  =\underbrace{\sum_{p=1}^{i-1}c_{p}A+\sum_{p=1}^{i-1}c_{p}A}_{=\sum
_{p=1}^{i-1}c_{p}A=\sum_{p=1}^{i-1}b_{p}A}+b_{i}A=\sum_{p=1}^{i-1}b_{p}%
A+b_{i}A=\sum_{p=1}^{i}b_{p}A,
\end{align*}
we obtain $\sum_{p=1}^{i}b_{p}A=\sum_{p=1}^{i}c_{p}A$. In other words,
(\ref{pf.lem.submodules.same-sum.main}) holds for $j=i$. This completes the
induction step. Thus, (\ref{pf.lem.submodules.same-sum.main}) is proven by induction.]

Now, (\ref{pf.lem.submodules.same-sum.main}) (applied to $j=k$) yields%
\[
\sum_{p=1}^{k}b_{p}A=\sum_{p=1}^{k}c_{p}A.
\]
Thus,
\[
b_{1}A+b_{2}A+\cdots+b_{k}A=\sum_{p=1}^{k}b_{p}A=\sum_{p=1}^{k}c_{p}%
A=c_{1}A+c_{2}A+\cdots+c_{k}A.
\]
This proves Lemma \ref{lem.submodules.same-sum}.
\end{proof}

\begin{proof}
[Proof of Proposition \ref{prop.J.grobas} (sketched).]For each $i\in\left\{
1,2,\ldots,k\right\}  $, we define a polynomial $b_{i}\in\mathcal{P}$ by%
\[
b_{i}=h_{n-k+i}\left(  x_{i},x_{i+1},\ldots,x_{k}\right)  -\sum_{t=0}%
^{i-1}\left(  -1\right)  ^{t}e_{t}\left(  x_{1},x_{2},\ldots,x_{i-1}\right)
a_{i-t}.
\]
Then, we must prove that the family $\left(  b_{i}\right)  _{i\in\left\{
1,2,\ldots,k\right\}  }$ is a Gr\"{o}bner basis of the ideal $J$. We shall
first prove that this family generates $J$.

For each $i\in\left\{  1,2,\ldots,k\right\}  $, we define $c_{i}\in
\mathcal{P}$ by $c_{i}=h_{n-k+i}-a_{i}$. Then, $J$ is the ideal of
$\mathcal{P}$ generated by the $k$ elements $c_{1},c_{2},\ldots,c_{k}$ (by the
definition of $J$). In other words,%
\begin{equation}
J=c_{1}\mathcal{P}+c_{2}\mathcal{P}+\cdots+c_{k}\mathcal{P}.
\label{pf.prop.J.grobas.J=1}%
\end{equation}

For each $i\in\left\{  1,2,\ldots,k\right\}  $, we have%
\begin{align*}
b_{i}  &  =\underbrace{h_{n-k+i}\left(  x_{i},x_{i+1},\ldots,x_{k}\right)
}_{\substack{=\sum_{t=0}^{i-1}\left(  -1\right)  ^{t}e_{t}\left(  x_{1}%
,x_{2},\ldots,x_{i-1}\right)  h_{n-k+i-t}\left(  x_{1},x_{2},\ldots
,x_{k}\right)  \\\text{(by Lemma \ref{lem.heh-id} (applied to }%
p=n-k+i\text{))}}}-\sum_{t=0}^{i-1}\left(  -1\right)  ^{t}e_{t}\left(
x_{1},x_{2},\ldots,x_{i-1}\right)  a_{i-t}\\
&  =\sum_{t=0}^{i-1}\left(  -1\right)  ^{t}e_{t}\left(  x_{1},x_{2}%
,\ldots,x_{i-1}\right)  h_{n-k+i-t}\left(  x_{1},x_{2},\ldots,x_{k}\right) \\
&  \ \ \ \ \ \ \ \ \ \ -\sum_{t=0}^{i-1}\left(  -1\right)  ^{t}e_{t}\left(
x_{1},x_{2},\ldots,x_{i-1}\right)  a_{i-t}\\
&  =\sum_{t=0}^{i-1}\left(  -1\right)  ^{t}e_{t}\left(  x_{1},x_{2}%
,\ldots,x_{i-1}\right)  \left(  \underbrace{h_{n-k+i-t}\left(  x_{1}%
,x_{2},\ldots,x_{k}\right)  }_{=h_{n-k+i-t}}-a_{i-t}\right) \\
&  =\sum_{t=0}^{i-1}\left(  -1\right)  ^{t}e_{t}\left(  x_{1},x_{2}%
,\ldots,x_{i-1}\right)  \underbrace{\left(  h_{n-k+i-t}-a_{i-t}\right)
}_{\substack{=c_{i-t}\\\text{(by the definition of }c_{i-t}\text{)}}}\\
&  =\sum_{t=0}^{i-1}\left(  -1\right)  ^{t}e_{t}\left(  x_{1},x_{2}%
,\ldots,x_{i-1}\right)  c_{i-t}\\
&  =\underbrace{\left(  -1\right)  ^{0}}_{=1}\underbrace{e_{0}\left(
x_{1},x_{2},\ldots,x_{0-1}\right)  }_{=1}\underbrace{c_{i-0}}_{=c_{i}}%
+\sum_{t=1}^{i-1}\underbrace{\left(  -1\right)  ^{t}e_{t}\left(  x_{1}%
,x_{2},\ldots,x_{i-1}\right)  }_{\in\mathcal{P}}c_{i-t}\\
&  \in c_{i}+\sum_{t=1}^{i-1}\underbrace{\mathcal{P}c_{i-t}}_{=c_{i-t}%
\mathcal{P}}=c_{i}+\sum_{t=1}^{i-1}c_{i-t}\mathcal{P}.
\end{align*}
Hence, Lemma \ref{lem.submodules.same-sum} (applied to $A=\mathcal{P}$) yields
that $b_{1}\mathcal{P}+b_{2}\mathcal{P}+\cdots+b_{k}\mathcal{P}=c_{1}%
\mathcal{P}+c_{2}\mathcal{P}+\cdots+c_{k}\mathcal{P}$ (as ideals of
$\mathcal{P}$). Comparing this with (\ref{pf.prop.J.grobas.J=1}), we obtain
$J=b_{1}\mathcal{P}+b_{2}\mathcal{P}+\cdots+b_{k}\mathcal{P}$. Thus, the
family $\left(  b_{i}\right)  _{i\in\left\{  1,2,\ldots,k\right\}  }$
generates the ideal $J$. Furthermore, for each $i\in\left\{  1,2,\ldots
,k\right\}  $, the $i$-th element%
\[
b_{i}=h_{n-k+i}\left(  x_{i},x_{i+1},\ldots,x_{k}\right)  -\sum_{t=0}%
^{i-1}\left(  -1\right)  ^{t}e_{t}\left(  x_{1},x_{2},\ldots,x_{i-1}\right)
a_{i-t}%
\]
of this family has leading term $x_{i}^{n-k+i}$ (because the polynomial
\newline$\sum_{t=0}^{i-1}\left(  -1\right)  ^{t}e_{t}\left(  x_{1}%
,x_{2},\ldots,x_{i-1}\right)  a_{i-t}$ has degree $<n-k+i$%
\ \ \ \ \footnote{\textit{Proof.} It clearly suffices to show that for each
$t\in\left\{  0,1,\ldots,i-1\right\}  $, the polynomial $e_{t}\left(
x_{1},x_{2},\ldots,x_{i-1}\right)  a_{i-t}$ has degree $<n-k+i$.
\par
So let us do this. Let $t\in\left\{  0,1,\ldots,i-1\right\}  $. Then, the
polynomial $a_{i-t}$ has degree $<n-k+\left(  i-t\right)  $ (by the definition
of $a_{1},a_{2},\ldots,a_{k}$). In other words, $\deg\left(  a_{i-t}\right)
<n-k+\left(  i-t\right)  $. Hence, the polynomial $e_{t}\left(  x_{1}%
,x_{2},\ldots,x_{i-1}\right)  a_{i-t}$ has degree%
\begin{align*}
\deg\left(  e_{t}\left(  x_{1},x_{2},\ldots,x_{i-1}\right)  a_{i-t}\right)
&  =\underbrace{\deg\left(  e_{t}\left(  x_{1},x_{2},\ldots,x_{i-1}\right)
\right)  }_{\leq t}+\underbrace{\deg\left(  a_{i-t}\right)  }_{<n-k+\left(
i-t\right)  }\\
&  <t+\left(  n-k+\left(  i-t\right)  \right)  =n-k+i.
\end{align*}
In other words, the polynomial $e_{t}\left(  x_{1},x_{2},\ldots,x_{i-1}%
\right)  a_{i-t}$ has degree $<n-k+i$. Qed.}, whereas the polynomial
$h_{n-k+i}\left(  x_{i},x_{i+1},\ldots,x_{k}\right)  $ is homogeneous of
degree $n-k+i$ with leading term $x_{i}^{n-k+i}$\ \ \ \ \footnote{Indeed,
every term of the polynomial $h_{n-k+i}\left(  x_{i},x_{i+1},\ldots
,x_{k}\right)  $ has the form $x_{i}^{u_{i}}x_{i+1}^{u_{i+1}}\cdots
x_{k}^{u_{k}}$ for some nonnegative integers $u_{i},u_{i+1},\ldots,u_{k}%
\in\mathbb{N}$ satisfying $u_{i}+u_{i+1}+\cdots+u_{k}=n-k+i$. Among these
terms, clearly the largest one is $x_{i}^{n-k+i}$.}). Thus, the leading terms
of the $k$ elements of this family are disjoint (in the sense that no two of
these leading terms have any indeterminates in common). Thus, clearly,
Buchberger's first criterion (see, e.g., \cite[detailed version, Proposition
3.9]{Grinbe17}) shows that this family is a Gr\"{o}bner basis.
\end{proof}

\begin{proof}
[Proof of Theorem \ref{thm.P/J} (sketched).]This follows using the
Macaulay-Buchberger basis theorem (e.g., \cite[detailed version, Proposition
3.10]{Grinbe17}) from Proposition \ref{prop.J.grobas}. (Indeed, if we let $G$
be the Gr\"{o}bner basis of $J$ constructed in Proposition \ref{prop.J.grobas}%
, then the monomials $\overline{x^{\alpha}}$ for all $\alpha\in\mathbb{N}^{k}$
satisfying $\left(  \alpha_{i}<n-k+i\text{ for each }i\right)  $ are precisely
the $G$-reduced monomials\footnote{because the $i$-th entry of the Gr\"{o}bner
basis $G$ has head term $x_{i}^{n-k+i}$}.)
\end{proof}

\section{\label{sect.S/J}Proof of Theorem \ref{thm.S/J}}

Next, we shall prove Theorem \ref{thm.S/J}.

\begin{convention}
For the rest of Section \ref{sect.S/J}, \textbf{we assume that }$a_{1}%
,a_{2},\ldots,a_{k}$ \textbf{belong to }$\mathcal{S}$\textbf{.}
\end{convention}

Thus, $a_{1},a_{2},\ldots,a_{k}$ are symmetric polynomials. Moreover, recall
that for each $i\in\left\{  1,2,\ldots,k\right\}  $, the polynomial $a_{i}$
has degree $<n-k+i$. In other words, for each $i\in\left\{  1,2,\ldots
,k\right\}  $, we have%
\begin{equation}
\deg\left(  a_{i}\right)  <n-k+i. \label{eq.degai.1}%
\end{equation}
Substituting $i-n+k$ for $i$ in this statement, we obtain the following: For
each $i\in\left\{  n-k+1,n-k+2,\ldots,n\right\}  $, we have%
\begin{equation}
\deg\left(  a_{n-k+i}\right)  <n-k+\left(  i-n+k\right)  =i.
\label{eq.degan-k+i.1}%
\end{equation}

Let $I$ be the ideal of $\mathcal{S}$ generated by the $k$ differences
(\ref{eq.gensJ}). Hence, these differences belong to $I$. Thus,%
\begin{equation}
h_{n-k+j}\equiv a_{j}\operatorname{mod}I\ \ \ \ \ \ \ \ \ \ \text{for each
}j\in\left\{  1,2,\ldots,k\right\}  . \label{eq.h=amodI}%
\end{equation}
Renaming the index $j$ as $i-n+k$ in this statement, we obtain%
\begin{equation}
h_{i}\equiv a_{i-n+k}\operatorname{mod}I\ \ \ \ \ \ \ \ \ \ \text{for each
}i\in\left\{  n-k+1,n-k+2,\ldots,n\right\}  . \label{eq.h=amodI2}%
\end{equation}

\begin{lemma}
\label{lem.basis-over-basis}Let $A$ be a commutative $\mathbf{k}$-algebra. Let
$B$ be a commutative $A$-algebra. Assume that the $A$-module $B$ is spanned by
the family $\left(  b_{u}\right)  _{u\in U}\in B^{U}$. Let $\mathcal{I}$ be an
ideal of $A$. Let $\left(  a_{v}\right)  _{v\in V}\in A^{V}$ be a family of
elements of $A$ such that the $\mathbf{k}$-module $A/\mathcal{I}$ is spanned
by the family $\left(  \overline{a_{v}}\right)  _{v\in V}\in\left(
A/\mathcal{I}\right)  ^{V}$. Then, the $\mathbf{k}$-module $B/\left(
\mathcal{I}B\right)  $ is spanned by the family $\left(  \overline{a_{v}b_{u}%
}\right)  _{\left(  u,v\right)  \in U\times V}\in\left(  B/\left(
\mathcal{I}B\right)  \right)  ^{U\times V}$.
\end{lemma}

\begin{proof}
[Proof of Lemma \ref{lem.basis-over-basis}.]Easy. Here is the proof under the
assumption that the set $U$ is finite\footnote{The case when $U$ is infinite
needs only minor modifications. But we shall only use the case when $U$ is
finite.}:

Let $x\in B/\left(  \mathcal{I}B\right)  $. Thus, $x=\overline{b}$ for some
$b\in B$. Consider this $b$. Recall that the $A$-module $B$ is spanned by the
family $\left(  b_{u}\right)  _{u\in U}$. Hence, $b=\sum_{u\in U}p_{u}b_{u}$
for some family $\left(  p_{u}\right)  _{u\in U}\in A^{U}$ of elements of $A$.
Consider this family $\left(  p_{u}\right)  _{u\in U}$.

Recall that the $\mathbf{k}$-module $A/\mathcal{I}$ is spanned by the family
$\left(  \overline{a_{v}}\right)  _{v\in V}\in\left(  A/\mathcal{I}\right)
^{V}$. Thus, for each $u\in U$, there exists a family $\left(  q_{u,v}\right)
_{v\in V}\in\mathbf{k}^{V}$ of elements of $\mathbf{k}$ such that
$\overline{p_{u}}=\sum_{v\in V}q_{u,v}\overline{a_{v}}$ (and such that all but
finitely many $v\in V$ satisfy $q_{u,v}=0$). Consider this family $\left(
q_{u,v}\right)  _{v\in V}$.

Now, recall that $B/\left(  \mathcal{I}B\right)  $ is an $A/\mathcal{I}%
$-module (since $B$ is an $A$-module, but each $i\in\mathcal{I}$ clearly acts
as $0$ on $B/\left(  \mathcal{I}B\right)  $). Now,%
\begin{align*}
x  &  =\overline{b}=\overline{\sum_{u\in U}p_{u}b_{u}}%
\ \ \ \ \ \ \ \ \ \ \left(  \text{since }b=\sum_{u\in U}p_{u}b_{u}\right) \\
&  =\sum_{u\in U}\underbrace{\overline{p_{u}}}_{=\sum_{v\in V}q_{u,v}%
\overline{a_{v}}}\overline{b_{u}}=\underbrace{\sum_{u\in U}\sum_{v\in V}%
}_{=\sum_{\left(  u,v\right)  \in U\times V}}q_{u,v}\underbrace{\overline
{a_{v}}\overline{b_{u}}}_{=\overline{a_{v}b_{u}}}=\sum_{\left(  u,v\right)
\in U\times V}q_{u,v}\overline{a_{v}b_{u}}.
\end{align*}
Thus, $x$ belongs to the $\mathbf{k}$-submodule of $B/\left(  \mathcal{I}%
B\right)  $ spanned by the family \newline$\left(  \overline{a_{v}b_{u}%
}\right)  _{\left(  u,v\right)  \in U\times V}$. Since we have proven this for
all $x\in B/\left(  \mathcal{I}B\right)  $, we thus conclude that the
$\mathbf{k}$-module $B/\left(  \mathcal{I}B\right)  $ is spanned by the family
$\left(  \overline{a_{v}b_{u}}\right)  _{\left(  u,v\right)  \in U\times V}%
\in\left(  B/\left(  \mathcal{I}B\right)  \right)  ^{U\times V}$. This proves
Lemma \ref{lem.basis-over-basis}.
\end{proof}

\begin{lemma}
\label{lem.freemod-span-basis}Let $M$ be a free $\mathbf{k}$-module with a
finite basis $\left(  b_{s}\right)  _{s\in S}$. Let $\left(  a_{u}\right)
_{u\in U}\in M^{U}$ be a family that spans $M$. Assume that $\left\vert
U\right\vert =\left\vert S\right\vert $. Then, $\left(  a_{u}\right)  _{u\in
U}$ is a basis of the $\mathbf{k}$-module $M$. (In other words: A spanning
family of $M$ whose size equals the size of a basis must itself be a basis, as
long as the sizes are finite.)
\end{lemma}

\begin{proof}
[Proof of Lemma \ref{lem.freemod-span-basis}.]Well-known (see, e.g.,
\cite[Exercise 2.5.18 \textbf{(b)}]{GriRei18}).
\end{proof}

\begin{lemma}
\label{lem.I.hi-red}Let $i$ be an integer such that $i>n-k$. Then,%
\[
h_{i}\equiv\left(  \text{some symmetric polynomial of degree }<i\right)
\operatorname{mod}I.
\]

\end{lemma}

\begin{proof}
[Proof of Lemma \ref{lem.I.hi-red} (sketched).]We shall prove Lemma
\ref{lem.I.hi-red} by strong induction on $i$. Thus, we assume (as the
induction hypothesis) that%
\begin{equation}
h_{j}\equiv\left(  \text{some symmetric polynomial of degree }<j\right)
\operatorname{mod}I \label{pf.lem.I.hi-red.IH}%
\end{equation}
for every $j\in\left\{  n-k+1,n-k+2,\ldots,i-1\right\}  $.

If $i\leq n$, then (\ref{eq.h=amodI2}) yields $h_{i}\equiv a_{i-n+k}%
\operatorname{mod}I$ (since $i\in\left\{  n-k+1,n-k+2,\ldots,n\right\}  $),
which clearly proves Lemma \ref{lem.I.hi-red} (since $a_{i-n+k}$ is a
symmetric polynomial of degree $<i$\ \ \ \ \footnote{by (\ref{eq.degan-k+i.1}%
)}). Thus, for the rest of this proof, we WLOG assume that $i>n$. Hence, each
$t\in\left\{  1,2,\ldots,k\right\}  $ satisfies \newline$i-t\in\left\{
n-k+1,n-k+2,\ldots,i-1\right\}  $ (since $\underbrace{i}_{>n}-\underbrace{t}%
_{\leq k}>n-k$ and $i-\underbrace{t}_{\geq1}\leq i-1$) and therefore%
\begin{equation}
h_{i-t}\equiv\left(  \text{some symmetric polynomial of degree }<i-t\right)
\operatorname{mod}I \label{pf.lem.I.hi-red.2}%
\end{equation}
(by (\ref{pf.lem.I.hi-red.IH}), applied to $j=i-t$).

But $i$ is a positive integer (since $i>n\geq0$). Hence, Corollary
\ref{cor.heh-id.0} (applied to $p=i$) yields%
\begin{align*}
h_{i}  &  =-\sum_{t=1}^{k}\left(  -1\right)  ^{t}e_{t}\underbrace{h_{i-t}%
}_{\substack{\equiv\left(  \text{some symmetric polynomial of degree
}<i-t\right)  \operatorname{mod}I\\\text{(by (\ref{pf.lem.I.hi-red.2}))}}}\\
&  \equiv-\sum_{t=1}^{k}\left(  -1\right)  ^{t}e_{t}\cdot\left(  \text{some
symmetric polynomial of degree }<i-t\right) \\
&  =\left(  \text{some symmetric polynomial of degree }<i\right)
\operatorname{mod}I.
\end{align*}
This completes the induction step. Thus, Lemma \ref{lem.I.hi-red} is proven.
\end{proof}

\begin{definition}
The \textit{size} of a partition $\lambda=\left(  \lambda_{1},\lambda
_{2},\lambda_{3},\ldots\right)  $ is defined as $\lambda_{1}+\lambda
_{2}+\lambda_{3}+\cdots$, and is denoted by $\left\vert \lambda\right\vert $.
\end{definition}

\begin{definition}
\label{def.Pk}Let $P_{k}$ denote the set of all partitions with at most $k$
parts. Thus, the elements of $P_{k}$ are weakly decreasing $k$-tuples of
nonnegative integers.
\end{definition}

\begin{proposition}
\label{prop.jacobi-trudi.Sh}Let $\lambda=\left(  \lambda_{1},\lambda
_{2},\ldots,\lambda_{k}\right)  $ be a partition in $P_{k}$. Then:

\textbf{(a)} We have%
\[
s_{\lambda}=\det\left(  \left(  h_{\lambda_{u}-u+v}\right)  _{1\leq u\leq
k,\ 1\leq v\leq k}\right)  .
\]

\textbf{(b)} Let $p\in\left\{  0,1,\ldots,k\right\}  $ be such that
$\lambda=\left(  \lambda_{1},\lambda_{2},\ldots,\lambda_{p}\right)  $. Then,%
\[
s_{\lambda}=\det\left(  \left(  h_{\lambda_{u}-u+v}\right)  _{1\leq u\leq
p,\ 1\leq v\leq p}\right)  .
\]

\end{proposition}

\begin{proof}
[Proof of Proposition \ref{prop.jacobi-trudi.Sh}.]\textbf{(b)} Proposition
\ref{prop.jacobi-trudi.Sh} \textbf{(b)} is the well-known Jacobi-Trudi
identity, and is proven in various places. (For instance, \cite[(2.4.16)]%
{GriRei18} states a similar formula for skew Schur functions; if we set
$\mu=\varnothing$ in it and apply both sides to the variables $x_{1}%
,x_{2},\ldots,x_{k}$, then we recover the claim of Proposition
\ref{prop.jacobi-trudi.Sh} \textbf{(b)}.)

\textbf{(a)} We have $\lambda=\left(  \lambda_{1},\lambda_{2},\ldots
,\lambda_{k}\right)  $. Hence, Proposition \ref{prop.jacobi-trudi.Sh}
\textbf{(a)} is the particular case of Proposition \ref{prop.jacobi-trudi.Sh}
\textbf{(b)} for $p=k$.
\end{proof}

\begin{lemma}
\label{lem.lam-sum-deg-1}Let $\lambda=\left(  \lambda_{1},\lambda_{2}%
,\ldots,\lambda_{\ell}\right)  $ be any partition. Let $i\in\left\{
1,2,\ldots,\ell\right\}  $ and $j\in\left\{  1,2,\ldots,\ell\right\}  $. Then,%
\[
\sum_{\substack{u\in\left\{  1,2,\ldots,\ell\right\}  ;\\u\neq i}}\left(
\lambda_{u}-u\right)  +\sum_{\substack{u\in\left\{  1,2,\ldots,\ell\right\}
;\\u\neq j}}u=\left\vert \lambda\right\vert -\left(  \lambda_{i}-i+j\right)
.
\]

\end{lemma}

\begin{proof}
[Proof of Lemma \ref{lem.lam-sum-deg-1}.]We have%
\begin{align*}
&  \underbrace{\sum_{\substack{u\in\left\{  1,2,\ldots,\ell\right\}  ;\\u\neq
i}}\left(  \lambda_{u}-u\right)  }_{=\sum_{u\in\left\{  1,2,\ldots
,\ell\right\}  }\left(  \lambda_{u}-u\right)  -\left(  \lambda_{i}-i\right)
}+\underbrace{\sum_{\substack{u\in\left\{  1,2,\ldots,\ell\right\}  ;\\u\neq
j}}u}_{=\sum_{u\in\left\{  1,2,\ldots,\ell\right\}  }u-j}\\
&  =\underbrace{\sum_{u\in\left\{  1,2,\ldots,\ell\right\}  }\left(
\lambda_{u}-u\right)  }_{=\sum_{u\in\left\{  1,2,\ldots,\ell\right\}  }%
\lambda_{u}-\sum_{u\in\left\{  1,2,\ldots,\ell\right\}  }u}-\left(
\lambda_{i}-i\right)  +\sum_{u\in\left\{  1,2,\ldots,\ell\right\}  }u-j\\
&  =\sum_{u\in\left\{  1,2,\ldots,\ell\right\}  }\lambda_{u}-\sum
_{u\in\left\{  1,2,\ldots,\ell\right\}  }u-\left(  \lambda_{i}-i\right)
+\sum_{u\in\left\{  1,2,\ldots,\ell\right\}  }u-j\\
&  =\underbrace{\sum_{u\in\left\{  1,2,\ldots,\ell\right\}  }\lambda_{u}%
}_{=\left\vert \lambda\right\vert }-\left(  \lambda_{i}-i\right)
-j=\left\vert \lambda\right\vert -\left(  \lambda_{i}-i\right)  -j=\left\vert
\lambda\right\vert -\left(  \lambda_{i}-i+j\right)  .
\end{align*}
This proves Lemma \ref{lem.lam-sum-deg-1}.
\end{proof}

Next, let us recall the definition of a cofactor of a matrix:

\begin{definition}
\label{def.det.cofactor}Let $\ell\in\mathbb{N}$. Let $R$ be a commutative
ring. Let $A\in R^{\ell\times\ell}$ be any $\ell\times\ell$-matrix. Let
$i\in\left\{  1,2,\ldots,\ell\right\}  $ and $j\in\left\{  1,2,\ldots
,\ell\right\}  $. Then:

\textbf{(a)} The $\left(  i,j\right)  $\textit{-th minor} of the matrix $A$ is
defined to be the determinant of the $\left(  \ell-1\right)  \times\left(
\ell-1\right)  $-matrix obtained from $A$ by removing the $i$-th row and the
$j$-th column.

\textbf{(b)} The $\left(  i,j\right)  $\textit{-th cofactor} of the matrix $A$
is defined to $\left(  -1\right)  ^{i+j}$ times the $\left(  i,j\right)  $-th
minor of $A$.
\end{definition}

It is known that any $\ell\times\ell$-matrix $A=\left(  a_{i,j}\right)
_{1\leq i\leq\ell,\ 1\leq j\leq\ell}$ over a commutative ring $R$ satisfies%
\begin{equation}
\det A=\sum_{j=1}^{\ell}a_{i,j}\cdot\left(  \text{the }\left(  i,j\right)
\text{-th cofactor of }A\right)  \label{eq.det.lap-exp}%
\end{equation}
for each $i\in\left\{  1,2,\ldots,\ell\right\}  $. (This is the Laplace
expansion of the determinant of $A$ along its $i$-th row.)

\begin{lemma}
\label{lem.I.cofactor1}Let $\lambda=\left(  \lambda_{1},\lambda_{2}%
,\ldots,\lambda_{\ell}\right)  $ be any partition. Let $i\in\left\{
1,2,\ldots,\ell\right\}  $ and $j\in\left\{  1,2,\ldots,\ell\right\}  $. Then,
the $\left(  i,j\right)  $-th cofactor of the matrix $\left(  h_{\lambda
_{u}-u+v}\right)  _{1\leq u\leq\ell,\ 1\leq v\leq\ell}$ is a homogeneous
symmetric polynomial of degree $\left\vert \lambda\right\vert -\left(
\lambda_{i}-i+j\right)  $.
\end{lemma}

\begin{proof}
[Proof of Lemma \ref{lem.I.cofactor1} (sketched).]This is a simple argument
that inflates in length by a multiple when put on paper. You will probably
have arrived at the proof long before you have finished reading the following.

For each $u\in\left\{  1,2,\ldots,\ell\right\}  $ and $v\in\left\{
1,2,\ldots,\ell\right\}  $, we define an integer $w\left(  u,v\right)  $ by%
\begin{equation}
w\left(  u,v\right)  =\lambda_{u}-u+v. \label{pf.lem.I.cofactor1.wuv=}%
\end{equation}

\begin{noncompile}
For each $m\in\mathbb{Z}$, we let $\mathcal{S}_{\deg=m}$ denote the $m$-th
graded component of the graded ring $\mathcal{S}$. This is the $\mathbf{k}%
$-submodule of $\mathcal{S}$ consisting of all homogeneous symmetric
polynomials of degree $m$. Hence,%
\begin{equation}
h_{m}\in\mathcal{S}_{\deg=m}\ \ \ \ \ \ \ \ \ \ \text{for each }m\in
\mathbb{Z}. \label{pf.lem.I.cofactor1.hm-hg}%
\end{equation}

Let $A$ be the matrix $\left(  h_{\lambda_{u}-u+v}\right)  _{1\leq u\leq
\ell,\ 1\leq v\leq\ell}$. Write this matrix $A$ in the form $A=\left(
a_{u,v}\right)  _{1\leq u\leq\ell,\ 1\leq v\leq\ell}$. Hence,%
\[
\left(  a_{u,v}\right)  _{1\leq u\leq\ell,\ 1\leq v\leq\ell}=A=\left(
h_{\lambda_{u}-u+v}\right)  _{1\leq u\leq\ell,\ 1\leq v\leq\ell}.
\]
Thus, for each $u\in\left\{  1,2,\ldots,\ell\right\}  $ and $v\in\left\{
1,2,\ldots,\ell\right\}  $, we have%
\begin{align}
a_{u,v}  &  =h_{\lambda_{u}-u+v}=h_{w\left(  u,v\right)  }%
\ \ \ \ \ \ \ \ \ \ \left(  \text{since }\lambda_{u}-u+v=w\left(  u,v\right)
\text{ (by (\ref{pf.lem.I.cofactor1.wuv=}))}\right) \nonumber\\
&  \in\Lambda_{\deg=w\left(  u,v\right)  } \label{pf.lem.I.cofactor1.auv-hg}%
\end{align}
(by (\ref{pf.lem.I.cofactor1.hm-hg}), applied to $m=w\left(  u,v\right)  $).
\end{noncompile}

Let $A$ be the matrix $\left(  h_{w\left(  u,v\right)  }\right)  _{1\leq
u\leq\ell,\ 1\leq v\leq\ell}$. Let $\mu$ be the $\left(  i,j\right)  $-th
minor of the matrix $A$. Thus, $\mu$ is the determinant of the $\left(
\ell-1\right)  \times\left(  \ell-1\right)  $-matrix obtained from $A$ by
removing the $i$-th row and the $j$-th column (by Definition
\ref{def.det.cofactor} \textbf{(a)}). The combinatorial definition of a
determinant (i.e., the definition of a determinant as a sum over all
permutations) thus shows that $\mu$ is a sum of $\left(  \ell-1\right)  !$
many products of the form%
\[
\pm h_{w\left(  i_{1},j_{1}\right)  }h_{w\left(  i_{2},j_{2}\right)  }\cdots
h_{w\left(  i_{\ell-1},j_{\ell-1}\right)  },
\]
where $i_{1},i_{2},\ldots,i_{\ell-1}$ are $\ell-1$ distinct elements of the
set $\left\{  1,2,\ldots,\ell\right\}  \setminus\left\{  i\right\}  $ and
where $j_{1},j_{2},\ldots,j_{\ell-1}$ are $\ell-1$ distinct elements of the
set $\left\{  1,2,\ldots,\ell\right\}  \setminus\left\{  j\right\}  $. Let us
refer to such products as \textit{diagonal products}. Hence, $\mu$ is a sum of
diagonal products.

We shall now claim the following:

\begin{statement}
\textit{Claim 1:} Each diagonal product is a homogeneous symmetric polynomial
of degree $\left\vert \lambda\right\vert -\left(  \lambda_{i}-i+j\right)  $.
\end{statement}

[\textit{Proof of Claim 1:} Let $d$ be a diagonal product. We must show that
$d$ is a homogeneous symmetric polynomial of degree $\left\vert \lambda
\right\vert -\left(  \lambda_{i}-i+j\right)  $.

We have assumed that $d$ is a diagonal product. In other words, $d$ is a
product of the form
\[
\pm h_{w\left(  i_{1},j_{1}\right)  }h_{w\left(  i_{2},j_{2}\right)  }\cdots
h_{w\left(  i_{\ell-1},j_{\ell-1}\right)  },
\]
where $i_{1},i_{2},\ldots,i_{\ell-1}$ are $\ell-1$ distinct elements of the
set $\left\{  1,2,\ldots,\ell\right\}  \setminus\left\{  i\right\}  $ and
where $j_{1},j_{2},\ldots,j_{\ell-1}$ are $\ell-1$ distinct elements of the
set $\left\{  1,2,\ldots,\ell\right\}  \setminus\left\{  j\right\}  $.
Consider these $i_{1},i_{2},\ldots,i_{\ell-1}$ and these $j_{1},j_{2}%
,\ldots,j_{\ell-1}$.

The numbers $i_{1},i_{2},\ldots,i_{\ell-1}$ are $\ell-1$ distinct elements of
the set $\left\{  1,2,\ldots,\ell\right\}  \setminus\left\{  i\right\}  $; but
the latter set has only $\ell-1$ elements altogether. Thus, these numbers
$i_{1},i_{2},\ldots,i_{\ell-1}$ must be precisely the $\ell-1$ elements of the
set $\left\{  1,2,\ldots,\ell\right\}  \setminus\left\{  i\right\}  $ in some
order. Similarly, the numbers $j_{1},j_{2},\ldots,j_{\ell-1}$ must be
precisely the $\ell-1$ elements of the set $\left\{  1,2,\ldots,\ell\right\}
\setminus\left\{  j\right\}  $ in some order.

For each $p\in\left\{  1,2,\ldots,\ell-1\right\}  $, the element $h_{w\left(
i_{p},j_{p}\right)  }$ of $\mathcal{S}$ is homogeneous of degree $w\left(
i_{p},j_{p}\right)  $ (because for each $m\in\mathbb{Z}$, the element $h_{m}$
of $\mathcal{S}$ is homogeneous of degree $m$). Hence, the product
$h_{w\left(  i_{1},j_{1}\right)  }h_{w\left(  i_{2},j_{2}\right)  }\cdots
h_{w\left(  i_{\ell-1},j_{\ell-1}\right)  }$ is homogeneous of degree
\begin{align*}
&  w\left(  i_{1},j_{1}\right)  +w\left(  i_{2},j_{2}\right)  +\cdots+w\left(
i_{\ell-1},j_{\ell-1}\right) \\
&  =\sum_{p\in\left\{  1,2,\ldots,\ell-1\right\}  }\underbrace{w\left(
i_{p},j_{p}\right)  }_{\substack{=\lambda_{i_{p}}-i_{p}+j_{p}\\\text{(by the
definition of }w\left(  i_{p},j_{p}\right)  \text{)}}}=\sum_{p\in\left\{
1,2,\ldots,\ell-1\right\}  }\left(  \lambda_{i_{p}}-i_{p}+j_{p}\right) \\
&  =\underbrace{\sum_{p\in\left\{  1,2,\ldots,\ell-1\right\}  }\left(
\lambda_{i_{p}}-i_{p}\right)  }_{\substack{=\left(  \lambda_{i_{1}}%
-i_{1}\right)  +\left(  \lambda_{i_{2}}-i_{2}\right)  +\cdots+\left(
\lambda_{i_{\ell-1}}-i_{\ell-1}\right)  \\=\sum_{u\in\left\{  1,2,\ldots
,\ell\right\}  \setminus\left\{  i\right\}  }\left(  \lambda_{u}-u\right)
\\\text{(since }i_{1},i_{2},\ldots,i_{\ell-1}\text{ are precisely}\\\text{the
}\ell-1\text{ elements of the set }\left\{  1,2,\ldots,\ell\right\}
\setminus\left\{  i\right\}  \\\text{in some order)}}}+\underbrace{\sum
_{p\in\left\{  1,2,\ldots,\ell-1\right\}  }j_{p}}_{\substack{=j_{1}%
+j_{2}+\cdots+j_{\ell-1}\\=\sum_{u\in\left\{  1,2,\ldots,\ell\right\}
\setminus\left\{  j\right\}  }u\\\text{(since }j_{1},j_{2},\ldots,j_{\ell
-1}\text{ are precisely}\\\text{the }\ell-1\text{ elements of the set
}\left\{  1,2,\ldots,\ell\right\}  \setminus\left\{  j\right\}  \\\text{in
some order)}}}\\
&  =\underbrace{\sum_{u\in\left\{  1,2,\ldots,\ell\right\}  \setminus\left\{
i\right\}  }}_{=\sum_{\substack{u\in\left\{  1,2,\ldots,\ell\right\}  ;\\u\neq
i}}}\left(  \lambda_{u}-u\right)  +\underbrace{\sum_{u\in\left\{
1,2,\ldots,\ell\right\}  \setminus\left\{  j\right\}  }}_{=\sum
_{\substack{u\in\left\{  1,2,\ldots,\ell\right\}  ;\\u\neq j}}}u\\
&  =\sum_{\substack{u\in\left\{  1,2,\ldots,\ell\right\}  ;\\u\neq i}}\left(
\lambda_{u}-u\right)  +\sum_{\substack{u\in\left\{  1,2,\ldots,\ell\right\}
;\\u\neq j}}u=\left\vert \lambda\right\vert -\left(  \lambda_{i}-i+j\right)
\end{align*}
(by Lemma \ref{lem.lam-sum-deg-1}). Thus, $d$ is homogeneous of degree
$\left\vert \lambda\right\vert -\left(  \lambda_{i}-i+j\right)  $ as well
(since $d=\pm h_{w\left(  i_{1},j_{1}\right)  }h_{w\left(  i_{2},j_{2}\right)
}\cdots h_{w\left(  i_{\ell-1},j_{\ell-1}\right)  }$). Hence, $d$ is a
homogeneous symmetric polynomial of degree $\left\vert \lambda\right\vert
-\left(  \lambda_{i}-i+j\right)  $ (since $d$ is clearly a symmetric
polynomial). This proves Claim 1.]

Now, $\mu$ is a sum of diagonal products; but each such diagonal product is a
homogeneous symmetric polynomial of degree $\left\vert \lambda\right\vert
-\left(  \lambda_{i}-i+j\right)  $ (by Claim 1). Hence, their sum $\mu$ is
also a homogeneous symmetric polynomial of degree $\left\vert \lambda
\right\vert -\left(  \lambda_{i}-i+j\right)  $.

Recall that $\mu$ is the $\left(  i,j\right)  $-th minor of the matrix $A$.
Hence, the $\left(  i,j\right)  $-th cofactor of the matrix $A$ is $\left(
-1\right)  ^{i+j}\mu$ (by Definition \ref{def.det.cofactor} \textbf{(b)}).
Thus, this cofactor is a homogeneous symmetric polynomial of degree
$\left\vert \lambda\right\vert -\left(  \lambda_{i}-i+j\right)  $ (since $\mu$
is a homogeneous symmetric polynomial of degree $\left\vert \lambda\right\vert
-\left(  \lambda_{i}-i+j\right)  $).

But
\begin{equation}
A=\left(  \underbrace{h_{w\left(  u,v\right)  }}_{\substack{=h_{\lambda
_{u}-u+v}\\\text{(by (\ref{pf.lem.I.cofactor1.wuv=}))}}}\right)  _{1\leq
u\leq\ell,\ 1\leq v\leq\ell}=\left(  h_{\lambda_{u}-u+v}\right)  _{1\leq
u\leq\ell,\ 1\leq v\leq\ell}. \label{pf.lem.I.cofactor1.A=}%
\end{equation}
We have shown that the $\left(  i,j\right)  $-th cofactor of the matrix $A$ is
a homogeneous symmetric polynomial of degree $\left\vert \lambda\right\vert
-\left(  \lambda_{i}-i+j\right)  $. In view of (\ref{pf.lem.I.cofactor1.A=}),
this rewrites as follows: The $\left(  i,j\right)  $-th cofactor of the matrix
$\left(  h_{\lambda_{u}-u+v}\right)  _{1\leq u\leq\ell,\ 1\leq v\leq\ell}$ is
a homogeneous symmetric polynomial of degree $\left\vert \lambda\right\vert
-\left(  \lambda_{i}-i+j\right)  $. This proves Lemma \ref{lem.I.cofactor1}.
\end{proof}

\begin{lemma}
\label{lem.I.sl-red}Let $\lambda\in P_{k}$ be a partition such that
$\lambda\notin P_{k,n}$. Then,%
\[
s_{\lambda}\equiv\left(  \text{some symmetric polynomial of degree
}<\left\vert \lambda\right\vert \right)  \operatorname{mod}I.
\]

\end{lemma}

\begin{proof}
[Proof of Lemma \ref{lem.I.sl-red} (sketched).]Write the partition $\lambda$
as $\lambda=\left(  \lambda_{1},\lambda_{2},\ldots,\lambda_{k}\right)  $.
(This can be done, since $\lambda\in P_{k}$.) Note that $k>0$ (since
otherwise, $\lambda\in P_{k}$ would lead to $\lambda=\varnothing\in P_{k,n}$,
which would contradict $\lambda\notin P_{k,n}$).

From $\lambda\in P_{k}$ and $\lambda\notin P_{k,n}$, we conclude that not all
parts of the partition $\lambda$ are $\leq n-k$. Thus, the first entry
$\lambda_{1}$ of $\lambda$ is $>n-k$ (since $\lambda_{1}\geq\lambda_{2}%
\geq\lambda_{3}\geq\cdots$). But $\lambda=\left(  \lambda_{1},\lambda
_{2},\ldots,\lambda_{k}\right)  $. Thus, Proposition
\ref{prop.jacobi-trudi.Sh} \textbf{(a)} yields%
\begin{equation}
s_{\lambda}=\det\left(  \left(  h_{\lambda_{u}-u+v}\right)  _{1\leq u\leq
k,\ 1\leq v\leq k}\right)  =\sum_{j=1}^{k}h_{\lambda_{1}-1+j}\cdot C_{j},
\label{pf.lem.I.sl-red.JT}%
\end{equation}
where $C_{j}$ denotes the $\left(  1,j\right)  $-th cofactor of the $k\times
k$-matrix $\left(  h_{\lambda_{u}-u+v}\right)  _{1\leq u\leq k,\ 1\leq v\leq
k}$. (Here, the last equality sign follows from (\ref{eq.det.lap-exp}),
applied to $\ell=k$ and $R=\mathcal{S}$ and $A=\left(  h_{\lambda_{u}%
-u+v}\right)  _{1\leq u\leq k,\ 1\leq v\leq k}$ and $a_{u,v}=h_{\lambda
_{u}-u+v}$ and $i=1$.)

For each $j\in\left\{  1,2,\ldots,k\right\}  $, we have $\lambda_{1}%
-1+j\geq\lambda_{1}-1+1=\lambda_{1}>n-k$ and therefore%
\begin{align}
&  h_{\lambda_{1}-1+j}\nonumber\\
&  \equiv\left(  \text{some symmetric polynomial of degree }<\lambda
_{1}-1+j\right)  \operatorname{mod}I \label{pf.lem.I.sl-red.deg1}%
\end{align}
(by Lemma \ref{lem.I.hi-red}, applied to $i=\lambda_{1}-1+j$).

For each $j\in\left\{  1,2,\ldots,k\right\}  $, the polynomial $C_{j}$ is the
$\left(  1,j\right)  $-th cofactor of the matrix $\left(  h_{\lambda_{u}%
-u+v}\right)  _{1\leq u\leq k,\ 1\leq v\leq k}$ (by its definition), and thus
is a homogeneous symmetric polynomial of degree $\left\vert \lambda\right\vert
-\left(  \lambda_{1}-1+j\right)  $ (by Lemma \ref{lem.I.cofactor1}, applied to
$\ell=k$ and $i=1$). Hence,%
\begin{equation}
C_{j}=\left(  \text{some symmetric polynomial of degree }\leq\left\vert
\lambda\right\vert -\left(  \lambda_{1}-1+j\right)  \right)
\label{pf.lem.I.sl-red.deg}%
\end{equation}
for each $j\in\left\{  1,2,\ldots,k\right\}  $.

Therefore, (\ref{pf.lem.I.sl-red.JT}) becomes%
\begin{align*}
s_{\lambda}  &  =\sum_{j=1}^{k}\underbrace{h_{\lambda_{1}-1+j}}%
_{\substack{\equiv\left(  \text{some symmetric polynomial of degree }%
<\lambda_{1}-1+j\right)  \operatorname{mod}I\\\text{(by
(\ref{pf.lem.I.sl-red.deg1}))}}}\\
&  \ \ \ \ \ \ \ \ \ \ \cdot\underbrace{C_{j}}_{\substack{=\left(  \text{some
symmetric polynomial of degree }\left\vert \lambda\right\vert -\left(
\lambda_{1}-1+j\right)  \right)  \\\text{(by (\ref{pf.lem.I.sl-red.deg}))}}}\\
&  \equiv\sum_{j=1}^{k}\left(  \text{some symmetric polynomial of degree
}<\lambda_{1}-1+j\right) \\
&  \ \ \ \ \ \ \ \ \ \ \cdot\left(  \text{some symmetric polynomial of degree
}\left\vert \lambda\right\vert -\left(  \lambda_{1}-1+j\right)  \right) \\
&  =\left(  \text{some symmetric polynomial of degree }<\left\vert
\lambda\right\vert \right)  \operatorname{mod}I.
\end{align*}
This proves Lemma \ref{lem.I.sl-red}.
\end{proof}

Recall Definition \ref{def.Pk}.

\begin{lemma}
\label{lem.I.sm-as-smaller}Let $N\in\mathbb{N}$. Let $f\in\mathcal{S}$ be a
symmetric polynomial of degree $<N$. Then, there exists a family $\left(
c_{\kappa}\right)  _{\kappa\in P_{k};\ \left\vert \kappa\right\vert <N}$ of
elements of $\mathbf{k}$ such that $f=\sum_{\substack{\kappa\in P_{k}%
;\\\left\vert \kappa\right\vert <N}}c_{\kappa}s_{\kappa}$.
\end{lemma}

\begin{proof}
[Proof of Lemma \ref{lem.I.sm-as-smaller}.]For each $d\in\mathbb{N}$, we let
$\mathcal{S}_{\deg=d}$ be the $d$-th graded part of the graded $\mathbf{k}%
$-module $\mathcal{S}$. This is the $\mathbf{k}$-submodule of $\mathcal{S}$
consisting of all homogeneous elements of $\mathcal{S}$ of degree $d$
(including the zero vector $0$, which is homogeneous of every degree).

Recall that the family $\left(  s_{\lambda}\right)  _{\lambda\in P_{k}}$ is a
graded basis of the graded $\mathbf{k}$-module $\mathcal{S}$. In other words,
for each $d\in\mathbb{N}$, the family $\left(  s_{\lambda}\right)
_{\lambda\in P_{k};\ \left\vert \lambda\right\vert =d}$ is a basis of the
$\mathbf{k}$-submodule $\mathcal{S}_{\deg=d}$ of $\mathcal{S}$. Hence, for
each $d\in\mathbb{N}$, we have%
\begin{align}
\mathcal{S}_{\deg=d}  &  =\left(  \text{the }\mathbf{k}\text{-linear span of
the family }\left(  s_{\lambda}\right)  _{\lambda\in P_{k};\ \left\vert
\lambda\right\vert =d}\right) \nonumber\\
&  =\sum_{\substack{\lambda\in P_{k};\\\left\vert \lambda\right\vert
=d}}\mathbf{k}s_{\lambda}. \label{pf.lem.I.sm-as-smaller.2}%
\end{align}

The polynomial $f$ has degree $<N$. Hence, we can write $f$ in the form
$f=\sum_{d=0}^{N-1}f_{d}$ for some $f_{0},f_{1},\ldots,f_{N-1}\in\mathcal{P}$,
where each $f_{d}$ is a homogeneous polynomial of degree $d$. Consider these
$f_{0},f_{1},\ldots,f_{N-1}$. These $N$ polynomials $f_{0},f_{1}%
,\ldots,f_{N-1}$ are the first $N$ homogeneous components of $f$, and thus are
symmetric (since $f$ is symmetric); in other words, $f_{0},f_{1}%
,\ldots,f_{N-1}$ are elements of $\mathcal{S}$. Thus, for each $d\in\left\{
0,1,\ldots,N-1\right\}  $, the polynomial $f_{d}$ is an element of
$\mathcal{S}$ and is homogeneous of degree $d$ (as we already know). In other
words, for each $d\in\left\{  0,1,\ldots,N-1\right\}  $, we have%
\begin{equation}
f_{d}\in\mathcal{S}_{\deg=d}. \label{pf.lem.I.sm-as-smaller.4}%
\end{equation}
Now,
\[
f=\sum_{d=0}^{N-1}\underbrace{f_{d}}_{\substack{\in\mathcal{S}_{\deg
=d}\\\text{(by (\ref{pf.lem.I.sm-as-smaller.4}))}}}\in\sum_{d=0}%
^{N-1}\underbrace{\mathcal{S}_{\deg=d}}_{\substack{=\sum_{\substack{\lambda\in
P_{k};\\\left\vert \lambda\right\vert =d}}\mathbf{k}s_{\lambda}\\\text{(by
(\ref{pf.lem.I.sm-as-smaller.2}))}}}=\underbrace{\sum_{d=0}^{N-1}%
\sum_{\substack{\lambda\in P_{k};\\\left\vert \lambda\right\vert =d}}}%
_{=\sum_{\substack{\lambda\in P_{k};\\\left\vert \lambda\right\vert <N}%
}}\mathbf{k}s_{\lambda}=\sum_{\substack{\lambda\in P_{k};\\\left\vert
\lambda\right\vert <N}}\mathbf{k}s_{\lambda}=\sum_{\substack{\kappa\in
P_{k};\\\left\vert \kappa\right\vert <N}}\mathbf{k}s_{\kappa}%
\]
(here, we have renamed the summation index $\lambda$ as $\kappa$ in the sum).
In other words, there exists a family $\left(  c_{\kappa}\right)  _{\kappa\in
P_{k};\ \left\vert \kappa\right\vert <N}$ of elements of $\mathbf{k}$ such
that $f=\sum_{\substack{\kappa\in P_{k};\\\left\vert \kappa\right\vert
<N}}c_{\kappa}s_{\kappa}$. This proves Lemma \ref{lem.I.sm-as-smaller}.
\end{proof}

\begin{lemma}
\label{lem.I.sl-red2}For each $\mu\in P_{k}$, the element $\overline{s_{\mu}%
}\in\mathcal{S}/I$ belongs to the $\mathbf{k}$-submodule of $\mathcal{S}/I$
spanned by the family $\left(  \overline{s_{\lambda}}\right)  _{\lambda\in
P_{k,n}}$.
\end{lemma}

\begin{proof}
[Proof of Lemma \ref{lem.I.sl-red2}.]Let $M$ be the $\mathbf{k}$-submodule of
$\mathcal{S}/I$ spanned by the family $\left(  \overline{s_{\lambda}}\right)
_{\lambda\in P_{k,n}}$. We thus must prove that $\overline{s_{\mu}}\in M$ for
each $\mu\in P_{k}$.

We shall prove this by strong induction on $\left\vert \mu\right\vert $. Thus,
we fix some $N\in\mathbb{N}$, and we assume (as induction hypothesis) that
\begin{equation}
\overline{s_{\kappa}}\in M\ \ \ \ \ \ \ \ \ \ \text{for each }\kappa\in
P_{k}\text{ satisfying }\left\vert \kappa\right\vert <N.
\label{pf.lem.I.sl-red2.IH}%
\end{equation}
Now, let $\mu\in P_{k}$ be such that $\left\vert \mu\right\vert =N$. We then
must show that $\overline{s_{\mu}}\in M$.

If $\mu\in P_{k,n}$, then this is obvious (since $\overline{s_{\mu}}$ then
belongs to the family that spans $M$). Thus, for the rest of this proof, we
WLOG assume that $\mu\notin P_{k,n}$. Hence, Lemma \ref{lem.I.sl-red} (applied
to $\lambda=\mu$) yields%
\[
s_{\mu}\equiv\left(  \text{some symmetric polynomial of degree }<\left\vert
\mu\right\vert \right)  \operatorname{mod}I.
\]
In other words, there exists some symmetric polynomial $f\in\mathcal{S}$ of
degree $<\left\vert \mu\right\vert $ such that $s_{\mu}\equiv
f\operatorname{mod}I$. Consider this $f$.

The polynomial $f$ is a symmetric polynomial of degree $<\left\vert
\mu\right\vert $. In other words, $f$ is a symmetric polynomial of degree $<N$
(since $\left\vert \mu\right\vert =N$). Hence, Lemma \ref{lem.I.sm-as-smaller}
shows that there exists a family $\left(  c_{\kappa}\right)  _{\kappa\in
P_{k};\ \left\vert \kappa\right\vert <N}$ of elements of $\mathbf{k}$ such
that $f=\sum_{\substack{\kappa\in P_{k};\\\left\vert \kappa\right\vert
<N}}c_{\kappa}s_{\kappa}$. Consider this family. From $f=\sum
_{\substack{\kappa\in P_{k};\\\left\vert \kappa\right\vert <N}}c_{\kappa
}s_{\kappa}$, we obtain%
\[
\overline{f}=\overline{\sum_{\substack{\kappa\in P_{k};\\\left\vert
\kappa\right\vert <N}}c_{\kappa}s_{\kappa}}=\sum_{\substack{\kappa\in
P_{k};\\\left\vert \kappa\right\vert <N}}c_{\kappa}\underbrace{\overline
{s_{\kappa}}}_{\substack{\in M\\\text{(by (\ref{pf.lem.I.sl-red2.IH}))}}%
}\in\sum_{\substack{\kappa\in P_{k};\\\left\vert \kappa\right\vert
<N}}c_{\kappa}M\subseteq M\ \ \ \ \ \ \ \ \ \ \left(  \text{since }M\text{ is
a }\mathbf{k}\text{-module}\right)  .
\]
But from $s_{\mu}\equiv f\operatorname{mod}I$, we obtain $\overline{s_{\mu}%
}=\overline{f}\in M$. This completes our induction step. Thus, we have proven
by strong induction that $\overline{s_{\mu}}\in M$ for each $\mu\in P_{k}$.
This proves Lemma \ref{lem.I.sl-red2}.
\end{proof}

\begin{proof}
[Proof of Theorem \ref{thm.S/J} (sketched).]Proposition \ref{prop.artin}
yields that $\left(  x^{\alpha}\right)  _{\alpha\in\mathbb{N}^{k};\ \alpha
_{i}<i\text{ for each }i}$ is a spanning set of the $\mathcal{S}$-module
$\mathcal{P}$.

Recall Definition \ref{def.Pk}. It is well-known that $\left(  s_{\lambda
}\right)  _{\lambda\in P_{k}}$ is a basis of the $\mathbf{k}$-module
$\mathcal{S}$. Hence, $\left(  \overline{s_{\lambda}}\right)  _{\lambda\in
P_{k}}$ is a spanning set of the $\mathbf{k}$-module $\mathcal{S}/I$. Thus,
$\left(  \overline{s_{\lambda}}\right)  _{\lambda\in P_{k,n}}$ is also a
spanning set of the $\mathbf{k}$-module $\mathcal{S}/I$ (because Lemma
\ref{lem.I.sl-red2} shows that every element of the first spanning set belongs
to the span of the second). It remains to prove that this spanning set is also
a basis.

In order to do so, we consider the family $\left(  \overline{s_{\lambda
}x^{\alpha}}\right)  _{\lambda\in P_{k,n};\ \alpha\in\mathbb{N}^{k}%
;\ \alpha_{i}<i\text{ for each }i}$ in the $\mathbf{k}$-module $\mathcal{P}%
/J$. This family spans $\mathcal{P}/J$ (by Lemma \ref{lem.basis-over-basis}),
because the family $\left(  \overline{s_{\lambda}}\right)  _{\lambda\in
P_{k,n}}$ spans $\mathcal{S}/I$ whereas the family $\left(  x^{\alpha}\right)
_{\alpha\in\mathbb{N}^{k};\ \alpha_{i}<i\text{ for each }i}$ spans
$\mathcal{P}$ over $\mathcal{S}$ (and because $I\mathcal{P}=J$). Moreover,
this family $\left(  \overline{s_{\lambda}x^{\alpha}}\right)  _{\lambda\in
P_{k,n};\ \alpha\in\mathbb{N}^{k};\ \alpha_{i}<i\text{ for each }i}$ has size
\begin{align*}
\underbrace{\left\vert P_{k,n}\right\vert }_{=\dbinom{n}{k}}\cdot
\underbrace{\left\vert \left\{  \alpha\in\mathbb{N}^{k}\ \mid\ \alpha
_{i}<i\text{ for each }i\right\}  \right\vert }_{=k!}  &  =\dbinom{n}{k}\cdot
k!\\
&  =n\left(  n-1\right)  \cdots\left(  n-k+1\right)  ,
\end{align*}
which is exactly the size of the basis $\left(  \overline{x^{\alpha}}\right)
_{\alpha\in\mathbb{N}^{k};\ \alpha_{i}<n-k+i\text{ for each }i}$ of the
$\mathbf{k}$-module $\mathcal{P}/J$ (this is a basis by Theorem \ref{thm.P/J}%
). Thus, this family $\left(  \overline{s_{\lambda}x^{\alpha}}\right)
_{\lambda\in P_{k,n};\ \alpha\in\mathbb{N}^{k};\ \alpha_{i}<i\text{ for each
}i}$ must be a basis of the $\mathbf{k}$-module $\mathcal{P}/J$ (by Lemma
\ref{lem.freemod-span-basis}), and hence is $\mathbf{k}$-linearly independent.
Thus, its subfamily $\left(  \overline{s_{\lambda}}\right)  _{\lambda\in
P_{k,n}}$ is also $\mathbf{k}$-linearly independent.

The canonical $\mathbf{k}$-linear map $\mathcal{S}/I\rightarrow\mathcal{P}/J$
(obtained as a quotient of the inclusion $\mathcal{S}\rightarrow\mathcal{P}$)
is injective (because it sends the spanning set $\left(  \overline{s_{\lambda
}}\right)  _{\lambda\in P_{k,n}}$ of $\mathcal{S}/I$ to the $\mathbf{k}%
$-linearly independent family $\left(  \overline{s_{\lambda}}\right)
_{\lambda\in P_{k,n}}$ in $\mathcal{P}/J$). Hence, the $\mathbf{k}$-linear
independency of the family $\left(  \overline{s_{\lambda}}\right)
_{\lambda\in P_{k,n}}$ in $\mathcal{P}/J$ yields the $\mathbf{k}$-linear
independency of the family $\left(  \overline{s_{\lambda}}\right)
_{\lambda\in P_{k,n}}$ in $\mathcal{S}/I$. Thus, the family $\left(
\overline{s_{\lambda}}\right)  _{\lambda\in P_{k,n}}$ in $\mathcal{S}/I$ is a
basis of $\mathcal{S}/I$ (since it is $\mathbf{k}$-linearly independent and
spans $\mathcal{S}/I$). This proves Theorem \ref{thm.S/J}.
\end{proof}

\section{\label{sect.symmetry}Symmetry of the multiplicative structure
constants}

\begin{convention}
\label{conv.symmetry-conv}For the rest of Section \ref{sect.symmetry},
\textbf{we assume that }$a_{1},a_{2},\ldots,a_{k}$ \textbf{belong to
}$\mathbf{k}$\textbf{.}

If $m\in\mathcal{S}$, then the notation $\overline{m}$ shall always mean the
projection of $m\in\mathcal{S}$ onto the quotient $\mathcal{S}/I$ (and not the
projection of $m\in\mathcal{P}$ onto the quotient $\mathcal{P}/J$).
\end{convention}

\begin{definition}
\label{def.omega-and-complement}\textbf{(a)} Let $\omega$ be the partition
$\left(  n-k,n-k,\ldots,n-k\right)  $ with $k$ entries equal to $n-k$. (This
is the largest partition in $P_{k,n}$.)

\textbf{(b)} Let $I$ be the ideal of $\mathcal{S}$ generated by the $k$
differences (\ref{eq.gensJ}). For each $\mu\in P_{k,n}$, let
$\operatorname*{coeff}\nolimits_{\mu}:\mathcal{S}/I\rightarrow\mathbf{k}$ be
the $\mathbf{k}$-linear map that sends $\overline{s_{\mu}}$ to $1$ while
sending all other $\overline{s_{\lambda}}$ (with $\lambda\in P_{k,n}$) to $0$.
(This is well-defined by Theorem \ref{thm.S/J}. Actually, $\left(
\operatorname*{coeff}\nolimits_{\mu}\right)  _{\mu\in P_{k,n}}$ is the dual
basis to the basis $\left(  \overline{s_{\lambda}}\right)  _{\lambda\in
P_{k,n}}$ of $\mathcal{S}/I$.)

\textbf{(c)} If $\lambda$ is any partition and if $p$ is a positive integer,
then $\lambda_{p}$ shall always denote the $p$-th entry of $\lambda$. Thus,
$\lambda=\left(  \lambda_{1},\lambda_{2},\lambda_{3},\ldots\right)  $ for
every partition $\lambda$.

\textbf{(d)} For every partition $\nu=\left(  \nu_{1},\nu_{2},\ldots,\nu
_{k}\right)  \in P_{k,n}$, we let $\nu^{\vee}$ denote the partition $\left(
n-k-\nu_{k},n-k-\nu_{k-1},\ldots,n-k-\nu_{1}\right)  \in P_{k,n}$. This
partition $\nu^{\vee}$ is called the \textit{complement} of $\nu$.
\end{definition}

We can now make a more substantial claim:

\begin{theorem}
\label{thm.coeffw}Each $\nu\in P_{k,n}$ and $f\in\mathcal{S}/I$ satisfy
$\operatorname*{coeff}\nolimits_{\omega}\left(  \overline{s_{\nu}}f\right)
=\operatorname*{coeff}\nolimits_{\nu^{\vee}}\left(  f\right)  $.
\end{theorem}

The proof of this theorem requires some preliminary work.

We first recall some basic notations from \cite[Chapter 2]{GriRei18}. If
$\lambda$ and $\mu$ are two partitions, then we say that $\mu\subseteq\lambda$
if and only if each positive integer $p$ satisfies $\mu_{p}\leq\lambda_{p}$. A
\textit{skew partition} means a pair $\left(  \lambda,\mu\right)  $ of two
partitions satisfying $\mu\subseteq\lambda$; such a pair is denoted by
$\lambda/\mu$. We refer to \cite[\S 2.7]{GriRei18} for the definition of a
\textit{vertical }$i$\textit{-strip} (where $i\in\mathbb{N}$).

Let $\Lambda$ be the ring of symmetric functions in infinitely many
indeterminates $x_{1},x_{2},x_{3},\ldots$ over $\mathbf{k}$. If $\mathbf{f}%
\in\Lambda$ is a symmetric function, then $\mathbf{f}\left(  x_{1}%
,x_{2},\ldots,x_{k}\right)  $ is a symmetric polynomial in $\mathcal{S}$; the
map%
\[
\Lambda\rightarrow\mathcal{S},\qquad\mathbf{f}\mapsto\mathbf{f}\left(
x_{1},x_{2},\ldots,x_{k}\right)
\]
is a surjective $\mathbf{k}$-algebra homomorphism. We shall use boldfaced
notations for symmetric functions in $\Lambda$ in order to distinguish them
from symmetric polynomials in $\mathcal{S}$. In particular:

\begin{itemize}
\item For any $i\in\mathbb{Z}$, we let $\mathbf{h}_{i}$ be the $i$-th complete
homogeneous symmetric function in $\Lambda$. (This is called $h_{i}$ in
\cite[Definition 2.2.1]{GriRei18}.)

\item For any $i\in\mathbb{Z}$, we let $\mathbf{e}_{i}$ be the $i$-th
elementary symmetric function in $\Lambda$. (This is called $e_{i}$ in
\cite[Definition 2.2.1]{GriRei18}.)

\item For any partition $\lambda$, we let $\mathbf{e}_{\lambda}$ be the
corresponding elementary symmetric function in $\Lambda$. (This is called
$e_{\lambda}$ in \cite[Definition 2.2.1]{GriRei18}.)

\item For any partition $\lambda$, we let $\mathbf{s}_{\lambda}$ be the
corresponding Schur function in $\Lambda$. (This is called $s_{\lambda}$ in
\cite[Definition 2.2.1]{GriRei18}.)

\item For any partitions $\lambda$ and $\mu$, we let $\mathbf{s}_{\lambda/\mu
}$ be the corresponding skew Schur function in $\Lambda$. (This is called
$s_{\lambda/\mu}$ in \cite[\S 2.3]{GriRei18}. Note that $\mathbf{s}%
_{\lambda/\mu}=0$ unless $\mu\subseteq\lambda$.)
\end{itemize}

Also, we shall use the \textit{skewing operators} as defined (e.g.) in
\cite[\S 2.8]{GriRei18}. We recall their main properties:

\begin{itemize}
\item For each $\mathbf{f}\in\Lambda$, the skewing operator $\mathbf{f}%
^{\perp}$ is a $\mathbf{k}$-linear map $\Lambda\rightarrow\Lambda$. It depends
$\mathbf{k}$-linearly on $\mathbf{f}$ (that is, we have $\left(
\alpha\mathbf{f}+\beta\mathbf{g}\right)  ^{\perp}=\alpha\mathbf{f}^{\perp
}+\beta\mathbf{g}^{\perp}$ for any $\alpha,\beta\in\mathbf{k}$ and
$\mathbf{f},\mathbf{g}\in\Lambda$).

\item For any partitions $\lambda$ and $\mu$, we have
\begin{equation}
\left(  \mathbf{s}_{\mu}\right)  ^{\perp}\left(  \mathbf{s}_{\lambda}\right)
=\mathbf{s}_{\lambda/\mu}. \label{eq.skewing.ss}%
\end{equation}
(This is \cite[(2.8.2)]{GriRei18}.)

\item For any $\mathbf{f},\mathbf{g}\in\Lambda$, we have
\begin{equation}
\left(  \mathbf{fg}\right)  ^{\perp}=\mathbf{g}^{\perp}\circ\mathbf{f}^{\perp
}. \label{eq.skewing.fg}%
\end{equation}
(This is \cite[Proposition 2.8.2(ii)]{GriRei18}, applied to $A=\Lambda$.)

\item We have $1^{\perp}=\operatorname*{id}$.
\end{itemize}

For each partition $\lambda$, let $\lambda^{t}$ denote the \textit{conjugate
partition} of $\lambda$; see \cite[Definition 2.2.8]{GriRei18} for its definition.

Recall the second Jacobi-Trudi identity (\cite[(2.4.17)]{GriRei18}):

\begin{proposition}
\label{prop.jt.e}Let $\lambda=\left(  \lambda_{1},\lambda_{2},\ldots
,\lambda_{\ell}\right)  $ and $\mu=\left(  \mu_{1},\mu_{2},\ldots,\mu_{\ell
}\right)  $ be two partitions. Then,%
\[
\mathbf{s}_{\lambda^{t}/\mu^{t}}=\det\left(  \left(  \mathbf{e}_{\lambda
_{i}-\mu_{j}-i+j}\right)  _{1\leq i\leq\ell,\ 1\leq j\leq\ell}\right)  .
\]

\end{proposition}

\begin{corollary}
\label{cor.jt.el}Let $\lambda=\left(  \lambda_{1},\lambda_{2},\ldots
,\lambda_{\ell}\right)  $ be a partition. Then,%
\[
\mathbf{s}_{\lambda^{t}}=\det\left(  \left(  \mathbf{e}_{\lambda_{i}%
-i+j}\right)  _{1\leq i\leq\ell,\ 1\leq j\leq\ell}\right)  .
\]

\end{corollary}

\begin{proof}
[Proof of Corollary \ref{cor.jt.el}.]This follows from Proposition
\ref{prop.jt.e}, applied to $\mu=\varnothing$ (since $\varnothing
^{t}=\varnothing$ and thus $\mathbf{s}_{\lambda^{t}/\varnothing^{t}%
}=\mathbf{s}_{\lambda^{t}/\varnothing}=\mathbf{s}_{\lambda^{t}}$).
\end{proof}

We also recall one of the Pieri rules (\cite[(2.7.2)]{GriRei18}):

\begin{proposition}
\label{prop.pieri.e}Let $\lambda$ be a partition, and let $i\in\mathbb{N}$.
Then,%
\[
\mathbf{s}_{\lambda}\mathbf{e}_{i}=\sum_{\substack{\mu\text{ is a
partition;}\\\mu/\lambda\text{ is a vertical }i\text{-strip}}}\mathbf{s}_{\mu
}.
\]

\end{proposition}

From this, we can easily derive the following:

\begin{corollary}
\label{cor.pieri.eskew}Let $\lambda$ be a partition, and let $i\in\mathbb{N}$.
Then,%
\[
\left(  \mathbf{e}_{i}\right)  ^{\perp}\mathbf{s}_{\lambda}=\sum
_{\substack{\mu\text{ is a partition;}\\\lambda/\mu\text{ is a vertical
}i\text{-strip}}}\mathbf{s}_{\mu}.
\]

\end{corollary}

Corollary \ref{cor.pieri.eskew} is also proven in \cite[(2.8.4)]{GriRei18}.

The next proposition is the claim of \cite[Exercise 2.9.1(b)]{GriRei18}:

\begin{proposition}
\label{prop.bernstein.1}Let $\lambda$ be a partition. Let $m\in\mathbb{Z}$ be
such that $m\geq\lambda_{1}$. Then,%
\[
\sum_{i\in\mathbb{N}}\left(  -1\right)  ^{i}\mathbf{h}_{m+i} \left(
\mathbf{e}_{i}\right)  ^{\perp}\mathbf{s}_{\lambda}=\mathbf{s}_{\left(
m,\lambda_{1},\lambda_{2},\lambda_{3},\ldots\right)  }.
\]

\end{proposition}

We shall use this to derive the following corollary:

\begin{corollary}
\label{cor.bernstein.2}Let $\lambda$ be a partition with at most $k$ parts.
Let $\overline{\lambda}$ be the partition $\left(  \lambda_{2},\lambda
_{3},\lambda_{4},\ldots\right)  $. Then,%
\[
\mathbf{s}_{\lambda}=\sum_{i=0}^{k-1}\left(  -1\right)  ^{i}\mathbf{h}%
_{\lambda_{1}+i}\sum_{\substack{\mu\text{ is a partition;}\\\overline{\lambda
}/\mu\text{ is a vertical }i\text{-strip}}}\mathbf{s}_{\mu}.
\]

\end{corollary}

\begin{proof}
[Proof of Corollary \ref{cor.bernstein.2}.]The partition $\overline{\lambda}$
is obtained from $\lambda$ by removing the first part. Hence, this partition
$\overline{\lambda}$ has at most $k-1$ parts (since $\lambda$ has at most $k$
parts). Thus, if $i\in\mathbb{N}$ satisfies $i\geq k$, then%
\begin{equation}
\text{there exists no partition }\mu\text{ such that }\overline{\lambda}%
/\mu\text{ is a vertical }i\text{-strip.} \label{pf.cor.bernstein.2.bound}%
\end{equation}

We have $\overline{\lambda}=\left(  \lambda_{2},\lambda_{3},\lambda_{4}%
,\ldots\right)  $, so that $\left(  \lambda_{2},\lambda_{3},\lambda_{4}%
,\ldots\right)  =\overline{\lambda}=\left(  \overline{\lambda}_{1}%
,\overline{\lambda}_{2},\overline{\lambda}_{3},\ldots\right)  $. Hence,%
\[
\left(  \lambda_{1},\lambda_{2},\lambda_{3},\lambda_{4},\ldots\right)
=\left(  \lambda_{1},\overline{\lambda}_{1},\overline{\lambda}_{2}%
,\overline{\lambda}_{3},\ldots\right)  .
\]
Also, clearly, $\lambda_{1}\geq\overline{\lambda}_{1}$ (since $\lambda_{1}%
\geq\lambda_{2}=\overline{\lambda}_{1}$). Hence, Proposition
\ref{prop.bernstein.1} (applied to $\overline{\lambda}$ and $\lambda_{1}$
instead of $\lambda$ and $m$) yields%
\[
\sum_{i\in\mathbb{N}}\left(  -1\right)  ^{i}\mathbf{h}_{\lambda_{1}+i}\left(
\mathbf{e}_{i}\right)  ^{\perp}\mathbf{s}_{\overline{\lambda}}=\mathbf{s}%
_{\left(  \lambda_{1},\overline{\lambda}_{1},\overline{\lambda}_{2}%
,\overline{\lambda}_{3},\ldots\right)  }=\mathbf{s}_{\lambda}%
\]
(since $\left(  \lambda_{1},\overline{\lambda}_{1},\overline{\lambda}%
_{2},\overline{\lambda}_{3},\ldots\right)  =\left(  \lambda_{1},\lambda
_{2},\lambda_{3},\lambda_{4},\ldots\right)  =\lambda$). Therefore,%
\begin{align*}
\mathbf{s}_{\lambda}  &  =\sum_{i\in\mathbb{N}}\left(  -1\right)
^{i}\mathbf{h}_{\lambda_{1}+i}\underbrace{\left(  \mathbf{e}_{i}\right)
^{\perp}\mathbf{s}_{\overline{\lambda}}}_{\substack{=\sum_{\substack{\mu\text{
is a partition;}\\\overline{\lambda}/\mu\text{ is a vertical }i\text{-strip}%
}}\mathbf{s}_{\mu}\\\text{(by Corollary \ref{cor.pieri.eskew})}}}=\sum
_{i\in\mathbb{N}}\left(  -1\right)  ^{i}\mathbf{h}_{\lambda_{1}+i}%
\sum_{\substack{\mu\text{ is a partition;}\\\overline{\lambda}/\mu\text{ is a
vertical }i\text{-strip}}}\mathbf{s}_{\mu}\\
&  =\sum_{i=0}^{k-1}\left(  -1\right)  ^{i}\mathbf{h}_{\lambda_{1}+i}%
\sum_{\substack{\mu\text{ is a partition;}\\\overline{\lambda}/\mu\text{ is a
vertical }i\text{-strip}}}\mathbf{s}_{\mu}+\sum_{i\geq k}\left(  -1\right)
^{i}\mathbf{h}_{\lambda_{1}+i}\underbrace{\sum_{\substack{\mu\text{ is a
partition;}\\\overline{\lambda}/\mu\text{ is a vertical }i\text{-strip}%
}}\mathbf{s}_{\mu}}_{\substack{=0\\\text{(by (\ref{pf.cor.bernstein.2.bound}%
))}}}\\
&  =\sum_{i=0}^{k-1}\left(  -1\right)  ^{i}\mathbf{h}_{\lambda_{1}+i}%
\sum_{\substack{\mu\text{ is a partition;}\\\overline{\lambda}/\mu\text{ is a
vertical }i\text{-strip}}}\mathbf{s}_{\mu}.
\end{align*}
This proves Corollary \ref{cor.bernstein.2}.
\end{proof}

\begin{convention}
We WLOG assume that $k>0$ for the rest of Section \ref{sect.symmetry} (since
otherwise, Theorem \ref{thm.coeffw} is trivial).
\end{convention}

Next, we define a filtration on the $\mathbf{k}$-module $\mathcal{S}/I$:

\begin{definition}
\label{def.Qp}For each $p\in\mathbb{Z}$, we let $Q_{p}$ denote the
$\mathbf{k}$-submodule of $\mathcal{S}/I$ spanned by the $\overline
{s_{\lambda}}$ with $\lambda\in P_{k,n}$ satisfying $\lambda_{k}\leq p$.
\end{definition}

Thus, $0=Q_{-1}\subseteq Q_{0}\subseteq Q_{1}\subseteq Q_{2}\subseteq\cdots$.
Theorem \ref{thm.S/J} shows that the $\mathbf{k}$-module $\mathcal{S}/I$ is
free with basis $\left(  \overline{s_{\lambda}}\right)  _{\lambda\in P_{k,n}}%
$; hence, $\mathcal{S}/I=Q_{n-k}$ (since each $\lambda\in P_{k,n}$ satisfies
$\lambda_{k}\leq n-k$).

Note that $\left(  Q_{0},Q_{1},Q_{2},\ldots\right)  $ is a filtration of the
$\mathbf{k}$-module $\mathcal{S}/I$, but not (in general) of the $\mathbf{k}%
$-algebra $\mathcal{S}/I$.

\begin{lemma}
\label{lem.coeffw.coeffQ}We have $\operatorname*{coeff}\nolimits_{\omega
}\left(  Q_{n-k-1}\right)  =0$.
\end{lemma}

\begin{proof}
[Proof of Lemma \ref{lem.coeffw.coeffQ}.]The map $\operatorname*{coeff}%
\nolimits_{\omega}$ is $\mathbf{k}$-linear; thus, it suffices to prove that
$\operatorname*{coeff}\nolimits_{\omega}\left(  \overline{s_{\lambda}}\right)
=0$ for each $\lambda\in P_{k,n}$ satisfying $\lambda_{k}\leq n-k-1$ (because
the $\mathbf{k}$-module $Q_{n-k-1}$ is spanned by the $\overline{s_{\lambda}}$
with $\lambda\in P_{k,n}$ satisfying $\lambda_{k}\leq n-k-1$). So let us fix
some $\lambda\in P_{k,n}$ satisfying $\lambda_{k}\leq n-k-1$. We must then
prove that $\operatorname*{coeff}\nolimits_{\omega}\left(  \overline
{s_{\lambda}}\right)  =0$.

We have $\lambda_{k}\leq n-k-1<n-k=\omega_{k}$. Thus, $\lambda_{k}\neq
\omega_{k}$, so that $\lambda\neq\omega$.

The definition of the map $\operatorname*{coeff}\nolimits_{\omega}$ yields
$\operatorname*{coeff}\nolimits_{\omega}\left(  \overline{s_{\lambda}}\right)
=%
\begin{cases}
1, & \text{if }\lambda=\omega;\\
0, & \text{if }\lambda\neq\omega
\end{cases}
=0$ (since $\lambda\neq\omega$). This completes our proof of Lemma
\ref{lem.coeffw.coeffQ}.
\end{proof}

\begin{lemma}
\label{lem.coeffw.0}Let $\lambda$ be a partition with at most $k$ parts.
Assume that $\lambda_{1}=n-k+1$. Let $\overline{\lambda}$ be the partition
$\left(  \lambda_{2},\lambda_{3},\lambda_{4},\ldots\right)  $. Then,
\[
\overline{s_{\lambda}}=\sum_{i=0}^{k-1}\left(  -1\right)  ^{i}a_{1+i}%
\sum_{\substack{\mu\text{ is a partition;}\\\overline{\lambda}/\mu\text{ is a
vertical }i\text{-strip}}}\overline{s_{\mu}}.
\]

\end{lemma}

\begin{proof}
[Proof of Lemma \ref{lem.coeffw.0}.]Corollary \ref{cor.bernstein.2} yields%
\[
\mathbf{s}_{\lambda}=\sum_{i=0}^{k-1}\left(  -1\right)  ^{i}\mathbf{h}%
_{\lambda_{1}+i}\sum_{\substack{\mu\text{ is a partition;}\\\overline{\lambda
}/\mu\text{ is a vertical }i\text{-strip}}}\mathbf{s}_{\mu}.
\]
This is an identity in $\Lambda$. Evaluating both of its sides at the $k$
variables $x_{1},x_{2},\ldots,x_{k}$, we obtain%
\begin{align*}
s_{\lambda}  &  =\sum_{i=0}^{k-1}\left(  -1\right)  ^{i}\underbrace{h_{\lambda
_{1}+i}}_{\substack{=h_{n-k+1+i}\\\text{(since }\lambda_{1}=n-k+1\text{)}%
}}\sum_{\substack{\mu\text{ is a partition;}\\\overline{\lambda}/\mu\text{ is
a vertical }i\text{-strip}}}s_{\mu}\\
&  =\sum_{i=0}^{k-1}\left(  -1\right)  ^{i}\underbrace{h_{n-k+1+i}%
}_{\substack{\equiv a_{1+i}\operatorname{mod}I\\\text{(by (\ref{eq.h=amodI}%
))}}}\sum_{\substack{\mu\text{ is a partition;}\\\overline{\lambda}/\mu\text{
is a vertical }i\text{-strip}}}s_{\mu}\\
&  \equiv\sum_{i=0}^{k-1}\left(  -1\right)  ^{i}a_{1+i}\sum_{\substack{\mu
\text{ is a partition;}\\\overline{\lambda}/\mu\text{ is a vertical
}i\text{-strip}}}s_{\mu}\operatorname{mod}I.
\end{align*}
Projecting both sides of this equality from $\mathcal{S}$ to $\mathcal{S}/I$,
we obtain%
\begin{equation}
\overline{s_{\lambda}}=\overline{\sum_{i=0}^{k-1}\left(  -1\right)
^{i}a_{1+i}\sum_{\substack{\mu\text{ is a partition;}\\\overline{\lambda}%
/\mu\text{ is a vertical }i\text{-strip}}}s_{\mu}}=\sum_{i=0}^{k-1}\left(
-1\right)  ^{i}a_{1+i}\sum_{\substack{\mu\text{ is a partition;}%
\\\overline{\lambda}/\mu\text{ is a vertical }i\text{-strip}}}\overline
{s_{\mu}}.\nonumber
\end{equation}
This proves Lemma \ref{lem.coeffw.0}.
\end{proof}

\begin{lemma}
\label{lem.coeffw.1}Let $\lambda$ be a partition with at most $k$ parts.
Assume that $\lambda_{1}=n-k+1$. Then, $\overline{s_{\lambda}}\in Q_{0}$.
\end{lemma}

\begin{proof}
[Proof of Lemma \ref{lem.coeffw.1}.]We shall prove Lemma \ref{lem.coeffw.1} by
strong induction on $\left\vert \lambda\right\vert $. Thus, we fix some
$N\in\mathbb{N}$, and we assume (as induction hypothesis) that Lemma
\ref{lem.coeffw.1} is already proven whenever $\left\vert \lambda\right\vert
<N$. We now must prove Lemma \ref{lem.coeffw.1} in the case when $\left\vert
\lambda\right\vert =N$.

So let $\lambda$ be as in Lemma \ref{lem.coeffw.1}, and assume that
$\left\vert \lambda\right\vert =N$. Let $\overline{\lambda}$ be the partition
$\left(  \lambda_{2},\lambda_{3},\lambda_{4},\ldots\right)  $. Then, Lemma
\ref{lem.coeffw.0} yields%
\begin{equation}
\overline{s_{\lambda}}=\sum_{i=0}^{k-1}\left(  -1\right)  ^{i}a_{1+i}%
\sum_{\substack{\mu\text{ is a partition;}\\\overline{\lambda}/\mu\text{ is a
vertical }i\text{-strip}}}\overline{s_{\mu}}. \label{pf.lem.coeffw.1.3}%
\end{equation}

But if $\mu$ is a partition such that $\overline{\lambda}/\mu$ is a vertical
$i$-strip, then%
\begin{equation}
\overline{s_{\mu}}\in Q_{0}. \label{pf.lem.coeffw.1.5}%
\end{equation}

[\textit{Proof of (\ref{pf.lem.coeffw.1.5}):} The partition $\lambda$ has at
most $k$ parts; thus, the partition $\overline{\lambda}$ has at most $k-1$ parts.

Now, let $\mu$ be a partition such that $\overline{\lambda}/\mu$ is a vertical
$i$-strip. Then, $\mu\subseteq\overline{\lambda}$, so that $\mu$ has at most
$k-1$ parts (since $\overline{\lambda}$ has at most $k-1$ parts). Thus,
$\mu_{k}=0\leq0$. Also, $\mu$ has at most $k$ parts (since $\mu$ has at most
$k-1$ parts). If $\mu_{1}\leq n-k$, then this yields that $\mu\in P_{k,n}$ and
therefore $\overline{s_{\mu}}\in Q_{0}$ (since $\mu\in P_{k,n}$ and $\mu
_{k}\leq0$). Thus, (\ref{pf.lem.coeffw.1.5}) is proven if $\mu_{1}\leq n-k$.
Hence, for the rest of this proof, we WLOG assume that we don't have $\mu
_{1}\leq n-k$. Hence, $\mu_{1}>n-k$.

But $\mu\subseteq\overline{\lambda}$, so that $\mu_{1}\leq\overline{\lambda
}_{1}=\lambda_{2}\leq\lambda_{1}=n-k+1$. Combining this with $\mu_{1}>n-k$, we
obtain $\mu_{1}=n-k+1$. Also, $\mu\subseteq\overline{\lambda}$, so that
\[
\left\vert \mu\right\vert \leq\left\vert \overline{\lambda}\right\vert
=\left\vert \lambda\right\vert -\underbrace{\lambda_{1}}_{=n-k+1\geq
1>0}<\left\vert \lambda\right\vert =N.
\]
Hence, we can apply Lemma \ref{lem.coeffw.1} to $\mu$ instead of $\lambda$ (by
the induction hypothesis). We thus obtain $\overline{s_{\mu}}\in Q_{0}$. This
completes the proof of (\ref{pf.lem.coeffw.1.5}).]

Now, (\ref{pf.lem.coeffw.1.3}) becomes%
\[
\overline{s_{\lambda}}=\sum_{i=0}^{k-1}\left(  -1\right)  ^{i}a_{1+i}%
\sum_{\substack{\mu\text{ is a partition;}\\\overline{\lambda}/\mu\text{ is a
vertical }i\text{-strip}}}\underbrace{\overline{s_{\mu}}}_{\substack{\in
Q_{0}\\\text{(by (\ref{pf.lem.coeffw.1.5}))}}}\in Q_{0}.
\]
Thus, we have proven Lemma \ref{lem.coeffw.1} for our $\lambda$. This
completes the induction step; thus, Lemma \ref{lem.coeffw.1} is proven.
\end{proof}

\begin{lemma}
\label{lem.coeffw.eisl}Let $i\in\mathbb{N}$ and $\lambda\in P_{k,n}$. Then,%
\[
\overline{e_{i}}\overline{s_{\lambda}}\equiv\sum_{\substack{\mu\in
P_{k,n};\\\mu/\lambda\text{ is a vertical }i\text{-strip}}}\overline{s_{\mu}%
}\operatorname{mod}Q_{0}.
\]

\end{lemma}

\begin{proof}
[Proof of Lemma \ref{lem.coeffw.eisl}.]If $\mu$ is a partition such that
$\mu/\lambda$ is a vertical $i$-strip and $\mu\notin P_{k,n}$, then%
\begin{equation}
\overline{s_{\mu}}\equiv0\operatorname{mod}Q_{0}.
\label{pf.lem.coeffw.eisl.mu}%
\end{equation}

[\textit{Proof of (\ref{pf.lem.coeffw.eisl.mu}):} Let $\mu$ be a partition
such that $\mu/\lambda$ is a vertical $i$-strip and $\mu\notin P_{k,n}$. We
must prove (\ref{pf.lem.coeffw.eisl.mu}).

If the partition $\mu$ has more than $k$ parts, then
(\ref{pf.lem.coeffw.eisl.mu}) easily follows\footnote{\textit{Proof.} Assume
that the partition $\mu$ has more than $k$ parts. Thus,
(\ref{eq.slam=0-too-long}) (applied to $\mu$ instead of $\lambda$) yields
$s_{\mu}=0$. Thus, $\overline{s_{\mu}}=0\equiv0\operatorname{mod}Q_{0}$. Thus,
(\ref{pf.lem.coeffw.eisl.mu}) holds.}. Hence, for the rest of this proof, we
WLOG assume that the partition $\mu$ has at most $k$ parts.

Since $\mu/\lambda$ is a vertical strip, we have $\mu_{1}\leq\lambda_{1}+1$.
But $\lambda_{1}\leq n-k$ (since $\lambda\in P_{k,n}$). If $\mu_{1}=n-k+1$,
then (\ref{pf.lem.coeffw.eisl.mu}) easily follows\footnote{\textit{Proof.}
Assume that $\mu_{1}=n-k+1$. Then, Lemma \ref{lem.coeffw.1} (applied to $\mu$
instead of $\lambda$) yields $\overline{s_{\mu}}\in Q_{0}$. Hence,
$\overline{s_{\mu}}\equiv0\operatorname{mod}Q_{0}$. Thus,
(\ref{pf.lem.coeffw.eisl.mu}) holds.}. Hence, for the rest of this proof, we
WLOG assume that $\mu_{1}\neq n-k+1$. Combining this with $\mu_{1}%
\leq\underbrace{\lambda_{1}}_{\leq n-k}+1\leq n-k+1$, we obtain $\mu
_{1}<n-k+1$, so that $\mu_{1}\leq n-k$. Hence, $\mu\in P_{k,n}$ (since $\mu$
has at most $k$ parts). This contradicts $\mu\notin P_{k,n}$. Thus,
$\overline{s_{\mu}}\equiv0\operatorname{mod}Q_{0}$ (because \textit{ex falso
quodlibet}). Hence, (\ref{pf.lem.coeffw.eisl.mu}) is proven.]

Proposition \ref{prop.pieri.e} yields%
\[
\mathbf{s}_{\lambda}\mathbf{e}_{i}=\sum_{\substack{\mu\text{ is a
partition;}\\\mu/\lambda\text{ is a vertical }i\text{-strip}}}\mathbf{s}_{\mu
}.
\]
This is an identity in $\Lambda$. Evaluating both of its sides at the $k$
variables $x_{1},x_{2},\ldots,x_{k}$, we obtain%
\[
s_{\lambda}e_{i}=\sum_{\substack{\mu\text{ is a partition;}\\\mu/\lambda\text{
is a vertical }i\text{-strip}}}s_{\mu}.
\]
Projecting both sides of this equality from $\mathcal{S}$ to $\mathcal{S}/I$,
we obtain%
\begin{align*}
\overline{s_{\lambda}}\overline{e_{i}}  &  =\overline{\sum_{\substack{\mu
\text{ is a partition;}\\\mu/\lambda\text{ is a vertical }i\text{-strip}%
}}s_{\mu}}=\sum_{\substack{\mu\text{ is a partition;}\\\mu/\lambda\text{ is a
vertical }i\text{-strip}}}\overline{s_{\mu}}\\
&  =\sum_{\substack{\mu\text{ is a partition;}\\\mu/\lambda\text{ is a
vertical }i\text{-strip;}\\\mu\in P_{k,n}}}\overline{s_{\mu}}+\sum
_{\substack{\mu\text{ is a partition;}\\\mu/\lambda\text{ is a vertical
}i\text{-strip;}\\\mu\notin P_{k,n}}}\underbrace{\overline{s_{\mu}}%
}_{\substack{\equiv0\operatorname{mod}Q_{0}\\\text{(by
(\ref{pf.lem.coeffw.eisl.mu}))}}}\\
&  \equiv\sum_{\substack{\mu\text{ is a partition;}\\\mu/\lambda\text{ is a
vertical }i\text{-strip;}\\\mu\in P_{k,n}}}\overline{s_{\mu}}=\sum
_{\substack{\mu\in P_{k,n};\\\mu/\lambda\text{ is a vertical }i\text{-strip}%
}}\overline{s_{\mu}}\operatorname{mod}Q_{0}.
\end{align*}
Thus, $\overline{e_{i}}\overline{s_{\lambda}}=\overline{s_{\lambda}}%
\overline{e_{i}}\equiv\sum_{\substack{\mu\in P_{k,n};\\\mu/\lambda\text{ is a
vertical }i\text{-strip}}}\overline{s_{\mu}}\operatorname{mod}Q_{0}$. This
proves Lemma \ref{lem.coeffw.eisl}.
\end{proof}

\begin{lemma}
\label{lem.coeffw.eiQ}Let $i\in\mathbb{Z}$ and $p\in\mathbb{Z}$. Then,
$\overline{e_{i}}Q_{p}\subseteq Q_{p+1}$.
\end{lemma}

\begin{proof}
[Proof of Lemma \ref{lem.coeffw.eiQ}.]Due to the definition of $Q_{p}$, it
suffices to prove that every $\lambda\in P_{k,n}$ satisfying $\lambda_{k}\leq
p$ satisfies $\overline{e_{i}}\overline{s_{\lambda}}\in Q_{p+1}$. So let us
fix $\lambda\in P_{k,n}$ satisfying $\lambda_{k}\leq p$. We must prove that
$\overline{e_{i}}\overline{s_{\lambda}}\in Q_{p+1}$.

From $\lambda_{k}\geq0$, we obtain $0\leq\lambda_{k}\leq p\leq p+1$.

We WLOG assume that $i\in\mathbb{N}$ (since otherwise, we have $e_{i}=0$ and
thus $\overline{e_{i}}\overline{s_{\lambda}}=\overline{0}\overline{s_{\lambda
}}=0\in Q_{p+1}$).

If $\mu\in P_{k,n}$ is such that $\mu/\lambda$ is a vertical $i$-strip, then%
\begin{equation}
\overline{s_{\mu}}\equiv0\operatorname{mod}Q_{p+1}.
\label{pf.lem.coeffw.eiQ.mu}%
\end{equation}

[\textit{Proof of (\ref{pf.lem.coeffw.eiQ.mu}):} Let $\mu\in P_{k,n}$ be such
that $\mu/\lambda$ is a vertical $i$-strip. We must prove
(\ref{pf.lem.coeffw.eiQ.mu}).

Since $\mu/\lambda$ is a vertical strip, we have $\mu_{k}\leq
\underbrace{\lambda_{k}}_{\leq p}+1\leq p+1$. From $\mu\in P_{k,n}$ and
$\mu_{k}\leq p+1$, we obtain $\overline{s_{\mu}}\in Q_{p+1}$. In other words,
$\overline{s_{\mu}}\equiv0\operatorname{mod}Q_{p+1}$. Thus,
(\ref{pf.lem.coeffw.eiQ.mu}) is proven.]

Lemma \ref{lem.coeffw.eisl} yields%
\[
\overline{e_{i}}\overline{s_{\lambda}}\equiv\sum_{\substack{\mu\in
P_{k,n};\\\mu/\lambda\text{ is a vertical }i\text{-strip}}}\overline{s_{\mu}%
}\operatorname{mod}Q_{0}.
\]
Hence,%
\[
\overline{e_{i}}\overline{s_{\lambda}}-\sum_{\substack{\mu\in P_{k,n}%
;\\\mu/\lambda\text{ is a vertical }i\text{-strip}}}\overline{s_{\mu}}\in
Q_{0}\subseteq Q_{p+1}\ \ \ \ \ \ \ \ \ \ \left(  \text{since }0\leq
p+1\right)  .
\]
Thus,%
\begin{equation}
\overline{e_{i}}\overline{s_{\lambda}}\equiv\sum_{\substack{\mu\in
P_{k,n};\\\mu/\lambda\text{ is a vertical }i\text{-strip}}%
}\underbrace{\overline{s_{\mu}}}_{\substack{\equiv0\operatorname{mod}%
Q_{p+1}\\\text{(by (\ref{pf.lem.coeffw.eiQ.mu}))}}}\equiv0\operatorname{mod}%
Q_{p+1}.\nonumber
\end{equation}
In other words, $\overline{e_{i}}\overline{s_{\lambda}}\in Q_{p+1}$. This
completes our proof of Lemma \ref{lem.coeffw.eiQ}.
\end{proof}

The next fact that we use from the theory of symmetric functions are some
basic properties of the Littlewood-Richardson coefficients. For any partitions
$\lambda,\mu,\nu$, we let $c_{\mu,\nu}^{\lambda}$ be the Littlewood-Richardson
coefficient as defined in \cite[Definition 2.5.8]{GriRei18}. Then, we have the
following fact (part of \cite[Remark 2.5.9]{GriRei18}):

\begin{proposition}
\label{prop.LR.props}Let $\lambda$ and $\mu$ be two partitions.

\textbf{(a)} We have
\[
\mathbf{s}_{\lambda/\mu}=\sum_{\nu\text{ is a partition}}c_{\mu,\nu}^{\lambda
}\mathbf{s}_{\nu}.
\]

\textbf{(b)} If $\nu$ is a partition, then $c_{\mu,\nu}^{\lambda}=0$ unless
$\nu\subseteq\lambda$.

\textbf{(c)} If $\nu$ is a partition, then $c_{\mu,\nu}^{\lambda}=0$ unless
$\left\vert \mu\right\vert +\left\vert \nu\right\vert =\left\vert
\lambda\right\vert $.
\end{proposition}

Next, let $\mathcal{Z}$ be the $\mathbf{k}$-submodule of $\Lambda$ spanned by
the $\mathbf{s}_{\lambda}$ with $\lambda\in P_{k,n}$. Then, $\left(
\mathbf{s}_{\lambda}\right)  _{\lambda\in P_{k,n}}$ is a basis of the
$\mathbf{k}$-module $\mathcal{Z}$ (since $\left(  \mathbf{s}_{\lambda}\right)
_{\lambda\text{ is a partition}}$ is a basis of the $\mathbf{k}$-module
$\Lambda$). We thus can define a $\mathbf{k}$-linear map $\delta
:\mathcal{Z}\rightarrow\mathcal{S}/I$ by setting%
\[
\delta\left(  \mathbf{s}_{\lambda}\right)  =\overline{s_{\lambda^{\vee}}%
}\ \ \ \ \ \ \ \ \ \ \text{for every }\lambda\in P_{k,n}.
\]
Notice that a partition $\lambda$ satisfies $\lambda\in P_{k,n}$ if and only
if $\lambda\subseteq\omega$.

\begin{lemma}
\label{lem.coeffw.Z-invar-skew}We have $\mathbf{f}^{\perp}\left(
\mathcal{Z}\right)  \subseteq\mathcal{Z}$ for each $\mathbf{f}\in\Lambda$.
\end{lemma}

\begin{proof}
[Proof of Lemma \ref{lem.coeffw.Z-invar-skew}.]Since $\mathbf{f}^{\perp}$
depends $\mathbf{k}$-linearly on $\mathbf{f}$, it suffices to check that
$\left(  \mathbf{s}_{\mu}\right)  ^{\perp}\left(  \mathcal{Z}\right)
\subseteq\mathcal{Z}$ for each partition $\mu$. So let us fix a partition
$\mu$; we then must prove that $\left(  \mathbf{s}_{\mu}\right)  ^{\perp
}\left(  \mathcal{Z}\right)  \subseteq\mathcal{Z}$.

Recall that $\mathcal{Z}$ is the $\mathbf{k}$-module spanned by the
$\mathbf{s}_{\lambda}$ with $\lambda\in P_{k,n}$. Hence, in order to prove
that $\left(  \mathbf{s}_{\mu}\right)  ^{\perp}\left(  \mathcal{Z}\right)
\subseteq\mathcal{Z}$, it suffices to check that $\left(  \mathbf{s}_{\mu
}\right)  ^{\perp}\left(  \mathbf{s}_{\lambda}\right)  \in\mathcal{Z}$ for
each $\lambda\in P_{k,n}$. So let us fix $\lambda\in P_{k,n}$; we must then
prove that $\left(  \mathbf{s}_{\mu}\right)  ^{\perp}\left(  \mathbf{s}%
_{\lambda}\right)  \in\mathcal{Z}$.

From (\ref{eq.skewing.ss}), we obtain%
\begin{align*}
\left(  \mathbf{s}_{\mu}\right)  ^{\perp}\left(  \mathbf{s}_{\lambda}\right)
&  =\mathbf{s}_{\lambda/\mu}=\sum_{\nu\text{ is a partition}}c_{\mu,\nu
}^{\lambda}\mathbf{s}_{\nu}\ \ \ \ \ \ \ \ \ \ \left(  \text{by Proposition
\ref{prop.LR.props} \textbf{(a)}}\right) \\
&  =\sum_{\substack{\nu\text{ is a partition;}\\\nu\subseteq\lambda}%
}c_{\mu,\nu}^{\lambda}\mathbf{s}_{\nu}+\sum_{\substack{\nu\text{ is a
partition;}\\\text{not }\nu\subseteq\lambda}}\underbrace{c_{\mu,\nu}^{\lambda
}}_{\substack{=0\\\text{(by Proposition \ref{prop.LR.props} \textbf{(b)})}%
}}\mathbf{s}_{\nu}\\
&  =\sum_{\substack{\nu\text{ is a partition;}\\\nu\subseteq\lambda}%
}c_{\mu,\nu}^{\lambda}\underbrace{\mathbf{s}_{\nu}}_{\substack{\in
\mathcal{Z}\\\text{(because }\nu\subseteq\lambda\\\text{and }\lambda\in
P_{k,n}\text{ lead to}\\\nu\in P_{k,n}\text{)}}}\in\mathcal{Z}.
\end{align*}
This completes our proof of Lemma \ref{lem.coeffw.Z-invar-skew}.
\end{proof}

\begin{lemma}
\label{lem.coeffw.del1}Let $i\in\mathbb{Z}$ and $\mathbf{f}\in\mathcal{Z}$.
Then,%
\[
\delta\left(  \left(  \mathbf{e}_{i}\right)  ^{\perp}\mathbf{f}\right)
\equiv\overline{e_{i}}\delta\left(  \mathbf{f}\right)  \operatorname{mod}%
Q_{0}.
\]
(Note that $\delta\left(  \left(  \mathbf{e}_{i}\right)  ^{\perp}%
\mathbf{f}\right)  $ is well-defined, since Lemma
\ref{lem.coeffw.Z-invar-skew} yields $\left(  \mathbf{e}_{i}\right)  ^{\perp
}\mathbf{f}\in\mathcal{Z}$.)
\end{lemma}

\begin{proof}
[Proof of Lemma \ref{lem.coeffw.del1}.]Both sides of the claim are
$\mathbf{k}$-linear in $\mathbf{f}$. Hence, we can WLOG assume that
$\mathbf{f}=\mathbf{s}_{\lambda}$ for some $\lambda\in P_{k,n}$ (since
$\left(  \mathbf{s}_{\lambda}\right)  _{\lambda\in P_{k,n}}$ is a basis of the
$\mathbf{k}$-module $\mathcal{Z}$). Assume this, and consider this $\lambda$.

We must prove that $\delta\left(  \left(  \mathbf{e}_{i}\right)  ^{\perp
}\mathbf{f}\right)  \equiv\overline{e_{i}}\delta\left(  \mathbf{f}\right)
\operatorname{mod}Q_{0}$. If $i<0$, then this is obvious (because if $i<0$,
then both $\mathbf{e}_{i}$ and $e_{i}$ equal $0$, and therefore both sides of
the congruence $\delta\left(  \left(  \mathbf{e}_{i}\right)  ^{\perp
}\mathbf{f}\right)  \equiv\overline{e_{i}}\delta\left(  \mathbf{f}\right)
\operatorname{mod}Q_{0}$ are equal to $0$). Hence, for the rest of this proof,
we WLOG assume that we don't have $i<0$. Thus, $i\geq0$, so that
$i\in\mathbb{N}$.

It is easy to see that if $\mu\in P_{k,n}$, then we have the following
equivalence of statements:%
\begin{equation}
\left(  \lambda/\mu\text{ is a vertical }i\text{-strip}\right)
\ \Longleftrightarrow\ \left(  \mu^{\vee}/\lambda^{\vee}\text{ is a vertical
}i\text{-strip}\right)  . \label{pf.lem.coeffw.del1.eq}%
\end{equation}
(Indeed, the skew Young diagram of $\mu^{\vee}/\lambda^{\vee}$ is obtained
from the skew Young diagram of $\lambda/\mu$ by a rotation by $180^{\circ}$.)

From $\mathbf{f}=\mathbf{s}_{\lambda}$, we obtain%
\begin{align*}
\left(  \mathbf{e}_{i}\right)  ^{\perp}\mathbf{f}  &  =\left(  \mathbf{e}%
_{i}\right)  ^{\perp}\mathbf{s}_{\lambda}=\sum_{\substack{\mu\text{ is a
partition;}\\\lambda/\mu\text{ is a vertical }i\text{-strip}}}\mathbf{s}_{\mu
}\ \ \ \ \ \ \ \ \ \ \left(  \text{by Corollary \ref{cor.pieri.eskew}}\right)
\\
&  =\sum_{\substack{\mu\in P_{k,n};\\\lambda/\mu\text{ is a vertical
}i\text{-strip}}}\mathbf{s}_{\mu}%
\end{align*}
(because if $\mu$ is a partition such that $\lambda/\mu$ is a vertical
$i$-strip, then $\mu\in P_{k,n}$ (since $\mu\subseteq\lambda$ and $\lambda\in
P_{k,n}$)). Applying the map $\delta$ to both sides of this equality, we find%
\begin{align}
\delta\left(  \left(  \mathbf{e}_{i}\right)  ^{\perp}\mathbf{f}\right)   &
=\delta\left(  \sum_{\substack{\mu\in P_{k,n};\\\lambda/\mu\text{ is a
vertical }i\text{-strip}}}\mathbf{s}_{\mu}\right)  =\underbrace{\sum
_{\substack{\mu\in P_{k,n};\\\lambda/\mu\text{ is a vertical }i\text{-strip}%
}}}_{\substack{=\sum_{\substack{\mu\in P_{k,n};\\\mu^{\vee}/\lambda^{\vee
}\text{ is a vertical }i\text{-strip}}}\\\text{(by
(\ref{pf.lem.coeffw.del1.eq}))}}}\underbrace{\delta\left(  \mathbf{s}_{\mu
}\right)  }_{\substack{=\overline{s_{\mu^{\vee}}}\\\text{(by the definition of
}\delta\text{)}}}\nonumber\\
&  =\sum_{\substack{\mu\in P_{k,n};\\\mu^{\vee}/\lambda^{\vee}\text{ is a
vertical }i\text{-strip}}}\overline{s_{\mu^{\vee}}}=\sum_{\substack{\mu\in
P_{k,n};\\\mu/\lambda^{\vee}\text{ is a vertical }i\text{-strip}}%
}\overline{s_{\mu}} \label{pf.lem.coeffw.del1.LHS}%
\end{align}
(here, we have substituted $\mu$ for $\mu^{\vee}$ in the sum, since the map
$P_{k,n}\rightarrow P_{k,n},\ \mu\mapsto\mu^{\vee}$ is a bijection).

On the other hand, from $\mathbf{f}=\mathbf{s}_{\lambda}$, we obtain
$\delta\left(  \mathbf{f}\right)  =\delta\left(  \mathbf{s}_{\lambda}\right)
=\overline{s_{\lambda^{\vee}}}$ (by the definition of $\delta$) and thus%
\[
\overline{e_{i}}\delta\left(  \mathbf{f}\right)  =\overline{e_{i}}%
\overline{s_{\lambda^{\vee}}}\equiv\sum_{\substack{\mu\in P_{k,n}%
;\\\mu/\lambda^{\vee}\text{ is a vertical }i\text{-strip}}}\overline{s_{\mu}%
}\operatorname{mod}Q_{0}\ \ \ \ \ \ \ \ \ \ \left(
\begin{array}
[c]{c}%
\text{by Lemma \ref{lem.coeffw.eisl}, applied}\\
\text{to }\lambda^{\vee}\text{ instead of }\lambda
\end{array}
\right)  .
\]
Comparing this with (\ref{pf.lem.coeffw.del1.LHS}), we obtain $\delta\left(
\left(  \mathbf{e}_{i}\right)  ^{\perp}\mathbf{f}\right)  \equiv
\overline{e_{i}}\delta\left(  \mathbf{f}\right)  \operatorname{mod}Q_{0}$.
This proves Lemma \ref{lem.coeffw.del1}.
\end{proof}

\begin{lemma}
\label{lem.coeffw.del2}Let $p\in\mathbb{N}$. Let $i_{1},i_{2},\ldots,i_{p}%
\in\mathbb{Z}$ and $\mathbf{f}\in\mathcal{Z}$. Then,%
\[
\delta\left(  \left(  \mathbf{e}_{i_{1}}\mathbf{e}_{i_{2}}\cdots
\mathbf{e}_{i_{p}}\right)  ^{\perp}\mathbf{f}\right)  \equiv\overline
{e_{i_{1}}e_{i_{2}}\cdots e_{i_{p}}}\delta\left(  \mathbf{f}\right)
\operatorname{mod}Q_{p-1}.
\]

\end{lemma}

\begin{proof}
[Proof of Lemma \ref{lem.coeffw.del2}.]We proceed by induction on $p$.

The \textit{induction base} (the case $p=0$) is obvious (since $1^{\perp
}=\operatorname*{id}$ and thus $1^{\perp}\mathbf{f}=\mathbf{f}$).

\textit{Induction step:} Let $q\in\mathbb{N}$. Assume (as the induction
hypothesis) that Lemma \ref{lem.coeffw.del2} holds for $p=q$. We must now
prove that Lemma \ref{lem.coeffw.del2} holds for $p=q+1$. In other words, we
must prove that every $i_{1},i_{2},\ldots,i_{q+1}\in\mathbb{Z}$ and
$\mathbf{f}\in\mathcal{Z}$ satisfy
\begin{equation}
\delta\left(  \left(  \mathbf{e}_{i_{1}}\mathbf{e}_{i_{2}}\cdots
\mathbf{e}_{i_{q+1}}\right)  ^{\perp}\mathbf{f}\right)  \equiv\overline
{e_{i_{1}}e_{i_{2}}\cdots e_{i_{q+1}}}\delta\left(  \mathbf{f}\right)
\operatorname{mod}Q_{q}. \label{pf.lem.coeffw.del2.goal}%
\end{equation}

So let $i_{1},i_{2},\ldots,i_{q+1}\in\mathbb{Z}$ and $\mathbf{f}\in
\mathcal{Z}$. We must prove (\ref{pf.lem.coeffw.del2.goal}).

Lemma \ref{lem.coeffw.eiQ} (applied to $i_{q+1}$ and $q-1$ instead of $i$ and
$p$) yields $\overline{e_{i_{q+1}}}Q_{q-1}\subseteq Q_{q}$.

The induction hypothesis yields
\[
\delta\left(  \left(  \mathbf{e}_{i_{1}}\mathbf{e}_{i_{2}}\cdots
\mathbf{e}_{i_{q}}\right)  ^{\perp}\mathbf{f}\right)  \equiv\overline
{e_{i_{1}}e_{i_{2}}\cdots e_{i_{q}}}\delta\left(  \mathbf{f}\right)
\operatorname{mod}Q_{q-1}.
\]
Multiplying both sides of this congruence by $\overline{e_{i_{q+1}}}$, we
obtain%
\begin{equation}
\overline{e_{i_{q+1}}}\delta\left(  \left(  \mathbf{e}_{i_{1}}\mathbf{e}%
_{i_{2}}\cdots\mathbf{e}_{i_{q}}\right)  ^{\perp}\mathbf{f}\right)
\equiv\overline{e_{i_{q+1}}}\overline{e_{i_{1}}e_{i_{2}}\cdots e_{i_{q}}%
}\delta\left(  \mathbf{f}\right)  \operatorname{mod}Q_{q}
\label{pf.lem.coeffw.del2.2}%
\end{equation}
(since $\overline{e_{i_{q+1}}}Q_{q-1}\subseteq Q_{q}$).

Applying Lemma \ref{lem.coeffw.Z-invar-skew} to $\mathbf{f}=\mathbf{e}_{i_{1}%
}\mathbf{e}_{i_{2}}\cdots\mathbf{e}_{i_{q}}$, we obtain $\left(
\mathbf{e}_{i_{1}}\mathbf{e}_{i_{2}}\cdots\mathbf{e}_{i_{q}}\right)  ^{\perp
}\left(  \mathcal{Z}\right)  \subseteq\mathcal{Z}$. Hence, $\left(
\mathbf{e}_{i_{1}}\mathbf{e}_{i_{2}}\cdots\mathbf{e}_{i_{q}}\right)  ^{\perp
}\mathbf{f}\in\mathcal{Z}$ (since $\mathbf{f}\in\mathcal{Z}$).

But (\ref{eq.skewing.fg}) (applied to $\mathbf{f}=\mathbf{e}_{i_{1}}%
\mathbf{e}_{i_{2}}\cdots\mathbf{e}_{i_{q}}$ and $\mathbf{g}=\mathbf{e}%
_{i_{q+1}}$) yields%
\[
\left(  \mathbf{e}_{i_{1}}\mathbf{e}_{i_{2}}\cdots\mathbf{e}_{i_{q+1}}\right)
^{\perp}=\left(  \mathbf{e}_{i_{q+1}}\right)  ^{\perp}\circ\left(
\mathbf{e}_{i_{1}}\mathbf{e}_{i_{2}}\cdots\mathbf{e}_{i_{q}}\right)  ^{\perp
}.
\]
Hence,%
\[
\left(  \mathbf{e}_{i_{1}}\mathbf{e}_{i_{2}}\cdots\mathbf{e}_{i_{q+1}}\right)
^{\perp}\mathbf{f}=\left(  \left(  \mathbf{e}_{i_{q+1}}\right)  ^{\perp}%
\circ\left(  \mathbf{e}_{i_{1}}\mathbf{e}_{i_{2}}\cdots\mathbf{e}_{i_{q}%
}\right)  ^{\perp}\right)  \mathbf{f}=\left(  \mathbf{e}_{i_{q+1}}\right)
^{\perp}\left(  \left(  \mathbf{e}_{i_{1}}\mathbf{e}_{i_{2}}\cdots
\mathbf{e}_{i_{q}}\right)  ^{\perp}\mathbf{f}\right)  .
\]
Applying the map $\delta$ to both sides of this equality, we find%
\begin{align*}
\delta\left(  \left(  \mathbf{e}_{i_{1}}\mathbf{e}_{i_{2}}\cdots
\mathbf{e}_{i_{q+1}}\right)  ^{\perp}\mathbf{f}\right)   &  =\delta\left(
\left(  \mathbf{e}_{i_{q+1}}\right)  ^{\perp}\left(  \left(  \mathbf{e}%
_{i_{1}}\mathbf{e}_{i_{2}}\cdots\mathbf{e}_{i_{q}}\right)  ^{\perp}%
\mathbf{f}\right)  \right) \\
&  \equiv\overline{e_{i_{q+1}}}\delta\left(  \left(  \mathbf{e}_{i_{1}%
}\mathbf{e}_{i_{2}}\cdots\mathbf{e}_{i_{q}}\right)  ^{\perp}\mathbf{f}\right)
\operatorname{mod}Q_{0}%
\end{align*}
(by Lemma \ref{lem.coeffw.del1}, applied to $i_{q+1}$ and $\left(
\mathbf{e}_{i_{1}}\mathbf{e}_{i_{2}}\cdots\mathbf{e}_{i_{q}}\right)  ^{\perp
}\mathbf{f}$ instead of $i$ and $\mathbf{f}$). Since $Q_{0}\subseteq Q_{q}$,
this yields%
\begin{align*}
\delta\left(  \left(  \mathbf{e}_{i_{1}}\mathbf{e}_{i_{2}}\cdots
\mathbf{e}_{i_{q+1}}\right)  ^{\perp}\mathbf{f}\right)   &  \equiv
\overline{e_{i_{q+1}}}\delta\left(  \left(  \mathbf{e}_{i_{1}}\mathbf{e}%
_{i_{2}}\cdots\mathbf{e}_{i_{q}}\right)  ^{\perp}\mathbf{f}\right) \\
&  \equiv\overline{e_{i_{q+1}}}\overline{e_{i_{1}}e_{i_{2}}\cdots e_{i_{q}}%
}\delta\left(  \mathbf{f}\right)  \ \ \ \ \ \ \ \ \ \ \left(  \text{by
(\ref{pf.lem.coeffw.del2.2})}\right) \\
&  =\overline{e_{i_{1}}e_{i_{2}}\cdots e_{i_{q+1}}}\delta\left(
\mathbf{f}\right)  \operatorname{mod}Q_{q}.
\end{align*}
Thus, (\ref{pf.lem.coeffw.del2.goal}) is proven. This completes the induction
step. Thus, Lemma \ref{lem.coeffw.del2} is proven.
\end{proof}

\begin{lemma}
\label{lem.coeffw.del3}Let $\lambda\in P_{k,n}$ and $\mathbf{f}\in\mathcal{Z}%
$. Then,%
\[
\delta\left(  \left(  \mathbf{s}_{\lambda}\right)  ^{\perp}\mathbf{f}\right)
\equiv\overline{s_{\lambda}}\delta\left(  \mathbf{f}\right)
\operatorname{mod}Q_{n-k-1}.
\]

\end{lemma}

\begin{proof}
[Proof of Lemma \ref{lem.coeffw.del3}.]Let $\ell=n-k$. From $\lambda\in
P_{k,n}$, we have $\lambda_{1}\leq n-k=\ell$.

Consider the conjugate partition $\lambda^{t}$ of $\lambda$. Then,
$\lambda^{t}$ has exactly $\lambda_{1}$ parts. Thus, $\lambda^{t}$ has
$\leq\ell$ parts (since $\lambda_{1}\leq\ell$). Therefore, $\lambda
^{t}=\left(  \left(  \lambda^{t}\right)  _{1},\left(  \lambda^{t}\right)
_{2},\ldots,\left(  \lambda^{t}\right)  _{\ell}\right)  $. Hence, Corollary
\ref{cor.jt.el} (applied to $\lambda^{t}$ instead of $\lambda$) yields%
\[
\mathbf{s}_{\left(  \lambda^{t}\right)  ^{t}}=\det\left(  \left(
\mathbf{e}_{\left(  \lambda^{t}\right)  _{i}-i+j}\right)  _{1\leq i\leq
\ell,\ 1\leq j\leq\ell}\right)  .
\]
In view of $\left(  \lambda^{t}\right)  ^{t}=\lambda$, this rewrites as
\begin{equation}
\mathbf{s}_{\lambda}=\det\left(  \left(  \mathbf{e}_{\left(  \lambda
^{t}\right)  _{i}-i+j}\right)  _{1\leq i\leq\ell,\ 1\leq j\leq\ell}\right)
=\sum_{\sigma\in S_{\ell}}\left(  -1\right)  ^{\sigma}\prod_{i=1}^{\ell
}\mathbf{e}_{\left(  \lambda^{t}\right)  _{i}-i+\sigma\left(  i\right)  }
\label{pf.lem.coeffw.del3.jt}%
\end{equation}
(where $S_{\ell}$ denotes the symmetric group of the set $\left\{
1,2,\ldots,\ell\right\}  $, and where $\left(  -1\right)  ^{\sigma}$ denotes
the sign of a permutation $\sigma\in S_{\ell}$). Hence,%
\[
\left(  \mathbf{s}_{\lambda}\right)  ^{\perp}\mathbf{f}=\left(  \sum
_{\sigma\in S_{\ell}}\left(  -1\right)  ^{\sigma}\prod_{i=1}^{\ell}%
\mathbf{e}_{\left(  \lambda^{t}\right)  _{i}-i+\sigma\left(  i\right)
}\right)  ^{\perp}\mathbf{f}=\sum_{\sigma\in S_{\ell}}\left(  -1\right)
^{\sigma}\left(  \prod_{i=1}^{\ell}\mathbf{e}_{\left(  \lambda^{t}\right)
_{i}-i+\sigma\left(  i\right)  }\right)  ^{\perp}\mathbf{f}.
\]
Applying the map $\delta$ to this equality, we obtain%
\begin{align}
\delta\left(  \left(  \mathbf{s}_{\lambda}\right)  ^{\perp}\mathbf{f}\right)
&  =\delta\left(  \sum_{\sigma\in S_{\ell}}\left(  -1\right)  ^{\sigma}\left(
\prod_{i=1}^{\ell}\mathbf{e}_{\left(  \lambda^{t}\right)  _{i}-i+\sigma\left(
i\right)  }\right)  ^{\perp}\mathbf{f}\right) \nonumber\\
&  =\sum_{\sigma\in S_{\ell}}\left(  -1\right)  ^{\sigma}\underbrace{\delta
\left(  \left(  \prod_{i=1}^{\ell}\mathbf{e}_{\left(  \lambda^{t}\right)
_{i}-i+\sigma\left(  i\right)  }\right)  ^{\perp}\mathbf{f}\right)
}_{\substack{\equiv\overline{\prod_{i=1}^{\ell}e_{\left(  \lambda^{t}\right)
_{i}-i+\sigma\left(  i\right)  }}\delta\left(  \mathbf{f}\right)
\operatorname{mod}Q_{\ell-1}\\\text{(by Lemma \ref{lem.coeffw.del2},
applied}\\\text{to }p=\ell\text{ and }i_{j}=\left(  \lambda^{t}\right)
_{j}-j+\sigma\left(  j\right)  \text{)}}}\ \ \ \ \ \ \ \ \ \ \left(
\text{since }\delta\text{ is }\mathbf{k}\text{-linear}\right) \nonumber\\
&  \equiv\sum_{\sigma\in S_{\ell}}\left(  -1\right)  ^{\sigma}\overline
{\prod_{i=1}^{\ell}e_{\left(  \lambda^{t}\right)  _{i}-i+\sigma\left(
i\right)  }}\delta\left(  \mathbf{f}\right)  \operatorname{mod}Q_{\ell-1}.
\label{pf.lem.coeffw.del3.LHS}%
\end{align}

On the other hand, (\ref{pf.lem.coeffw.del3.jt}) is an identity in $\Lambda$.
Evaluating both of its sides at the $k$ variables $x_{1},x_{2},\ldots,x_{k}$,
we obtain%
\[
s_{\lambda}=\sum_{\sigma\in S_{\ell}}\left(  -1\right)  ^{\sigma}\prod
_{i=1}^{\ell}e_{\left(  \lambda^{t}\right)  _{i}-i+\sigma\left(  i\right)  }.
\]
Hence,%
\[
\overline{s_{\lambda}}\delta\left(  \mathbf{f}\right)  =\overline{\sum
_{\sigma\in S_{\ell}}\left(  -1\right)  ^{\sigma}\prod_{i=1}^{\ell}e_{\left(
\lambda^{t}\right)  _{i}-i+\sigma\left(  i\right)  }}\delta\left(
\mathbf{f}\right)  =\sum_{\sigma\in S_{\ell}}\left(  -1\right)  ^{\sigma
}\overline{\prod_{i=1}^{\ell}e_{\left(  \lambda^{t}\right)  _{i}%
-i+\sigma\left(  i\right)  }}\delta\left(  \mathbf{f}\right)  .
\]
Thus, (\ref{pf.lem.coeffw.del3.LHS}) rewrites as $\delta\left(  \left(
\mathbf{s}_{\lambda}\right)  ^{\perp}\mathbf{f}\right)  \equiv\overline
{s_{\lambda}}\delta\left(  \mathbf{f}\right)  \operatorname{mod}Q_{\ell-1}$.
In other words, $\delta\left(  \left(  \mathbf{s}_{\lambda}\right)  ^{\perp
}\mathbf{f}\right)  \equiv\overline{s_{\lambda}}\delta\left(  \mathbf{f}%
\right)  \operatorname{mod}Q_{n-k-1}$ (since $\ell=n-k$). This proves Lemma
\ref{lem.coeffw.del3}.
\end{proof}

\begin{lemma}
\label{lem.coeffw.prod2}Let $\lambda\in P_{k,n}$ and $\mu\in P_{k,n}$. Then,%
\[
\operatorname*{coeff}\nolimits_{\omega}\left(  \overline{s_{\lambda}s_{\mu}%
}\right)  =%
\begin{cases}
1, & \text{if }\lambda=\mu^{\vee};\\
0, & \text{if }\lambda\neq\mu^{\vee}%
\end{cases}
.
\]

\end{lemma}

\begin{proof}
[Proof of Lemma \ref{lem.coeffw.prod2}.]From $\mu\in P_{k,n}$, we obtain
$\mu^{\vee}\in P_{k,n}$. Hence, $\mathbf{s}_{\mu^{\vee}}\in\mathcal{Z}$ and
\begin{align*}
\delta\left(  \mathbf{s}_{\mu^{\vee}}\right)   &  =\overline{s_{\left(
\mu^{\vee}\right)  ^{\vee}}}\ \ \ \ \ \ \ \ \ \ \left(  \text{by the
definition of }\delta\right) \\
&  =\overline{s_{\mu}}\ \ \ \ \ \ \ \ \ \ \left(  \text{since }\left(
\mu^{\vee}\right)  ^{\vee}=\mu\right)  .
\end{align*}

Also, Lemma \ref{lem.coeffw.del3} (applied to $\mathbf{f}=\mathbf{s}%
_{\mu^{\vee}}$) yields%
\[
\delta\left(  \left(  \mathbf{s}_{\lambda}\right)  ^{\perp}\mathbf{s}%
_{\mu^{\vee}}\right)  \equiv\overline{s_{\lambda}}\delta\left(  \mathbf{s}%
_{\mu^{\vee}}\right)  \operatorname{mod}Q_{n-k-1}%
\]
(since $\mathbf{s}_{\mu^{\vee}}\in\mathcal{Z}$). In other words,
$\delta\left(  \left(  \mathbf{s}_{\lambda}\right)  ^{\perp}\mathbf{s}%
_{\mu^{\vee}}\right)  -\overline{s_{\lambda}}\delta\left(  \mathbf{s}%
_{\mu^{\vee}}\right)  \in Q_{n-k-1}$. Hence,%
\[
\operatorname*{coeff}\nolimits_{\omega}\left(  \delta\left(  \left(
\mathbf{s}_{\lambda}\right)  ^{\perp}\mathbf{s}_{\mu^{\vee}}\right)
-\overline{s_{\lambda}}\delta\left(  \mathbf{s}_{\mu^{\vee}}\right)  \right)
\in\operatorname*{coeff}\nolimits_{\omega}\left(  Q_{n-k-1}\right)  =0
\]
(by Lemma \ref{lem.coeffw.coeffQ}). Thus,
\begin{align}
\operatorname*{coeff}\nolimits_{\omega}\left(  \delta\left(  \left(
\mathbf{s}_{\lambda}\right)  ^{\perp}\mathbf{s}_{\mu^{\vee}}\right)  \right)
&  =\operatorname*{coeff}\nolimits_{\omega}\left(  \overline{s_{\lambda}%
}\underbrace{\delta\left(  \mathbf{s}_{\mu^{\vee}}\right)  }_{=\overline
{s_{\mu}}}\right)  =\operatorname*{coeff}\nolimits_{\omega}\left(
\overline{s_{\lambda}}\overline{s_{\mu}}\right) \nonumber\\
&  =\operatorname*{coeff}\nolimits_{\omega}\left(  \overline{s_{\lambda}%
s_{\mu}}\right)  . \label{pf.lem.coeffw.prod2.3}%
\end{align}

Applying (\ref{eq.skewing.ss}) to $\lambda$ and $\mu^{\vee}$ instead of $\mu$
and $\lambda$, we obtain $\left(  \mathbf{s}_{\lambda}\right)  ^{\perp
}\mathbf{s}_{\mu^{\vee}}=\mathbf{s}_{\mu^{\vee}/\lambda}$. Thus,
(\ref{pf.lem.coeffw.prod2.3}) rewrites as%
\begin{equation}
\operatorname*{coeff}\nolimits_{\omega}\left(  \delta\left(  \mathbf{s}%
_{\mu^{\vee}/\lambda}\right)  \right)  =\operatorname*{coeff}\nolimits_{\omega
}\left(  \overline{s_{\lambda}s_{\mu}}\right)  . \label{pf.lem.coeffw.prod2.4}%
\end{equation}

We are in one of the following three cases:

\textit{Case 1:} We have $\lambda=\mu^{\vee}$.

\textit{Case 2:} We have $\lambda\subseteq\mu^{\vee}$ but not $\lambda
=\mu^{\vee}$.

\textit{Case 3:} We don't have $\lambda\subseteq\mu^{\vee}$.

Let us first consider Case 1. In this case, we have $\lambda=\mu^{\vee}$.
Thus, $\mathbf{s}_{\mu^{\vee}/\lambda}=\mathbf{s}_{\mu^{\vee}/\mu^{\vee}%
}=1=\mathbf{s}_{\varnothing}$ and thus%
\begin{align*}
\delta\left(  \mathbf{s}_{\mu^{\vee}/\lambda}\right)   &  =\delta\left(
\mathbf{s}_{\varnothing}\right)  =\overline{s_{\varnothing^{\vee}}%
}\ \ \ \ \ \ \ \ \ \ \left(  \text{by the definition of }\delta\right) \\
&  =\overline{s_{\omega}}\ \ \ \ \ \ \ \ \ \ \left(  \text{since }%
\varnothing^{\vee}=\omega\right)  .
\end{align*}
Therefore, $\operatorname*{coeff}\nolimits_{\omega}\left(  \delta\left(
\mathbf{s}_{\mu^{\vee}/\lambda}\right)  \right)  =\operatorname*{coeff}%
\nolimits_{\omega}\left(  \overline{s_{\omega}}\right)  =1$ (by the definition
of $\operatorname*{coeff}\nolimits_{\omega}$). Comparing this with%
\[%
\begin{cases}
1, & \text{if }\lambda=\mu^{\vee};\\
0, & \text{if }\lambda\neq\mu^{\vee}%
\end{cases}
=1\ \ \ \ \ \ \ \ \ \ \left(  \text{since }\lambda=\mu^{\vee}\right)  ,
\]
we obtain $\operatorname*{coeff}\nolimits_{\omega}\left(  \overline
{s_{\lambda}s_{\mu}}\right)  =%
\begin{cases}
1, & \text{if }\lambda=\mu^{\vee};\\
0, & \text{if }\lambda\neq\mu^{\vee}%
\end{cases}
$. Hence, Lemma \ref{lem.coeffw.prod2} is proven in Case 1.

Let us next consider Case 2. In this case, we have $\lambda\subseteq\mu^{\vee
}$ but not $\lambda=\mu^{\vee}$. Hence, $\left\vert \lambda\right\vert
<\left\vert \mu^{\vee}\right\vert $ and $\lambda\neq\mu^{\vee}$.

Now, every partition $\nu$ satisfying $\left\vert \lambda\right\vert
+\left\vert \nu\right\vert =\left\vert \mu^{\vee}\right\vert $ and
$\nu\subseteq\mu^{\vee}$ must satisfy%
\begin{equation}
\nu\in P_{k,n}\text{ and }\operatorname*{coeff}\nolimits_{\omega}\left(
\delta\left(  \mathbf{s}_{\nu}\right)  \right)  =0.
\label{pf.lem.coeffw.prod2.c2.1}%
\end{equation}

[\textit{Proof of (\ref{pf.lem.coeffw.prod2.c2.1}):} Let $\nu$ be a partition
satisfying $\left\vert \lambda\right\vert +\left\vert \nu\right\vert
=\left\vert \mu^{\vee}\right\vert $ and $\nu\subseteq\mu^{\vee}$. We must
prove (\ref{pf.lem.coeffw.prod2.c2.1}).

First of all, from $\nu\subseteq\mu^{\vee}$ and $\mu^{\vee}\in P_{k,n}$, we
obtain $\nu\in P_{k,n}$. It thus remains to show that $\operatorname*{coeff}%
\nolimits_{\omega}\left(  \delta\left(  \mathbf{s}_{\nu}\right)  \right)  =0$.

The definition of $\delta$ yields $\delta\left(  \mathbf{s}_{\nu}\right)
=\overline{s_{\nu^{\vee}}}$ (since $\nu\in P_{k,n}$). But $\left\vert
\lambda\right\vert +\left\vert \nu\right\vert =\left\vert \mu^{\vee
}\right\vert $ yields $\left\vert \nu\right\vert =\left\vert \mu^{\vee
}\right\vert -\left\vert \lambda\right\vert >0$ (since $\left\vert
\lambda\right\vert <\left\vert \mu^{\vee}\right\vert $).

But every partition $\kappa\in P_{k,n}$ satisfies $\left\vert \kappa^{\vee
}\right\vert =\underbrace{k\left(  n-k\right)  }_{=\left\vert \omega
\right\vert }-\left\vert \kappa\right\vert =\left\vert \omega\right\vert
-\left\vert \kappa\right\vert $. Applying this to $\kappa=\nu$, we obtain
\[
\left\vert \nu^{\vee}\right\vert =\left\vert \omega\right\vert
-\underbrace{\left\vert \nu\right\vert }_{>0}<\left\vert \omega\right\vert .
\]
Hence, $\left\vert \nu^{\vee}\right\vert \neq\left\vert \omega\right\vert $,
so that $\nu^{\vee}\neq\omega$.

But the definition of $\operatorname*{coeff}\nolimits_{\omega}$ yields
$\operatorname*{coeff}\nolimits_{\omega}\left(  \overline{s_{\nu^{\vee}}%
}\right)  =%
\begin{cases}
1, & \text{if }\nu^{\vee}=\omega;\\
0, & \text{if }\nu^{\vee}\neq\omega
\end{cases}
=0$ (since $\nu^{\vee}\neq\omega$). In view of $\delta\left(  \mathbf{s}_{\nu
}\right)  =\overline{s_{\nu^{\vee}}}$, this rewrites as $\operatorname*{coeff}%
\nolimits_{\omega}\left(  \delta\left(  \mathbf{s}_{\nu}\right)  \right)  =0$.
This completes the proof of (\ref{pf.lem.coeffw.prod2.c2.1}).]

Proposition \ref{prop.LR.props} \textbf{(a)} (applied to $\mu^{\vee}$ and
$\lambda$ instead of $\lambda$ and $\mu$) yields
\begin{align*}
\mathbf{s}_{\mu^{\vee}/\lambda}  &  =\sum_{\nu\text{ is a partition}%
}c_{\lambda,\nu}^{\mu^{\vee}}\mathbf{s}_{\nu}\\
&  =\sum_{\substack{\nu\text{ is a partition;}\\\nu\subseteq\mu^{\vee}%
}}c_{\lambda,\nu}^{\mu^{\vee}}\mathbf{s}_{\nu}+\sum_{\substack{\nu\text{ is a
partition;}\\\text{not }\nu\subseteq\mu^{\vee}}}\underbrace{c_{\lambda,\nu
}^{\mu^{\vee}}}_{\substack{=0\\\text{(by Proposition \ref{prop.LR.props}
\textbf{(b)},}\\\text{applied to }\mu^{\vee}\text{ and }\lambda\text{ instead
of }\lambda\text{ and }\mu\text{)}}}\mathbf{s}_{\nu}\\
&  =\sum_{\substack{\nu\text{ is a partition;}\\\nu\subseteq\mu^{\vee}%
}}c_{\lambda,\nu}^{\mu^{\vee}}\mathbf{s}_{\nu}\\
&  =\sum_{\substack{\nu\text{ is a partition;}\\\nu\subseteq\mu^{\vee
};\\\left\vert \lambda\right\vert +\left\vert \nu\right\vert =\left\vert
\mu^{\vee}\right\vert }}c_{\lambda,\nu}^{\mu^{\vee}}\mathbf{s}_{\nu}%
+\sum_{\substack{\nu\text{ is a partition;}\\\nu\subseteq\mu^{\vee
};\\\text{not }\left\vert \lambda\right\vert +\left\vert \nu\right\vert
=\left\vert \mu^{\vee}\right\vert }}\underbrace{c_{\lambda,\nu}^{\mu^{\vee}}%
}_{\substack{=0\\\text{(by Proposition \ref{prop.LR.props} \textbf{(c)}%
,}\\\text{applied to }\mu^{\vee}\text{ and }\lambda\text{ instead of }%
\lambda\text{ and }\mu\text{)}}}\mathbf{s}_{\nu}\\
&  =\sum_{\substack{\nu\text{ is a partition;}\\\nu\subseteq\mu^{\vee
};\\\left\vert \lambda\right\vert +\left\vert \nu\right\vert =\left\vert
\mu^{\vee}\right\vert }}c_{\lambda,\nu}^{\mu^{\vee}}\mathbf{s}_{\nu}.
\end{align*}
Applying the map $\delta$ to this equality, we find%
\begin{align*}
\delta\left(  \mathbf{s}_{\mu^{\vee}/\lambda}\right)   &  =\delta\left(
\sum_{\substack{\nu\text{ is a partition;}\\\nu\subseteq\mu^{\vee
};\\\left\vert \lambda\right\vert +\left\vert \nu\right\vert =\left\vert
\mu^{\vee}\right\vert }}c_{\lambda,\nu}^{\mu^{\vee}}\mathbf{s}_{\nu}\right)
=\sum_{\substack{\nu\text{ is a partition;}\\\nu\subseteq\mu^{\vee
};\\\left\vert \lambda\right\vert +\left\vert \nu\right\vert =\left\vert
\mu^{\vee}\right\vert }}c_{\lambda,\nu}^{\mu^{\vee}}\delta\left(
\mathbf{s}_{\nu}\right) \\
&  \ \ \ \ \ \ \ \ \ \ \left(
\begin{array}
[c]{c}%
\text{since every partition }\nu\text{ satisfying }\nu\subseteq\mu^{\vee
}\text{ and }\left\vert \lambda\right\vert +\left\vert \nu\right\vert
=\left\vert \mu^{\vee}\right\vert \\
\text{must satisfy }\nu\in P_{k,n}\text{ (by (\ref{pf.lem.coeffw.prod2.c2.1}))
and thus }\mathbf{s}_{\nu}\in\mathcal{Z}%
\end{array}
\right)  .
\end{align*}
Applying the map $\operatorname*{coeff}\nolimits_{\omega}$ to this equality,
we find%
\begin{align*}
\operatorname*{coeff}\nolimits_{\omega}\left(  \delta\left(  \mathbf{s}%
_{\mu^{\vee}/\lambda}\right)  \right)   &  =\operatorname*{coeff}%
\nolimits_{\omega}\left(  \sum_{\substack{\nu\text{ is a partition;}%
\\\nu\subseteq\mu^{\vee};\\\left\vert \lambda\right\vert +\left\vert
\nu\right\vert =\left\vert \mu^{\vee}\right\vert }}c_{\lambda,\nu}^{\mu^{\vee
}}\delta\left(  \mathbf{s}_{\nu}\right)  \right) \\
&  =\sum_{\substack{\nu\text{ is a partition;}\\\nu\subseteq\mu^{\vee
};\\\left\vert \lambda\right\vert +\left\vert \nu\right\vert =\left\vert
\mu^{\vee}\right\vert }}c_{\lambda,\nu}^{\mu^{\vee}}%
\underbrace{\operatorname*{coeff}\nolimits_{\omega}\left(  \delta\left(
\mathbf{s}_{\nu}\right)  \right)  }_{\substack{=0\\\text{(by
(\ref{pf.lem.coeffw.prod2.c2.1}))}}}=0.
\end{align*}
Comparing this with%
\[%
\begin{cases}
1, & \text{if }\lambda=\mu^{\vee};\\
0, & \text{if }\lambda\neq\mu^{\vee}%
\end{cases}
=0\ \ \ \ \ \ \ \ \ \ \left(  \text{since }\lambda\neq\mu^{\vee}\right)  ,
\]
we obtain $\operatorname*{coeff}\nolimits_{\omega}\left(  \overline
{s_{\lambda}s_{\mu}}\right)  =%
\begin{cases}
1, & \text{if }\lambda=\mu^{\vee};\\
0, & \text{if }\lambda\neq\mu^{\vee}%
\end{cases}
$. Hence, Lemma \ref{lem.coeffw.prod2} is proven in Case 2.

Let us finally consider Case 3. In this case, we don't have $\lambda
\subseteq\mu^{\vee}$. Hence, we don't have $\lambda=\mu^{\vee}$ either. Thus,
$\lambda\neq\mu^{\vee}$.

Also, $\mathbf{s}_{\mu^{\vee}/\lambda}=0$ (since we don't have $\lambda
\subseteq\mu^{\vee}$). Thus,%
\[
\operatorname*{coeff}\nolimits_{\omega}\left(  \delta\left(
\underbrace{\mathbf{s}_{\mu^{\vee}/\lambda}}_{=0}\right)  \right)
=\operatorname*{coeff}\nolimits_{\omega}\left(  \delta\left(  0\right)
\right)  =0.
\]
Comparing this with%
\[%
\begin{cases}
1, & \text{if }\lambda=\mu^{\vee};\\
0, & \text{if }\lambda\neq\mu^{\vee}%
\end{cases}
=0\ \ \ \ \ \ \ \ \ \ \left(  \text{since }\lambda\neq\mu^{\vee}\right)  ,
\]
we obtain $\operatorname*{coeff}\nolimits_{\omega}\left(  \overline
{s_{\lambda}s_{\mu}}\right)  =%
\begin{cases}
1, & \text{if }\lambda=\mu^{\vee};\\
0, & \text{if }\lambda\neq\mu^{\vee}%
\end{cases}
$. Hence, Lemma \ref{lem.coeffw.prod2} is proven in Case 3.

We have now proven Lemma \ref{lem.coeffw.prod2} in all three Cases 1, 2 and 3.
Thus, Lemma \ref{lem.coeffw.prod2} always holds.
\end{proof}

\begin{proof}
[Proof of Theorem \ref{thm.coeffw}.]Write $f\in\mathcal{S}/I$ in the form
$f=\sum_{\lambda\in P_{k,n}}\alpha_{\lambda}\overline{s_{\lambda}}$ with
$\alpha_{\lambda}\in\mathbf{k}$. (This is possible, since $\left(
\overline{s_{\lambda}}\right)  _{\lambda\in P_{k,n}}$ is a basis of the
$\mathbf{k}$-module $\mathcal{S}/I$.) Then, the definition of
$\operatorname*{coeff}\nolimits_{\nu^{\vee}}$ yields $\operatorname*{coeff}%
\nolimits_{\nu^{\vee}}\left(  f\right)  =\alpha_{\nu^{\vee}}$.

On the other hand,%
\begin{align*}
\operatorname*{coeff}\nolimits_{\omega}\left(  \overline{s_{\nu}%
}\underbrace{f}_{=\sum_{\lambda\in P_{k,n}}\alpha_{\lambda}\overline
{s_{\lambda}}}\right)   &  =\operatorname*{coeff}\nolimits_{\omega}\left(
\overline{s_{\nu}}\sum_{\lambda\in P_{k,n}}\alpha_{\lambda}\overline
{s_{\lambda}}\right)  =\sum_{\lambda\in P_{k,n}}\alpha_{\lambda}%
\operatorname*{coeff}\nolimits_{\omega}\left(  \underbrace{\overline{s_{\nu}%
}\overline{s_{\lambda}}}_{=\overline{s_{\lambda}s_{\nu}}}\right) \\
&  =\sum_{\lambda\in P_{k,n}}\alpha_{\lambda}\underbrace{\operatorname*{coeff}%
\nolimits_{\omega}\left(  \overline{s_{\lambda}s_{\nu}}\right)  }_{\substack{=%
\begin{cases}
1, & \text{if }\lambda=\nu^{\vee};\\
0, & \text{if }\lambda\neq\nu^{\vee}%
\end{cases}
\\\text{(by Lemma \ref{lem.coeffw.prod2}, applied to }\mu=\nu\text{)}}}\\
&  =\sum_{\lambda\in P_{k,n}}\alpha_{\lambda}%
\begin{cases}
1, & \text{if }\lambda=\nu^{\vee};\\
0, & \text{if }\lambda\neq\nu^{\vee}%
\end{cases}
=\alpha_{\nu^{\vee}}%
\end{align*}
(since $\nu^{\vee}\in P_{k,n}$). Comparing this with $\operatorname*{coeff}%
\nolimits_{\nu^{\vee}}\left(  f\right)  =\alpha_{\nu^{\vee}}$, we obtain
$\operatorname*{coeff}\nolimits_{\omega}\left(  \overline{s_{\nu}}f\right)
=\operatorname*{coeff}\nolimits_{\nu^{\vee}}\left(  f\right)  $. This proves
Theorem \ref{thm.coeffw}.
\end{proof}

\begin{definition}
For any three partitions $\alpha,\beta,\gamma\in P_{k,n}$, let $g_{\alpha
,\beta,\gamma}=\operatorname*{coeff}\nolimits_{\gamma^{\vee}}\left(
\overline{s_{\alpha}}\overline{s_{\beta}}\right)  \in\mathbf{k}$.
\end{definition}

These scalars $g_{\alpha,\beta,\gamma}$ are thus the structure constants of
the $\mathbf{k}$-algebra $\mathcal{S}/I$ in the basis $\left(  \overline
{s_{\lambda}}\right)  _{\lambda\in P_{k,n}}$ (although slightly reindexed). As
a consequence of Theorem \ref{thm.coeffw}, we obtain the following $S_{3}%
$-property of these structure constants:

\begin{corollary}
\label{cor.S3sym}We have%
\[
g_{\alpha,\beta,\gamma}=g_{\alpha,\gamma,\beta}=g_{\beta,\alpha,\gamma
}=g_{\beta,\gamma,\alpha}=g_{\gamma,\alpha,\beta}=g_{\gamma,\beta,\alpha
}=\operatorname*{coeff}\nolimits_{\omega}\left(  \overline{s_{\alpha}s_{\beta
}s_{\gamma}}\right)
\]
for any $\alpha,\beta,\gamma\in P_{k,n}$.
\end{corollary}

\begin{proof}
[Proof of Corollary \ref{cor.S3sym}.]Let $\alpha,\beta,\gamma\in P_{k,n}$. It
clearly suffices to prove $g_{\alpha,\beta,\gamma}=\operatorname*{coeff}%
\nolimits_{\omega}\left(  \overline{s_{\alpha}s_{\beta}s_{\gamma}}\right)  $,
since the rest of the claim then follows by analogy.

Theorem \ref{thm.coeffw} (applied to $\nu=\gamma$ and $f=\overline{s_{\alpha}%
}\overline{s_{\beta}}$) yields%
\[
\operatorname*{coeff}\nolimits_{\omega}\left(  \overline{s_{\gamma}}%
\overline{s_{\alpha}}\overline{s_{\beta}}\right)  =\operatorname*{coeff}%
\nolimits_{\gamma^{\vee}}\left(  \overline{s_{\alpha}}\overline{s_{\beta}%
}\right)  =g_{\alpha,\beta,\gamma}%
\]
(by the definition of $g_{\alpha,\beta,\gamma}$). Thus, $g_{\alpha
,\beta,\gamma}=\operatorname*{coeff}\nolimits_{\omega}\left(
\underbrace{\overline{s_{\gamma}}\overline{s_{\alpha}}\overline{s_{\beta}}%
}_{=\overline{s_{\alpha}s_{\beta}s_{\gamma}}}\right)  =\operatorname*{coeff}%
\nolimits_{\omega}\left(  \overline{s_{\alpha}s_{\beta}s_{\gamma}}\right)  $.
This completes our proof of Corollary \ref{cor.S3sym}.
\end{proof}

\section{\label{sect.redh}Complete homogeneous symmetric polynomials}

In this section, we shall further explore the projections $\overline{h_{i}}$
of complete homogeneous symmetric polynomials $h_{i}$ onto $\mathcal{S}/I$.
This exploration will culminate in a second proof of Theorem \ref{thm.coeffw}.

\begin{convention}
Convention \ref{conv.symmetry-conv} remains in place for the whole Section
\ref{sect.redh}.

We shall also use all the notations introduced in Section \ref{sect.symmetry}.

If $j\in\mathbb{N}$, then the expression \textquotedblleft$1^{j}%
$\textquotedblright\ in a tuple stands for $j$ consecutive entries equal to
$1$ (that is, $\underbrace{1,1,\ldots,1}_{j\text{ times}}$). Thus, $\left(
m,1^{j}\right)  =\left(  m,\underbrace{1,1,\ldots,1}_{j\text{ times}}\right)
$ for any $m\in\mathbb{N}$ and $j\in\mathbb{N}$.
\end{convention}

\subsection{A reduction formula for $h_{n+m}$}

The following result helps us reduce complete homogeneous symmetric
polynomials $h_{n+m}$ modulo the ideal $I$:

\begin{proposition}
\label{prop.redh.1}Let $m$ be a positive integer. Then,%
\[
h_{n+m}\equiv\sum_{j=0}^{k-1}\left(  -1\right)  ^{j}a_{k-j}s_{\left(
m,1^{j}\right)  }\operatorname{mod}I.
\]

\end{proposition}

We shall derive Proposition \ref{prop.redh.1} from the following identity
between symmetric functions in $\Lambda$:

\begin{proposition}
\label{prop.redh.Lam}Let $m$ be a positive integer. Then,%
\[
\mathbf{h}_{n+m}=\sum_{j=0}^{n}\left(  -1\right)  ^{j}\mathbf{h}%
_{n-j}\mathbf{s}_{\left(  m,1^{j}\right)  }.
\]

\end{proposition}

\begin{proof}
[Proof of Proposition \ref{prop.redh.Lam}.]Let $j\in\mathbb{N}$.

In \cite[Exercise 2.9.14(b)]{GriRei18}, it is shown that%
\begin{equation}
\sum_{i=0}^{b}\left(  -1\right)  ^{i}\mathbf{h}_{a+i+1}\mathbf{e}%
_{b-i}=\mathbf{s}_{\left(  a+1,1^{b}\right)  } \label{pf.prop.redh.Lam.ab}%
\end{equation}
for all $a,b\in\mathbb{N}$. Applying this equality to $a=m-1$ and $b=j$, we
obtain%
\begin{equation}
\sum_{i=0}^{j}\left(  -1\right)  ^{i}\mathbf{h}_{m+i}\mathbf{e}_{j-i}%
=\mathbf{s}_{\left(  m,1^{j}\right)  }. \label{pf.prop.redh.Lam.1}%
\end{equation}

Now, forget that we fixed $j$. We thus have proven (\ref{pf.prop.redh.Lam.1})
for each $j\in\mathbb{N}$.

Also, for any $N\in\mathbb{N}$, we have%
\[
\sum_{\substack{\left(  i,j\right)  \in\mathbb{N}^{2};\\i+j=N}}\left(
-1\right)  ^{i}\mathbf{e}_{i}\mathbf{h}_{j}=\delta_{0,N}%
\]
(where $\delta_{0,N}$ is a Kronecker delta). (This is \cite[(2.4.4)]%
{GriRei18}, with $n$ renamed as $N$.) Thus, for any $N\in\mathbb{N}$, we have%
\begin{align*}
\delta_{0,N}  &  =\sum_{\substack{\left(  i,j\right)  \in\mathbb{N}%
^{2};\\i+j=N}}\left(  -1\right)  ^{i}\mathbf{e}_{i}\mathbf{h}_{j}%
=\sum_{\substack{\left(  i,j\right)  \in\mathbb{N}^{2};\\i+j=N}}\left(
-1\right)  ^{i}\mathbf{h}_{j}\mathbf{e}_{i}\\
&  =\sum_{i=0}^{N}\left(  -1\right)  ^{i}\mathbf{h}_{N-i}\mathbf{e}%
_{i}\ \ \ \ \ \ \ \ \ \ \left(
\begin{array}
[c]{c}%
\text{here, we have substituted }\left(  i,N-i\right) \\
\text{for }\left(  i,j\right)  \text{ in the sum}%
\end{array}
\right) \\
&  =\sum_{j=0}^{N}\left(  -1\right)  ^{j}\mathbf{h}_{N-j}\mathbf{e}%
_{j}\ \ \ \ \ \ \ \ \ \ \left(
\begin{array}
[c]{c}%
\text{here, we have renamed the}\\
\text{summation index }i\text{ as }j
\end{array}
\right)  .
\end{align*}
Thus, for any $N\in\mathbb{N}$, we have%
\begin{equation}
\sum_{j=0}^{N}\left(  -1\right)  ^{j}\mathbf{h}_{N-j}\mathbf{e}_{j}%
=\delta_{0,N}. \label{pf.prop.redh.Lam.3}%
\end{equation}

For each $i\in\left\{  0,1,\ldots,n\right\}  $, we have%
\begin{align}
&  \sum_{j=i}^{n}\left(  -1\right)  ^{j-i}\mathbf{h}_{n-j}\mathbf{e}%
_{j-i}\nonumber\\
&  =\sum_{j=0}^{n-i}\left(  -1\right)  ^{j}\mathbf{h}_{n-i-j}\mathbf{e}%
_{j}\ \ \ \ \ \ \ \ \ \ \left(  \text{here, we have substituted }j\text{ for
}j-i\text{ in the sum}\right) \nonumber\\
&  =\delta_{0,n-i}\ \ \ \ \ \ \ \ \ \ \left(  \text{by
(\ref{pf.prop.redh.Lam.3}), applied to }n-i\text{ instead of }N\right)
\nonumber\\
&  =\delta_{i,n}. \label{pf.prop.redh.Lam.5}%
\end{align}

Now,%
\begin{align*}
&  \sum_{j=0}^{n}\left(  -1\right)  ^{j}\mathbf{h}_{n-j}\underbrace{\mathbf{s}%
_{\left(  m,1^{j}\right)  }}_{\substack{=\sum_{i=0}^{j}\left(  -1\right)
^{i}\mathbf{h}_{m+i}\mathbf{e}_{j-i}\\\text{(by (\ref{pf.prop.redh.Lam.1}))}%
}}\\
&  =\sum_{j=0}^{n}\left(  -1\right)  ^{j}\mathbf{h}_{n-j}\sum_{i=0}^{j}\left(
-1\right)  ^{i}\mathbf{h}_{m+i}\mathbf{e}_{j-i}=\underbrace{\sum_{j=0}^{n}%
\sum_{i=0}^{j}}_{=\sum_{i=0}^{n}\sum_{j=i}^{n}}\left(  -1\right)
^{j}\mathbf{h}_{n-j}\left(  -1\right)  ^{i}\mathbf{h}_{m+i}\mathbf{e}_{j-i}\\
&  =\sum_{i=0}^{n}\sum_{j=i}^{n}\left(  -1\right)  ^{j}\mathbf{h}_{n-j}\left(
-1\right)  ^{i}\mathbf{h}_{m+i}\mathbf{e}_{j-i}=\sum_{i=0}^{n}\mathbf{h}%
_{m+i}\sum_{j=i}^{n}\left(  -1\right)  ^{j-i}\mathbf{h}_{n-j}\mathbf{e}%
_{j-i}\\
&  =\sum_{i=0}^{n}\mathbf{h}_{m+i}\underbrace{\sum_{j=i}^{n}\left(  -1\right)
^{j-i}\mathbf{h}_{n-j}\mathbf{e}_{j-i}}_{\substack{=\delta_{i,n}\\\text{(by
(\ref{pf.prop.redh.Lam.5}))}}}=\sum_{i=0}^{n}\mathbf{h}_{m+i}\delta
_{i,n}=\mathbf{h}_{m+n}=\mathbf{h}_{n+m}.
\end{align*}
This proves Proposition \ref{prop.redh.Lam}.
\end{proof}

\begin{proof}
[Proof of Proposition \ref{prop.redh.1}.]For each integer $j\geq k$, we have%
\begin{equation}
s_{\left(  m,1^{j}\right)  }=0. \label{pf.prop.redh.1.=0}%
\end{equation}

[\textit{Proof of (\ref{pf.prop.redh.1.=0}):} Let $j\geq k$ be an integer.
Then, the partition $\left(  m,1^{j}\right)  $ has $j+1$ parts; thus, this
partition has more than $k$ parts (since $j+1>j\geq k$). Thus,
(\ref{eq.slam=0-too-long}) (applied to $\lambda=\left(  m,1^{j}\right)  $)
yields $s_{\left(  m,1^{j}\right)  }=0$. This proves (\ref{pf.prop.redh.1.=0}).]

Proposition \ref{prop.redh.Lam} yields%
\[
\mathbf{h}_{n+m}=\sum_{j=0}^{n}\left(  -1\right)  ^{j}\mathbf{h}%
_{n-j}\mathbf{s}_{\left(  m,1^{j}\right)  }.
\]
This is an identity in $\Lambda$. Evaluating both of its sides at the $k$
variables $x_{1},x_{2},\ldots,x_{k}$, we obtain%
\begin{align*}
h_{n+m}  &  =\sum_{j=0}^{n}\left(  -1\right)  ^{j}h_{n-j}s_{\left(
m,1^{j}\right)  }\\
&  =\sum_{j=0}^{k-1}\left(  -1\right)  ^{j}h_{n-j}s_{\left(  m,1^{j}\right)
}+\sum_{j=k}^{n}\left(  -1\right)  ^{j}h_{n-j}\underbrace{s_{\left(
m,1^{j}\right)  }}_{\substack{=0\\\text{(by (\ref{pf.prop.redh.1.=0}))}}}\\
&  =\sum_{j=0}^{k-1}\left(  -1\right)  ^{j}\underbrace{h_{n-j}}%
_{\substack{=h_{n-k+\left(  k-j\right)  }\equiv a_{k-j}\operatorname{mod}%
I\\\text{(by (\ref{eq.h=amodI}), applied to }k-j\\\text{instead of }j\text{)}%
}}s_{\left(  m,1^{j}\right)  }\equiv\sum_{j=0}^{k-1}\left(  -1\right)
^{j}a_{k-j}s_{\left(  m,1^{j}\right)  }\operatorname{mod}I.
\end{align*}
This proves Proposition \ref{prop.redh.1}.
\end{proof}

\subsection{Lemmas on free modules}

Next, we state a basic lemma from commutative algebra:

\begin{lemma}
\label{lem.freemod.2}Let $r\in\mathbb{N}$. Let $X$ and $Y$ be two free
$\mathbf{k}$-modules of rank $r$. Then, every surjective $\mathbf{k}$-linear
map from $X$ to $Y$ is a $\mathbf{k}$-module isomorphism.
\end{lemma}

\begin{proof}
[Proof of Lemma \ref{lem.freemod.2}.]Let $f:X\rightarrow Y$ be a surjective
$\mathbf{k}$-linear map from $X$ to $Y$. We must prove that $f$ is a
$\mathbf{k}$-module isomorphism.

There is clearly a $\mathbf{k}$-module isomorphism $j:Y\rightarrow X$ (since
$X$ and $Y$ are free $\mathbf{k}$-modules of the same rank). Consider this
$j$. Then, the composition $j\circ f$ is surjective (since $j$ and $f$ are
surjective), and thus is a surjective endomorphism of the finitely generated
$\mathbf{k}$-module $X$. But \cite[Exercise 2.5.18(a)]{GriRei18} shows that
any surjective endomorphism of a finitely generated $\mathbf{k}$-module is a
$\mathbf{k}$-module isomorphism. Hence, we conclude that $j\circ f$ is a
$\mathbf{k}$-module isomorphism. Thus, $f$ is a $\mathbf{k}$-module
isomorphism (since $j$ is a $\mathbf{k}$-module isomorphism). This proves
Lemma \ref{lem.freemod.2}.
\end{proof}

\begin{lemma}
\label{lem.freemod.3}Let $Z$ be a $\mathbf{k}$-module. Let $U$, $X$ and $Y$ be
$\mathbf{k}$-submodules of $Z$ such that $Z=X\oplus Y$ and $X\subseteq U$. Let
$r\in\mathbb{N}$. Assume that the $\mathbf{k}$-module $X$ has a basis with $r$
elements, whereas the $\mathbf{k}$-module $U$ can be spanned by $r$ elements.
Then, $X=U$.
\end{lemma}

\begin{proof}
[Proof of Lemma \ref{lem.freemod.3}.]Let $\pi:Z\rightarrow X$ be the canonical
projection from the direct sum $Z=X\oplus Y$ onto its addend $X$. Let
$\iota:X\rightarrow U$ be the canonical injection. Then, the composition%
\[
X\overset{\iota}{\longrightarrow}U\overset{\pi\mid_{U}}{\longrightarrow}X
\]
is just $\operatorname*{id}\nolimits_{X}$ (since $\pi\mid_{X}%
=\operatorname*{id}\nolimits_{X}$). Hence, the map $\pi\mid_{U}$ is surjective.

We assumed that the $\mathbf{k}$-module $U$ can be spanned by $r$ elements.
Thus, there is a surjective $\mathbf{k}$-module homomorphism $u:\mathbf{k}%
^{r}\rightarrow U$. Consider this $u$.

Both $\mathbf{k}$-modules $\mathbf{k}^{r}$ and $X$ are free of rank $r$ (since
$X$ has a basis with $r$ elements). The composition
\[
\mathbf{k}^{r}\overset{u}{\longrightarrow}U\overset{\pi\mid_{U}%
}{\longrightarrow}X
\]
is surjective (since both $u$ and $\pi\mid_{U}$ are surjective), and thus is a
$\mathbf{k}$-module isomorphism (by Lemma \ref{lem.freemod.2}, applied to
$\mathbf{k}^{r}$ and $X$ instead of $X$ and $Y$). Hence, it is injective.
Thus, $u$ is injective. Since $u$ is also surjective, we thus conclude that
$u$ is bijective, and therefore a $\mathbf{k}$-module isomorphism. Since both
$u$ and the composition $\mathbf{k}^{r}\overset{u}{\longrightarrow
}U\overset{\pi\mid_{U}}{\longrightarrow}X$ are $\mathbf{k}$-module
isomorphisms, we now conclude that the map $\pi\mid_{U}$ is a $\mathbf{k}%
$-module isomorphism. Hence, it has an inverse. But this inverse must be
$\iota$ (since the composition $X\overset{\iota}{\longrightarrow}%
U\overset{\pi\mid_{U}}{\longrightarrow}X$ is $\operatorname*{id}\nolimits_{X}%
$). Thus, $\iota$ is a $\mathbf{k}$-module isomorphism, too. Thus, in
particular, $\iota$ is surjective. Therefore, $U=\iota\left(  X\right)  =X$.
This proves Lemma \ref{lem.freemod.3}.
\end{proof}

\subsection{The symmetric polynomials $h_{\nu}$}

\begin{definition}
Let $\ell\in\mathbb{N}$, and let $\nu=\left(  \nu_{1},\nu_{2},\ldots,\nu
_{\ell}\right)  \in\mathbb{Z}^{\ell}$ be any $\ell$-tuple of integers. Then,
we define the symmetric polynomial $h_{\nu}\in\mathcal{S}$ as follows:%
\[
h_{\nu}=h_{\nu_{1}}h_{\nu_{2}}\cdots h_{\nu_{\ell}}.
\]

\end{definition}

Note that the polynomial $h_{\nu}$ does not change if we permute the entries
of the $\ell$-tuple $\nu$. If an $\ell$-tuple $\nu$ of integers contains any
negative entries, then $h_{\nu}=0$ (since $h_{i}=0$ for any $i<0$). Also, if
an $\ell$-tuple $\nu$ of integers contains any entry $=0$, then we can remove
this entry without changing $h_{\nu}$ (since $h_{0}=1$).

\subsection{The submodules $L_{p}$ and $H_{p}$ of $\mathcal{S}/I$}

It is time to define two further filtrations of the $\mathbf{k}$-module
$\mathcal{S}/I$ (in addition to the filtration $\left(  Q_{p}\right)
_{p\in\mathbb{Z}}$ from Definition \ref{def.Qp}):

\begin{definition}
\label{def.LpHp}\textbf{(a)} If $\lambda$ is a partition, then $\ell\left(
\lambda\right)  $ shall denote the \textit{length} of $\lambda$; this is
defined as the number of positive entries of $\lambda$. Note that $\ell\left(
\lambda\right)  \leq k$ for each $\lambda\in P_{k,n}$.

\textbf{(b)} For each $p\in\mathbb{Z}$, we let $L_{p}$ denote the $\mathbf{k}%
$-submodule of $\mathcal{S}/I$ spanned by the $\overline{s_{\lambda}}$ with
$\lambda\in P_{k,n}$ satisfying $\ell\left(  \lambda\right)  \leq p$.

\textbf{(c)} For each $p\in\mathbb{Z}$, we let $H_{p}$ denote the $\mathbf{k}%
$-submodule of $\mathcal{S}/I$ spanned by the $\overline{h_{\lambda}}$ with
$\lambda\in P_{k,n}$ satisfying $\ell\left(  \lambda\right)  \leq p$.
\end{definition}

The only partition $\lambda$ satisfying $\ell\left(  \lambda\right)  \leq0$ is
the empty partition $\varnothing=\left(  {}\right)  $; it belongs to $P_{k,n}$
and satisfies $\overline{s_{\lambda}}=1$. Hence, $L_{0}$ is the $\mathbf{k}%
$-submodule of $\mathcal{S}/I$ spanned by $1$. Similarly, $H_{0}$ is the same
$\mathbf{k}$-submodule.

Also, $L_{k}$ is the $\mathbf{k}$-submodule of $\mathcal{S}/I$ spanned by all
$\overline{s_{\lambda}}$ with $\lambda\in P_{k,n}$ (because each $\lambda\in
P_{k,n}$ satisfies $\ell\left(  \lambda\right)  \leq k$). But the latter
$\mathbf{k}$-submodule is $\mathcal{S}/I$ itself (by Theorem \ref{thm.S/J}).
Thus, we conclude that $L_{k}$ is $\mathcal{S}/I$ itself. In other words,%
\[
L_{k}=\mathcal{S}/I.
\]

Clearly, $L_{0}\subseteq L_{1}\subseteq L_{2}\subseteq\cdots$ and
$H_{0}\subseteq H_{1}\subseteq H_{2}\subseteq\cdots$. We shall soon see that
the families $\left(  L_{p}\right)  _{p\in\mathbb{Z}}$ and $\left(
H_{p}\right)  _{p\in\mathbb{Z}}$ are identical (Proposition \ref{prop.Lp=Hp})
and are filtrations of the $\mathbf{k}$-algebra $\mathcal{S}/I$ (Proposition
\ref{prop.Lpalg}). First let us show a basic fact:

\begin{lemma}
\label{lem.Hp-nu}Let $p\in\mathbb{N}$ be such that $p\leq k$. Let $\nu=\left(
\nu_{1},\nu_{2},\ldots,\nu_{p}\right)  \in\mathbb{Z}^{p}$. Assume that
$\nu_{i}\leq n$ for each $i\in\left\{  1,2,\ldots,p\right\}  $. Then,
$\overline{h_{\nu}}\in H_{p}$.
\end{lemma}

(The condition \textquotedblleft$p\leq k$\textquotedblright\ can be removed
from this lemma, but we aren't yet at the point where this is easy to see. We
will show this in Proposition \ref{prop.Hp-nu2} below.)

\begin{proof}
[Proof of Lemma \ref{lem.Hp-nu}.]We WLOG assume that $\nu_{1}\geq\nu_{2}%
\geq\cdots\geq\nu_{p}$ (since otherwise, we can just permute the entries of
$\nu$ to achieve this). Let $j$ be the number of $i\in\left\{  1,2,\ldots
,p\right\}  $ satisfying $\nu_{i}>n-k$. Then,%
\[
\nu_{1}\geq\nu_{2}\geq\cdots\geq\nu_{j}>n-k\geq\nu_{j+1}\geq\nu_{j+2}%
\geq\cdots\geq\nu_{p}.
\]

We WLOG assume that all of the $\nu_{1},\nu_{2},\ldots,\nu_{p}$ are
nonnegative (since otherwise, we have $h_{\nu}=0$ and thus $\overline{h_{\nu}%
}=0\in H_{p}$).

Now,
\begin{equation}
\overline{h_{\nu_{i}}}\in\mathbf{k}\ \ \ \ \ \ \ \ \ \ \text{for each }%
i\in\left\{  1,2,\ldots,j\right\}  . \label{pf.lem.Hp-nu.high}%
\end{equation}

[\textit{Proof of (\ref{pf.lem.Hp-nu.high}):} Let $i\in\left\{  1,2,\ldots
,j\right\}  $. Then, $\nu_{i}>n-k$ (since $\nu_{1}\geq\nu_{2}\geq\cdots\geq
\nu_{j}>n-k$), but also $\nu_{i}\leq n$ (by the assumptions of Lemma
\ref{lem.Hp-nu}). Thus, $n-k<\nu_{i}\leq n$, so that $\nu_{i}\in\left\{
n-k+1,n-k+2,\ldots,n\right\}  $ and thus $\nu_{i}-\left(  n-k\right)
\in\left\{  1,2,\ldots,k\right\}  $. Hence, (\ref{eq.h=amodI}) (applied to
$\nu_{i}-\left(  n-k\right)  $ instead of $j$) yields $h_{\nu_{i}}\equiv
a_{\nu_{i}-\left(  n-k\right)  }\operatorname{mod}I$. Hence, $\overline
{h_{\nu_{i}}}=\overline{a_{\nu_{i}-\left(  n-k\right)  }}\in\mathbf{k}$. This
proves (\ref{pf.lem.Hp-nu.high}).]

Furthermore, $\left(  \nu_{j+1},\nu_{j+2},\ldots,\nu_{p}\right)  $ is a
partition (since $\nu_{j+1}\geq\nu_{j+2}\geq\cdots\geq\nu_{p}$ and since all
of the $\nu_{1},\nu_{2},\ldots,\nu_{p}$ are nonnegative) with at most $k$
entries (indeed, its number of entries is $\leq p-j\leq p\leq k$), and all of
its entries are $\leq n-k$ (since $n-k\geq\nu_{j+1}\geq\nu_{j+2}\geq\cdots
\geq\nu_{p}$). Hence, $\left(  \nu_{j+1},\nu_{j+2},\ldots,\nu_{p}\right)  $
belongs to $P_{k,n}$.

From $\left(  \nu_{j+1},\nu_{j+2},\ldots,\nu_{p}\right)  \in P_{k,n}$ and
$\ell\left(  \nu_{j+1},\nu_{j+2},\ldots,\nu_{p}\right)  \leq p-j\leq p$, we
obtain $\overline{h_{\left(  \nu_{j+1},\nu_{j+2},\ldots,\nu_{p}\right)  }}\in
H_{p}$ (by the definition of $H_{p}$).

Now, the definition of $h_{\nu}$ yields $h_{\nu}=h_{\nu_{1}}h_{\nu_{2}}\cdots
h_{\nu_{p}}$, so that%
\begin{align*}
\overline{h_{\nu}}  &  =\overline{h_{\nu_{1}}h_{\nu_{2}}\cdots h_{\nu_{p}}%
}=\overline{h_{\nu_{1}}}\overline{h_{\nu_{2}}}\cdots\overline{h_{\nu_{p}}%
}=\underbrace{\left(  \overline{h_{\nu_{1}}}\overline{h_{\nu_{2}}}%
\cdots\overline{h_{\nu_{j}}}\right)  }_{\substack{\in\mathbf{k}\\\text{(by
(\ref{pf.lem.Hp-nu.high}))}}}\underbrace{\left(  \overline{h_{\nu_{j+1}}%
}\overline{h_{\nu_{j+2}}}\cdots\overline{h_{\nu_{p}}}\right)  }%
_{\substack{=\overline{h_{\nu_{j+1}}h_{\nu_{j+2}}\cdots h_{\nu_{p}}%
}\\=\overline{h_{\left(  \nu_{j+1},\nu_{j+2},\ldots,\nu_{p}\right)  }}\in
H_{p}}}\\
&  \in\mathbf{k}H_{p}\subseteq H_{p}.
\end{align*}
This proves Lemma \ref{lem.Hp-nu}.
\end{proof}

\begin{lemma}
\label{lem.Lp-basis}Let $p\in\mathbb{Z}$. Then, the family $\left(
\overline{s_{\lambda}}\right)  _{\lambda\in P_{k,n};\ \ell\left(
\lambda\right)  \leq p}$ is a basis of the $\mathbf{k}$-module $L_{p}$.
\end{lemma}

\begin{proof}
[Proof of Lemma \ref{lem.Lp-basis}.]Theorem \ref{thm.S/J} yields that $\left(
\overline{s_{\lambda}}\right)  _{\lambda\in P_{k,n}}$ is a basis of the
$\mathbf{k}$-module $\mathcal{S}/I$. Hence, this family $\left(
\overline{s_{\lambda}}\right)  _{\lambda\in P_{k,n}}$ is $\mathbf{k}$-linearly
independent. Thus, its subfamily $\left(  \overline{s_{\lambda}}\right)
_{\lambda\in P_{k,n};\ \ell\left(  \lambda\right)  \leq p}$ is $\mathbf{k}%
$-linearly independent as well. Moreover, this subfamily $\left(
\overline{s_{\lambda}}\right)  _{\lambda\in P_{k,n};\ \ell\left(
\lambda\right)  \leq p}$ spans the $\mathbf{k}$-module $L_{p}$ (by the
definition of $L_{p}$). Hence, this subfamily $\left(  \overline{s_{\lambda}%
}\right)  _{\lambda\in P_{k,n};\ \ell\left(  \lambda\right)  \leq p}$ is a
basis of the $\mathbf{k}$-module $L_{p}$. This proves Lemma \ref{lem.Lp-basis}.
\end{proof}

\begin{lemma}
\label{lem.Lp=Hp.restr}Let $p\in\left\{  0,1,\ldots,k\right\}  $. Then,
$L_{p}=H_{p}$.
\end{lemma}

(This lemma holds more generally for all $p\in\mathbb{Z}$, as we shall see in
Lemma \ref{prop.Lp=Hp} below.)

\begin{proof}
[Proof of Lemma \ref{lem.Lp=Hp.restr}.]Let $\lambda\in P_{k,n}$ be such that
$\ell\left(  \lambda\right)  \leq p$. We shall show that $\overline
{s_{\lambda}}\in H_{p}$.

Indeed, let $S_{p}$ denote the group of permutations of $\left\{
1,2,\ldots,p\right\}  $. For each $\sigma\in S_{p}$, let $\left(  -1\right)
^{\sigma}$ denote the sign of $\sigma$.

For each $\sigma\in S_{p}$, we have%
\begin{equation}
\overline{\prod_{i=1}^{p}h_{\lambda_{i}-i+\sigma\left(  i\right)  }}\in H_{p}.
\label{pf.prop.Lp=Hp.1}%
\end{equation}

[\textit{Proof of (\ref{pf.prop.Lp=Hp.1}):} Let $\sigma\in S_{p}$. Then, each
$i\in\left\{  1,2,\ldots,p\right\}  $ satisfies
\[
\underbrace{\lambda_{i}}_{\substack{\leq n-k\\\text{(since }\lambda\in
P_{k,n}\text{)}}}-\underbrace{i}_{\geq0}+\underbrace{\sigma\left(  i\right)
}_{\leq p\leq k}\leq n-k+0+k=n.
\]
Thus, Lemma \ref{lem.Hp-nu} (applied to $\left(  \lambda_{1}-1+\sigma\left(
1\right)  ,\lambda_{2}-2+\sigma\left(  2\right)  ,\ldots,\lambda_{p}%
-p+\sigma\left(  p\right)  \right)  $ and $\lambda_{i}-i+\sigma\left(
i\right)  $ instead of $\nu$ and $\nu_{i}$) yields
\[
\overline{h_{\left(  \lambda_{1}-1+\sigma\left(  1\right)  ,\lambda
_{2}-2+\sigma\left(  2\right)  ,\ldots,\lambda_{p}-p+\sigma\left(  p\right)
\right)  }}\in H_{p}%
\]
(since $p\leq k$). In view of%
\[
h_{\left(  \lambda_{1}-1+\sigma\left(  1\right)  ,\lambda_{2}-2+\sigma\left(
2\right)  ,\ldots,\lambda_{p}-p+\sigma\left(  p\right)  \right)  }=\prod
_{i=1}^{p}h_{\lambda_{i}-i+\sigma\left(  i\right)  },
\]
this rewrites as $\overline{\prod_{i=1}^{p}h_{\lambda_{i}-i+\sigma\left(
i\right)  }}\in H_{p}$. Thus, (\ref{pf.prop.Lp=Hp.1}) is proven.]

We have $\ell\left(  \lambda\right)  \leq p$ and thus $\lambda=\left(
\lambda_{1},\lambda_{2},\ldots,\lambda_{p}\right)  $. Hence, Proposition
\ref{prop.jacobi-trudi.Sh} \textbf{(b)} yields%
\[
s_{\lambda}=\det\left(  \left(  h_{\lambda_{u}-u+v}\right)  _{1\leq u\leq
p,\ 1\leq v\leq p}\right)  =\sum_{\sigma\in S_{p}}\left(  -1\right)  ^{\sigma
}\prod_{i=1}^{p}h_{\lambda_{i}-i+\sigma\left(  i\right)  }%
\]
(by the definition of a determinant). Projecting both sides of this equality
onto $\mathcal{S}/I$, we obtain%
\[
\overline{s_{\lambda}}=\overline{\sum_{\sigma\in S_{p}}\left(  -1\right)
^{\sigma}\prod_{i=1}^{p}h_{\lambda_{i}-i+\sigma\left(  i\right)  }}%
=\sum_{\sigma\in S_{p}}\left(  -1\right)  ^{\sigma}\underbrace{\overline
{\prod_{i=1}^{p}h_{\lambda_{i}-i+\sigma\left(  i\right)  }}}_{\substack{\in
H_{p}\\\text{(by (\ref{pf.prop.Lp=Hp.1}))}}}\in H_{p}.
\]

Now, forget that we fixed $\lambda$. We thus have proven that%
\[
\overline{s_{\lambda}}\in H_{p}\ \ \ \ \ \ \ \ \ \ \text{for each }\lambda\in
P_{k,n}\text{ satisfying }\ell\left(  \lambda\right)  \leq p\text{.}%
\]
Therefore, $L_{p}\subseteq H_{p}$ (since $L_{p}$ is the $\mathbf{k}$-submodule
of $\mathcal{S}/I$ spanned by the $\overline{s_{\lambda}}$ with $\lambda\in
P_{k,n}$ satisfying $\ell\left(  \lambda\right)  \leq p$).

Lemma \ref{lem.Lp-basis} yields that the family $\left(  \overline{s_{\lambda
}}\right)  _{\lambda\in P_{k,n};\ \ell\left(  \lambda\right)  \leq p}$ is a
basis of the $\mathbf{k}$-module $L_{p}$.

Now, let $L_{p}^{\prime}$ be the $\mathbf{k}$-submodule of $\mathcal{S}/I$
spanned by the $\overline{s_{\lambda}}$ with $\lambda\in P_{k,n}$ satisfying
$\ell\left(  \lambda\right)  >p$. Recall (from Theorem \ref{thm.S/J}) that
$\left(  \overline{s_{\lambda}}\right)  _{\lambda\in P_{k,n}}$ is a basis of
the $\mathbf{k}$-module $\mathcal{S}/I$. Hence, $\mathcal{S}/I=L_{p}\oplus
L_{p}^{\prime}$ (since each $\lambda\in P_{k,n}$ satisfies either $\ell\left(
\lambda\right)  \leq p$ or $\ell\left(  \lambda\right)  >p$ but not both). Let
$r$ be the number of all $\lambda\in P_{k,n}$ satisfying $\ell\left(
\lambda\right)  \leq p$. Then, the $\mathbf{k}$-module $H_{p}$ can be spanned
by $r$ elements (namely, by the $\overline{h_{\lambda}}$ with $\lambda\in
P_{k,n}$ satisfying $\ell\left(  \lambda\right)  \leq p$), whereas the
$\mathbf{k}$-module $L_{p}$ has a basis with $r$ elements (namely, the family
$\left(  \overline{s_{\lambda}}\right)  _{\lambda\in P_{k,n};\ \ell\left(
\lambda\right)  \leq p}$). Thus, Lemma \ref{lem.freemod.3} (applied to
$Z=\mathcal{S}/I$, $X=L_{p}$, $Y=L_{p}^{\prime}$ and $U=H_{p}$) yields
$L_{p}=H_{p}$. This proves Lemma \ref{lem.Lp=Hp.restr}.
\end{proof}

\begin{proposition}
\label{prop.Lp=Hp}Let $p\in\mathbb{Z}$. Then, $L_{p}=H_{p}$.
\end{proposition}

\begin{proof}
[Proof of Proposition \ref{prop.Lp=Hp}.]If $p$ is negative, then both $L_{p}$
and $H_{p}$ equal $0$ (since there exists no $\lambda\in P_{k,n}$ satisfying
$\ell\left(  \lambda\right)  \leq p$ in this case). Thus, if $p$ is negative,
then $L_{p}=H_{p}$ is obviously true. Hence, for the rest of this proof, we
WLOG assume that $p$ is not negative. Thus, $p\in\mathbb{N}$.

If $p\in\left\{  0,1,\ldots,k\right\}  $, then $L_{p}=H_{p}$ follows from
Lemma \ref{lem.Lp=Hp.restr}. Hence, for the rest of this proof, we WLOG assume
that $p\notin\left\{  0,1,\ldots,k\right\}  $. Thus, $p>k$ (since
$p\in\mathbb{N}$). Hence, $k<p$, so that $H_{k}\subseteq H_{p}$ (since
$H_{0}\subseteq H_{1}\subseteq H_{2}\subseteq\cdots$). But Lemma
\ref{lem.Lp=Hp.restr} (applied to $k$ instead of $p$) yields $L_{k}=H_{k}$.

But recall that $L_{k}=\mathcal{S}/I$. Thus, $\mathcal{S}/I=L_{k}%
=H_{k}\subseteq H_{p}$. Thus, $H_{p}\supseteq\mathcal{S}/I\supseteq L_{p}$.

On the other hand, $k<p$ and thus $L_{k}\subseteq L_{p}$ (since $L_{0}%
\subseteq L_{1}\subseteq L_{2}\subseteq\cdots$). Hence, $L_{p}\supseteq
L_{k}=\mathcal{S}/I\supseteq H_{p}$. Combining this with $H_{p}\supseteq
L_{p}$, we obtain $L_{p}=H_{p}$. This proves Proposition \ref{prop.Lp=Hp}.
\end{proof}

\begin{corollary}
\label{cor.Hp-basis}Let $p\in\mathbb{Z}$. Then, the family $\left(
\overline{h_{\lambda}}\right)  _{\lambda\in P_{k,n};\ \ell\left(
\lambda\right)  \leq p}$ is a basis of the $\mathbf{k}$-module $L_{p}$.
\end{corollary}

\begin{proof}
[Proof of Corollary \ref{cor.Hp-basis}.]Lemma \ref{lem.Lp-basis} yields that
the family $\left(  \overline{s_{\lambda}}\right)  _{\lambda\in P_{k,n}%
;\ \ell\left(  \lambda\right)  \leq p}$ is a basis of the $\mathbf{k}$-module
$L_{p}$. On the other hand, the family $\left(  \overline{h_{\lambda}}\right)
_{\lambda\in P_{k,n};\ \ell\left(  \lambda\right)  \leq p}$ spans the
$\mathbf{k}$-module $H_{p}$ (by the definition of $H_{p}$). In other words,
the family $\left(  \overline{h_{\lambda}}\right)  _{\lambda\in P_{k,n}%
;\ \ell\left(  \lambda\right)  \leq p}$ spans the $\mathbf{k}$-module $L_{p}$
(since Proposition \ref{prop.Lp=Hp} yields $L_{p}=H_{p}$). Since $\left\vert
\left\{  \lambda\in P_{k,n}\ \mid\ \ell\left(  \lambda\right)  \leq p\right\}
\right\vert =\left\vert \left\{  \lambda\in P_{k,n}\ \mid\ \ell\left(
\lambda\right)  \leq p\right\}  \right\vert $, we can therefore apply Lemma
\ref{lem.freemod-span-basis} to $L_{p}$, $\left(  \overline{s_{\lambda}%
}\right)  _{\lambda\in P_{k,n};\ \ell\left(  \lambda\right)  \leq p}$ and
$\left(  \overline{h_{\lambda}}\right)  _{\lambda\in P_{k,n};\ \ell\left(
\lambda\right)  \leq p}$ instead of $M$, $\left(  b_{s}\right)  _{s\in S}$ and
$\left(  a_{u}\right)  _{u\in U}$. We thus conclude that $\left(
\overline{h_{\lambda}}\right)  _{\lambda\in P_{k,n};\ \ell\left(
\lambda\right)  \leq p}$ is a basis of the $\mathbf{k}$-module $L_{p}$. This
proves Corollary \ref{cor.Hp-basis}.
\end{proof}

\begin{theorem}
\label{thm.S/J-h-basis}The family $\left(  \overline{h_{\lambda}}\right)
_{\lambda\in P_{k,n}}$ is a basis of the $\mathbf{k}$-module $\mathcal{S}/I$.
\end{theorem}

\begin{proof}
[Proof of Theorem \ref{thm.S/J-h-basis}.]Corollary \ref{cor.Hp-basis} (applied
to $p=k$) shows that the family $\left(  \overline{h_{\lambda}}\right)
_{\lambda\in P_{k,n};\ \ell\left(  \lambda\right)  \leq k}$ is a basis of the
$\mathbf{k}$-module $L_{k}$. In view of $\left(  \overline{h_{\lambda}%
}\right)  _{\lambda\in P_{k,n};\ \ell\left(  \lambda\right)  \leq k}=\left(
\overline{h_{\lambda}}\right)  _{\lambda\in P_{k,n}}$ (since each $\lambda\in
P_{k,n}$ satisfies $\ell\left(  \lambda\right)  \leq k$) and $L_{k}%
=\mathcal{S}/I$, this rewrites as follows: The family $\left(  \overline
{h_{\lambda}}\right)  _{\lambda\in P_{k,n}}$ is a basis of the $\mathbf{k}%
$-module $\mathcal{S}/I$. This proves Theorem \ref{thm.S/J-h-basis}.
\end{proof}

\begin{proposition}
\label{prop.Hp-nu2}Let $p\in\mathbb{N}$. Let $\nu=\left(  \nu_{1},\nu
_{2},\ldots,\nu_{p}\right)  \in\mathbb{Z}^{p}$. Assume that $\nu_{i}\leq n$
for each $i\in\left\{  1,2,\ldots,p\right\}  $. Then, $\overline{h_{\nu}}\in
H_{p}$.
\end{proposition}

\begin{proof}
[Proof of Proposition \ref{prop.Hp-nu2}.]If $p\leq k$, then this follows from
Lemma \ref{lem.Hp-nu}. Thus, for the rest of this proof, we WLOG assume that
$p>k$. Hence, $k<p$, so that $H_{k}\subseteq H_{p}$ (since $H_{0}\subseteq
H_{1}\subseteq H_{2}\subseteq\cdots$). But Proposition \ref{prop.Lp=Hp}
(applied to $k$ instead of $p$) yields $H_{k}=L_{k}=\mathcal{S}/I$. Now,
$\overline{h_{\nu}}\in\mathcal{S}/I=H_{k}\subseteq H_{p}$. This proves
Proposition \ref{prop.Hp-nu2}.
\end{proof}

We recall that the $\mathbf{k}$-submodules of a given $\mathbf{k}$-algebra $A$
form a monoid under multiplication: The product $XY$ of two $\mathbf{k}%
$-submodules $X$ and $Y$ of $A$ is defined as the $\mathbf{k}$-linear span of
all products $xy$ with $x\in X$ and $y\in Y$. The neutral element of this
monoid is $\mathbf{k}\cdot1_{A}$. We shall specifically use this monoid in the
case when $A=\mathcal{S}/I$.

\begin{proposition}
\label{prop.Lpalg}The family $\left(  L_{p}\right)  _{p\in\mathbb{N}}$ is a
filtration of the $\mathbf{k}$-algebra $\mathcal{S}/I$; that is, we have%
\begin{align}
L_{0}  &  \subseteq L_{1}\subseteq L_{2}\subseteq\cdots
,\ \ \ \ \ \ \ \ \ \ \bigcup_{p\in\mathbb{N}}L_{p}=\mathcal{S}/I,\nonumber\\
1  &  \in L_{0},\ \ \ \ \ \ \ \ \ \ \text{and}\nonumber\\
L_{a}L_{b}  &  \subseteq L_{a+b}\ \ \ \ \ \ \ \ \ \ \text{for every }%
a,b\in\mathbb{N}. \label{eq.prop.Lpalg.3}%
\end{align}

\end{proposition}

\begin{proof}
[Proof of Proposition \ref{prop.Lpalg}.]We already know that $L_{0}\subseteq
L_{1}\subseteq L_{2}\subseteq\cdots$. Also, $1\in L_{0}$ (since $L_{0}$ is the
$\mathbf{k}$-submodule of $\mathcal{S}/I$ spanned by $1$). Also,
$L_{k}=\mathcal{S}/I$, so that $\mathcal{S}/I=L_{k}\subseteq\bigcup
_{p\in\mathbb{N}}L_{p}$. Combining this with $\bigcup_{p\in\mathbb{N}}%
L_{p}\subseteq\mathcal{S}/I$, we obtain $\bigcup_{p\in\mathbb{N}}%
L_{p}=\mathcal{S}/I$.

Hence, it remains to prove that $L_{a}L_{b}\subseteq L_{a+b}$ for every
$a,b\in\mathbb{N}$. So let us fix $a,b\in\mathbb{N}$. We must prove that
$L_{a}L_{b}\subseteq L_{a+b}$.

If $a+b\geq k$, then this is obvious (because if $a+b\geq k$, then $k\leq
a+b$, hence $L_{k}\subseteq L_{a+b}$ (since $L_{0}\subseteq L_{1}\subseteq
L_{2}\subseteq\cdots$), hence $L_{a}L_{b}\subseteq\mathcal{S}/I=L_{k}\subseteq
L_{a+b}$). Hence, we WLOG assume that $a+b<k$.

We must prove that $L_{a}L_{b}\subseteq L_{a+b}$. It clearly suffices to show
that $fg\in L_{a+b}$ for each $f\in L_{a}$ and $g\in L_{b}$. So let us fix
$f\in L_{a}$ and $g\in L_{b}$; we must prove that $fg\in L_{a+b}$.

Proposition \ref{prop.Lp=Hp} yields that $L_{a}=H_{a}$. Thus, $f\in
L_{a}=H_{a}$, so that $f$ is a $\mathbf{k}$-linear combination of the
$\overline{h_{\lambda}}$ with $\lambda\in P_{k,n}$ satisfying $\ell\left(
\lambda\right)  \leq a$ (because $H_{a}$ is the $\mathbf{k}$-submodule of
$\mathcal{S}/I$ spanned by these $\overline{h_{\lambda}}$). Since the claim we
are proving (that is, $fg\in L_{a+b}$) depends $\mathbf{k}$-linearly on $f$,
we can thus WLOG assume that $f$ is one of those $\overline{h_{\lambda}}$. In
other words, we can WLOG assume that $f=\overline{h_{\alpha}}$ for some
$\alpha\in P_{k,n}$ satisfying $\ell\left(  \alpha\right)  \leq a$. Assume
this, and consider this $\alpha$. For similar reasons, we WLOG assume that
$g=\overline{h_{\beta}}$ for some $\beta\in P_{k,n}$ satisfying $\ell\left(
\beta\right)  \leq b$. Consider this $\beta$.

Note that each entry of $\alpha$ is $\leq n-k$ (since $\alpha\in P_{k,n}$),
and therefore $\leq n$. Thus, we can consider $\alpha$ as an $a$-tuple of
elements of $\left\{  0,1,\ldots,n\right\}  $ (since $\ell\left(
\alpha\right)  \leq a$). Likewise, consider $\beta$ as a $b$-tuple of elements
of $\left\{  0,1,\ldots,n\right\}  $.

Let $\gamma$ be the concatenation of the $a$-tuple $\alpha$ with the $b$-tuple
$\beta$. Thus, $\gamma$ is an $\left(  a+b\right)  $-tuple of elements of
$\left\{  0,1,\ldots,n\right\}  $ (since $\alpha$ is an $a$-tuple of elements
of $\left\{  0,1,\ldots,n\right\}  $ and since $\beta$ is a $b$-tuple of
elements of $\left\{  0,1,\ldots,n\right\}  $), and satisfies $h_{\gamma
}=h_{\alpha}h_{\beta}$. (But $\gamma$ is not necessarily a partition.)
Moreover, $a+b\leq k$ (since $a+b<k$). Finally, write $\gamma$ in the form
$\gamma=\left(  \gamma_{1},\gamma_{2},\ldots,\gamma_{a+b}\right)  $; then, we
have $\gamma_{i}\leq n$ for each $i\in\left\{  1,2,\ldots,a+b\right\}  $
(because $\gamma$ is an $\left(  a+b\right)  $-tuple of elements of $\left\{
0,1,\ldots,n\right\}  $). Hence, Lemma \ref{lem.Hp-nu} (applied to $p=a+b$,
$\nu=\gamma$ and $\nu_{i}=\gamma_{i}$) yields $\overline{h_{\gamma}}\in
H_{a+b}$. But Proposition \ref{prop.Lp=Hp} yields that $L_{a+b}=H_{a+b}$.

From $f=\overline{h_{\alpha}}$ and $g=\overline{h_{\beta}}$, we obtain
$fg=\overline{h_{\alpha}}\overline{h_{\beta}}=\overline{h_{\alpha}h_{\beta}%
}=\overline{h_{\gamma}}$ (since $h_{\alpha}h_{\beta}=h_{\gamma}$). Thus,
$fg=\overline{h_{\gamma}}\in H_{a+b}=L_{a+b}$ (since $L_{a+b}=H_{a+b}$). This
completes our proof of Proposition \ref{prop.Lpalg}.
\end{proof}

\begin{corollary}
\label{cor.Lpalgfil}We have $\left(  L_{1}\right)  ^{m}\subseteq L_{m}$ for
each $m\in\mathbb{N}$.
\end{corollary}

\begin{proof}
[Proof of Corollary \ref{cor.Lpalgfil}.]This follows by induction on $m$,
using the facts (which we proved in Proposition \ref{prop.Lpalg}) that $1\in
L_{0}$ and that $L_{a}L_{b}\subseteq L_{a+b}$ for every $a,b\in\mathbb{N}$.
\end{proof}

\subsection{A formula for hook-shaped Schur functions}

\begin{lemma}
\label{lem.sm1j.1}Let $m$ be a positive integer. Let $j\in\mathbb{N}$. Then,%
\[
\mathbf{s}_{\left(  m,1^{j}\right)  }=\sum_{i=1}^{m}\left(  -1\right)
^{i-1}\mathbf{h}_{m-i}\mathbf{e}_{j+i}.
\]

\end{lemma}

\begin{proof}
[Proof of Lemma \ref{lem.sm1j.1}.]For each $N\in\mathbb{N}$, we have%
\begin{equation}
\sum_{p=0}^{N}\left(  -1\right)  ^{p}\mathbf{h}_{N-p}\mathbf{e}_{p}%
=\delta_{0,N}. \label{pf.lem.sm1j.1.1}%
\end{equation}
(This is just the equality (\ref{pf.prop.redh.Lam.3}), with $j$ renamed as $p$.)

From $m>0$ and $j\geq0$, we obtain $m+j>0$, so that $\delta_{0,m+j}=0$. The
equality (\ref{pf.lem.sm1j.1.1}) (applied to $N=m+j$) becomes%
\[
\sum_{p=0}^{m+j}\left(  -1\right)  ^{p}\mathbf{h}_{m+j-p}\mathbf{e}_{p}%
=\delta_{0,m+j}=0.
\]
Thus,%
\begin{align*}
0  &  =\sum_{p=0}^{m+j}\left(  -1\right)  ^{p}\mathbf{h}_{m+j-p}\mathbf{e}%
_{p}\\
&  =\sum_{i=-m}^{j}\left(  -1\right)  ^{j-i}\mathbf{h}_{m+i}\mathbf{e}%
_{j-i}\ \ \ \ \ \ \ \ \ \ \left(
\begin{array}
[c]{c}%
\text{here, we have substituted }j-i\\
\text{for }p\text{ in the sum}%
\end{array}
\right) \\
&  =\underbrace{\sum_{i=-m}^{-1}\left(  -1\right)  ^{j-i}\mathbf{h}%
_{m+i}\mathbf{e}_{j-i}}_{\substack{=\sum_{i=1}^{m}\left(  -1\right)
^{j+i}\mathbf{h}_{m-i}\mathbf{e}_{j+i}\\\text{(here, we have substituted
}-i\text{ for }i\\\text{in the sum)}}}+\sum_{i=0}^{j}\underbrace{\left(
-1\right)  ^{j-i}}_{=\left(  -1\right)  ^{j}\left(  -1\right)  ^{i}}%
\mathbf{h}_{m+i}\mathbf{e}_{j-i}\\
&  =\sum_{i=1}^{m}\left(  -1\right)  ^{j+i}\mathbf{h}_{m-i}\mathbf{e}%
_{j+i}+\left(  -1\right)  ^{j}\underbrace{\sum_{i=0}^{j}\left(  -1\right)
^{i}\mathbf{h}_{m+i}\mathbf{e}_{j-i}}_{\substack{=\mathbf{s}_{\left(
m,1^{j}\right)  }\\\text{(by (\ref{pf.prop.redh.Lam.1}))}}}\\
&  =\sum_{i=1}^{m}\left(  -1\right)  ^{j+i}\mathbf{h}_{m-i}\mathbf{e}%
_{j+i}+\left(  -1\right)  ^{j}\mathbf{s}_{\left(  m,1^{j}\right)  }.
\end{align*}
Solving this equality for $\mathbf{s}_{\left(  m,1^{j}\right)  }$, we obtain%
\[
\mathbf{s}_{\left(  m,1^{j}\right)  }=-\dfrac{1}{\left(  -1\right)  ^{j}}%
\sum_{i=1}^{m}\left(  -1\right)  ^{j+i}\mathbf{h}_{m-i}\mathbf{e}_{j+i}%
=\sum_{i=1}^{m}\left(  -1\right)  ^{i-1}\mathbf{h}_{m-i}\mathbf{e}_{j+i}.
\]
This proves Lemma \ref{lem.sm1j.1}.
\end{proof}

\subsection{The submodules $C$ and $R_{p}$ of $\mathcal{S}/I$}

Next, we introduce some more $\mathbf{k}$-submodules of $\mathcal{S}/I$:

\begin{definition}
\textbf{(a)} Let $C$ be the $\mathbf{k}$-submodule of $\mathcal{S}/I$ spanned
by the $\overline{e_{i}}$ with $i\in\mathbb{N}$.

\textbf{(b)} For each $p\in\mathbb{Z}$, we let $R_{p}$ be the $\mathbf{k}%
$-submodule of $\mathcal{S}/I$ spanned by the $\overline{h_{i}}$ with
$i\in\mathbb{N}$ satisfying $i\leq p$.
\end{definition}

We recall that $e_{i}=0$ for every $i>k$. Thus, $\overline{e_{i}}=0$ for every
$i>k$. Hence, the $\mathbf{k}$-module $C$ is spanned by $\overline{e_{0}%
},\overline{e_{1}},\ldots,\overline{e_{k}}$ (because all the other among its
designated generators $\overline{e_{i}}$ are $0$). Also, the definition of $C$
yields $\overline{e_{0}}\in C$, so that $1=\overline{e_{0}}\in C$. Thus, each
$i\in\mathbb{N}$ satisfies $C^{i}=\underbrace{1}_{\in C}C^{i}\subseteq
CC^{i}=C^{i+1}$. In other words, $C^{0}\subseteq C^{1}\subseteq C^{2}%
\subseteq\cdots$.

Note that $R_{0}\subseteq R_{1}\subseteq R_{2}\subseteq\cdots$. Also:

\begin{proposition}
\label{prop.R.Rn-k}We have $R_{n-k}=L_{1}$.
\end{proposition}

\begin{proof}
[Proof of Proposition \ref{prop.R.Rn-k}.]We WLOG assume that $k\neq0$, because
the case when $k=0$ is trivial for its own reasons\footnote{\textit{Proof.}
Assume that $k=0$. Then, $\mathcal{S}=\mathbf{k}$ and $I=0$, whence
$\mathcal{S}/I=\mathbf{k}\cdot1$. Both $\mathbf{k}$-submodules $R_{n-k}$ and
$L_{1}$ contain $1$ (since $1=\overline{h_{0}}$ and since $1=\overline
{s_{\varnothing}}$); hence, both of these $\mathbf{k}$-submodules must be the
whole $\mathcal{S}/I$ (since $\mathcal{S}/I=\mathbf{k}\cdot1$) and therefore
must be equal. So we have proven $R_{n-k}=L_{1}$. In other words, we have
proven Proposition \ref{prop.R.Rn-k} under the assumption that $k=0$.}. Thus,
$k>0$, and therefore the partition $\left(  i\right)  $ belongs to $P_{k,n}$
for each $i\in\left\{  0,1,\ldots,n-k\right\}  $.

Recall that $L_{1}$ was defined as the $\mathbf{k}$-submodule of
$\mathcal{S}/I$ spanned by the $\overline{s_{\lambda}}$ with $\lambda\in
P_{k,n}$ satisfying $\ell\left(  \lambda\right)  \leq1$. But the $\lambda\in
P_{k,n}$ satisfying $\ell\left(  \lambda\right)  \leq1$ are exactly the
partitions of the form $\left(  i\right)  $ for $i\in\left\{  0,1,\ldots
,n-k\right\}  $. Hence, $L_{1}$ is the $\mathbf{k}$-submodule of
$\mathcal{S}/I$ spanned by the $\overline{s_{\left(  i\right)  }}$ with
$i\in\left\{  0,1,\ldots,n-k\right\}  $. Since we have $s_{\left(  i\right)
}=h_{i}$ for each $i\in\left\{  0,1,\ldots,n-k\right\}  $, we can rewrite this
as follows: $L_{1}$ is the $\mathbf{k}$-submodule of $\mathcal{S}/I$ spanned
by the $\overline{h_{i}}$ with $i\in\left\{  0,1,\ldots,n-k\right\}  $. In
other words, $L_{1}$ is the $\mathbf{k}$-submodule of $\mathcal{S}/I$ spanned
by the $\overline{h_{i}}$ with $i\in\mathbb{N}$ satisfying $i\leq n-k$. But
this is precisely the definition of the $\mathbf{k}$-submodule $R_{n-k}$.
Hence, $L_{1}=R_{n-k}$. This proves Proposition \ref{prop.R.Rn-k}.
\end{proof}

It is easy to see that $R_{n-k}=R_{n-k+1}=\cdots=R_{n}$, but the sequence
$\left(  R_{0},R_{1},R_{2},\ldots\right)  $ may and may not grow after its
$n$-th term depending on the choice of $a_{1},a_{2},\ldots,a_{k}$. So the
family $\left(  R_{p}\right)  _{p\in\mathbb{Z}}$ is a filtration of some
$\mathbf{k}$-submodule of $\mathcal{S}/I$, but it isn't easy to say which
specific $\mathbf{k}$-submodule it is.

\begin{lemma}
\label{lem.R.Cp}We have $R_{p}\subseteq C^{p}$ for each $p\in\mathbb{N}$.
\end{lemma}

\begin{proof}
[Proof of Lemma \ref{lem.R.Cp}.]We have%
\begin{equation}
\overline{e_{i}}\in C\ \ \ \ \ \ \ \ \ \ \text{for each }i\in\mathbb{N}
\label{pf.lem.R.Cp.eiinC}%
\end{equation}
(by the definition of $C$).

Let $p\in\mathbb{N}$. Recall that $R_{p}$ is the $\mathbf{k}$-submodule of
$\mathcal{S}/I$ spanned by the $\overline{h_{i}}$ with $i\in\mathbb{N}$
satisfying $i\leq p$. Hence, in order to prove that $R_{p}\subseteq C^{p}$, it
suffices to show that $\overline{h_{i}}\in C^{p}$ for each $i\in\mathbb{N}$
satisfying $i\leq p$.

We first claim that%
\begin{equation}
\overline{h_{i}}\in C^{i}\ \ \ \ \ \ \ \ \ \ \text{for each }i\in\mathbb{N}.
\label{pf.lem.R.Cp.1}%
\end{equation}

[\textit{Proof of (\ref{pf.lem.R.Cp.1}):} We shall prove (\ref{pf.lem.R.Cp.1})
by strong induction on $i$. So we fix $j\in\mathbb{N}$, and we assume (as
induction hypothesis) that (\ref{pf.lem.R.Cp.1}) holds for all $i<j$. We must
now prove that (\ref{pf.lem.R.Cp.1}) holds for $i=j$. In other words, we must
prove that $\overline{h_{j}}\in C^{j}$.

If $j=0$, then this is obvious (because in this case, we have $\overline
{h_{j}}=\overline{h_{0}}=\overline{1}=1\in C^{0}$). Thus, we WLOG assume that
$j\neq0$. Hence, $j$ is a positive integer. Thus, Corollary \ref{cor.heh-id.0}
(applied to $j$ instead of $p$) yields%
\[
h_{j}=-\sum_{t=1}^{k}\left(  -1\right)  ^{t}e_{t}h_{j-t}.
\]
Hence,%
\begin{align*}
\overline{h_{j}}  &  =\overline{-\sum_{t=1}^{k}\left(  -1\right)  ^{t}%
e_{t}h_{j-t}}=-\sum_{t=1}^{k}\left(  -1\right)  ^{t}\underbrace{\overline
{e_{t}}}_{\substack{\in C\\\text{(by (\ref{pf.lem.R.Cp.eiinC}))}%
}}\underbrace{\overline{h_{j-t}}}_{\substack{\in C^{j-t}\\\text{(by the
induction}\\\text{hypothesis, since }j-t<j\text{)}}}\\
&  \in-\sum_{t=1}^{k}\left(  -1\right)  ^{t}\underbrace{CC^{j-t}%
}_{\substack{=C^{j-t+1}\subseteq C^{j}\\\text{(since }j-t+1\leq j\text{
and}\\C^{0}\subseteq C^{1}\subseteq C^{2}\subseteq\cdots\text{)}}%
}\subseteq-\sum_{t=1}^{k}\left(  -1\right)  ^{t}C^{j}\subseteq C^{j}.
\end{align*}
In other words, (\ref{pf.lem.R.Cp.1}) holds for $i=j$. This completes the
induction step. Thus, (\ref{pf.lem.R.Cp.1}) is proven.]

Now, let us fix $i\in\mathbb{N}$ satisfying $i\leq p$. Then, $C^{i}\subseteq
C^{p}$ (since $i\leq p$ and $C^{0}\subseteq C^{1}\subseteq C^{2}%
\subseteq\cdots$). But (\ref{pf.lem.R.Cp.1}) yields $\overline{h_{i}}\in
C^{i}\subseteq C^{p}$.

Now, forget that we fixed $i$. We thus have shown that $\overline{h_{i}}\in
C^{p}$ for each $i\in\mathbb{N}$ satisfying $i\leq p$. As we have said, this
proves Lemma \ref{lem.R.Cp}.
\end{proof}

\begin{lemma}
\label{lem.RC.sm1j}Let $m$ be a positive integer. Let $j\in\mathbb{N}$. Then,
$\overline{s_{\left(  m,1^{j}\right)  }}\in R_{m-1}C$.
\end{lemma}

\begin{proof}
[Proof of Lemma \ref{lem.RC.sm1j}.]Lemma \ref{lem.sm1j.1} yields%
\[
\mathbf{s}_{\left(  m,1^{j}\right)  }=\sum_{i=1}^{m}\left(  -1\right)
^{i-1}\mathbf{h}_{m-i}\mathbf{e}_{j+i}.
\]
This is an equality in $\Lambda$. If we evaluate both of its sides at
$x_{1},x_{2},\ldots,x_{k}$, then we obtain%
\[
s_{\left(  m,1^{j}\right)  }=\sum_{i=1}^{m}\left(  -1\right)  ^{i-1}%
h_{m-i}e_{j+i}.
\]
Thus,%
\begin{align*}
\overline{s_{\left(  m,1^{j}\right)  }}  &  =\overline{\sum_{i=1}^{m}\left(
-1\right)  ^{i-1}h_{m-i}e_{j+i}}=\sum_{i=1}^{m}\left(  -1\right)
^{i-1}\underbrace{\overline{h_{m-i}}}_{\substack{\in R_{m-1}\\\text{(by the
definition of }R_{m-1}\text{,}\\\text{since }m-i\leq m-1\text{)}%
}}\underbrace{\overline{e_{j+i}}}_{\substack{\in C\\\text{(by the definition
of }C\text{)}}}\\
&  \in\sum_{i=1}^{m}\left(  -1\right)  ^{i-1}R_{m-1}C\subseteq R_{m-1}C.
\end{align*}
This proves Lemma \ref{lem.RC.sm1j}.
\end{proof}

\begin{corollary}
\label{cor.RC.h}Let $m$ be a positive integer. Then, $\overline{h_{n+m}}\in
R_{m-1}C$.
\end{corollary}

\begin{proof}
[Proof of Corollary \ref{cor.RC.h}.]Proposition \ref{prop.redh.1} yields%
\[
h_{n+m}\equiv\sum_{j=0}^{k-1}\left(  -1\right)  ^{j}a_{k-j}s_{\left(
m,1^{j}\right)  }\operatorname{mod}I.
\]
Thus,%
\begin{align*}
\overline{h_{n+m}}  &  =\overline{\sum_{j=0}^{k-1}\left(  -1\right)
^{j}a_{k-j}s_{\left(  m,1^{j}\right)  }}=\sum_{j=0}^{k-1}\left(  -1\right)
^{j}a_{k-j}\underbrace{\overline{s_{\left(  m,1^{j}\right)  }}}_{\substack{\in
R_{m-1}C\\\text{(by Lemma \ref{lem.RC.sm1j})}}}\\
&  \in\sum_{j=0}^{k-1}\left(  -1\right)  ^{j}a_{k-j}R_{m-1}C\subseteq
R_{m-1}C.
\end{align*}
This proves Corollary \ref{cor.RC.h}.
\end{proof}

\begin{lemma}
\label{lem.hi-again}Let $j\in\mathbb{N}$ be such that $j\leq n$.

\textbf{(a)} We have $\overline{h_{j}}\in L_{1}$.

\textbf{(b)} Assume that $n>k$ and $j\neq n-k$. Then, $\overline{h_{j}}\in
R_{n-k-1}$.
\end{lemma}

\begin{proof}
[Proof of Lemma \ref{lem.hi-again}.]\textbf{(a)} We are in one of the
following two cases:

\textit{Case 1:} We have $j\leq n-k$.

\textit{Case 2:} We have $j>n-k$.

Let us first consider Case 1. In this case, we have $j\leq n-k$. Recall that
$R_{n-k}$ was defined as the $\mathbf{k}$-submodule of $\mathcal{S}/I$ spanned
by the $\overline{h_{i}}$ with $i\in\mathbb{N}$ satisfying $i\leq n-k$. Hence,
$\overline{h_{j}}\in R_{n-k}$ (since $j\in\mathbb{N}$ and $j\leq n-k$). Thus,
$\overline{h_{j}}\in R_{n-k}=L_{1}$ (by Proposition \ref{prop.R.Rn-k}). Thus,
Lemma \ref{lem.hi-again} \textbf{(a)} is proven in Case 1.

Let us now consider Case 2. In this case, we have $j>n-k$. Hence, $n-k<j\leq
n$, so that $j\in\left\{  n-k+1,n-k+2,\ldots,n\right\}  $ and therefore
$j-\left(  n-k\right)  \in\left\{  1,2,\ldots,k\right\}  $. Hence,
(\ref{eq.h=amodI}) (applied to $j-\left(  n-k\right)  $ instead of $j$) yields
$h_{j}\equiv a_{j-\left(  n-k\right)  }\operatorname{mod}I$. Hence,
$\overline{h_{j}}=\overline{a_{j-\left(  n-k\right)  }}\in\mathbf{k}$.

But $0\leq n-k$ and thus $\overline{h_{0}}\in R_{n-k}$ (by the definition of
$R_{n-k}$). Hence, $1=\overline{h_{0}}\in R_{n-k}$, so that $\mathbf{k}%
\subseteq R_{n-k}$ and thus $\overline{h_{j}}\in\mathbf{k}\subseteq
R_{n-k}=L_{1}$ (by Proposition \ref{prop.R.Rn-k}). Thus, Lemma
\ref{lem.hi-again} \textbf{(a)} is proven in Case 2.

We have now proven Lemma \ref{lem.hi-again} \textbf{(a)} in each of the two
Cases 1 and 2. Thus, Lemma \ref{lem.hi-again} \textbf{(a)} is proven.

\textbf{(b)} We are in one of the following two cases:

\textit{Case 1:} We have $j\leq n-k$.

\textit{Case 2:} We have $j>n-k$.

Let us first consider Case 1. In this case, we have $j\leq n-k$. Thus, $j<n-k$
(since $j\neq n-k$), so that $j\leq n-k-1$. Thus, $n-k-1\geq j\geq0$, so that
$n-k-1\in\mathbb{N}$. Recall that $R_{n-k-1}$ is defined as the $\mathbf{k}%
$-submodule of $\mathcal{S}/I$ spanned by the $\overline{h_{i}}$ with
$i\in\mathbb{N}$ satisfying $i\leq n-k-1$. Hence, $\overline{h_{j}}\in
R_{n-k-1}$ (since $j\in\mathbb{N}$ and $j\leq n-k-1$). Thus, Lemma
\ref{lem.hi-again} \textbf{(b)} is proven in Case 1.

Let us now consider Case 2. In this case, we have $j>n-k$. Hence, $n-k<j\leq
n$, so that $j\in\left\{  n-k+1,n-k+2,\ldots,n\right\}  $ and therefore
$j-\left(  n-k\right)  \in\left\{  1,2,\ldots,k\right\}  $. Hence,
(\ref{eq.h=amodI}) (applied to $j-\left(  n-k\right)  $ instead of $j$) yields
$h_{j}\equiv a_{j-\left(  n-k\right)  }\operatorname{mod}I$. Hence,
$\overline{h_{j}}=\overline{a_{j-\left(  n-k\right)  }}\in\mathbf{k}$.

But $n-k>0$ (since $n>k$), and thus $1\leq n-k$, so that $0\leq n-k-1$. Hence,
$\overline{h_{0}}\in R_{n-k-1}$ (by the definition of $R_{n-k-1}$). Hence,
$1=\overline{h_{0}}\in R_{n-k-1}$, so that $\mathbf{k}\subseteq R_{n-k-1}$ and
thus $\overline{h_{j}}\in\mathbf{k}\subseteq R_{n-k-1}$. Thus, Lemma
\ref{lem.hi-again} \textbf{(b)} is proven in Case 2.

We have now proven Lemma \ref{lem.hi-again} \textbf{(b)} in each of the two
Cases 1 and 2. Thus, Lemma \ref{lem.hi-again} \textbf{(b)} is proven.
\end{proof}

\subsection{Connection to the $Q_{p}$}

\begin{convention}
We WLOG assume that $k>0$ from now on.
\end{convention}

Now, let us recall Definition \ref{def.Qp}.

\begin{proposition}
\label{prop.Lk-1}We have $L_{k-1}=Q_{0}$.
\end{proposition}

\begin{proof}
[Proof of Proposition \ref{prop.Lk-1}.]Recall the following:

\begin{itemize}
\item We have defined $L_{k-1}$ as the $\mathbf{k}$-submodule of
$\mathcal{S}/I$ spanned by the $\overline{s_{\lambda}}$ with $\lambda\in
P_{k,n}$ satisfying $\ell\left(  \lambda\right)  \leq k-1$.

\item We have defined $Q_{0}$ as the $\mathbf{k}$-submodule of $\mathcal{S}/I$
spanned by the $\overline{s_{\lambda}}$ with $\lambda\in P_{k,n}$ satisfying
$\lambda_{k}\leq0$.
\end{itemize}

Comparing these two definitions, we conclude that $L_{k-1}=Q_{0}$ (because for
any $\lambda\in P_{k,n}$, the statement $\left(  \ell\left(  \lambda\right)
\leq k-1\right)  $ is equivalent to the statement $\left(  \lambda_{k}%
\leq0\right)  $). This proves Proposition \ref{prop.Lk-1}.
\end{proof}

\begin{lemma}
\label{lem.R.1}We have $\left(  L_{1}\right)  ^{k-1}\subseteq Q_{0}$.
\end{lemma}

\begin{proof}
[Proof of Lemma \ref{lem.R.1}.]Corollary \ref{cor.Lpalgfil} yields $\left(
L_{1}\right)  ^{k-1}\subseteq L_{k-1}=Q_{0}$ (by Proposition \ref{prop.Lk-1}).
This proves Lemma \ref{lem.R.1}.
\end{proof}

\begin{lemma}
\label{lem.C.Qp}Let $p\in\mathbb{Z}$. Then, $CQ_{p}\subseteq Q_{p+1}$.
\end{lemma}

\begin{proof}
[Proof of Lemma \ref{lem.C.Qp}.]Lemma \ref{lem.coeffw.eiQ} shows that
$\overline{e_{i}}Q_{p}\subseteq Q_{p+1}$ for each $i\in\mathbb{N}$. Thus,
$CQ_{p}\subseteq Q_{p+1}$ (since the $\mathbf{k}$-module $C$ is spanned by the
$\overline{e_{i}}$ with $i\in\mathbb{N}$). This proves Lemma \ref{lem.C.Qp}.
\end{proof}

\begin{corollary}
\label{cor.Cq.Qp}Let $p\in\mathbb{Z}$ and $q\in\mathbb{N}$. Then, $C^{q}%
Q_{p}\subseteq Q_{p+q}$.
\end{corollary}

\begin{proof}
[Proof of Corollary \ref{cor.Cq.Qp}.]This follows by induction on $q$, where
the induction step uses Lemma \ref{lem.C.Qp}.
\end{proof}

\subsection{Criteria for $\operatorname*{coeff}\nolimits_{\omega}\left(
\overline{h_{\nu}}\right)  =0$}

We shall now show two sufficient criteria for when a $p$-tuple $\nu
\in\mathbb{Z}^{p}$ satisfies $\operatorname*{coeff}\nolimits_{\omega}\left(
\overline{h_{\nu}}\right)  =0$.

\begin{theorem}
\label{thm.coeff-h.1}Let $p\in\mathbb{N}$ be such that $p\leq k$. Let
$\nu=\left(  \nu_{1},\nu_{2},\ldots,\nu_{p}\right)  \in\mathbb{Z}^{p}$ be an
$p$-tuple of integers. Let $q\in\left\{  1,2,\ldots,p\right\}  $ be such that%
\[
\nu_{1}\geq\nu_{2}\geq\cdots\geq\nu_{q}>n\geq\nu_{q+1}\geq\nu_{q+2}\geq
\cdots\geq\nu_{p}%
\]
and $\nu_{q}\leq2n-k-q$.

Assume also that
\begin{equation}
\nu_{i}\leq2n-k+1\ \ \ \ \ \ \ \ \ \ \text{for each }i\in\left\{
1,2,\ldots,p\right\}  . \label{eq.thm.coeff-h.1.ass1}%
\end{equation}

Then, $\operatorname*{coeff}\nolimits_{\omega}\left(  \overline{h_{\nu}%
}\right)  =0$.
\end{theorem}

\begin{proof}
[Proof of Theorem \ref{thm.coeff-h.1}.]From $\nu_{q}\leq2n-k-q$, we obtain
$2n-k-q\geq\nu_{q}>n$, so that $n-k-q>0$. Thus, $n-k-q-1\in\mathbb{N}$.

If any of the entries $\nu_{1},\nu_{2},\ldots,\nu_{p}$ of $\nu$ is negative,
then Theorem \ref{thm.coeff-h.1} holds for easy reasons\footnote{Indeed, in
this case we have $\nu_{i}<0$ for some $i\in\left\{  1,2,\ldots,p\right\}  $,
and therefore $h_{\nu_{i}}=0$ for this $i$, and therefore
\[
h_{\nu}=h_{\nu_{1}}h_{\nu_{2}}\cdots h_{\nu_{p}}=\left(  h_{\nu_{1}}h_{\nu
_{2}}\cdots h_{\nu_{i-1}}\right)  \underbrace{h_{\nu_{i}}}_{=0}\left(
h_{\nu_{i+1}}h_{\nu_{i+2}}\cdots h_{\nu_{p}}\right)  =0,
\]
and therefore $\operatorname*{coeff}\nolimits_{\omega}\left(  \overline
{h_{\nu}}\right)  =0$, qed.}. Hence, we WLOG assume that none of the entries
$\nu_{1},\nu_{2},\ldots,\nu_{p}$ of $\nu$ is negative. Thus, all of the
entries $\nu_{1},\nu_{2},\ldots,\nu_{p}$ are nonnegative integers.

From $p\leq k$, we obtain $p-1\leq k-1$ and thus $L_{p-1}\subseteq L_{k-1}$
(since $L_{0}\subseteq L_{1}\subseteq L_{2}\subseteq\cdots$). Thus,%
\begin{equation}
L_{p-1}\subseteq L_{k-1}=Q_{0} \label{pf.thm.coeff-h.1.Lp-1}%
\end{equation}
(by Proposition \ref{prop.Lk-1}).

From $n\geq\nu_{q+1}\geq\nu_{q+2}\geq\cdots\geq\nu_{p}$, we conclude that
$\nu_{j}\leq n$ for each $j\in\left\{  q+1,q+2,\ldots,p\right\}  $. In other
words, $\nu_{q+i}\leq n$ for each $i\in\left\{  1,2,\ldots,p-q\right\}  $.
Hence, Proposition \ref{prop.Hp-nu2} (applied to $p-q$, $\left(  \nu_{q+1}%
,\nu_{q+2},\ldots,\nu_{p}\right)  $ and $\nu_{q+i}$ instead of $p$, $\nu$ and
$\nu_{i}$) yields $\overline{h_{\left(  \nu_{q+1},\nu_{q+2},\ldots,\nu
_{p}\right)  }}\in H_{p-q}$. But Proposition \ref{prop.Lp=Hp} (applied to
$p-q$ instead of $p$) yields $L_{p-q}=H_{p-q}$. Thus,%
\begin{equation}
\overline{h_{\left(  \nu_{q+1},\nu_{q+2},\ldots,\nu_{p}\right)  }}\in
H_{p-q}=L_{p-q}. \label{pf.thm.coeff-h.1.1}%
\end{equation}

Next, we claim that%
\begin{equation}
\overline{h_{\nu_{i}}}\in L_{1}C\ \ \ \ \ \ \ \ \ \ \text{for each }%
i\in\left\{  1,2,\ldots,q-1\right\}  . \label{pf.thm.coeff-h.1.2}%
\end{equation}

[\textit{Proof of (\ref{pf.thm.coeff-h.1.2}):} Let $i\in\left\{
1,2,\ldots,q-1\right\}  $. Then, $\nu_{i}>n$ (since $\nu_{1}\geq\nu_{2}%
\geq\cdots\geq\nu_{q}>n$), so that $\nu_{i}-n$ is a positive integer. Thus,
Corollary \ref{cor.RC.h} (applied to $m=\nu_{i}-n$) yields $\overline
{h_{\nu_{i}}}\in R_{\nu_{i}-n-1}C$.

But $\nu_{i}\leq2n-k+1$ (by (\ref{eq.thm.coeff-h.1.ass1})), so that $\nu
_{i}-n-1\leq n-k$. Thus, $R_{\nu_{i}-n-1}\subseteq R_{n-k}$ (since
$R_{0}\subseteq R_{1}\subseteq R_{2}\subseteq\cdots$). Thus, $R_{\nu_{i}%
-n-1}\subseteq R_{n-k}=L_{1}$ (by Proposition \ref{prop.R.Rn-k}). Hence,
$\overline{h_{\nu_{i}}}\in\underbrace{R_{\nu_{i}-n-1}}_{\subseteq L_{1}%
}C\subseteq L_{1}C$. This proves (\ref{pf.thm.coeff-h.1.2}).]

From (\ref{pf.thm.coeff-h.1.2}), we obtain%
\[
\overline{h_{\nu_{1}}}\overline{h_{\nu_{2}}}\cdots\overline{h_{\nu_{q-1}}}%
\in\left(  L_{1}C\right)  ^{q-1}=\underbrace{\left(  L_{1}\right)  ^{q-1}%
}_{\substack{\subseteq L_{q-1}\\\text{(by Corollary \ref{cor.Lpalgfil})}%
}}C^{q-1}\subseteq L_{q-1}C^{q-1}.
\]

Also, $\nu_{q}>n$, so that $\nu_{q}-n$ is a positive integer. Thus, Corollary
\ref{cor.RC.h} (applied to $m=\nu_{q}-n$) yields $\overline{h_{\nu_{q}}}\in
R_{\nu_{q}-n-1}C$. But $\nu_{q}\leq2n-k-q$ and thus $\nu_{q}-n-1\leq n-k-q-1$.
Hence, $R_{\nu_{q}-n-1}\subseteq R_{n-k-q-1}$ (since $R_{0}\subseteq
R_{1}\subseteq R_{2}\subseteq\cdots$). Thus, $\overline{h_{\nu_{q}}}%
\in\underbrace{R_{\nu_{q}-n-1}}_{\subseteq R_{n-k-q-1}}C\subseteq
R_{n-k-q-1}C$.

Recall that $h_{\nu}=h_{\nu_{1}}h_{\nu_{2}}\cdots h_{\nu_{p}}$. Thus,%
\begin{align*}
\overline{h_{\nu}} &  =\overline{h_{\nu_{1}}h_{\nu_{2}}\cdots h_{\nu_{p}}%
}=\overline{h_{\nu_{1}}}\overline{h_{\nu_{2}}}\cdots\overline{h_{\nu_{p}}}\\
&  =\underbrace{\left(  \overline{h_{\nu_{1}}}\overline{h_{\nu_{2}}}%
\cdots\overline{h_{\nu_{q-1}}}\right)  }_{\in L_{q-1}C^{q-1}}%
\underbrace{\overline{h_{\nu_{q}}}}_{\in R_{n-k-q-1}C}\underbrace{\left(
\overline{h_{\nu_{q+1}}}\overline{h_{\nu_{q+2}}}\cdots\overline{h_{\nu_{p}}%
}\right)  }_{\substack{=\overline{h_{\nu_{q+1}}h_{\nu_{q+2}}\cdots h_{\nu_{p}%
}}\\=\overline{h_{\left(  \nu_{q+1},\nu_{q+2},\ldots,\nu_{p}\right)  }}\in
L_{p-q}\\\text{(by (\ref{pf.thm.coeff-h.1.1}))}}}\\
&  \in L_{q-1}C^{q-1}R_{n-k-q-1}CL_{p-q}=\underbrace{C^{q-1}C}_{=C^{q}%
}\underbrace{R_{n-k-q-1}}_{\substack{\subseteq C^{n-k-q-1}\\\text{(by Lemma
\ref{lem.R.Cp},}\\\text{applied to }n-k-q-1\text{ instead of }p\text{)}%
}}\underbrace{L_{q-1}L_{p-q}}_{\substack{\subseteq L_{\left(  q-1\right)
+\left(  p-q\right)  }\\\text{(by (\ref{eq.prop.Lpalg.3}))}}}\\
&  \subseteq\underbrace{C^{q}C^{n-k-q-1}}_{=C^{q+\left(  n-k-q-1\right)
}=C^{n-k-1}}\underbrace{L_{\left(  q-1\right)  +\left(  p-q\right)  }%
}_{\substack{=L_{p-1}\subseteq Q_{0}\\\text{(by (\ref{pf.thm.coeff-h.1.Lp-1}%
))}}}\subseteq C^{n-k-1}Q_{0}\subseteq Q_{0+\left(  n-k-1\right)  }%
\end{align*}
(by Corollary \ref{cor.Cq.Qp}, applied to $n-k-1$ and $0$ instead of $q$ and
$p$). In other words, $\overline{h_{\nu}}\in Q_{n-k-1}$. Hence,
$\operatorname*{coeff}\nolimits_{\omega}\left(  \overline{h_{\nu}}\right)
\in\operatorname*{coeff}\nolimits_{\omega}\left(  Q_{n-k-1}\right)  =0$ (by
Lemma \ref{lem.coeffw.coeffQ}), and thus $\operatorname*{coeff}%
\nolimits_{\omega}\left(  \overline{h_{\nu}}\right)  =0$. This proves Theorem
\ref{thm.coeff-h.1}.
\end{proof}

\begin{theorem}
\label{thm.coeff-h.2}Assume that $n>k$. Let $\gamma=\left(  \gamma_{1}%
,\gamma_{2},\ldots,\gamma_{k}\right)  \in\mathbb{Z}^{k}$ be a $k$-tuple of
integers such that $\gamma\neq\omega$.

Assume that
\begin{equation}
\gamma_{i}\leq2n-k-i\ \ \ \ \ \ \ \ \ \ \text{for each }i\in\left\{
1,2,\ldots,k\right\}  . \label{eq.lem.coeff-h.2.ass1}%
\end{equation}
Then, $\operatorname*{coeff}\nolimits_{\omega}\left(  \overline{h_{\gamma}%
}\right)  =0$.
\end{theorem}

\begin{proof}
[Proof of Theorem \ref{thm.coeff-h.2}.]We have $k\neq0$%
\ \ \ \ \footnote{\textit{Proof.} Assume the contrary. Thus, $k=0$. Now,
$\gamma\in\mathbb{Z}^{k}=\mathbb{Z}^{0}$ (since $k=0$), whence $\gamma=\left(
{}\right)  $. But $k=0$ also leads to $\omega=\left(  {}\right)  $, and thus
$\gamma=\left(  {}\right)  =\omega$. But this contradicts $\gamma\neq\omega$.
This contradiction shows that our assumption was false. Qed.}. Thus, $k>0$;
hence, $\gamma_{1}$ is well-defined.

If any of the entries $\gamma_{1},\gamma_{2},\ldots,\gamma_{k}$ of $\gamma$ is
negative, then Theorem \ref{thm.coeff-h.2} holds for easy
reasons\footnote{Indeed, in this case we have $\gamma_{i}<0$ for some
$i\in\left\{  1,2,\ldots,k\right\}  $, and therefore $h_{\gamma_{i}}=0$ for
this $i$, and therefore
\[
h_{\gamma}=h_{\gamma_{1}}h_{\gamma_{2}}\cdots h_{\gamma_{k}}=\left(
h_{\gamma_{1}}h_{\gamma_{2}}\cdots h_{\gamma_{i-1}}\right)
\underbrace{h_{\gamma_{i}}}_{=0}\left(  h_{\gamma_{i+1}}h_{\gamma_{i+2}}\cdots
h_{\gamma_{k}}\right)  =0,
\]
and therefore $\operatorname*{coeff}\nolimits_{\omega}\left(  \overline
{h_{\gamma}}\right)  =0$, qed.}. Hence, we WLOG assume that none of the
entries $\gamma_{1},\gamma_{2},\ldots,\gamma_{k}$ of $\gamma$ is negative.
Thus, all of the entries $\gamma_{1},\gamma_{2},\ldots,\gamma_{k}$ are
nonnegative integers. In other words, $\left(  \gamma_{1},\gamma_{2}%
,\ldots,\gamma_{k}\right)  \in\mathbb{N}^{k}$.

Let $\nu=\left(  \nu_{1},\nu_{2},\ldots,\nu_{k}\right)  \in\mathbb{Z}^{k}$ be
the weakly decreasing permutation of the $k$-tuple $\gamma=\left(  \gamma
_{1},\gamma_{2},\ldots,\gamma_{k}\right)  $. Thus, $h_{\nu_{1}}h_{\nu_{2}%
}\cdots h_{\nu_{k}}=h_{\gamma_{1}}h_{\gamma_{2}}\cdots h_{\gamma_{k}}$. Hence,
$h_{\nu}=h_{\nu_{1}}h_{\nu_{2}}\cdots h_{\nu_{k}}=h_{\gamma_{1}}h_{\gamma_{2}%
}\cdots h_{\gamma_{k}}=h_{\gamma}$.

Recall that $\left(  \nu_{1},\nu_{2},\ldots,\nu_{k}\right)  $ is a permutation
of $\left(  \gamma_{1},\gamma_{2},\ldots,\gamma_{k}\right)  $. In other words,
there exists a permutation $\sigma\in S_{k}$ such that
\begin{equation}
\left(  \nu_{i}=\gamma_{\sigma\left(  i\right)  }\text{ for each }i\in\left\{
1,2,\ldots,k\right\}  \right)  . \label{pf.lem.coeff-h.2.sig}%
\end{equation}
Consider this $\sigma$.

Recall that $\left(  \nu_{1},\nu_{2},\ldots,\nu_{k}\right)  $ is weakly
decreasing. Thus, $\nu_{1}\geq\nu_{2}\geq\cdots\geq\nu_{k}$. Also, $\left(
\nu_{1},\nu_{2},\ldots,\nu_{k}\right)  \in\mathbb{N}^{k}$ (since $\left(
\nu_{1},\nu_{2},\ldots,\nu_{k}\right)  $ is a permutation of $\left(
\gamma_{1},\gamma_{2},\ldots,\gamma_{k}\right)  \in\mathbb{N}^{k}$).

For each $i\in\left\{  1,2,\ldots,k\right\}  $, we have%
\begin{align}
\nu_{i}  &  =\gamma_{\sigma\left(  i\right)  }\ \ \ \ \ \ \ \ \ \ \left(
\text{by (\ref{pf.lem.coeff-h.2.sig})}\right) \nonumber\\
&  \leq2n-k-\sigma\left(  i\right)  \label{pf.lem.coeff-h.2.ineq}%
\end{align}
(by (\ref{eq.lem.coeff-h.2.ass1}), applied to $\sigma\left(  i\right)  $
instead of $i$).

We are in one of the following two cases:

\textit{Case 1:} We have $\nu_{1}\leq n$.

\textit{Case 2:} We have $\nu_{1}>n$.

Let us first consider Case 1. In this case, we have $\nu_{1}\leq n$. But
recall that $\gamma\neq\omega$. Hence, there exists at least one $q\in\left\{
1,2,\ldots,k\right\}  $ satisfying $\nu_{q}\neq n-k$%
\ \ \ \ \footnote{\textit{Proof.} Assume the contrary. Thus, $\nu_{i}=n-k$ for
each $i\in\left\{  1,2,\ldots,k\right\}  $. Now, let $j\in\left\{
1,2,\ldots,k\right\}  $ be arbitrary. Then, $\nu_{\sigma^{-1}\left(  j\right)
}=n-k$ (since $\nu_{i}=n-k$ for each $i\in\left\{  1,2,\ldots,k\right\}  $).
But (\ref{pf.lem.coeff-h.2.sig}) (applied to $i=\sigma^{-1}\left(  j\right)
$) yields $\nu_{\sigma^{-1}\left(  j\right)  }=\gamma_{\sigma\left(
\sigma^{-1}\left(  j\right)  \right)  }=\gamma_{j}$. Hence, $\gamma_{j}%
=\nu_{\sigma^{-1}\left(  j\right)  }=n-k$. Now, forget that we fixed $j$. We
thus have proven that $\gamma_{j}=n-k$ for each $j\in\left\{  1,2,\ldots
,k\right\}  $. Hence, $\gamma=\left(  n-k,n-k,\ldots,n-k\right)  =\omega$.
This contradicts $\gamma\neq\omega$. This contradiction shows that our
assumption was false, qed.}. Consider such a $q$.

Next, we claim that%
\begin{equation}
\overline{h_{\nu_{i}}}\in L_{1}\ \ \ \ \ \ \ \ \ \ \text{for each }%
i\in\left\{  1,2,\ldots,k\right\}  . \label{pf.lem.coeff-h.2.1}%
\end{equation}

[\textit{Proof of (\ref{pf.lem.coeff-h.2.1}):} Let $i\in\left\{
1,2,\ldots,k\right\}  $. We have $\nu_{1}\geq\nu_{2}\geq\cdots\geq\nu_{k}$,
thus $\nu_{i}\leq\nu_{1}\leq n$. Now, $\nu_{i}\leq n$ and $\nu_{i}%
\in\mathbb{N}$ (since $\left(  \nu_{1},\nu_{2},\ldots,\nu_{k}\right)
\in\mathbb{N}^{k}$). Hence, Lemma \ref{lem.hi-again} \textbf{(a)} (applied to
$j=\nu_{i}$) yields $\overline{h_{\nu_{i}}}\in L_{1}$. This proves
(\ref{pf.lem.coeff-h.2.1}).]

Also, $\nu_{1}\geq\nu_{2}\geq\cdots\geq\nu_{k}$, thus $\nu_{q}\leq\nu_{1}\leq
n$. Also, $n>k$ and $\nu_{q}\in\mathbb{N}$ (since $\left(  \nu_{1},\nu
_{2},\ldots,\nu_{k}\right)  \in\mathbb{N}^{k}$) and $\nu_{q}\neq n-k$. Hence,
Lemma \ref{lem.hi-again} \textbf{(b)} (applied to $j=\nu_{q}$) yields
$\overline{h_{\nu_{q}}}\in R_{n-k-1}$. From $n>k$, we obtain $n-k>0$, so that
$n-k\geq1$, and thus $n-k-1\in\mathbb{N}$.

Now, $h_{\nu}=h_{\nu_{1}}h_{\nu_{2}}\cdots h_{\nu_{k}}=\prod_{i=1}^{k}%
h_{\nu_{i}}$, so that%
\begin{align*}
\overline{h_{\nu}}  &  =\overline{\prod_{i=1}^{k}h_{\nu_{i}}}=\prod_{i=1}%
^{k}\overline{h_{\nu_{i}}}=\left(  \prod_{\substack{i\in\left\{
1,2,\ldots,k\right\}  ;\\i\neq q}}\underbrace{\overline{h_{\nu_{i}}}%
}_{\substack{\in L_{1}\\\text{(by (\ref{pf.lem.coeff-h.2.1}))}}}\right)
\underbrace{\overline{h_{\nu_{q}}}}_{\in R_{n-k-1}}\\
&  \in\underbrace{\left(  \prod_{\substack{i\in\left\{  1,2,\ldots,k\right\}
;\\i\neq q}}L_{1}\right)  }_{\substack{=\left(  L_{1}\right)  ^{k-1}\subseteq
L_{k-1}\\\text{(by Corollary \ref{cor.Lpalgfil})}}}\underbrace{R_{n-k-1}%
}_{\substack{\subseteq C^{n-k-1}\\\text{(by Lemma \ref{lem.R.Cp}%
,}\\\text{applied to }n-k-1\text{ instead of }p\text{)}}}\subseteq
\underbrace{L_{k-1}}_{\substack{=Q_{0}\\\text{(by Proposition \ref{prop.Lk-1}%
)}}}C^{n-k-1}\\
&  =Q_{0}C^{n-k-1}=C^{n-k-1}Q_{0}\subseteq Q_{0+\left(  n-k-1\right)  }%
\end{align*}
(by Corollary \ref{cor.Cq.Qp}, applied to $n-k-1$ and $0$ instead of $q$ and
$p$). In other words, $\overline{h_{\nu}}\in Q_{n-k-1}$. In view of $h_{\nu
}=h_{\gamma}$, this rewrites as $\overline{h_{\gamma}}\in Q_{n-k-1}$. Hence,
$\operatorname*{coeff}\nolimits_{\omega}\left(  \overline{h_{\gamma}}\right)
\in\operatorname*{coeff}\nolimits_{\omega}\left(  Q_{n-k-1}\right)  =0$ (by
Lemma \ref{lem.coeffw.coeffQ}), and thus $\operatorname*{coeff}%
\nolimits_{\omega}\left(  \overline{h_{\gamma}}\right)  =0$. Thus, Theorem
\ref{thm.coeff-h.2} is proven in Case 1.

Let us now consider Case 2. In this case, we have $\nu_{1}>n$. Hence, there
exists at least one $r\in\left\{  1,2,\ldots,k\right\}  $ such that $\nu
_{r}>n$ (namely, $r=1$). Let $q$ be the \textbf{largest} such $r$. Thus,
$\nu_{q}>n$, but each $r>q$ satisfies $\nu_{r}\leq n$. Hence,%
\[
\nu_{1}\geq\nu_{2}\geq\cdots\geq\nu_{q}>n\geq\nu_{q+1}\geq\nu_{q+2}\geq
\cdots\geq\nu_{k}%
\]
(since $\nu_{1}\geq\nu_{2}\geq\cdots\geq\nu_{k}$). Also, $\nu_{q}\leq
2n-k-q$\ \ \ \ \footnote{\textit{Proof.} The map $\sigma$ is a permutation,
and thus injective. Hence, $\left\vert \sigma\left(  \left\{  1,2,\ldots
,q\right\}  \right)  \right\vert =\left\vert \left\{  1,2,\ldots,q\right\}
\right\vert =q$. Thus, $\sigma\left(  \left\{  1,2,\ldots,q\right\}  \right)
$ cannot be a subset of $\left\{  1,2,\ldots,q-1\right\}  $ (because this
would lead to $\left\vert \sigma\left(  \left\{  1,2,\ldots,q\right\}
\right)  \right\vert \leq\left\vert \left\{  1,2,\ldots,q-1\right\}
\right\vert =q-1<q$, which would contradict $\left\vert \sigma\left(  \left\{
1,2,\ldots,q\right\}  \right)  \right\vert =q$). In other words, not every
$i\in\left\{  1,2,\ldots,q\right\}  $ satisfies $\sigma\left(  i\right)
\in\left\{  1,2,\ldots,q-1\right\}  $. In other words, there exists some
$i\in\left\{  1,2,\ldots,q\right\}  $ that satisfies $\sigma\left(  i\right)
\notin\left\{  1,2,\ldots,q-1\right\}  $. Consider such an $i$.
\par
From $i\in\left\{  1,2,\ldots,q\right\}  $, we obtain $i\leq q$ and thus
$\nu_{i}\geq\nu_{q}$ (since $\nu_{1}\geq\nu_{2}\geq\cdots\geq\nu_{k}$). Hence,
$\nu_{q}\leq\nu_{i}$. From $\sigma\left(  i\right)  \notin\left\{
1,2,\ldots,q-1\right\}  $, we obtain $\sigma\left(  i\right)  >q-1$, so that
$\sigma\left(  i\right)  \geq q$. Now, (\ref{pf.lem.coeff-h.2.ineq}) yields
$\nu_{i}\leq2n-k-\underbrace{\sigma\left(  i\right)  }_{\geq q}\leq2n-k-q$.
Now, $\nu_{q}\leq\nu_{i}\leq2n-k-q$, qed.}. Furthermore,
(\ref{pf.lem.coeff-h.2.ineq}) shows that%
\[
\nu_{i}\leq2n-k-\underbrace{\sigma\left(  i\right)  }_{\geq-1}\leq2n-k+1
\]
for each $i\in\left\{  1,2,\ldots,k\right\}  $. Hence, Theorem
\ref{thm.coeff-h.1} (applied to $p=k$) yields $\operatorname*{coeff}%
\nolimits_{\omega}\left(  \overline{h_{\nu}}\right)  =0$. In view of $h_{\nu
}=h_{\gamma}$, this rewrites as $\operatorname*{coeff}\nolimits_{\omega
}\left(  \overline{h_{\gamma}}\right)  =0$. Thus, Theorem \ref{thm.coeff-h.2}
is proven in Case 2.

We have now proven Theorem \ref{thm.coeff-h.2} in both Cases 1 and 2. Hence,
Theorem \ref{thm.coeff-h.2} always holds.
\end{proof}

\subsection{A criterion for $\operatorname*{coeff}\nolimits_{\omega}\left(
\overline{s_{\lambda}}\right)  =0$}

\begin{theorem}
\label{thm.coeff-s.2}Let $\lambda$ be a partition with at most $k$ parts.
Assume that $\lambda_{1}\leq2\left(  n-k\right)  $ and $\lambda\neq\omega$.
Then, $\operatorname*{coeff}\nolimits_{\omega}\left(  \overline{s_{\lambda}%
}\right)  =0$.
\end{theorem}

\begin{proof}
[Proof of Theorem \ref{thm.coeff-s.2}.]We have $n>k$%
\ \ \ \ \footnote{\textit{Proof.} Assume the contrary. Thus, $n\leq k$ and
therefore $n=k$ (since $n\geq k$). Hence, $n-k=0$. Thus, $\lambda_{1}%
\leq2\underbrace{\left(  n-k\right)  }_{=0}=0$, so that $\lambda_{1}=0$ and
thus $\lambda=\varnothing$ (since $\lambda$ is a partition). But from $n-k=0$,
we also obtain $\omega=\varnothing$ (since $\omega=\left(  n-k,n-k,\ldots
,n-k\right)  $). Thus, $\lambda=\varnothing=\omega$. But this contradicts
$\lambda\neq\omega$. This contradiction shows that our assumption was wrong,
qed.}.

We have $\lambda=\left(  \lambda_{1},\lambda_{2},\ldots,\lambda_{k}\right)  $
(since the partition $\lambda$ has at most $k$ parts). Proposition
\ref{prop.jacobi-trudi.Sh} \textbf{(a)} yields%
\[
s_{\lambda}=\det\left(  \left(  h_{\lambda_{u}-u+v}\right)  _{1\leq u\leq
k,\ 1\leq v\leq k}\right)  =\sum_{\sigma\in S_{k}}\left(  -1\right)  ^{\sigma
}\prod_{i=1}^{k}h_{\lambda_{i}-i+\sigma\left(  i\right)  }%
\]
(by the definition of a determinant). Hence,%
\begin{equation}
\overline{s_{\lambda}}=\overline{\sum_{\sigma\in S_{k}}\left(  -1\right)
^{\sigma}\prod_{i=1}^{k}h_{\lambda_{i}-i+\sigma\left(  i\right)  }}%
=\sum_{\sigma\in S_{k}}\left(  -1\right)  ^{\sigma}\overline{\prod_{i=1}%
^{k}h_{\lambda_{i}-i+\sigma\left(  i\right)  }}. \label{pf.thm.coeff-s.2.jt2}%
\end{equation}

Now, we claim that each $\sigma\in S_{k}$ satisfies%
\begin{equation}
\operatorname*{coeff}\nolimits_{\omega}\left(  \overline{\prod_{i=1}%
^{k}h_{\lambda_{i}-i+\sigma\left(  i\right)  }}\right)  =0.
\label{pf.thm.coeff-s.2.claim}%
\end{equation}

[\textit{Proof of (\ref{pf.thm.coeff-s.2.claim}):} Let $\sigma\in S_{k}$.
Define a $k$-tuple $\gamma=\left(  \gamma_{1},\gamma_{2},\ldots,\gamma
_{k}\right)  \in\mathbb{Z}^{k}$ of integers by%
\begin{equation}
\left(  \gamma_{i}=\lambda_{i}-i+\sigma\left(  i\right)
\ \ \ \ \ \ \ \ \ \ \text{for each }i\in\left\{  1,2,\ldots,k\right\}
\right)  . \label{pf.thm.coeff-s.2.claim.pf.gammai=}%
\end{equation}
Then, $\gamma\neq\omega$\ \ \ \ \footnote{\textit{Proof.} Assume the contrary.
Thus, $\gamma=\omega$.
\par
Let $i\in\left\{  1,2,\ldots,k-1\right\}  $. From $\gamma=\omega$, we obtain
$\gamma_{i}=\omega_{i}=n-k$. Comparing this with
(\ref{pf.thm.coeff-s.2.claim.pf.gammai=}), we find $\lambda_{i}-i+\sigma
\left(  i\right)  =n-k$. The same argument (applied to $i+1$ instead of $i$)
yields $\lambda_{i+1}-\left(  i+1\right)  +\sigma\left(  i+1\right)  =n-k$.
But $\lambda_{i}\geq\lambda_{i+1}$ (since $\lambda$ is a partition). Hence,%
\[
\underbrace{\lambda_{i}}_{\geq\lambda_{i+1}}-\underbrace{i}_{<i+1}%
+\sigma\left(  i+1\right)  >\lambda_{i+1}-\left(  i+1\right)  +\sigma\left(
i+1\right)  =n-k=\lambda_{i}-i+\sigma\left(  i\right)
\]
(since $\lambda_{i}-i+\sigma\left(  i\right)  =n-k$). If we subtract
$\lambda_{i}-i$ from this inequality, we obtain $\sigma\left(  i+1\right)
>\sigma\left(  i\right)  $. In other words, $\sigma\left(  i\right)
<\sigma\left(  i+1\right)  $.
\par
Now, forget that we fixed $i$. We thus have shown that each $i\in\left\{
1,2,\ldots,k-1\right\}  $ satisfies $\sigma\left(  i\right)  <\sigma\left(
i+1\right)  $. In other words, we have $\sigma\left(  1\right)  <\sigma\left(
2\right)  <\cdots<\sigma\left(  k\right)  $. Hence, $\sigma$ is a strictly
increasing map from $\left\{  1,2,\ldots,k\right\}  $ to $\left\{
1,2,\ldots,k\right\}  $. But the only such map is $\operatorname*{id}$. Thus,
$\sigma=\operatorname*{id}$. Hence, for each $i\in\left\{  1,2,\ldots
,k\right\}  $, we have
\begin{align*}
\gamma_{i}  &  =\lambda_{i}-i+\underbrace{\sigma}_{=\operatorname*{id}}\left(
i\right)  \ \ \ \ \ \ \ \ \ \ \left(  \text{by
(\ref{pf.thm.coeff-s.2.claim.pf.gammai=})}\right) \\
&  =\lambda_{i}-i+\operatorname*{id}\left(  i\right)  =\lambda_{i}%
-i+i=\lambda_{i}.
\end{align*}
Thus, $\gamma=\lambda$. Comparing this with $\gamma=\omega$, we obtain
$\lambda=\omega$. This contradicts $\lambda\neq\omega$. This contradiction
shows that our assumption was wrong, qed.}. Moreover,%
\[
\gamma_{i}\leq2n-k-i\ \ \ \ \ \ \ \ \ \ \text{for each }i\in\left\{
1,2,\ldots,k\right\}
\]
\footnote{\textit{Proof.} Let $i\in\left\{  1,2,\ldots,k\right\}  $. Then,
$\lambda_{1}\geq\lambda_{i}$ (since $\lambda$ is a partition), so that
$\lambda_{i}\leq\lambda_{1}\leq2\left(  n-k\right)  $. Now,
(\ref{pf.thm.coeff-s.2.claim.pf.gammai=}) yields%
\[
\gamma_{i}=\underbrace{\lambda_{i}}_{\leq2\left(  n-k\right)  }%
-i+\underbrace{\sigma\left(  i\right)  }_{\leq k}\leq2\left(  n-k\right)
-i+k=2n-k-i,
\]
qed.}. Hence, Theorem \ref{thm.coeff-h.2} yields $\operatorname*{coeff}%
\nolimits_{\omega}\left(  \overline{h_{\gamma}}\right)  =0$. In view of%
\[
h_{\gamma}=h_{\gamma_{1}}h_{\gamma_{2}}\cdots h_{\gamma_{k}}=\prod_{i=1}%
^{k}\underbrace{h_{\gamma_{i}}}_{\substack{=h_{\lambda_{i}-i+\sigma\left(
i\right)  }\\\text{(by (\ref{pf.thm.coeff-s.2.claim.pf.gammai=}))}}%
}=\prod_{i=1}^{k}h_{\lambda_{i}-i+\sigma\left(  i\right)  },
\]
this rewrites as $\operatorname*{coeff}\nolimits_{\omega}\left(
\overline{\prod_{i=1}^{k}h_{\lambda_{i}-i+\sigma\left(  i\right)  }}\right)
=0$. Thus, (\ref{pf.thm.coeff-s.2.claim}) is proven.]

From (\ref{pf.thm.coeff-s.2.jt2}), we obtain%
\begin{align*}
\operatorname*{coeff}\nolimits_{\omega}\left(  \overline{s_{\lambda}}\right)
&  =\operatorname*{coeff}\nolimits_{\omega}\left(  \sum_{\sigma\in S_{k}%
}\left(  -1\right)  ^{\sigma}\overline{\prod_{i=1}^{k}h_{\lambda_{i}%
-i+\sigma\left(  i\right)  }}\right) \\
&  =\sum_{\sigma\in S_{k}}\left(  -1\right)  ^{\sigma}%
\underbrace{\operatorname*{coeff}\nolimits_{\omega}\left(  \overline
{\prod_{i=1}^{k}h_{\lambda_{i}-i+\sigma\left(  i\right)  }}\right)
}_{\substack{=0\\\text{(by (\ref{pf.thm.coeff-s.2.claim}))}}}=0.
\end{align*}
This proves Theorem \ref{thm.coeff-s.2}.
\end{proof}

\section{\label{sect.redh-another-proof}Another proof of Theorem
\ref{thm.coeffw}}

We can use Theorem \ref{thm.coeff-s.2} to obtain a second proof of Theorem
\ref{thm.coeffw}. To that end, we shall use a few more basic facts about
Littlewood-Richardson coefficients. First we introduce a few notations (only
for this section):

\begin{convention}
Convention \ref{conv.symmetry-conv} remains in place for the whole Section
\ref{sect.redh-another-proof}.

We shall also use all the notations introduced in Section \ref{sect.symmetry}.
\end{convention}

\subsection{Some basics on Littlewood-Richardson coefficients}

\begin{definition}
\label{def.LRcoeffs.domi}Let $a\in\mathbb{N}$.

\textbf{(a)} We let $\operatorname*{Par}\nolimits_{a}$ denote the set of all
partitions with size $a$. (That is, $\operatorname*{Par}\nolimits_{a}=\left\{
\lambda\text{ is a partition}\ \mid\ \left\vert \lambda\right\vert =a\right\}
$.)

\textbf{(b)} If $\lambda$ and $\mu$ are two partitions with size $a$, then we
write $\lambda\triangleright\mu$ if and only if we have%
\[
\lambda_{1}+\lambda_{2}+\cdots+\lambda_{i}\geq\mu_{1}+\mu_{2}+\cdots+\mu
_{i}\ \ \ \ \ \ \ \ \ \ \text{for each }i\in\left\{  1,2,\ldots,a\right\}  .
\]
This defines a binary relation $\triangleright$ on $\operatorname*{Par}%
\nolimits_{a}$. This relation is the smaller-or-equal relation of a partial
order on $\operatorname*{Par}\nolimits_{a}$, which is called the
\textit{dominance order}.
\end{definition}

Here is another way to describe the dominance order:

\begin{remark}
\label{rmk.LRcoeffs.domi=domi}Let $a\in\mathbb{N}$. Let $\lambda$ and $\mu$ be
two partitions with size $a$. Then, we have $\lambda\triangleright\mu$ if and
only if we have%
\begin{equation}
\lambda_{1}+\lambda_{2}+\cdots+\lambda_{i}\geq\mu_{1}+\mu_{2}+\cdots+\mu
_{i}\ \ \ \ \ \ \ \ \ \ \text{for each }i\geq1.
\label{eq.rmk.LRcoeffs.domi=domi.1}%
\end{equation}

\end{remark}

\begin{proof}
[Proof of Remark \ref{rmk.LRcoeffs.domi=domi}.]$\Longleftarrow:$ Assume that
we have (\ref{eq.rmk.LRcoeffs.domi=domi.1}). We must prove that $\lambda
\triangleright\mu$.

For each $i\in\left\{  1,2,\ldots,a\right\}  $, we have $i\geq1$ and therefore
$\lambda_{1}+\lambda_{2}+\cdots+\lambda_{i}\geq\mu_{1}+\mu_{2}+\cdots+\mu_{i}$
(by (\ref{eq.rmk.LRcoeffs.domi=domi.1})). In other words, we have%
\[
\lambda_{1}+\lambda_{2}+\cdots+\lambda_{i}\geq\mu_{1}+\mu_{2}+\cdots+\mu
_{i}\ \ \ \ \ \ \ \ \ \ \text{for each }i\in\left\{  1,2,\ldots,a\right\}  .
\]
In other words, we have $\lambda\triangleright\mu$ (by the definition of the
relation $\triangleright$). This proves the \textquotedblleft$\Longleftarrow
$\textquotedblright\ direction of Remark \ref{rmk.LRcoeffs.domi=domi}.

$\Longrightarrow:$ Assume that $\lambda\triangleright\mu$. We must prove that
we have (\ref{eq.rmk.LRcoeffs.domi=domi.1}).

We have assumed that $\lambda\triangleright\mu$. In other words, we have%
\begin{equation}
\lambda_{1}+\lambda_{2}+\cdots+\lambda_{i}\geq\mu_{1}+\mu_{2}+\cdots+\mu
_{i}\ \ \ \ \ \ \ \ \ \ \text{for each }i\in\left\{  1,2,\ldots,a\right\}
\label{pf.rmk.LRcoeffs.domi=domi.1}%
\end{equation}
(by the definition of the relation $\triangleright$).

Now, let $i\geq1$. Our goal is to show that $\lambda_{1}+\lambda_{2}%
+\cdots+\lambda_{i}\geq\mu_{1}+\mu_{2}+\cdots+\mu_{i}$. If $i\in\left\{
1,2,\ldots,a\right\}  $, then this follows from
(\ref{pf.rmk.LRcoeffs.domi=domi.1}). Hence, for the rest of this proof, we
WLOG assume that we don't have $i\in\left\{  1,2,\ldots,a\right\}  $. Hence,
$i\geq a+1$ (because $i\geq1$), so that $a+1\leq i<i+1$. But $\lambda$ is a
partition; thus, $\lambda_{1}\geq\lambda_{2}\geq\lambda_{3}\geq\cdots$. Now,
recall that $\lambda$ is a partition of size $a$; hence, $\left\vert
\lambda\right\vert =a$. Thus,%
\begin{align*}
a  &  =\left\vert \lambda\right\vert =\lambda_{1}+\lambda_{2}+\lambda
_{3}+\cdots=\sum_{p=1}^{\infty}\lambda_{p}=\sum_{p=1}^{i+1}\underbrace{\lambda
_{p}}_{\substack{\geq\lambda_{i+1}\\\text{(since }p\leq i+1\\\text{and
}\lambda_{1}\geq\lambda_{2}\geq\lambda_{3}\geq\cdots\text{)}}%
}+\underbrace{\sum_{p=i+2}^{\infty}\lambda_{p}}_{\substack{\geq0\\\text{(since
all }\lambda_{p}\text{ are }\geq0\text{)}}}\\
&  \geq\sum_{p=1}^{i+1}\lambda_{i+1}=\left(  i+1\right)  \lambda_{i+1}.
\end{align*}
Hence, $\lambda_{i+1}\leq\dfrac{a}{i+1}<1$ (since $a<a+1<i+1$). Thus,
$\lambda_{i+1}=0$ (since $\lambda_{i+1}\in\mathbb{N}$). Furthermore, from
$\lambda_{1}\geq\lambda_{2}\geq\lambda_{3}\geq\cdots$, we conclude that each
$p\in\left\{  i+1,i+2,i+3,\ldots\right\}  $ satisfies $\lambda_{i+1}%
\geq\lambda_{p}$ and thus $\lambda_{p}\leq\lambda_{i+1}=0$ and therefore
$\lambda_{p}=0$ (since $\lambda_{p}\in\mathbb{N}$). Hence, $\sum
_{p=i+1}^{\infty}\underbrace{\lambda_{p}}_{=0}=\sum_{p=i+1}^{\infty}0=0$.
Thus,%
\[
a=\sum_{p=1}^{\infty}\lambda_{p}=\sum_{p=1}^{i}\lambda_{p}+\underbrace{\sum
_{p=i+1}^{\infty}\lambda_{p}}_{=0}=\sum_{p=1}^{i}\lambda_{p}=\lambda
_{1}+\lambda_{2}+\cdots+\lambda_{i}.
\]
The same argument (applied to the partition $\mu$ instead of $\lambda$) yields%
\[
a=\mu_{1}+\mu_{2}+\cdots+\mu_{i}.
\]
Comparing these two equalities, we find $\lambda_{1}+\lambda_{2}%
+\cdots+\lambda_{i}=\mu_{1}+\mu_{2}+\cdots+\mu_{i}$. Hence, $\lambda
_{1}+\lambda_{2}+\cdots+\lambda_{i}\geq\mu_{1}+\mu_{2}+\cdots+\mu_{i}$.

Now, forget that we fixed $i$. We thus have shown that
\[
\lambda_{1}+\lambda_{2}+\cdots+\lambda_{i}\geq\mu_{1}+\mu_{2}+\cdots+\mu
_{i}\ \ \ \ \ \ \ \ \ \ \text{for each }i\geq1.
\]
In other words, (\ref{eq.rmk.LRcoeffs.domi=domi.1}). This proves the
\textquotedblleft$\Longrightarrow$\textquotedblright\ direction of Remark
\ref{rmk.LRcoeffs.domi=domi}.
\end{proof}

\begin{definition}
Let $\mu$ and $\nu$ be two partitions. Then, we define two new partitions
$\mu+\nu$ and $\mu\sqcup\nu$ as follows:

\begin{itemize}
\item The partition $\mu+\nu$ is defined as $\left(  \mu_{1}+\nu_{1},\mu
_{2}+\nu_{2},\mu_{3}+\nu_{3},\ldots\right)  $.

\item The partition $\mu\sqcup\nu$ is defined as the result of sorting the
list $\left(  \mu_{1},\mu_{2},\ldots,\mu_{\ell\left(  \mu\right)  },\nu
_{1},\nu_{2},\ldots,\nu_{\ell\left(  \nu\right)  }\right)  $ in decreasing order.
\end{itemize}
\end{definition}

We shall use the following fact:

\begin{proposition}
\label{prop.LRcoeffs.dom}Let $a\in\mathbb{N}$ and $b\in\mathbb{N}$ be such
that $a\leq b$. Let $\mu\in\operatorname*{Par}\nolimits_{a}$, $\nu
\in\operatorname*{Par}\nolimits_{b-a}$ and $\lambda\in\operatorname*{Par}%
\nolimits_{b}$ be such that $c_{\mu,\nu}^{\lambda}\neq0$. Then, $\mu
+\nu\triangleright\lambda\triangleright\mu\sqcup\nu$.
\end{proposition}

Proposition \ref{prop.LRcoeffs.dom} is precisely \cite[Exercise 2.9.17(c)]%
{GriRei18} (with $k$ and $n$ renamed as $a$ and $b$).

\begin{corollary}
\label{cor.LRcoeffs.dom0}Let $\lambda$, $\mu$ and $\nu$ be three partitions
such that $\lambda_{1}>\mu_{1}+\nu_{1}$. Then, $c_{\mu,\nu}^{\lambda}=0$.
\end{corollary}

\begin{proof}
[Proof of Corollary \ref{cor.LRcoeffs.dom0}.]Assume the contrary. Thus,
$c_{\mu,\nu}^{\lambda}\neq0$.

Let $a=\left\vert \mu\right\vert $; thus, $\mu\in\operatorname*{Par}%
\nolimits_{a}$. Let $b=\left\vert \lambda\right\vert $; thus, $\lambda
\in\operatorname*{Par}\nolimits_{b}$.

Proposition \ref{prop.LR.props} \textbf{(c)} shows that $c_{\mu,\nu}^{\lambda
}=0$ unless $\left\vert \mu\right\vert +\left\vert \nu\right\vert =\left\vert
\lambda\right\vert $. Hence, $\left\vert \mu\right\vert +\left\vert
\nu\right\vert =\left\vert \lambda\right\vert $ (since $c_{\mu,\nu}^{\lambda
}\neq0$). Thus, $\left\vert \nu\right\vert =\underbrace{\left\vert
\lambda\right\vert }_{=b}-\underbrace{\left\vert \mu\right\vert }_{=a}=b-a$.
Hence, $b-a=\left\vert \nu\right\vert \geq0$, so that $a\leq b$. Also, from
$\left\vert \nu\right\vert =b-a$, we obtain $\nu\in\operatorname*{Par}%
\nolimits_{b-a}$. Thus, Proposition \ref{prop.LRcoeffs.dom} yields $\mu
+\nu\triangleright\lambda\triangleright\mu\sqcup\nu$.

But $b=\left\vert \lambda\right\vert \geq\lambda_{1}>\mu_{1}+\nu_{1}\geq0$, so
that $1\in\left\{  1,2,\ldots,b\right\}  $.

Now, from $\mu+\nu\triangleright\lambda$, we conclude that%
\begin{align*}
\left(  \mu+\nu\right)  _{1}+\left(  \mu+\nu\right)  _{2}+\cdots+\left(
\mu+\nu\right)  _{i}  &  \geq\lambda_{1}+\lambda_{2}+\cdots+\lambda_{i}\\
&  \ \ \ \ \ \ \ \ \ \ \text{for each }i\in\left\{  1,2,\ldots,b\right\}
\end{align*}
(by the definition of the relation $\triangleright$, since $\mu+\nu$ and
$\lambda$ are two partitions of size $b$). Applying this to $i=1$, we obtain
$\left(  \mu+\nu\right)  _{1}\geq\lambda_{1}$ (since $1\in\left\{
1,2,\ldots,b\right\}  $). But the definition of $\mu+\nu$ yields $\left(
\mu+\nu\right)  _{1}=\mu_{1}+\nu_{1}<\lambda_{1}$ (since $\lambda_{1}>\mu
_{1}+\nu_{1}$). This contradicts $\left(  \mu+\nu\right)  _{1}\geq\lambda_{1}%
$. This contradiction shows that our assumption was false. Hence, Corollary
\ref{cor.LRcoeffs.dom0} is proven.
\end{proof}

Next, we recall the Littlewood-Richardson rule itself:

\begin{proposition}
\label{prop.LRrule}Let $\lambda$ and $\mu$ be two partitions. Then,%
\[
\mathbf{s}_{\lambda}\mathbf{s}_{\mu}=\sum_{\rho\text{ is a partition}%
}c_{\lambda,\mu}^{\rho}\mathbf{s}_{\rho}.
\]

\end{proposition}

Proposition \ref{prop.LRrule} is precisely \cite[(2.5.6)]{GriRei18} (with
$\lambda$, $\mu$ and $\nu$ renamed as $\rho$, $\lambda$ and $\mu$).

\begin{corollary}
\label{cor.LRcoeffs.reduced}Let $\lambda\in P_{k,n}$ and $\mu\in P_{k,n}$.
Then,%
\[
s_{\lambda}s_{\mu}=\sum_{\substack{\rho\text{ is a partition with at most
}k\text{ parts;}\\\rho_{1}\leq2\left(  n-k\right)  }}c_{\lambda,\mu}^{\rho
}s_{\rho}.
\]

\end{corollary}

\begin{proof}
[Proof of Corollary \ref{cor.LRcoeffs.reduced}.]If $\rho$ is a partition
satisfying $\rho_{1}>2\left(  n-k\right)  $, then%
\begin{equation}
c_{\lambda,\mu}^{\rho}=0. \label{pf.cor.LRcoeffs.reduced.1}%
\end{equation}

[\textit{Proof of (\ref{pf.cor.LRcoeffs.reduced.1}):} Let $\rho$ be a
partition satisfying $\rho_{1}>2\left(  n-k\right)  $.

We have $\lambda\in P_{k,n}$; thus, each part of $\lambda$ is $\leq n-k$.
Thus, $\lambda_{1}\leq n-k$. Similarly, $\mu_{1}\leq n-k$. Hence,
$\underbrace{\lambda_{1}}_{\leq n-k}+\underbrace{\mu_{1}}_{\leq n-k}%
\leq2\left(  n-k\right)  <\rho_{1}$. In other words, $\rho_{1}>\lambda_{1}%
+\mu_{1}$. Hence, Corollary \ref{cor.LRcoeffs.dom0} (applied to $\rho$,
$\lambda$ and $\mu$ instead of $\lambda$, $\mu$ and $\nu$) yields
$c_{\lambda,\mu}^{\rho}=0$. This proves (\ref{pf.cor.LRcoeffs.reduced.1}).]

Proposition \ref{prop.LRrule} yields%
\[
\mathbf{s}_{\lambda}\mathbf{s}_{\mu}=\sum_{\rho\text{ is a partition}%
}c_{\lambda,\mu}^{\rho}\mathbf{s}_{\rho}.
\]
This is an equality in $\Lambda$. Evaluating both of its sides at the $k$
indeterminates $x_{1},x_{2},\ldots,x_{k}$, we find%
\begin{align*}
s_{\lambda}s_{\mu}  &  =\sum_{\rho\text{ is a partition}}c_{\lambda,\mu}%
^{\rho}s_{\rho}\\
&  =\sum_{\substack{\rho\text{ is a partition;}\\\rho_{1}\leq2\left(
n-k\right)  }}c_{\lambda,\mu}^{\rho}s_{\rho}+\sum_{\substack{\rho\text{ is a
partition;}\\\rho_{1}>2\left(  n-k\right)  }}\underbrace{c_{\lambda,\mu}%
^{\rho}}_{\substack{=0\\\text{(by (\ref{pf.cor.LRcoeffs.reduced.1}))}}%
}s_{\rho}\\
&  \ \ \ \ \ \ \ \ \ \ \left(
\begin{array}
[c]{c}%
\text{since each partition }\rho\text{ satisfies either }\rho_{1}\leq2\left(
n-k\right) \\
\text{or }\rho_{1}>2\left(  n-k\right)  \text{ (but not both)}%
\end{array}
\right) \\
&  =\sum_{\substack{\rho\text{ is a partition;}\\\rho_{1}\leq2\left(
n-k\right)  }}c_{\lambda,\mu}^{\rho}s_{\rho}\\
&  =\sum_{\substack{\rho\text{ is a partition with at most }k\text{
parts;}\\\rho_{1}\leq2\left(  n-k\right)  }}c_{\lambda,\mu}^{\rho}s_{\rho
}+\sum_{\substack{\rho\text{ is a partition with more than }k\text{
parts;}\\\rho_{1}\leq2\left(  n-k\right)  }}c_{\lambda,\mu}^{\rho
}\underbrace{s_{\rho}}_{\substack{=0\\\text{(by (\ref{eq.slam=0-too-long}))}%
}}\\
&  =\sum_{\substack{\rho\text{ is a partition with at most }k\text{
parts;}\\\rho_{1}\leq2\left(  n-k\right)  }}c_{\lambda,\mu}^{\rho}s_{\rho}.
\end{align*}
This proves Corollary \ref{cor.LRcoeffs.reduced}.
\end{proof}

Next, let us recall another known fact on skew Schur functions:

\begin{proposition}
\label{prop.skew-schur.complement}Let $\lambda$ be any partition. Then,
$\mathbf{s}_{\omega/\lambda^{\vee}}=\mathbf{s}_{\lambda}$.
\end{proposition}

\begin{proof}
[Proof of Proposition \ref{prop.skew-schur.complement}.]From \cite[Exercise
2.9.15(a)]{GriRei18} (applied to $n-k$ and $\varnothing$ instead of $m$ and
$\mu$), we obtain $\mathbf{s}_{\lambda/\varnothing}=\mathbf{s}_{\varnothing
^{\vee}/\lambda^{\vee}}$. In view of $\varnothing^{\vee}=\omega$, this
rewrites as $\mathbf{s}_{\lambda/\varnothing}=\mathbf{s}_{\omega/\lambda
^{\vee}}$. Thus, $\mathbf{s}_{\omega/\lambda^{\vee}}=\mathbf{s}_{\lambda
/\varnothing}=\mathbf{s}_{\lambda}$. This proves Proposition
\ref{prop.skew-schur.complement}.
\end{proof}

\begin{corollary}
\label{cor.LRcoeffs.omega}Let $\lambda$ and $\mu$ be two partitions. Then,%
\[
c_{\lambda,\mu}^{\omega}=%
\begin{cases}
1, & \text{if }\lambda\in P_{k,n}\text{ and }\mu=\lambda^{\vee};\\
0, & \text{else}%
\end{cases}
.
\]

\end{corollary}

\begin{proof}
[Proof of Corollary \ref{cor.LRcoeffs.omega}.]Proposition \ref{prop.LR.props}
\textbf{(a)} (applied to $\omega$ and $\lambda$ instead of $\lambda$ and $\mu
$) shows that
\begin{equation}
\mathbf{s}_{\omega/\lambda}=\sum_{\nu\text{ is a partition}}c_{\lambda,\nu
}^{\omega}\mathbf{s}_{\nu}. \label{pf.cor.LRcoeffs.omega.1}%
\end{equation}
On the other hand, it is easy to see that%
\begin{equation}
\mathbf{s}_{\omega/\lambda}=\sum_{\nu\text{ is a partition}}%
\begin{cases}
1, & \text{if }\lambda\in P_{k,n}\text{ and }\nu=\lambda^{\vee};\\
0, & \text{else}%
\end{cases}
\mathbf{s}_{\nu}. \label{pf.cor.LRcoeffs.omega.2}%
\end{equation}

[\textit{Proof of (\ref{pf.cor.LRcoeffs.omega.2}):} We are in one of the
following two cases:

\textit{Case 1:} We have $\lambda\in P_{k,n}$.

\textit{Case 2:} We have $\lambda\notin P_{k,n}$.

Let us first consider Case 1. In this case, we have $\lambda\in P_{k,n}$.
Thus, $\lambda^{\vee}$ is well-defined, and we have $\left(  \lambda^{\vee
}\right)  ^{\vee}=\lambda$. Hence, Proposition
\ref{prop.skew-schur.complement} (applied to $\lambda^{\vee}$ instead of
$\lambda$) yields
\begin{align*}
\mathbf{s}_{\omega/\left(  \lambda^{\vee}\right)  ^{\vee}}  &  =\mathbf{s}%
_{\lambda^{\vee}}=\sum_{\nu\text{ is a partition}}\underbrace{%
\begin{cases}
1, & \text{if }\nu=\lambda^{\vee};\\
0, & \text{else}%
\end{cases}
}_{\substack{=%
\begin{cases}
1, & \text{if }\lambda\in P_{k,n}\text{ and }\nu=\lambda^{\vee};\\
0, & \text{else}%
\end{cases}
\\\text{(since }\lambda\in P_{k,n}\text{ holds)}}}\mathbf{s}_{\nu}\\
&  =\sum_{\nu\text{ is a partition}}%
\begin{cases}
1, & \text{if }\lambda\in P_{k,n}\text{ and }\nu=\lambda^{\vee};\\
0, & \text{else}%
\end{cases}
\mathbf{s}_{\nu}.
\end{align*}
In view of $\left(  \lambda^{\vee}\right)  ^{\vee}=\lambda$, this rewrites as
\[
\mathbf{s}_{\omega/\lambda}=\sum_{\nu\text{ is a partition}}%
\begin{cases}
1, & \text{if }\lambda\in P_{k,n}\text{ and }\nu=\lambda^{\vee};\\
0, & \text{else}%
\end{cases}
\mathbf{s}_{\nu}.
\]
Thus, (\ref{pf.cor.LRcoeffs.omega.2}) is proven in Case 1.

Now, let us consider Case 2. In this case, we have $\lambda\notin P_{k,n}$.
Hence, $\lambda\not \subseteq \omega$ (since $\lambda\subseteq\omega$ holds if
and only if $\lambda\in P_{k,n}$). Thus, $\mathbf{s}_{\omega/\lambda}=0$.
Comparing this with%
\[
\sum_{\nu\text{ is a partition}}\underbrace{%
\begin{cases}
1, & \text{if }\lambda\in P_{k,n}\text{ and }\nu=\lambda^{\vee};\\
0, & \text{else}%
\end{cases}
}_{\substack{=0\\\text{(since }\lambda\notin P_{k,n}\text{)}}}\mathbf{s}_{\nu
}=0,
\]
we obtain%
\[
\mathbf{s}_{\omega/\lambda}=\sum_{\nu\text{ is a partition}}%
\begin{cases}
1, & \text{if }\lambda\in P_{k,n}\text{ and }\nu=\lambda^{\vee};\\
0, & \text{else}%
\end{cases}
\mathbf{s}_{\nu}.
\]
Thus, (\ref{pf.cor.LRcoeffs.omega.2}) is proven in Case 2.

We have now proven (\ref{pf.cor.LRcoeffs.omega.2}) in each of the two Cases 1
and 2. Thus, (\ref{pf.cor.LRcoeffs.omega.2}) always holds.]

Now, comparing (\ref{pf.cor.LRcoeffs.omega.2}) with
(\ref{pf.cor.LRcoeffs.omega.1}), we obtain%
\[
\sum_{\nu\text{ is a partition}}c_{\lambda,\nu}^{\omega}\mathbf{s}_{\nu}%
=\sum_{\nu\text{ is a partition}}%
\begin{cases}
1, & \text{if }\lambda\in P_{k,n}\text{ and }\nu=\lambda^{\vee};\\
0, & \text{else}%
\end{cases}
\mathbf{s}_{\nu}.
\]
Since the family $\left(  \mathbf{s}_{\nu}\right)  _{\nu\text{ is a
partition}}$ is a basis of the $\mathbf{k}$-module $\Lambda$, we can compare
the coefficients of $\mathbf{s}_{\mu}$ on both sides of this equality. We thus
obtain%
\[
c_{\lambda,\mu}^{\omega}=%
\begin{cases}
1, & \text{if }\lambda\in P_{k,n}\text{ and }\mu=\lambda^{\vee};\\
0, & \text{else}%
\end{cases}
.
\]
This proves Corollary \ref{cor.LRcoeffs.omega}.
\end{proof}

\subsection{Another proof of Theorem \ref{thm.coeffw}}

We are now ready to prove Theorem \ref{thm.coeffw} again. More precisely, we
shall prove Lemma \ref{lem.coeffw.prod2} (as we know that Theorem
\ref{thm.coeffw} quickly follows from Lemma \ref{lem.coeffw.prod2}).

\begin{proof}
[Second proof of Lemma \ref{lem.coeffw.prod2}.]If $k=0$, then Lemma
\ref{lem.coeffw.prod2} holds\footnote{\textit{Proof.} Assume that $k=0$. Then,
$P_{k,n}=\left\{  \varnothing\right\}  $, so that $\lambda\in P_{k,n}=\left\{
\varnothing\right\}  $ and thus $\lambda=\varnothing$. Similarly,
$\mu=\varnothing$. Therefore, $\lambda=\mu^{\vee}$ holds. Also, $\omega
=\varnothing$. Moreover, from $\lambda=\varnothing$, we obtain $s_{\lambda
}=s_{\varnothing}=1$; similarly, $s_{\mu}=1$. Thus, $\underbrace{s_{\lambda}%
}_{=1}\underbrace{s_{\mu}}_{=1}=1=s_{\varnothing}=s_{\omega}$ (since
$\varnothing=\omega$). Hence, $\operatorname*{coeff}\nolimits_{\omega}\left(
\overline{s_{\lambda}s_{\mu}}\right)  =\operatorname*{coeff}\nolimits_{\omega
}\left(  \overline{s_{\omega}}\right)  =1$. Comparing this with $%
\begin{cases}
1, & \text{if }\lambda=\mu^{\vee};\\
0, & \text{if }\lambda\neq\mu^{\vee}%
\end{cases}
=1$ (since $\lambda=\mu^{\vee}$ holds), we obtain $\operatorname*{coeff}%
\nolimits_{\omega}\left(  \overline{s_{\lambda}s_{\mu}}\right)  =%
\begin{cases}
1, & \text{if }\lambda=\mu^{\vee};\\
0, & \text{if }\lambda\neq\mu^{\vee}%
\end{cases}
$. Thus, Lemma \ref{lem.coeffw.prod2} holds. Qed.}. Hence, for the rest of
this proof, we WLOG assume that $k\neq0$. Thus, $k>0$. Hence, $\omega
_{1}=n-k\leq2\left(  n-k\right)  $. Thus, $\omega$ is a partition $\rho$ with
at most $k$ parts that satisfies $\rho_{1}\leq2\left(  n-k\right)  $ (since
$\omega_{1}\leq2\left(  n-k\right)  $).

Corollary \ref{cor.LRcoeffs.reduced} yields%
\[
s_{\lambda}s_{\mu}=\sum_{\substack{\rho\text{ is a partition with at most
}k\text{ parts;}\\\rho_{1}\leq2\left(  n-k\right)  }}c_{\lambda,\mu}^{\rho
}s_{\rho}.
\]
Hence,%
\[
\overline{s_{\lambda}s_{\mu}}=\overline{\sum_{\substack{\rho\text{ is a
partition with at most }k\text{ parts;}\\\rho_{1}\leq2\left(  n-k\right)
}}c_{\lambda,\mu}^{\rho}s_{\rho}}=\sum_{\substack{\rho\text{ is a partition
with at most }k\text{ parts;}\\\rho_{1}\leq2\left(  n-k\right)  }%
}c_{\lambda,\mu}^{\rho}\overline{s_{\rho}}.
\]
Thus,%
\begin{align*}
\operatorname*{coeff}\nolimits_{\omega}\left(  \overline{s_{\lambda}s_{\mu}%
}\right)   &  =\operatorname*{coeff}\nolimits_{\omega}\left(  \sum
_{\substack{\rho\text{ is a partition with at most }k\text{ parts;}\\\rho
_{1}\leq2\left(  n-k\right)  }}c_{\lambda,\mu}^{\rho}\overline{s_{\rho}%
}\right) \\
&  =\sum_{\substack{\rho\text{ is a partition with at most }k\text{
parts;}\\\rho_{1}\leq2\left(  n-k\right)  }}c_{\lambda,\mu}^{\rho
}\operatorname*{coeff}\nolimits_{\omega}\left(  \overline{s_{\rho}}\right) \\
&  =c_{\lambda,\mu}^{\omega}\operatorname*{coeff}\nolimits_{\omega}\left(
\overline{s_{\omega}}\right)  +\sum_{\substack{\rho\text{ is a partition with
at most }k\text{ parts;}\\\rho_{1}\leq2\left(  n-k\right)  ;\\\rho\neq\omega
}}c_{\lambda,\mu}^{\rho}\underbrace{\operatorname*{coeff}\nolimits_{\omega
}\left(  \overline{s_{\rho}}\right)  }_{\substack{=0\\\text{(by Theorem
\ref{thm.coeff-s.2},}\\\text{applied to }\rho\text{ instead of }%
\lambda\text{)}}}\\
&  \ \ \ \ \ \ \ \ \ \ \left(
\begin{array}
[c]{c}%
\text{here, we have split off the addend for }\rho=\omega\\
\text{from the sum, since }\omega\text{ is a partition }\rho\text{ with}\\
\text{at most }k\text{ parts that satisfies }\rho_{1}\leq2\left(  n-k\right)
\end{array}
\right) \\
&  =c_{\lambda,\mu}^{\omega}\underbrace{\operatorname*{coeff}\nolimits_{\omega
}\left(  \overline{s_{\omega}}\right)  }_{\substack{=1\\\text{(by the
definition of }\operatorname*{coeff}\nolimits_{\omega}\text{)}}}\\
&  =c_{\lambda,\mu}^{\omega}=%
\begin{cases}
1, & \text{if }\lambda\in P_{k,n}\text{ and }\mu=\lambda^{\vee};\\
0, & \text{else}%
\end{cases}
\ \ \ \ \ \ \ \ \ \ \left(  \text{by Corollary \ref{cor.LRcoeffs.omega}%
}\right) \\
&  =%
\begin{cases}
1, & \text{if }\mu=\lambda^{\vee};\\
0, & \text{if }\mu\neq\lambda^{\vee}%
\end{cases}
\ \ \ \ \ \ \ \ \ \ \left(  \text{since }\lambda\in P_{k,n}\text{
holds}\right) \\
&  =%
\begin{cases}
1, & \text{if }\lambda=\mu^{\vee};\\
0, & \text{if }\lambda\neq\mu^{\vee}%
\end{cases}
\end{align*}
(since $\mu=\lambda^{\vee}$ holds if and only if $\lambda=\mu^{\vee}$). Thus,
Lemma \ref{lem.coeffw.prod2} is proven again.
\end{proof}

\section{\label{sect.hm}The $h$-basis and the $m$-basis}

\begin{convention}
For the rest of Section \ref{sect.hm}, \textbf{we assume that }$a_{1}%
,a_{2},\ldots,a_{k}$ \textbf{belong to }$\mathcal{S}$\textbf{.}
\end{convention}

\subsection{A lemma on the $s$-basis}

For future use, we shall show a technical lemma, which improves on Lemma
\ref{lem.I.sl-red2}:

\begin{lemma}
\label{lem.sm-as-Pkn}Let $N\in\mathbb{N}$. Let $f\in\mathcal{S}$ be a
symmetric polynomial of degree $<N$. Then, in $\mathcal{S}/I$, we have
\[
\overline{f}\in\sum_{\substack{\kappa\in P_{k,n};\\\left\vert \kappa
\right\vert <N}}\mathbf{k}\overline{s_{\kappa}}.
\]

\end{lemma}

\begin{proof}
[Proof of Lemma \ref{lem.sm-as-Pkn}.]We shall prove Lemma \ref{lem.sm-as-Pkn}
by strong induction on $N$. Thus, we fix some $M\in\mathbb{N}$, and we assume
(as the induction hypothesis) that Lemma \ref{lem.sm-as-Pkn} holds whenever
$N<M$. We now must prove that Lemma \ref{lem.sm-as-Pkn} holds for $N=M$.

Let $f\in\mathcal{S}$ be a symmetric polynomial of degree $<M$. Then, in
$\mathcal{S}/I$, we shall show that $\overline{f}\in\sum_{\substack{\kappa\in
P_{k,n};\\\left\vert \kappa\right\vert <M}}\mathbf{k}\overline{s_{\kappa}}$.

Indeed, let $U$ be the $\mathbf{k}$-submodule $\sum_{\substack{\kappa\in
P_{k,n};\\\left\vert \kappa\right\vert <M}}\mathbf{k}\overline{s_{\kappa}}$ of
$\mathcal{S}/I$. Hence, $U$ is the $\mathbf{k}$-submodule of $\mathcal{S}/I$
spanned by the family $\left(  \overline{s_{\kappa}}\right)  _{\kappa\in
P_{k,n};\ \left\vert \kappa\right\vert <M}$. Hence,%
\begin{equation}
\overline{s_{\kappa}}\in U\ \ \ \ \ \ \ \ \ \ \text{for each }\kappa\in
P_{k,n}\text{ satisfying }\left\vert \kappa\right\vert <M.
\label{pf.lem.sm-as-Pkn.inU}%
\end{equation}
We are going to show that $\overline{f}\in U$.

Lemma \ref{lem.I.sm-as-smaller} (applied to $N=M$) shows that there exists a
family $\left(  c_{\kappa}\right)  _{\kappa\in P_{k};\ \left\vert
\kappa\right\vert <M}$ of elements of $\mathbf{k}$ such that $f=\sum
_{\substack{\kappa\in P_{k};\\\left\vert \kappa\right\vert <M}}c_{\kappa
}s_{\kappa}$. Consider this family. Thus,%
\begin{equation}
f=\sum_{\substack{\kappa\in P_{k};\\\left\vert \kappa\right\vert <M}%
}c_{\kappa}s_{\kappa}=\sum_{\substack{\mu\in P_{k};\\\left\vert \mu\right\vert
<M}}c_{\mu}s_{\mu} \label{pf.lem.sm-as-Pkn.f=1}%
\end{equation}
(here, we have renamed the summation index $\kappa$ as $\mu$).

Now, let $\mu\in P_{k}$ satisfy $\left\vert \mu\right\vert <M$. We shall show
that $\overline{s_{\mu}}\in U$.

[\textit{Proof:} If $\mu\in P_{k,n}$, then this follows directly from
(\ref{pf.lem.sm-as-Pkn.inU}) (applied to $\kappa=\mu$). Hence, for the rest of
this proof, we WLOG assume that $\mu\notin P_{k,n}$. Thus, Lemma
\ref{lem.I.sl-red} (applied to $\mu$ instead of $\lambda$) shows that%
\[
s_{\mu}\equiv\left(  \text{some symmetric polynomial of degree }<\left\vert
\mu\right\vert \right)  \operatorname{mod}I.
\]
In other words, there exists a symmetric polynomial $g\in\mathcal{S}$ of
degree $<\left\vert \mu\right\vert $ such that $s_{\mu}\equiv
g\operatorname{mod}I$. Consider this $g$. We have $\left\vert \mu\right\vert
<M$. Hence, Lemma \ref{lem.sm-as-Pkn} holds for $N=\left\vert \mu\right\vert $
(by our induction hypothesis). Thus, we can apply Lemma \ref{lem.sm-as-Pkn} to
$g$ and $\left\vert \mu\right\vert $ instead of $f$ and $N$. We thus conclude
that%
\[
\overline{g}\in\sum_{\substack{\kappa\in P_{k,n};\\\left\vert \kappa
\right\vert <\left\vert \mu\right\vert }}\mathbf{k}\overline{s_{\kappa}}.
\]
But from $s_{\mu}\equiv g\operatorname{mod}I$, we obtain%
\[
\overline{s_{\mu}}=\overline{g}\in\sum_{\substack{\kappa\in P_{k,n}%
;\\\left\vert \kappa\right\vert <\left\vert \mu\right\vert }}\mathbf{k}%
\overline{s_{\kappa}}\subseteq\sum_{\substack{\kappa\in P_{k,n};\\\left\vert
\kappa\right\vert <M}}\mathbf{k}\overline{s_{\kappa}}%
\]
(since each $\kappa\in P_{k,n}$ satisfying $\left\vert \kappa\right\vert
<\left\vert \mu\right\vert $ must also satisfy $\left\vert \kappa\right\vert
<M$ (because $\left\vert \kappa\right\vert <\left\vert \mu\right\vert <M$),
and therefore the sum $\sum_{\substack{\kappa\in P_{k,n};\\\left\vert
\kappa\right\vert <\left\vert \mu\right\vert }}\mathbf{k}\overline{s_{\kappa}%
}$ is a subsum of the sum $\sum_{\substack{\kappa\in P_{k,n};\\\left\vert
\kappa\right\vert <M}}\mathbf{k}\overline{s_{\kappa}}$). Hence,%
\[
\overline{s_{\mu}}\in\sum_{\substack{\kappa\in P_{k,n};\\\left\vert
\kappa\right\vert <M}}\mathbf{k}\overline{s_{\kappa}}%
=U\ \ \ \ \ \ \ \ \ \ \left(  \text{since }U\text{ is defined as }%
\sum_{\substack{\kappa\in P_{k,n};\\\left\vert \kappa\right\vert
<M}}\mathbf{k}\overline{s_{\kappa}}\right)  ,
\]
qed.]

Forget that we fixed $\mu$. We thus have shown that%
\begin{equation}
\overline{s_{\mu}}\in U\ \ \ \ \ \ \ \ \ \ \text{for each }\mu\in P_{k}\text{
satisfying }\left\vert \mu\right\vert <M. \label{pf.lem.sm-as-Pkn.4}%
\end{equation}

Now, (\ref{pf.lem.sm-as-Pkn.f=1}) yields%
\[
\overline{f}=\overline{\sum_{\substack{\mu\in P_{k};\\\left\vert
\mu\right\vert <M}}c_{\mu}s_{\mu}}=\sum_{\substack{\mu\in P_{k};\\\left\vert
\mu\right\vert <M}}c_{\mu}\underbrace{\overline{s_{\mu}}}_{\substack{\in
U\\\text{(by (\ref{pf.lem.sm-as-Pkn.4}))}}}\in\sum_{\substack{\mu\in
P_{k};\\\left\vert \mu\right\vert <M}}c_{\mu}U\subseteq U
\]
(since $U$ is a $\mathbf{k}$-module). In other words, $\overline{f}\in
\sum_{\substack{\kappa\in P_{k,n};\\\left\vert \kappa\right\vert
<M}}\mathbf{k}\overline{s_{\kappa}}$ (since $U=\sum_{\substack{\kappa\in
P_{k,n};\\\left\vert \kappa\right\vert <M}}\mathbf{k}\overline{s_{\kappa}}$).

Forget that we fixed $f$. We thus have shown that if $f\in\mathcal{S}$ is a
symmetric polynomial of degree $<M$, then $\overline{f}\in\sum
_{\substack{\kappa\in P_{k,n};\\\left\vert \kappa\right\vert <M}%
}\mathbf{k}\overline{s_{\kappa}}$. In other words, Lemma \ref{lem.sm-as-Pkn}
holds for $N=M$. This completes the induction step. Hence, Lemma
\ref{lem.sm-as-Pkn} is proven by induction.
\end{proof}

\subsection{The $h$-basis}

In Theorem \ref{thm.S/J-h-basis}, we have shown that the family $\left(
\overline{h_{\lambda}}\right)  _{\lambda\in P_{k,n}}$ is a basis of the
$\mathbf{k}$-module $\mathcal{S}/I$ under the condition that $a_{1}%
,a_{2},\ldots,a_{k}\in\mathbf{k}$. We shall soon prove this again, this time
under the weaker condition that $a_{1},a_{2},\ldots,a_{k}\in\mathcal{S}$. The
vehicle of the proof will be a triangularity property for the change-of-basis
matrix between the bases $\left(  \overline{h_{\lambda}}\right)  _{\lambda\in
P_{k,n}}$ and $\left(  \overline{s_{\lambda}}\right)  _{\lambda\in P_{k,n}}$
of $\mathcal{S}/I$. We refer to \cite[Definition 11.1.16(c)]{GriRei18} for the
concepts that we shall be using. The triangularity is defined with respect to
a certain partial order on the set $P_{k,n}$:

\begin{definition}
\label{def.size-antidomin}We define a binary relation $\geq^{\ast}$ on the set
$P_{k,n}$ as follows: For two partitions $\lambda\in P_{k,n}$ and $\mu\in
P_{k,n}$, we set $\lambda\geq^{\ast}\mu$ if and only if

\begin{itemize}
\item \textbf{either} $\left\vert \lambda\right\vert >\left\vert
\mu\right\vert $

\item \textbf{or} $\left\vert \lambda\right\vert =\left\vert \mu\right\vert $
and $\lambda_{1}+\lambda_{2}+\cdots+\lambda_{i}\leq\mu_{1}+\mu_{2}+\cdots
+\mu_{i}$ for all $i\geq1$.
\end{itemize}

It is clear that this relation $\geq^{\ast}$ is the greater-or-equal relation
of a partial order on $P_{k,n}$. This order will be called the
\textit{size-then-antidominance order}.
\end{definition}

Note that the condition \textquotedblleft$\left\vert \lambda\right\vert
=\left\vert \mu\right\vert $ and $\lambda_{1}+\lambda_{2}+\cdots+\lambda
_{i}\leq\mu_{1}+\mu_{2}+\cdots+\mu_{i}$ for all $i\geq1$\textquotedblright\ in
Definition \ref{def.size-antidomin} can also be restated as \textquotedblleft%
$\mu\triangleright\lambda$\textquotedblright, where $\triangleright$ means the
dominance relation (defined in Definition \ref{def.LRcoeffs.domi}
\textbf{(b)}). Indeed, this follows easily from Remark
\ref{rmk.LRcoeffs.domi=domi} (applied to $\mu$ and $\lambda$ instead of
$\lambda$ and $\mu$).

For future reference, we need two simple criteria for the $\geq^{\ast}$ relation:

\begin{remark}
\label{rmk.size-antidomin.domin}Let $\lambda\in P_{k,n}$ and $\mu\in P_{k,n}$.

\textbf{(a)} If $\left\vert \lambda\right\vert >\left\vert \mu\right\vert $,
then $\lambda\geq^{\ast}\mu$.

\textbf{(b)} Let $a\in\mathbb{N}$. If both $\lambda$ and $\mu$ are partitions
of size $a$ and satisfy $\mu\triangleright\lambda$, then $\lambda\geq^{\ast
}\mu$. (See Definition \ref{def.LRcoeffs.domi} \textbf{(b)} for the meaning of
\textquotedblleft$\triangleright$\textquotedblright.)
\end{remark}

\begin{proof}
[Proof of Remark \ref{rmk.size-antidomin.domin}.]\textbf{(a)} This follows
immediately from the definition of the relation $\geq^{\ast}$.

\textbf{(b)} Assume that both $\lambda$ and $\mu$ are partitions of size $a$
and satisfy $\mu\triangleright\lambda$. Now, both partitions $\lambda$ and
$\mu$ have size $a$; in other words, $\left\vert \lambda\right\vert =a$ and
$\left\vert \mu\right\vert =a$. Hence, $\left\vert \lambda\right\vert
=a=\left\vert \mu\right\vert $.

We have $\mu\triangleright\lambda$. In other words, we have
\[
\mu_{1}+\mu_{2}+\cdots+\mu_{i}\geq\lambda_{1}+\lambda_{2}+\cdots+\lambda
_{i}\ \ \ \ \ \ \ \ \ \ \text{for each }i\geq1
\]
(by Remark \ref{rmk.LRcoeffs.domi=domi}, applied to $\mu$ and $\lambda$
instead of $\lambda$ and $\mu$). In other words, $\lambda_{1}+\lambda
_{2}+\cdots+\lambda_{i}\leq\mu_{1}+\mu_{2}+\cdots+\mu_{i}$ for all $i\geq1$.
Hence, we have $\left\vert \lambda\right\vert =\left\vert \mu\right\vert $ and
$\lambda_{1}+\lambda_{2}+\cdots+\lambda_{i}\leq\mu_{1}+\mu_{2}+\cdots+\mu_{i}$
for all $i\geq1$. Therefore, $\lambda\geq^{\ast}\mu$ (by the definition of the
relation $\geq^{\ast}$). This proves Remark \ref{rmk.size-antidomin.domin}
\textbf{(b)}.
\end{proof}

Now, we can put the size-then-antidominance order to use. Recall that Theorem
\ref{thm.S/J} yields that the family $\left(  \overline{s_{\lambda}}\right)
_{\lambda\in P_{k,n}}$ is a basis of the $\mathbf{k}$-module $\mathcal{S}/I$.

\begin{theorem}
\label{thm.s-h-triangularity}The family $\left(  \overline{h_{\lambda}%
}\right)  _{\lambda\in P_{k,n}}$ expands unitriangularly in the family
$\left(  \overline{s_{\lambda}}\right)  _{\lambda\in P_{k,n}}$. Here, the word
\textquotedblleft expands unitriangularly\textquotedblright\ is understood
according to \cite[Definition 11.1.16(c)]{GriRei18}, with the poset structure
on $P_{k,n}$ being given by the size-then-antidominance order.
\end{theorem}

\begin{example}
For this example, let $n=5$ and $k=3$. Assume that $a_{1},a_{2}\in\mathbf{k}$.
Then, the expansion of the $\overline{h_{\lambda}}$ in the basis $\left(
\overline{s_{\lambda}}\right)  _{\lambda\in P_{k,n}}$ looks as follows:%
\begin{align*}
\overline{h_{\varnothing}}  &  =\overline{s_{\varnothing}};\\
\overline{h_{\left(  1\right)  }}  &  =\overline{s_{\left(  1\right)  }};\\
\overline{h_{\left(  2\right)  }}  &  =\overline{s_{\left(  2\right)  }};\\
\overline{h_{\left(  1,1\right)  }}  &  =\overline{s_{\left(  2\right)  }%
}+\overline{s_{\left(  1,1\right)  }};\\
\overline{h_{\left(  2,1\right)  }}  &  =a_{1}\overline{s_{\varnothing}%
}+\overline{s_{\left(  2,1\right)  }};\\
\overline{h_{\left(  1,1,1\right)  }}  &  =a_{1}\overline{s_{\varnothing}%
}+\overline{s_{\left(  1,1,1\right)  }}+2\overline{s_{\left(  2,1\right)  }%
};\\
\overline{h_{\left(  2,2\right)  }}  &  =a_{1}\overline{s_{\left(  1\right)
}}+\overline{s_{\left(  2,2\right)  }};\\
\overline{h_{\left(  2,1,1\right)  }}  &  =-a_{2}\overline{s_{\varnothing}%
}+2a_{1}\overline{s_{\left(  1\right)  }}+\overline{s_{\left(  2,1,1\right)
}}+\overline{s_{\left(  2,2\right)  }};\\
\overline{h_{\left(  2,2,1\right)  }}  &  =-a_{2}\overline{s_{\left(
1\right)  }}+a_{1}\overline{s_{\left(  1,1\right)  }}+2a_{1}\overline
{s_{\left(  2\right)  }}+\overline{s_{\left(  2,2,1\right)  }};\\
\overline{h_{\left(  2,2,2\right)  }}  &  =a_{1}^{2}\overline{s_{\varnothing}%
}-a_{2}\overline{s_{\left(  1,1\right)  }}+2a_{1}\overline{s_{\left(
2,1\right)  }}+\overline{s_{\left(  2,2,2\right)  }}.
\end{align*}
These equalities hold for arbitrary $a_{1},a_{2}\in\mathcal{S}$, not only for
$a_{1},a_{2}\in\mathbf{k}$; but in the general case they are not expansions in
the basis $\left(  \overline{s_{\lambda}}\right)  _{\lambda\in P_{k,n}}$,
since $a_{1},a_{2}$ themselves can be expanded further.
\end{example}

Our proof of Theorem \ref{thm.s-h-triangularity} will use the concept of
Kostka numbers. Let us recall their definition:

\begin{definition}
\textbf{(a)} See \cite[\S 2.2]{GriRei18} for the definition of a
\textit{column-strict tableau of shape }$\lambda$ (where $\lambda$ is a
partition), and also for a definition of $\operatorname*{cont}\left(
T\right)  $ where $T$ is such a tableau.

\textbf{(b)} Let $\lambda$ and $\mu$ be two partitions. Then, the
\textit{Kostka number }$K_{\lambda,\mu}$ is defined to be the number of all
column-strict tableaux $T$ of shape $\lambda$ having $\operatorname*{cont}%
\left(  T\right)  =\mu$.
\end{definition}

This definition of $K_{\lambda,\mu}$ is a particular case of the definition of
$K_{\lambda,\mu}$ in \cite[Exercise 2.2.13]{GriRei18}.

We shall use the following properties of Kostka numbers:

\begin{lemma}
\label{lem.kostka.facts}\textbf{(a)} If $a\in\mathbb{N}$, then we have
$K_{\lambda,\mu}=0$ for any partitions $\lambda\in\operatorname*{Par}%
\nolimits_{a}$ and $\mu\in\operatorname*{Par}\nolimits_{a}$ that don't satisfy
$\lambda\triangleright\mu$.

\textbf{(b)} If $a\in\mathbb{N}$, then we have $K_{\lambda,\lambda}=1$ for any
$\lambda\in\operatorname*{Par}\nolimits_{a}$.

\textbf{(c)} If $\lambda$ and $\mu$ are two partitions such that $\left\vert
\lambda\right\vert \neq\left\vert \mu\right\vert $, then $K_{\lambda,\mu}=0$.

\textbf{(d)} For any partition $\mu$, we have%
\[
\mathbf{h}_{\mu}=\sum_{\lambda\in\operatorname*{Par}}K_{\lambda,\mu}%
\mathbf{s}_{\lambda},
\]
where $\operatorname*{Par}$ denotes the set of all partitions.

\textbf{(e)} For any $a\in\mathbb{N}$ and any $\lambda\in\operatorname*{Par}%
\nolimits_{a}$, we have%
\[
\mathbf{h}_{\lambda}=\sum_{\mu\in\operatorname*{Par}\nolimits_{a}}%
K_{\mu,\lambda}\mathbf{s}_{\mu}.
\]

\textbf{(f)} For any $a\in\mathbb{N}$ and any $\lambda\in\operatorname*{Par}%
\nolimits_{a}$, we have%
\[
h_{\lambda}=\sum_{\substack{\mu\in P_{k};\\\mu\in\operatorname*{Par}%
\nolimits_{a}}}K_{\mu,\lambda}s_{\mu}.
\]

\end{lemma}

\begin{proof}
[Proof of Lemma \ref{lem.kostka.facts}.]\textbf{(a)} This is \cite[Exercise
2.2.13(d)]{GriRei18}, applied to $a$ instead of $n$.

\textbf{(b)} This is \cite[Exercise 2.2.13(e)]{GriRei18}, applied to $a$
instead of $n$.

\textbf{(c)} Let $\lambda$ and $\mu$ be two partitions such that $\left\vert
\lambda\right\vert \neq\left\vert \mu\right\vert $. Let $T$ be a column-strict
tableau of shape $\lambda$ having $\operatorname*{cont}\left(  T\right)  =\mu
$. We shall derive a contradiction.

Indeed, the tableau $T$ has shape $\lambda$, and thus has $\left\vert
\lambda\right\vert $ many cells. Hence,%
\begin{align*}
\left\vert \lambda\right\vert  &  =\left(  \text{the number of cells of
}T\right)  =\left(  \text{the number of entries of }T\right) \\
&  =\left\vert \operatorname*{cont}\left(  T\right)  \right\vert =\left\vert
\mu\right\vert \ \ \ \ \ \ \ \ \ \ \left(  \text{since }\operatorname*{cont}%
\left(  T\right)  =\mu\right)  .
\end{align*}
This contradicts $\left\vert \lambda\right\vert \neq\left\vert \mu\right\vert
$.

Now, forget that we fixed $T$. We thus have found a contradiction whenever $T$
is a column-strict tableau of shape $\lambda$ having $\operatorname*{cont}%
\left(  T\right)  =\mu$. Hence, there exists no such tableau. In other words,
the number of such tableaux is $0$. In other words, $K_{\lambda,\mu}=0$ (since
$K_{\lambda,\mu}$ is defined to be the number of such tableaux). This proves
Lemma \ref{lem.kostka.facts} \textbf{(c)}.

\textbf{(d)} This is \cite[Exercise 2.7.10(a)]{GriRei18}.

\textbf{(e)} Let $\operatorname*{Par}$ denote the set of all partitions. Then,
Lemma \ref{lem.kostka.facts} \textbf{(d)} yields that%
\[
\mathbf{h}_{\mu}=\sum_{\lambda\in\operatorname*{Par}}K_{\lambda,\mu}%
\mathbf{s}_{\lambda}\ \ \ \ \ \ \ \ \ \ \text{for any partition }\mu.
\]
Hence, for any partition $\mu$, we have%
\begin{align*}
\mathbf{h}_{\mu}  &  =\sum_{\lambda\in\operatorname*{Par}}K_{\lambda,\mu
}\mathbf{s}_{\lambda}=\sum_{\substack{\lambda\in\operatorname*{Par}%
;\\\left\vert \lambda\right\vert =\left\vert \mu\right\vert }}K_{\lambda,\mu
}\mathbf{s}_{\lambda}+\sum_{\substack{\lambda\in\operatorname*{Par}%
;\\\left\vert \lambda\right\vert \neq\left\vert \mu\right\vert }%
}\underbrace{K_{\lambda,\mu}}_{\substack{=0\\\text{(by Lemma
\ref{lem.kostka.facts} \textbf{(c)})}}}\mathbf{s}_{\lambda}\\
&  =\sum_{\substack{\lambda\in\operatorname*{Par};\\\left\vert \lambda
\right\vert =\left\vert \mu\right\vert }}K_{\lambda,\mu}\mathbf{s}_{\lambda
}+\underbrace{\sum_{\substack{\lambda\in\operatorname*{Par};\\\left\vert
\lambda\right\vert \neq\left\vert \mu\right\vert }}0\mathbf{s}_{\lambda}}%
_{=0}=\sum_{\substack{\lambda\in\operatorname*{Par};\\\left\vert
\lambda\right\vert =\left\vert \mu\right\vert }}K_{\lambda,\mu}\mathbf{s}%
_{\lambda}.
\end{align*}
Renaming $\mu$ and $\lambda$ as $\lambda$ and $\mu$ in this equality, we
obtain the following: For any partition $\lambda$, we have%
\begin{equation}
\mathbf{h}_{\lambda}=\sum_{\substack{\mu\in\operatorname*{Par};\\\left\vert
\mu\right\vert =\left\vert \lambda\right\vert }}K_{\mu,\lambda}\mathbf{s}%
_{\mu}. \label{pf.lem.kostka.facts.e.1}%
\end{equation}

Now, let $a\in\mathbb{N}$ and $\lambda\in\operatorname*{Par}\nolimits_{a}$.
Then, $\left\vert \lambda\right\vert =a$. Now, (\ref{pf.lem.kostka.facts.e.1})
becomes%
\begin{align*}
\mathbf{h}_{\lambda}  &  =\sum_{\substack{\mu\in\operatorname*{Par}%
;\\\left\vert \mu\right\vert =\left\vert \lambda\right\vert }}K_{\mu,\lambda
}\mathbf{s}_{\mu}=\underbrace{\sum_{\substack{\mu\in\operatorname*{Par}%
;\\\left\vert \mu\right\vert =a}}}_{=\sum_{\mu\in\operatorname*{Par}%
\nolimits_{a}}}K_{\mu,\lambda}\mathbf{s}_{\mu}\ \ \ \ \ \ \ \ \ \ \left(
\text{since }\left\vert \lambda\right\vert =a\right) \\
&  =\sum_{\mu\in\operatorname*{Par}\nolimits_{a}}K_{\mu,\lambda}%
\mathbf{s}_{\mu}.
\end{align*}
This proves Lemma \ref{lem.kostka.facts} \textbf{(e)}.

\textbf{(f)} Let $a\in\mathbb{N}$ and $\lambda\in\operatorname*{Par}%
\nolimits_{a}$. Lemma \ref{lem.kostka.facts} \textbf{(e)} yields
$\mathbf{h}_{\lambda}=\sum_{\mu\in\operatorname*{Par}\nolimits_{a}}%
K_{\mu,\lambda}\mathbf{s}_{\mu}$. This is an identity in $\Lambda$. Evaluating
both of its sides at the $k$ variables $x_{1},x_{2},\ldots,x_{k}$, we obtain%
\begin{align*}
h_{\lambda}  &  =\sum_{\mu\in\operatorname*{Par}\nolimits_{a}}K_{\mu,\lambda
}s_{\mu}=\underbrace{\sum_{\substack{\mu\in\operatorname*{Par}\nolimits_{a}%
;\\\mu\text{ has at most }k\text{ parts}}}}_{=\sum_{\substack{\mu
\in\operatorname*{Par}\nolimits_{a};\\\mu\in P_{k}}}=\sum_{\substack{\mu\in
P_{k};\\\mu\in\operatorname*{Par}\nolimits_{a}}}}K_{\mu,\lambda}s_{\mu}%
+\sum_{\substack{\mu\in\operatorname*{Par}\nolimits_{a};\\\mu\text{ has more
than }k\text{ parts}}}K_{\mu,\lambda}\underbrace{s_{\mu}}%
_{\substack{=0\\\text{(by (\ref{eq.slam=0-too-long}), applied to }%
\mu\\\text{instead of }\lambda\text{)}}}\\
&  =\sum_{\substack{\mu\in P_{k};\\\mu\in\operatorname*{Par}\nolimits_{a}%
}}K_{\mu,\lambda}s_{\mu}+\underbrace{\sum_{\substack{\mu\in\operatorname*{Par}%
\nolimits_{a};\\\mu\text{ has more than }k\text{ parts}}}K_{\mu,\lambda}%
0}_{=0}=\sum_{\substack{\mu\in P_{k};\\\mu\in\operatorname*{Par}\nolimits_{a}%
}}K_{\mu,\lambda}s_{\mu}.
\end{align*}
This proves Lemma \ref{lem.kostka.facts} \textbf{(f)}.
\end{proof}

\begin{proof}
[Proof of Theorem \ref{thm.s-h-triangularity}.]Let $<^{\ast}$ denote the
smaller relation of the size-then-antidominance order on $P_{k,n}$. Thus, two
partitions $\lambda$ and $\mu$ satisfy $\mu<^{\ast}\lambda$ if and only if
$\mu\neq\lambda$ and $\lambda\geq^{\ast}\mu$.

Our goal is to show that the family $\left(  \overline{h_{\lambda}}\right)
_{\lambda\in P_{k,n}}$ expands unitriangularly in the family $\left(
\overline{s_{\lambda}}\right)  _{\lambda\in P_{k,n}}$. In other words, our
goal is to show that each $\lambda\in P_{k,n}$ satisfies%
\begin{equation}
\overline{h_{\lambda}}=\overline{s_{\lambda}}+\left(  \text{a }\mathbf{k}%
\text{-linear combination of the elements }\overline{s_{\mu}}\text{ for }%
\mu\in P_{k,n}\text{ satisfying }\mu<^{\ast}\lambda\right)
\label{pf.thm.s-h-triangularity.tria}%
\end{equation}
(because \cite[Remark 11.1.17(c)]{GriRei18} shows that the family $\left(
\overline{h_{\lambda}}\right)  _{\lambda\in P_{k,n}}$ expands unitriangularly
in the family $\left(  \overline{s_{\lambda}}\right)  _{\lambda\in P_{k,n}}$
if and only if every $\lambda\in P_{k,n}$ satisfies
(\ref{pf.thm.s-h-triangularity.tria})). So let us prove
(\ref{pf.thm.s-h-triangularity.tria}).

Fix $\lambda\in P_{k,n}$. Define $a\in\mathbb{N}$ by $a=\left\vert
\lambda\right\vert $. Thus, $\lambda\in\operatorname*{Par}\nolimits_{a}$.
Hence, Lemma \ref{lem.kostka.facts} \textbf{(f)} yields%
\begin{equation}
h_{\lambda}=\sum_{\substack{\mu\in P_{k};\\\mu\in\operatorname*{Par}%
\nolimits_{a}}}K_{\mu,\lambda}s_{\mu}=K_{\lambda,\lambda}s_{\lambda}%
+\sum_{\substack{\mu\in P_{k};\\\mu\in\operatorname*{Par}\nolimits_{a}%
;\\\mu\neq\lambda}}K_{\mu,\lambda}s_{\mu} \label{pf.thm.s-h-triangularity.3}%
\end{equation}
(here, we have split off the addend for $\mu=\lambda$, since $\lambda\in
P_{k,n}\subseteq P_{k}$ and $\lambda\in\operatorname*{Par}\nolimits_{a}$).

Now, let $M$ be the $\mathbf{k}$-submodule of $\mathcal{S}/I$ spanned by the
elements $\overline{s_{\mu}}$ for $\mu\in P_{k,n}$ satisfying $\mu<^{\ast
}\lambda$. Thus, we have%
\begin{equation}
\overline{s_{\mu}}\in M\ \ \ \ \ \ \ \ \ \ \text{for each }\mu\in
P_{k,n}\text{ satisfying }\mu<^{\ast}\lambda.
\label{pf.thm.s-h-triangularity.inM}%
\end{equation}
Also, $0\in M$ (since $M$ is a $\mathbf{k}$-submodule of $\mathcal{S}/I$).

We shall next show that%
\begin{equation}
K_{\mu,\lambda}\overline{s_{\mu}}\in M\ \ \ \ \ \ \ \ \ \ \text{for each }%
\mu\in P_{k}\text{ satisfying }\mu\in\operatorname*{Par}\nolimits_{a}\text{
and }\mu\neq\lambda. \label{pf.thm.s-h-triangularity.4}%
\end{equation}

[\textit{Proof of (\ref{pf.thm.s-h-triangularity.4}):} Let $\mu\in P_{k}$ be
such that $\mu\in\operatorname*{Par}\nolimits_{a}$ and $\mu\neq\lambda$. We
must prove that $K_{\mu,\lambda}\overline{s_{\mu}}\in M$.

If $K_{\mu,\lambda}=0$, then this follows immediately from $\underbrace{K_{\mu
,\lambda}}_{=0}\overline{s_{\mu}}=0\overline{s_{\mu}}=0\in M$. Hence, for the
rest of this proof, we WLOG assume that $K_{\mu,\lambda}\neq0$.

If $\mu$ and $\lambda$ would not satisfy $\mu\triangleright\lambda$, then we
would have $K_{\mu,\lambda}=0$ (by Lemma \ref{lem.kostka.facts} \textbf{(a)},
applied to $\mu$ and $\lambda$ instead of $\lambda$ and $\mu$), which would
contradict $K_{\mu,\lambda}\neq0$. Hence, $\mu$ and $\lambda$ must satisfy
$\mu\triangleright\lambda$. Both $\lambda$ and $\mu$ are partitions of size
$a$ (since $\lambda\in\operatorname*{Par}\nolimits_{a}$ and $\mu
\in\operatorname*{Par}\nolimits_{a}$). Thus, $\left\vert \lambda\right\vert
=a$ and $\left\vert \mu\right\vert =a$.

Now, we are in one of the following two cases:

\textit{Case 1:} We have $\mu\in P_{k,n}$.

\textit{Case 2:} We have $\mu\notin P_{k,n}$.

Let us first consider Case 1. In this case, we have $\mu\in P_{k,n}$. Thus,
$\lambda\geq^{\ast}\mu$ (by Remark \ref{rmk.size-antidomin.domin}
\textbf{(b)}) and thus $\mu<^{\ast}\lambda$ (since $\mu\neq\lambda$). Hence,
(\ref{pf.thm.s-h-triangularity.inM}) shows that $\overline{s_{\mu}}\in M$.
Thus, $K_{\mu,\lambda}\underbrace{\overline{s_{\mu}}}_{\in M}\in
K_{\mu,\lambda}M\subseteq M$ (since $M$ is a $\mathbf{k}$-submodule of
$\mathcal{S}/I$). Thus, (\ref{pf.thm.s-h-triangularity.4}) is proven in Case 1.

Let us next consider Case 2. In this case, we have $\mu\notin P_{k,n}$. Hence,
Lemma \ref{lem.I.sl-red} (applied to $\mu$ instead of $\lambda$) shows that%
\[
s_{\mu}\equiv\left(  \text{some symmetric polynomial of degree }<\left\vert
\mu\right\vert \right)  \operatorname{mod}I.
\]
In other words, there exists some symmetric polynomial $f\in\mathcal{S}$ of
degree $<\left\vert \mu\right\vert $ such that $s_{\mu}\equiv
f\operatorname{mod}I$. Consider this $f$. Lemma \ref{lem.sm-as-Pkn} (applied
to $N=\left\vert \mu\right\vert $) yields that in $\mathcal{S}/I$, we have
\begin{equation}
\overline{f}\in\sum_{\substack{\kappa\in P_{k,n};\\\left\vert \kappa
\right\vert <\left\vert \mu\right\vert }}\mathbf{k}\overline{s_{\kappa}}.
\label{pf.thm.s-h-triangularity.4.pf.c2.5}%
\end{equation}

Now, let $\kappa\in P_{k,n}$ be such that $\left\vert \kappa\right\vert
<\left\vert \mu\right\vert $. Then, $\left\vert \kappa\right\vert <\left\vert
\mu\right\vert =a=\left\vert \lambda\right\vert $, so that $\left\vert
\lambda\right\vert >\left\vert \kappa\right\vert $. Thus, Remark
\ref{rmk.size-antidomin.domin} \textbf{(a)} (applied to $\kappa$ instead of
$\mu$) yields $\lambda\geq^{\ast}\kappa$. Also, $\left\vert \kappa\right\vert
\neq\left\vert \lambda\right\vert $ (since $\left\vert \kappa\right\vert
<\left\vert \lambda\right\vert $) and thus $\kappa\neq\lambda$.\ Combining
this with $\lambda\geq^{\ast}\kappa$, we obtain $\kappa<^{\ast}\lambda$.
Hence, (\ref{pf.thm.s-h-triangularity.inM}) (applied to $\kappa$ instead of
$\mu$) yields $\overline{s_{\kappa}}\in M$. Hence, $\mathbf{k}%
\underbrace{\overline{s_{\kappa}}}_{\in M}\subseteq\mathbf{k}M\subseteq M$
(since $M$ is a $\mathbf{k}$-module).

Forget that we fixed $\kappa$. We thus have shown that $\mathbf{k}%
\overline{s_{\kappa}}\subseteq M$ for each $\kappa\in P_{k,n}$ satisfying
$\left\vert \kappa\right\vert <\left\vert \mu\right\vert $. Hence,
$\sum_{\substack{\kappa\in P_{k,n};\\\left\vert \kappa\right\vert <\left\vert
\mu\right\vert }}\underbrace{\mathbf{k}\overline{s_{\kappa}}}_{\subseteq
M}\subseteq\sum_{\substack{\kappa\in P_{k,n};\\\left\vert \kappa\right\vert
<\left\vert \mu\right\vert }}M\subseteq M$ (since $M$ is a $\mathbf{k}%
$-module). Thus, (\ref{pf.thm.s-h-triangularity.4.pf.c2.5}) becomes
$\overline{f}\in\sum_{\substack{\kappa\in P_{k,n};\\\left\vert \kappa
\right\vert <\left\vert \mu\right\vert }}\mathbf{k}\overline{s_{\kappa}}=M$.
But $s_{\mu}\equiv f\operatorname{mod}I$ and thus $\overline{s_{\mu}%
}=\overline{f}\in M$. Thus, $K_{\mu,\lambda}\underbrace{\overline{s_{\mu}}%
}_{\in M}\in K_{\mu,\lambda}M\subseteq M$ (since $M$ is a $\mathbf{k}%
$-submodule of $\mathcal{S}/I$). Hence, (\ref{pf.thm.s-h-triangularity.4}) is
proven in Case 2.

We have now proven (\ref{pf.thm.s-h-triangularity.4}) in both Cases 1 and 2.
Hence, (\ref{pf.thm.s-h-triangularity.4}) always holds.]

Now, from (\ref{pf.thm.s-h-triangularity.3}), we obtain%
\begin{align*}
\overline{h_{\lambda}}  &  =\overline{K_{\lambda,\lambda}s_{\lambda}%
+\sum_{\substack{\mu\in P_{k};\\\mu\in\operatorname*{Par}\nolimits_{a}%
;\\\mu\neq\lambda}}K_{\mu,\lambda}s_{\mu}}=\underbrace{K_{\lambda,\lambda}%
}_{\substack{=1\\\text{(by Lemma \ref{lem.kostka.facts} \textbf{(b)})}%
}}\overline{s_{\lambda}}+\sum_{\substack{\mu\in P_{k};\\\mu\in
\operatorname*{Par}\nolimits_{a};\\\mu\neq\lambda}}\underbrace{K_{\mu,\lambda
}\overline{s_{\mu}}}_{\substack{\in M\\\text{(by
(\ref{pf.thm.s-h-triangularity.4}))}}}\\
&  \in\overline{s_{\lambda}}+\underbrace{\sum_{\substack{\mu\in P_{k};\\\mu
\in\operatorname*{Par}\nolimits_{a};\\\mu\neq\lambda}}M}_{\substack{\subseteq
M\\\text{(since }M\text{ is a }\mathbf{k}\text{-module)}}}\subseteq
\overline{s_{\lambda}}+M.
\end{align*}
In other words, $\overline{h_{\lambda}}-\overline{s_{\lambda}}\in M$. In other
words, $\overline{h_{\lambda}}-\overline{s_{\lambda}}$ is a $\mathbf{k}%
$-linear combination of the elements $\overline{s_{\mu}}$ for $\mu\in P_{k,n}$
satisfying $\mu<^{\ast}\lambda$ (since $M$ was defined as the $\mathbf{k}%
$-submodule of $\mathcal{S}/I$ spanned by these elements). In other words,%
\[
\overline{h_{\lambda}}=\overline{s_{\lambda}}+\left(  \text{a }\mathbf{k}%
\text{-linear combination of the elements }\overline{s_{\mu}}\text{ for }%
\mu\in P_{k,n}\text{ satisfying }\mu<^{\ast}\lambda\right)  .
\]
Thus, (\ref{pf.thm.s-h-triangularity.tria}) is proven. As we already have
explained, this completes the proof of Theorem \ref{thm.s-h-triangularity}.
\end{proof}

We can now prove Theorem \ref{thm.S/J-h-basis} again. Better yet, we can prove
the following more general fact:

\begin{theorem}
\label{thm.S/J-h-basis-ainS}The family $\left(  \overline{h_{\lambda}}\right)
_{\lambda\in P_{k,n}}$ is a basis of the $\mathbf{k}$-module $\mathcal{S}/I$.
\end{theorem}

Theorem \ref{thm.S/J-h-basis-ainS} makes the exact same claim as Theorem
\ref{thm.S/J-h-basis}, but is nevertheless more general because we have stated
it in a more general context (namely, $a_{1},a_{2},\ldots,a_{k}\in\mathcal{S}$
rather than $a_{1},a_{2},\ldots,a_{k}\in\mathbf{k}$).

\begin{proof}
[Proof of Theorem \ref{thm.S/J-h-basis-ainS}.]Consider the finite set
$P_{k,n}$ as a poset (using the size-then-antidominance order).

Theorem \ref{thm.s-h-triangularity} says that the family $\left(
\overline{h_{\lambda}}\right)  _{\lambda\in P_{k,n}}$ expands unitriangularly
in the family $\left(  \overline{s_{\lambda}}\right)  _{\lambda\in P_{k,n}}$.
Hence, the family $\left(  \overline{h_{\lambda}}\right)  _{\lambda\in
P_{k,n}}$ expands invertibly triangularly\footnote{See \cite[Definition
11.1.16(b)]{GriRei18} for the meaning of this word.} in the family $\left(
\overline{s_{\lambda}}\right)  _{\lambda\in P_{k,n}}$. Thus, \cite[Corollary
11.1.19(e)]{GriRei18} (applied to $\mathcal{S}/I$, $P_{k,n}$, $\left(
\overline{h_{\lambda}}\right)  _{\lambda\in P_{k,n}}$ and $\left(
\overline{s_{\lambda}}\right)  _{\lambda\in P_{k,n}}$ instead of $M$, $S$,
$\left(  e_{s}\right)  _{s\in S}$ and $\left(  f_{s}\right)  _{s\in S}$) shows
that the family $\left(  \overline{h_{\lambda}}\right)  _{\lambda\in P_{k,n}}$
is a basis of the $\mathbf{k}$-module $\mathcal{S}/I$ if and only if the
family $\left(  \overline{s_{\lambda}}\right)  _{\lambda\in P_{k,n}}$ is a
basis of the $\mathbf{k}$-module $\mathcal{S}/I$. Hence, the family $\left(
\overline{h_{\lambda}}\right)  _{\lambda\in P_{k,n}}$ is a basis of the
$\mathbf{k}$-module $\mathcal{S}/I$ (since the family $\left(  \overline
{s_{\lambda}}\right)  _{\lambda\in P_{k,n}}$ is a basis of the $\mathbf{k}%
$-module $\mathcal{S}/I$). Thus, Theorem \ref{thm.S/J-h-basis-ainS} is proven.
(And therefore, Theorem \ref{thm.S/J-h-basis} is proven again.)
\end{proof}

\subsection{The $m$-basis}

Next, we recall another well-known family of symmetric polynomials:

\begin{definition}
\label{def.mlam}For any partition $\lambda$, we let $m_{\lambda}$ denote the
monomial symmetric polynomial in $x_{1},x_{2},\ldots,x_{k}$ corresponding to
the partition $\lambda$. This monomial symmetric polynomial is what is called
$m_{\lambda}\left(  x_{1},x_{2},\ldots,x_{k}\right)  $ in \cite[Chapter
2]{GriRei18}. Note that%
\begin{equation}
m_{\lambda}=0\ \ \ \ \ \ \ \ \ \ \text{if }\lambda\text{ has more than
}k\text{ parts.} \label{eq.mlam=0-too-long}%
\end{equation}

\end{definition}

If $\lambda$ is any partition, then the monomial symmetric polynomial
$m_{\lambda}=m_{\lambda}\left(  x_{1},x_{2},\ldots,x_{k}\right)  $ is
symmetric and thus belongs to $\mathcal{S}$.

We now claim the following:

\begin{theorem}
\label{thm.S/J-m-basis}The family $\left(  \overline{m_{\lambda}}\right)
_{\lambda\in P_{k,n}}$ is a basis of the $\mathbf{k}$-module $\mathcal{S}/I$.
\end{theorem}

We shall prove this further below; a different proof has been given in by
Weinfeld in \cite[Corollary 6.2]{Weinfe19}.

Our proof of Theorem \ref{thm.S/J-m-basis} will again rely on the concept of
unitriangularity and on a partial order on the set $P_{k,n}$. The partial
order, this time, is not the size-then-antidominance order, but a simpler one
(the \textquotedblleft graded dominance order\textquotedblright):

\begin{definition}
\label{def.size-domin}We define a binary relation $\geq_{\ast}$ on the set
$P_{k,n}$ as follows: For two partitions $\lambda\in P_{k,n}$ and $\mu\in
P_{k,n}$, we set $\lambda\geq_{\ast}\mu$ if and only if

\begin{itemize}
\item $\left\vert \lambda\right\vert =\left\vert \mu\right\vert $ and
$\lambda_{1}+\lambda_{2}+\cdots+\lambda_{i}\geq\mu_{1}+\mu_{2}+\cdots+\mu_{i}$
for all $i\geq1$.
\end{itemize}

It is clear that this relation $\geq_{\ast}$ is the greater-or-equal relation
of a partial order on $P_{k,n}$. This order will be called the \textit{graded
dominance order}.
\end{definition}

Note that the condition \textquotedblleft$\left\vert \lambda\right\vert
=\left\vert \mu\right\vert $ and $\lambda_{1}+\lambda_{2}+\cdots+\lambda
_{i}\geq\mu_{1}+\mu_{2}+\cdots+\mu_{i}$ for all $i\geq1$\textquotedblright\ in
Definition \ref{def.size-domin} can also be restated as \textquotedblleft%
$\lambda\triangleright\mu$\textquotedblright, where $\triangleright$ means the
dominance relation (defined in Definition \ref{def.LRcoeffs.domi}
\textbf{(b)}). Indeed, this follows easily from Remark
\ref{rmk.LRcoeffs.domi=domi}.

For future reference, we need a simple criterion for the $\geq_{\ast}$ relation:

\begin{remark}
\label{rmk.size-domin.domin}Let $\lambda\in P_{k,n}$ and $\mu\in P_{k,n}$.

Let $a\in\mathbb{N}$. If both $\lambda$ and $\mu$ are partitions of size $a$
and satisfy $\lambda\triangleright\mu$, then $\lambda\geq_{\ast}\mu$. (See
Definition \ref{def.LRcoeffs.domi} \textbf{(b)} for the meaning of
\textquotedblleft$\triangleright$\textquotedblright.)
\end{remark}

\begin{proof}
[Proof of Remark \ref{rmk.size-domin.domin}.]Assume that both $\lambda$ and
$\mu$ are partitions of size $a$ and satisfy $\lambda\triangleright\mu$. Now,
both partitions $\lambda$ and $\mu$ have size $a$; in other words, $\left\vert
\lambda\right\vert =a$ and $\left\vert \mu\right\vert =a$. Hence, $\left\vert
\lambda\right\vert =a=\left\vert \mu\right\vert $.

We have $\lambda\triangleright\mu$. In other words, we have
\[
\lambda_{1}+\lambda_{2}+\cdots+\lambda_{i}\geq\mu_{1}+\mu_{2}+\cdots+\mu
_{i}\ \ \ \ \ \ \ \ \ \ \text{for each }i\geq1
\]
(by Remark \ref{rmk.LRcoeffs.domi=domi}). Hence, we have $\left\vert
\lambda\right\vert =\left\vert \mu\right\vert $ and $\lambda_{1}+\lambda
_{2}+\cdots+\lambda_{i}\geq\mu_{1}+\mu_{2}+\cdots+\mu_{i}$ for all $i\geq1$.
Therefore, $\lambda\geq_{\ast}\mu$ (by the definition of the relation
$\geq_{\ast}$). This proves Remark \ref{rmk.size-domin.domin}.
\end{proof}

Now, we can put the graded dominance order to use. Recall that Theorem
\ref{thm.S/J} yields that the family $\left(  \overline{s_{\lambda}}\right)
_{\lambda\in P_{k,n}}$ is a basis of the $\mathbf{k}$-module $\mathcal{S}/I$.

\begin{theorem}
\label{thm.s-m-triangularity}The family $\left(  \overline{s_{\lambda}%
}\right)  _{\lambda\in P_{k,n}}$ expands unitriangularly in the family
$\left(  \overline{m_{\lambda}}\right)  _{\lambda\in P_{k,n}}$. Here, the word
\textquotedblleft expands unitriangularly\textquotedblright\ is understood
according to \cite[Definition 11.1.16(c)]{GriRei18}, with the poset structure
on $P_{k,n}$ being given by the graded dominance order.
\end{theorem}

\begin{example}
For this example, let $n=5$ and $k=3$. Then, the expansion of the
$\overline{s_{\lambda}}$ in the basis $\left(  \overline{m_{\lambda}}\right)
_{\lambda\in P_{k,n}}$ looks as follows:%
\begin{align*}
\overline{s_{\varnothing}}  &  =\overline{m_{\varnothing}};\\
\overline{s_{\left(  1\right)  }}  &  =\overline{m_{\left(  1\right)  }};\\
\overline{s_{\left(  2\right)  }}  &  =\overline{m_{\left(  1,1\right)  }%
}+\overline{m_{\left(  2\right)  }};\\
\overline{s_{\left(  1,1\right)  }}  &  =\overline{m_{\left(  1,1\right)  }%
};\\
\overline{s_{\left(  2,1\right)  }}  &  =2\overline{m_{\left(  1,1,1\right)
}}+\overline{m_{\left(  2,1\right)  }};\\
\overline{s_{\left(  1,1,1\right)  }}  &  =\overline{m_{\left(  1,1,1\right)
}};\\
\overline{s_{\left(  2,2\right)  }}  &  =\overline{m_{\left(  2,1,1\right)  }%
}+\overline{m_{\left(  2,2\right)  }};\\
\overline{s_{\left(  2,1,1\right)  }}  &  =\overline{m_{\left(  2,1,1\right)
}};\\
\overline{s_{\left(  2,2,1\right)  }}  &  =\overline{m_{\left(  2,2,1\right)
}};\\
\overline{s_{\left(  2,2,2\right)  }}  &  =\overline{m_{\left(  2,2,2\right)
}}.
\end{align*}
The coefficients in these expansions are Kostka numbers; the $a_{1}%
,a_{2},\ldots,a_{k}$ do not appear in them. (This will become clear in the
proof of Theorem \ref{thm.s-m-triangularity}.)
\end{example}

To prove Theorem \ref{thm.s-m-triangularity}, we shall use the monomial
symmetric functions $\mathbf{m}_{\lambda}$:

\begin{itemize}
\item For any partition $\lambda$, we let $\mathbf{m}_{\lambda}$ be the
corresponding monomial symmetric function in $\Lambda$. (This is called
$m_{\lambda}$ in \cite[(2.1.1)]{GriRei18}.)
\end{itemize}

We shall furthermore use the following property of the dominance order:

\begin{lemma}
\label{lem.dominance.Pkn}Let $a\in\mathbb{N}$. Let $\lambda\in
\operatorname*{Par}\nolimits_{a}$ and $\mu\in\operatorname*{Par}\nolimits_{a}$
be such that $\lambda\triangleright\mu$. Assume that $\lambda\in P_{k,n}$ and
$\mu\in P_{k}$. Then, $\mu\in P_{k,n}$.
\end{lemma}

\begin{proof}
[Proof of Lemma \ref{lem.dominance.Pkn}.]We have $\lambda\triangleright\mu$.
In other words, we have
\[
\lambda_{1}+\lambda_{2}+\cdots+\lambda_{i}\geq\mu_{1}+\mu_{2}+\cdots+\mu
_{i}\ \ \ \ \ \ \ \ \ \ \text{for each }i\geq1
\]
(by Remark \ref{rmk.LRcoeffs.domi=domi}). Applying this to $i=1$, we obtain
$\lambda_{1}\geq\mu_{1}$. Hence, $\mu_{1}\leq\lambda_{1}\leq n-k$ (since
$\lambda\in P_{k,n}$). Thus, all parts of the partition $\mu$ are $\leq n-k$
(since $\mu_{1}\geq\mu_{2}\geq\mu_{3}\geq\cdots$). Hence, $\mu\in P_{k,n}$
(since $\mu\in P_{k}$). This proves Lemma \ref{lem.dominance.Pkn}.
\end{proof}

Also, we shall again use Kostka numbers, specifically their following properties:

\begin{lemma}
\label{lem.kostka.m}\textbf{(a)} For any $a\in\mathbb{N}$ and any $\lambda
\in\operatorname*{Par}\nolimits_{a}$, we have%
\[
\mathbf{s}_{\lambda}=\sum_{\mu\in\operatorname*{Par}\nolimits_{a}}%
K_{\lambda,\mu}\mathbf{m}_{\mu}.
\]

\textbf{(b)} For any $a\in\mathbb{N}$ and $\lambda\in\operatorname*{Par}%
\nolimits_{a}$, we have%
\[
s_{\lambda}=\sum_{\substack{\mu\in P_{k};\\\mu\in\operatorname*{Par}%
\nolimits_{a}}}K_{\lambda,\mu}m_{\mu}.
\]

\textbf{(c)} For any $a\in\mathbb{N}$ and $\lambda\in\operatorname*{Par}%
\nolimits_{a}$ satisfying $\lambda\in P_{k,n}$, we have%
\[
s_{\lambda}=\sum_{\substack{\mu\in P_{k,n};\\\mu\in\operatorname*{Par}%
\nolimits_{a}}}K_{\lambda,\mu}m_{\mu}.
\]

\end{lemma}

\begin{proof}
[Proof of Lemma \ref{lem.kostka.m}.]\textbf{(a)} This is \cite[Exercise
2.2.13(c)]{GriRei18}.

\textbf{(b)} Let $a\in\mathbb{N}$ and $\lambda\in\operatorname*{Par}%
\nolimits_{a}$. Lemma \ref{lem.kostka.m} \textbf{(a)} yields $\mathbf{s}%
_{\lambda}=\sum_{\mu\in\operatorname*{Par}\nolimits_{a}}K_{\lambda,\mu
}\mathbf{m}_{\mu}$. This is an identity in $\Lambda$. Evaluating both of its
sides at the $k$ variables $x_{1},x_{2},\ldots,x_{k}$, we obtain%
\begin{align*}
s_{\lambda}  &  =\sum_{\mu\in\operatorname*{Par}\nolimits_{a}}K_{\lambda,\mu
}m_{\mu}=\underbrace{\sum_{\substack{\mu\in\operatorname*{Par}\nolimits_{a}%
;\\\mu\text{ has at most }k\text{ parts}}}}_{=\sum_{\substack{\mu
\in\operatorname*{Par}\nolimits_{a};\\\mu\in P_{k}}}=\sum_{\substack{\mu\in
P_{k};\\\mu\in\operatorname*{Par}\nolimits_{a}}}}K_{\lambda,\mu}m_{\mu}%
+\sum_{\substack{\mu\in\operatorname*{Par}\nolimits_{a};\\\mu\text{ has more
than }k\text{ parts}}}K_{\lambda,\mu}\underbrace{m_{\mu}}%
_{\substack{=0\\\text{(by (\ref{eq.mlam=0-too-long}), applied to }%
\mu\\\text{instead of }\lambda\text{)}}}\\
&  =\sum_{\substack{\mu\in P_{k};\\\mu\in\operatorname*{Par}\nolimits_{a}%
}}K_{\lambda,\mu}m_{\mu}+\underbrace{\sum_{\substack{\mu\in\operatorname*{Par}%
\nolimits_{a};\\\mu\text{ has more than }k\text{ parts}}}K_{\lambda,\mu}%
0}_{=0}=\sum_{\substack{\mu\in P_{k};\\\mu\in\operatorname*{Par}\nolimits_{a}%
}}K_{\lambda,\mu}m_{\mu}.
\end{align*}
This proves Lemma \ref{lem.kostka.m} \textbf{(b)}.

\textbf{(c)} Let $a\in\mathbb{N}$ and $\lambda\in\operatorname*{Par}%
\nolimits_{a}$ satisfy $\lambda\in P_{k,n}$.

Fix some $\mu\in P_{k}$ such that $\mu\in\operatorname*{Par}\nolimits_{a}$ and
$\mu\notin P_{k,n}$. Then, we don't have $\lambda\triangleright\mu$ (since
otherwise, Lemma \ref{lem.dominance.Pkn} would yield that $\mu\in P_{k,n}$;
but this would contradict $\mu\notin P_{k,n}$). Hence, Lemma
\ref{lem.kostka.facts} \textbf{(a)} yields $K_{\lambda,\mu}=0$.

Forget that we fixed $\mu$. We thus have shown that%
\begin{equation}
K_{\lambda,\mu}=0\ \ \ \ \ \ \ \ \ \ \text{for every }\mu\in P_{k}\text{
satisfying }\mu\in\operatorname*{Par}\nolimits_{a}\text{ and }\mu\notin
P_{k,n}. \label{pf.lem.kostka.m.c.1}%
\end{equation}

Now, Lemma \ref{lem.kostka.m} \textbf{(b)} yields%
\begin{align*}
s_{\lambda}  &  =\sum_{\substack{\mu\in P_{k};\\\mu\in\operatorname*{Par}%
\nolimits_{a}}}K_{\lambda,\mu}m_{\mu}=\underbrace{\sum_{\substack{\mu\in
P_{k};\\\mu\in\operatorname*{Par}\nolimits_{a};\\\mu\in P_{k,n}}%
}}_{\substack{=\sum_{\substack{\mu\in P_{k,n};\\\mu\in\operatorname*{Par}%
\nolimits_{a}}}\\\text{(since }P_{k,n}\subseteq P_{k}\text{)}}}K_{\lambda,\mu
}m_{\mu}+\sum_{\substack{\mu\in P_{k};\\\mu\in\operatorname*{Par}%
\nolimits_{a};\\\mu\notin P_{k,n}}}\underbrace{K_{\lambda,\mu}}%
_{\substack{=0\\\text{(by (\ref{pf.lem.kostka.m.c.1}))}}}m_{\mu}\\
&  =\sum_{\substack{\mu\in P_{k,n};\\\mu\in\operatorname*{Par}\nolimits_{a}%
}}K_{\lambda,\mu}m_{\mu}+\underbrace{\sum_{\substack{\mu\in P_{k};\\\mu
\in\operatorname*{Par}\nolimits_{a};\\\mu\notin P_{k,n}}}0m_{\mu}}_{=0}%
=\sum_{\substack{\mu\in P_{k,n};\\\mu\in\operatorname*{Par}\nolimits_{a}%
}}K_{\lambda,\mu}m_{\mu}.
\end{align*}
This proves Lemma \ref{lem.kostka.m} \textbf{(c)}.
\end{proof}

\begin{proof}
[Proof of Theorem \ref{thm.s-m-triangularity}.]Let $<_{\ast}$ denote the
smaller relation of the graded dominance order on $P_{k,n}$. Thus, two
partitions $\lambda$ and $\mu$ satisfy $\mu<_{\ast}\lambda$ if and only if
$\mu\neq\lambda$ and $\lambda\geq_{\ast}\mu$.

Our goal is to show that the family $\left(  \overline{s_{\lambda}}\right)
_{\lambda\in P_{k,n}}$ expands unitriangularly in the family $\left(
\overline{m_{\lambda}}\right)  _{\lambda\in P_{k,n}}$. In other words, our
goal is to show that each $\lambda\in P_{k,n}$ satisfies%
\begin{equation}
\overline{s_{\lambda}}=\overline{m_{\lambda}}+\left(  \text{a }\mathbf{k}%
\text{-linear combination of the elements }\overline{m_{\mu}}\text{ for }%
\mu\in P_{k,n}\text{ satisfying }\mu<_{\ast}\lambda\right)
\label{pf.thm.s-m-triangularity.tria}%
\end{equation}
(because \cite[Remark 11.1.17(c)]{GriRei18} shows that the family $\left(
\overline{s_{\lambda}}\right)  _{\lambda\in P_{k,n}}$ expands unitriangularly
in the family $\left(  \overline{m_{\lambda}}\right)  _{\lambda\in P_{k,n}}$
if and only if every $\lambda\in P_{k,n}$ satisfies
(\ref{pf.thm.s-m-triangularity.tria})). So let us prove
(\ref{pf.thm.s-m-triangularity.tria}).

Fix $\lambda\in P_{k,n}$. Define $a\in\mathbb{N}$ by $a=\left\vert
\lambda\right\vert $. Thus, $\lambda\in\operatorname*{Par}\nolimits_{a}$.
Hence, Lemma \ref{lem.kostka.m} \textbf{(c)} yields%
\begin{equation}
s_{\lambda}=\sum_{\substack{\mu\in P_{k,n};\\\mu\in\operatorname*{Par}%
\nolimits_{a}}}K_{\lambda,\mu}m_{\mu}=K_{\lambda,\lambda}m_{\lambda}%
+\sum_{\substack{\mu\in P_{k,n};\\\mu\in\operatorname*{Par}\nolimits_{a}%
;\\\mu\neq\lambda}}K_{\lambda,\mu}m_{\mu} \label{pf.thm.s-m-triangularity.3}%
\end{equation}
(here, we have split off the addend for $\mu=\lambda$, since $\lambda\in
P_{k,n}$ and $\lambda\in\operatorname*{Par}\nolimits_{a}$).

Now, let $M$ be the $\mathbf{k}$-submodule of $\mathcal{S}/I$ spanned by the
elements $\overline{m_{\mu}}$ for $\mu\in P_{k,n}$ satisfying $\mu<_{\ast
}\lambda$. Thus, we have%
\begin{equation}
\overline{m_{\mu}}\in M\ \ \ \ \ \ \ \ \ \ \text{for each }\mu\in
P_{k,n}\text{ satisfying }\mu<_{\ast}\lambda.
\label{pf.thm.s-m-triangularity.inM}%
\end{equation}
Also, $0\in M$ (since $M$ is a $\mathbf{k}$-submodule of $\mathcal{S}/I$).

We shall next show that%
\begin{equation}
K_{\lambda,\mu}\overline{m_{\mu}}\in M\ \ \ \ \ \ \ \ \ \ \text{for each }%
\mu\in P_{k,n}\text{ satisfying }\mu\in\operatorname*{Par}\nolimits_{a}\text{
and }\mu\neq\lambda. \label{pf.thm.s-m-triangularity.4}%
\end{equation}

[\textit{Proof of (\ref{pf.thm.s-m-triangularity.4}):} Let $\mu\in P_{k,n}$ be
such that $\mu\in\operatorname*{Par}\nolimits_{a}$ and $\mu\neq\lambda$. We
must prove that $K_{\lambda,\mu}\overline{m_{\mu}}\in M$.

If $K_{\lambda,\mu}=0$, then this follows immediately from
$\underbrace{K_{\lambda,\mu}}_{=0}\overline{m_{\mu}}=0\overline{m_{\mu}}=0\in
M$. Hence, for the rest of this proof, we WLOG assume that $K_{\lambda,\mu
}\neq0$.

If $\lambda$ and $\mu$ would not satisfy $\lambda\triangleright\mu$, then we
would have $K_{\lambda,\mu}=0$ (by Lemma \ref{lem.kostka.facts} \textbf{(a)}),
which would contradict $K_{\lambda,\mu}\neq0$. Hence, $\lambda$ and $\mu$ must
satisfy $\lambda\triangleright\mu$. Both $\lambda$ and $\mu$ are partitions of
size $a$ (since $\lambda\in\operatorname*{Par}\nolimits_{a}$ and $\mu
\in\operatorname*{Par}\nolimits_{a}$). Thus, $\left\vert \lambda\right\vert
=a$ and $\left\vert \mu\right\vert =a$. Thus, $\lambda\geq_{\ast}\mu$ (by
Remark \ref{rmk.size-domin.domin}) and thus $\mu<_{\ast}\lambda$ (since
$\mu\neq\lambda$). Hence, (\ref{pf.thm.s-m-triangularity.inM}) shows that
$\overline{m_{\mu}}\in M$. Thus, $K_{\lambda,\mu}\underbrace{\overline{m_{\mu
}}}_{\in M}\in K_{\lambda,\mu}M\subseteq M$ (since $M$ is a $\mathbf{k}%
$-submodule of $\mathcal{S}/I$). Thus, (\ref{pf.thm.s-m-triangularity.4}) is proven.]

Now, from (\ref{pf.thm.s-m-triangularity.3}), we obtain%
\begin{align*}
\overline{s_{\lambda}}  &  =\overline{K_{\lambda,\lambda}m_{\lambda}%
+\sum_{\substack{\mu\in P_{k,n};\\\mu\in\operatorname*{Par}\nolimits_{a}%
;\\\mu\neq\lambda}}K_{\lambda,\mu}m_{\mu}}=\underbrace{K_{\lambda,\lambda}%
}_{\substack{=1\\\text{(by Lemma \ref{lem.kostka.facts} \textbf{(b)})}%
}}\overline{m_{\lambda}}+\sum_{\substack{\mu\in P_{k,n};\\\mu\in
\operatorname*{Par}\nolimits_{a};\\\mu\neq\lambda}}\underbrace{K_{\lambda,\mu
}\overline{m_{\mu}}}_{\substack{\in M\\\text{(by
(\ref{pf.thm.s-m-triangularity.4}))}}}\\
&  \in\overline{m_{\lambda}}+\underbrace{\sum_{\substack{\mu\in P_{k,n}%
;\\\mu\in\operatorname*{Par}\nolimits_{a};\\\mu\neq\lambda}}M}%
_{\substack{\subseteq M\\\text{(since }M\text{ is a }\mathbf{k}\text{-module)}%
}}\subseteq\overline{m_{\lambda}}+M.
\end{align*}
In other words, $\overline{s_{\lambda}}-\overline{m_{\lambda}}\in M$. In other
words, $\overline{s_{\lambda}}-\overline{m_{\lambda}}$ is a $\mathbf{k}%
$-linear combination of the elements $\overline{m_{\mu}}$ for $\mu\in P_{k,n}$
satisfying $\mu<_{\ast}\lambda$ (since $M$ was defined as the $\mathbf{k}%
$-submodule of $\mathcal{S}/I$ spanned by these elements). In other words,%
\[
\overline{s_{\lambda}}=\overline{m_{\lambda}}+\left(  \text{a }\mathbf{k}%
\text{-linear combination of the elements }\overline{m_{\mu}}\text{ for }%
\mu\in P_{k,n}\text{ satisfying }\mu<_{\ast}\lambda\right)  .
\]
Thus, (\ref{pf.thm.s-m-triangularity.tria}) is proven. As we already have
explained, this completes the proof of Theorem \ref{thm.s-m-triangularity}.
\end{proof}

\begin{proof}
[Proof of Theorem \ref{thm.S/J-m-basis}.]Consider the finite set $P_{k,n}$ as
a poset (using the graded dominance order).

Theorem \ref{thm.s-m-triangularity} says that the family $\left(
\overline{s_{\lambda}}\right)  _{\lambda\in P_{k,n}}$ expands unitriangularly
in the family $\left(  \overline{m_{\lambda}}\right)  _{\lambda\in P_{k,n}}$.
Hence, the family $\left(  \overline{s_{\lambda}}\right)  _{\lambda\in
P_{k,n}}$ expands invertibly triangularly\footnote{See \cite[Definition
11.1.16(b)]{GriRei18} for the meaning of this word.} in the family $\left(
\overline{m_{\lambda}}\right)  _{\lambda\in P_{k,n}}$. Thus, \cite[Corollary
11.1.19(e)]{GriRei18} (applied to $\mathcal{S}/I$, $P_{k,n}$, $\left(
\overline{s_{\lambda}}\right)  _{\lambda\in P_{k,n}}$ and $\left(
\overline{m_{\lambda}}\right)  _{\lambda\in P_{k,n}}$ instead of $M$, $S$,
$\left(  e_{s}\right)  _{s\in S}$ and $\left(  f_{s}\right)  _{s\in S}$) shows
that the family $\left(  \overline{s_{\lambda}}\right)  _{\lambda\in P_{k,n}}$
is a basis of the $\mathbf{k}$-module $\mathcal{S}/I$ if and only if the
family $\left(  \overline{m_{\lambda}}\right)  _{\lambda\in P_{k,n}}$ is a
basis of the $\mathbf{k}$-module $\mathcal{S}/I$. Hence, the family $\left(
\overline{m_{\lambda}}\right)  _{\lambda\in P_{k,n}}$ is a basis of the
$\mathbf{k}$-module $\mathcal{S}/I$ (since the family $\left(  \overline
{s_{\lambda}}\right)  _{\lambda\in P_{k,n}}$ is a basis of the $\mathbf{k}%
$-module $\mathcal{S}/I$). Thus, Theorem \ref{thm.S/J-m-basis} is proven.
\end{proof}

\subsection{The e-basis}

We recall one more classical basis of $\mathcal{S}$:

\begin{definition}
\label{def.elam}For any partition $\lambda$, we let $e_{\lambda}$ denote the
elementary symmetric polynomial in $x_{1},x_{2},\ldots,x_{k}$ corresponding to
the partition $\lambda$. This elementary symmetric polynomial is what is
called $e_{\lambda}\left(  x_{1},x_{2},\ldots,x_{k}\right)  $ in \cite[Chapter
2]{GriRei18}. It is explicitly given by%
\[
e_{\lambda}=e_{\lambda_{1}}e_{\lambda_{2}}e_{\lambda_{3}}\cdots.
\]
Note that%
\[
e_{\lambda}=0\ \ \ \ \ \ \ \ \ \ \text{if }\lambda_{1}>k.
\]

\end{definition}

If $\lambda$ is any partition, then the elementary symmetric polynomial
$e_{\lambda}=e_{\lambda}\left(  x_{1},x_{2},\ldots,x_{k}\right)  $ is
symmetric and thus belongs to $\mathcal{S}$.

It is well-known (and goes back to Gauss) that the $\mathbf{k}$-algebra
$\mathcal{S}$ is generated by the algebraically independent elements
$e_{1},e_{2},\ldots,e_{k}$. Equivalently, the family $\left(  e_{\lambda^{t}%
}\right)  _{\lambda\in P_{k}}$ is a basis of the $\mathbf{k}$-module
$\mathcal{S}$ (see Definition \ref{def.Pk} for the meaning of $P_{k}$). Again,
we can obtain a basis of $\mathcal{S}/I$ by restricting this family:

\begin{theorem}
\label{thm.S/J-e-basis}The family $\left(  \overline{e_{\lambda^{t}}}\right)
_{\lambda\in P_{k,n}}$ is a basis of the $\mathbf{k}$-module $\mathcal{S}/I$.
\end{theorem}

This is a result of Weinfeld, proved in \cite[Theorem 6.2]{Weinfe19}.

\subsection{Non-bases}

What other known families of symmetric functions give rise to bases of
$\mathcal{S}/I$ ? Here is an example of a family that does not lead to such a
basis (at least not in an obvious way):

\begin{remark}
\label{rmk.S/J-p-not-basis}Let $n=4$ and $k=2$. Let $a_{1},a_{2}\in\mathbf{k}%
$. For each partition $\lambda$, let $p_{\lambda}$ be the corresponding power
sum symmetric polynomial, i.e., the $p_{\lambda}\left(  x_{1},x_{2}%
,\ldots,x_{k}\right)  $ from \cite[Definition 2.2.1]{GriRei18}. Then, the
family $\left(  \overline{p_{\lambda}}\right)  _{\lambda\in P_{k,n}}$ is not a
basis of the $\mathbf{k}$-module $\mathcal{S}/I$ (unless $\mathbf{k}=0$).
\end{remark}

\begin{proof}
[Proof of Remark \ref{rmk.S/J-p-not-basis}.]Straightforward computations yield
the following expansions of the $\overline{p_{\lambda}}$ in the basis $\left(
\overline{s_{\lambda}}\right)  _{\lambda\in P_{k,n}}$ of $\mathcal{S}/I$:%
\begin{align*}
\overline{p_{\varnothing}}  &  =\overline{s_{\varnothing}};\\
\overline{p_{\left(  1\right)  }}  &  =\overline{s_{\left(  1\right)  }};\\
\overline{p_{\left(  2\right)  }}  &  =-\overline{s_{\left(  1,1\right)  }%
}+\overline{s_{\left(  2\right)  }};\\
\overline{p_{\left(  1,1\right)  }}  &  =\overline{s_{\left(  1,1\right)  }%
}+\overline{s_{\left(  2\right)  }};\\
\overline{p_{\left(  2,1\right)  }}  &  =a_{1}\overline{s_{\varnothing}};\\
\overline{p_{\left(  2,2\right)  }}  &  =2a_{2}\overline{s_{\varnothing}%
}-a_{1}\overline{s_{\left(  1\right)  }}+2\overline{s_{\left(  2,2\right)  }}.
\end{align*}
Thus, $\overline{p_{\left(  2,1\right)  }}-a_{1}\overline{p_{\varnothing}}=0$.
Hence, the family $\left(  \overline{p_{\lambda}}\right)  _{\lambda\in
P_{k,n}}$ fails to be $\mathbf{k}$-linearly independent, and thus cannot be a
basis of $\mathcal{S}/I$. This proves Remark \ref{rmk.S/J-p-not-basis}.
\end{proof}

It is natural to wonder for which pairs $\left(  n,k\right)  $ the family
$\left(  \overline{p_{\lambda}}\right)  _{\lambda\in P_{k,n}}$ is a basis of
$\mathcal{S}/I$. The following table (made using SageMath) collects some
answers:%
\[%
\begin{tabular}
[c]{c||ccccccc}
& $k=1$ & $k=2$ & $k=3$ & $k=4$ & $k=5$ & $k=6$ & $k=7$\\\hline\hline
$n=2$ & yes &  &  &  &  &  & \\
$n=3$ & yes & yes &  &  &  &  & \\
$n=4$ & yes & no & yes &  &  &  & \\
$n=5$ & yes & st & st & yes &  &  & \\
$n=6$ & yes & st & no & no & yes &  & \\
$n=7$ & yes & st & st & st & st & yes & \\
$n=8$ & yes & st & st & no & st & no & yes
\end{tabular}
\ \ \ .
\]
Here, \textquotedblleft yes\textquotedblright\ means that the family is a
basis; \textquotedblleft no\textquotedblright\ means that the family is not a
basis; \textquotedblleft st\textquotedblright\ means that the answer depends
on the characteristic of $\mathbf{k}$. Interestingly, the answer never depends
on $a_{1},a_{2},\ldots,a_{k}$ in the cases tabulated above. (We have omitted
the trivial cases $k=0$ and $k=n$ from the table, since $\mathcal{S}%
/I\cong\mathbf{k}$ in these cases. We note that the \textquotedblleft
yes\textquotedblright es for $k=n-1$ are fairly obvious, since $\left(
\overline{p_{\lambda}}\right)  _{\lambda\in P_{k,n}}=\left(  \overline
{h_{\lambda}}\right)  _{\lambda\in P_{k,n}}$ in this case. The
\textquotedblleft yes\textquotedblright es for $k=1$ hold for the same reason.)

\begin{question}
Which other patterns in the above table can be explained? Is there a reason
why the \textquotedblleft no\textquotedblright s appear for even $n$'s?
\end{question}

Another non-basis is the family $\left(  \overline{h_{\lambda^{t}}}\right)
_{\lambda\in P_{k,n}}$:

\begin{remark}
\label{rmk.S/J-ht-not-basis}Let $n=3$ and $k=2$. Let $a_{1},a_{2}\in
\mathbf{k}$. Then, the family $\left(  \overline{h_{\lambda^{t}}}\right)
_{\lambda\in P_{k,n}}$ is not a basis of the $\mathbf{k}$-module
$\mathcal{S}/I$ (unless $\mathbf{k}=0$).
\end{remark}

\begin{proof}
[Proof of Remark \ref{rmk.S/J-ht-not-basis}.]It is easy to see that
$\overline{h_{\left(  1,1\right)  ^{t}}}-a_{1}\overline{h_{\varnothing^{t}}%
}=0$ (indeed, this follows from $h_{\left(  1,1\right)  ^{t}}-a_{1}%
h_{\varnothing^{t}}=h_{2}-a_{1}\in I$). Hence, the family $\left(
\overline{h_{\lambda^{t}}}\right)  _{\lambda\in P_{k,n}}$ fails to be
$\mathbf{k}$-linearly independent, and thus cannot be a basis of
$\mathcal{S}/I$. This proves Remark \ref{rmk.S/J-ht-not-basis}.
\end{proof}

The followng table shows for which pairs $\left(  n,k\right)  $ the family
$\left(  \overline{h_{\lambda^{t}}}\right)  _{\lambda\in P_{k,n}}$ is a basis
of $\mathcal{S}/I$. The following table collects some answers:%
\[%
\begin{tabular}
[c]{c||ccccccc}
& $k=1$ & $k=2$ & $k=3$ & $k=4$ & $k=5$ & $k=6$ & $k=7$\\\hline\hline
$n=2$ & yes &  &  &  &  &  & \\
$n=3$ & yes & no &  &  &  &  & \\
$n=4$ & yes & yes & no &  &  &  & \\
$n=5$ & yes & st & no & no &  &  & \\
$n=6$ & yes & no & yes & no & no &  & \\
$n=7$ & yes & st & st & no & no & no & \\
$n=8$ & yes & st & no & yes & no & no & no
\end{tabular}
\ \ \ .
\]
The \textquotedblleft yes\textquotedblright s in the $k=1$ column are easily
explained (they are saying that $\left(  1,x,x^{2},\ldots,x^{n-k}\right)  $ is
a basis of $\mathbf{k}\left[  x\right]  /\left(  x^{n-k+1}-a_{1}\right)  $),
and so are the \textquotedblleft no\textquotedblright s in the $k>n/2$ region
(indeed, in these cases, $\overline{h_{n-k+1}}-a_{1}\overline{h_{\varnothing}%
}=0$ provides a $\mathbf{k}$-linear dependence relation between the
$\overline{h_{\lambda^{t}}}$, as in our above proof of Remark
\ref{rmk.S/J-ht-not-basis}). The \textquotedblleft yes\textquotedblright s in
the $k=n/2$ cases follow from Theorem \ref{thm.S/J-h-basis} (since the family
$\left(  \overline{h_{\lambda^{t}}}\right)  _{\lambda\in P_{k,n}}$ is a
relabeling of the family $\left(  \overline{h_{\lambda}}\right)  _{\lambda\in
P_{k,n}}$ in these cases).

\begin{question}
Which other patterns exist in the above table?
\end{question}

\section{\label{sect.pieri}Pieri rules for multiplying by $\overline{h_{j}}$}

\begin{convention}
Convention \ref{conv.symmetry-conv} remains in place for the whole Section
\ref{sect.pieri}.

We shall also use all the notations introduced in Section \ref{sect.symmetry}.
\end{convention}

In this section, we shall explore formulas for expanding products of the form
$\overline{s_{\lambda}h_{j}}$ in the basis $\left(  \overline{s_{\mu}}\right)
_{\mu\in P_{k,n}}$. We begin with the simplest case -- that of $j=1$:

\subsection{Multiplying by $\overline{h_{1}}$}

\begin{proposition}
\label{prop.pieri-h1}Let $\lambda\in P_{k,n}$. Assume that $k>0$.

\textbf{(a)} If $\lambda_{1}<n-k$, then%
\[
\overline{s_{\lambda}h_{1}}=\sum_{\substack{\mu\in P_{k,n};\\\mu/\lambda\text{
is a single box}}}\overline{s_{\mu}}.
\]

\textbf{(b)} Let $\overline{\lambda}$ be the partition $\left(  \lambda
_{2},\lambda_{3},\lambda_{4},\ldots\right)  $. If $\lambda_{1}=n-k$, then%
\[
\overline{s_{\lambda}h_{1}}=\sum_{\substack{\mu\in P_{k,n};\\\mu/\lambda\text{
is a single box}}}\overline{s_{\mu}}+\sum_{i=0}^{k-1}\left(  -1\right)
^{i}a_{1+i}\sum_{\substack{\mu\in P_{k,n};\\\overline{\lambda}/\mu\text{ is a
vertical }i\text{-strip}}}\overline{s_{\mu}}.
\]

\end{proposition}

\begin{proof}
[Proof of Proposition \ref{prop.pieri-h1}.]We have $\mathbf{h}_{1}%
=\mathbf{e}_{1}$, thus%
\[
\mathbf{s}_{\lambda}\mathbf{h}_{1}=\mathbf{s}_{\lambda}\mathbf{e}_{1}%
=\sum_{\substack{\mu\text{ is a partition;}\\\mu/\lambda\text{ is a vertical
}1\text{-strip}}}\mathbf{s}_{\mu}%
\]
(by Proposition \ref{prop.pieri.e}, applied to $i=1$). Evaluating both sides
of this identity at the $k$ variables $x_{1},x_{2},\ldots,x_{k}$, we find%
\[
s_{\lambda}h_{1}=\sum_{\substack{\mu\text{ is a partition;}\\\mu/\lambda\text{
is a vertical }1\text{-strip}}}s_{\mu}=\sum_{\substack{\mu\text{ is a
partition;}\\\mu/\lambda\text{ is a single box}}}s_{\mu}%
\]
(because a skew diagram $\mu/\lambda$ is a vertical $1$-strip if and only if
it is a single box). This becomes%
\begin{align}
s_{\lambda}h_{1}  &  =\sum_{\substack{\mu\text{ is a partition;}\\\mu
/\lambda\text{ is a single box}}}s_{\mu}\nonumber\\
&  =\sum_{\substack{\mu\text{ is a partition;}\\\mu/\lambda\text{ is a single
box;}\\\mu\text{ has at most }k\text{ parts}}}s_{\mu}+\sum_{\substack{\mu
\text{ is a partition;}\\\mu/\lambda\text{ is a single box;}\\\mu\text{ has
more than }k\text{ parts}}}\underbrace{s_{\mu}}_{\substack{=0\\\text{(by
(\ref{eq.slam=0-too-long})}\\\text{(applied to }\mu\text{ instead of }%
\lambda\text{))}}}\nonumber\\
&  =\sum_{\substack{\mu\text{ is a partition;}\\\mu/\lambda\text{ is a single
box;}\\\mu\text{ has at most }k\text{ parts}}}s_{\mu}.
\label{pf.prop.pieri-h1.2}%
\end{align}

\textbf{(a)} Assume that $\lambda_{1}<n-k$. Then, each partition $\mu$
satisfying%
\begin{equation}
\left(  \mu/\lambda\text{ is a single box}\right)  \wedge\left(  \mu\text{ has
at most }k\text{ parts}\right)  \label{pf.prop.pieri-h1.a.1ass}%
\end{equation}
must satisfy%
\begin{equation}
\mu\in P_{k,n}. \label{pf.prop.pieri-h1.a.1}%
\end{equation}

[\textit{Proof of (\ref{pf.prop.pieri-h1.a.1}):} Let $\mu$ be a partition
satisfying (\ref{pf.prop.pieri-h1.a.1ass}). We must prove that $\mu\in
P_{k,n}$.

We have $\mu_{1}\leq\lambda_{1}+1$ (since $\mu/\lambda$ is a single box) and
thus $\mu_{1}\leq\lambda_{1}+1\leq n-k$ (since $\lambda_{1}<n-k$). Hence, each
part of $\mu$ is $\leq n-k$ (since $\mu$ is a partition). Thus, $\mu\in
P_{k,n}$ (since $\mu$ has at most $k$ parts). This proves
(\ref{pf.prop.pieri-h1.a.1}).]

Now, (\ref{pf.prop.pieri-h1.2}) becomes%
\[
s_{\lambda}h_{1}=\sum_{\substack{\mu\text{ is a partition;}\\\mu/\lambda\text{
is a single box;}\\\mu\text{ has at most }k\text{ parts}}}s_{\mu}%
=\sum_{\substack{\mu\in P_{k,n};\\\mu/\lambda\text{ is a single box}}}s_{\mu}%
\]
(because (\ref{pf.prop.pieri-h1.a.1}) yields the equality $\sum_{\substack{\mu
\text{ is a partition;}\\\mu/\lambda\text{ is a single box;}\\\mu\text{ has at
most }k\text{ parts}}}=\sum_{\substack{\mu\in P_{k,n};\\\mu/\lambda\text{ is a
single box}}}$ of summation signs). Projecting both sides of this equality
onto $\mathcal{S}/I$, we obtain%
\[
\overline{s_{\lambda}h_{1}}=\overline{\sum_{\substack{\mu\in P_{k,n}%
;\\\mu/\lambda\text{ is a single box}}}s_{\mu}}=\sum_{\substack{\mu\in
P_{k,n};\\\mu/\lambda\text{ is a single box}}}\overline{s_{\mu}}.
\]
This proves Proposition \ref{prop.pieri-h1} \textbf{(a)}.

\textbf{(b)} Assume that $\lambda_{1}=n-k$. Let $\nu$ be the partition
$\left(  \lambda_{1}+1,\lambda_{2},\lambda_{3},\ldots\right)  $. Then,
$\nu/\lambda$ is a single box, which lies in the first row. The definition of
$\nu$ yields $\nu_{1}=\lambda_{1}+1=n-k+1$ (since $\lambda_{1}=n-k$) and thus
$\nu_{1}>n-k$; hence, not all parts of $\nu$ are $\leq n-k$. Thus, $\nu\notin
P_{k,n}$.

Clearly, $\overline{\lambda}\in P_{k,n}$. Hence, if $i\in\mathbb{N}$, and if
$\mu$ is any partition such that $\overline{\lambda}/\mu$ is a vertical
$i$-strip, then $\mu\in P_{k,n}$ (since $\mu\subseteq\overline{\lambda}$).
Thus, for each $i\in\mathbb{N}$, we have the following equality of summation
signs:%
\begin{equation}
\sum_{\substack{\mu\text{ is a partition;}\\\overline{\lambda}/\mu\text{ is a
vertical }i\text{-strip}}}=\sum_{\substack{\mu\in P_{k,n};\\\overline{\lambda
}/\mu\text{ is a vertical }i\text{-strip}}}. \label{pf.prop.pieri-h1.b.-1}%
\end{equation}

The partition $\nu$ has at most $k$ parts (since $\lambda$ has at most $k$
parts, and since $k>0$). The definition of $\nu$ yields $\nu_{1}=\lambda
_{1}+1=n-k+1$ (since $\lambda_{1}=n-k$) and $\left(  \nu_{2},\nu_{3},\nu
_{4},\ldots\right)  =\left(  \lambda_{2},\lambda_{3},\lambda_{4}%
,\ldots\right)  =\overline{\lambda}$. Hence, Lemma \ref{lem.coeffw.0} (applied
to $\nu$ and $\nu_{i}$ instead of $\lambda$ and $\lambda_{i}$) yields%
\begin{equation}
\overline{s_{\nu}}=\sum_{i=0}^{k-1}\left(  -1\right)  ^{i}a_{1+i}%
\sum_{\substack{\mu\text{ is a partition;}\\\overline{\lambda}/\mu\text{ is a
vertical }i\text{-strip}}}\overline{s_{\mu}}=\sum_{i=0}^{k-1}\left(
-1\right)  ^{i}a_{1+i}\sum_{\substack{\mu\in P_{k,n};\\\overline{\lambda}%
/\mu\text{ is a vertical }i\text{-strip}}}\overline{s_{\mu}}
\label{pf.prop.pieri-h1.b.0}%
\end{equation}
(by (\ref{pf.prop.pieri-h1.b.-1})).

Each partition $\mu$ satisfying%
\begin{equation}
\left(  \mu/\lambda\text{ is a single box}\right)  \wedge\left(  \mu\text{ has
at most }k\text{ parts}\right)  \wedge\left(  \mu\neq\nu\right)
\label{pf.prop.pieri-h1.b.1ass}%
\end{equation}
must satisfy%
\begin{equation}
\mu\in P_{k,n}. \label{pf.prop.pieri-h1.b.1}%
\end{equation}

[\textit{Proof of (\ref{pf.prop.pieri-h1.b.1}):} Let $\mu$ be a partition
satisfying (\ref{pf.prop.pieri-h1.b.1ass}). We must prove that $\mu\in
P_{k,n}$.

We know that $\mu/\lambda$ is a single box. If we had $\mu_{1}>\lambda_{1}$,
then this box would lie in the first row, which would yield that $\mu=\nu$
(because $\nu$ is the partition obtained from $\lambda$ by adding a box in the
first row); but this would contradict $\mu\neq\nu$. Hence, we cannot have
$\mu_{1}>\lambda_{1}$. Thus, we have $\mu_{1}\leq\lambda_{1}=n-k$. Hence, each
part of $\mu$ is $\leq n-k$ (since $\mu$ is a partition). Thus, $\mu\in
P_{k,n}$ (since $\mu$ has at most $k$ parts). This proves
(\ref{pf.prop.pieri-h1.b.1}).]

Conversely, each $\mu\in P_{k,n}$ satisfies $\mu\neq\nu$ (because $\nu\notin
P_{k,n}$) and has at most $k$ parts. Combining this with
(\ref{pf.prop.pieri-h1.b.1}), we obtain the following equality of summation
signs:%
\begin{equation}
\sum_{\substack{\mu\text{ is a partition;}\\\mu/\lambda\text{ is a single
box;}\\\mu\text{ has at most }k\text{ parts;}\\\mu\neq\nu}}=\sum
_{\substack{\mu\in P_{k,n};\\\mu/\lambda\text{ is a single box}}}.
\label{pf.prop.pieri-h1.b.2}%
\end{equation}

Now, (\ref{pf.prop.pieri-h1.2}) becomes%
\begin{align*}
s_{\lambda}h_{1}  &  =\sum_{\substack{\mu\text{ is a partition;}\\\mu
/\lambda\text{ is a single box;}\\\mu\text{ has at most }k\text{ parts}%
}}s_{\mu}=s_{\nu}+\sum_{\substack{\mu\text{ is a partition;}\\\mu
/\lambda\text{ is a single box;}\\\mu\text{ has at most }k\text{ parts;}%
\\\mu\neq\nu}}s_{\mu}\\
&  \ \ \ \ \ \ \ \ \ \ \left(
\begin{array}
[c]{c}%
\text{here, we have split off the addend for }\mu=\nu\text{ from the sum}\\
\text{(since }\nu/\lambda\text{ is a single box, and since }\nu\text{ has at
most }k\text{ parts)}%
\end{array}
\right) \\
&  =s_{\nu}+\sum_{\substack{\mu\in P_{k,n};\\\mu/\lambda\text{ is a single
box}}}s_{\mu}\ \ \ \ \ \ \ \ \ \ \left(  \text{by (\ref{pf.prop.pieri-h1.b.2}%
)}\right)  .
\end{align*}
Projecting both sides of this equality onto $\mathcal{S}/I$, we obtain%
\begin{align*}
\overline{s_{\lambda}h_{1}}  &  =\overline{s_{\nu}+\sum_{\substack{\mu\in
P_{k,n};\\\mu/\lambda\text{ is a single box}}}s_{\mu}}=\overline{s_{\nu}}%
+\sum_{\substack{\mu\in P_{k,n};\\\mu/\lambda\text{ is a single box}%
}}\overline{s_{\mu}}=\sum_{\substack{\mu\in P_{k,n};\\\mu/\lambda\text{ is a
single box}}}\overline{s_{\mu}}+\overline{s_{\nu}}\\
&  =\sum_{\substack{\mu\in P_{k,n};\\\mu/\lambda\text{ is a single box}%
}}\overline{s_{\mu}}+\sum_{i=0}^{k-1}\left(  -1\right)  ^{i}a_{1+i}%
\sum_{\substack{\mu\in P_{k,n};\\\overline{\lambda}/\mu\text{ is a vertical
}i\text{-strip}}}\overline{s_{\mu}}%
\end{align*}
(by (\ref{pf.prop.pieri-h1.b.0})). This proves Proposition \ref{prop.pieri-h1}
\textbf{(b)}.
\end{proof}

\subsection{Multiplying by $\overline{h_{n-k}}$}

On the other end of the spectrum is the case of $j=n-k$; this case also turns
out to have a simple answer:

\begin{proposition}
\label{prop.pieri-hlast}Let $\lambda\in P_{k,n}$. Assume that $k>0$.

\textbf{(a)} We have%
\[
\overline{s_{\lambda}h_{n-k}}=\overline{s_{\left(  n-k,\lambda_{1},\lambda
_{2},\lambda_{3},\ldots\right)  }}-\sum_{i=1}^{k}\left(  -1\right)  ^{i}%
a_{i}\sum_{\substack{\mu\in P_{k,n};\\\lambda/\mu\text{ is a vertical
}i\text{-strip}}}\overline{s_{\mu}}.
\]

\textbf{(b)} If $\lambda_{k}>0$, then%
\[
\overline{s_{\lambda}h_{n-k}}=-\sum_{i=1}^{k}\left(  -1\right)  ^{i}a_{i}%
\sum_{\substack{\mu\in P_{k,n};\\\lambda/\mu\text{ is a vertical
}i\text{-strip}}}\overline{s_{\mu}}.
\]

\end{proposition}

\begin{proof}
[Proof of Proposition \ref{prop.pieri-hlast}.]We have $\lambda\in P_{k,n}$,
thus $\lambda_{1}\leq n-k$. Hence, $n-k\geq\lambda_{1}$. Thus, $\left(
n-k,\lambda_{1},\lambda_{2},\lambda_{3},\ldots\right)  $ is a partition.

\textbf{(a)} We have
\begin{equation}
\left(  \mathbf{e}_{i}\right)  ^{\perp}\mathbf{s}_{\lambda}%
=0\ \ \ \ \ \ \ \ \ \ \text{for every integer }i>k.
\label{pf.prop.pieri-hlast.a.1}%
\end{equation}

[\textit{Proof of (\ref{pf.prop.pieri-hlast.a.1}):} Let $i>k$ be an integer.
The partition $\lambda$ has at most $k$ parts (since $\lambda\in P_{k,n}$). In
other words, the Young diagram of $\lambda$ contains at most $k$ rows. Hence,
this diagram contains no vertical $i$-strip (since a vertical $i$-strip would
involve more than $k$ rows (because $i>k$)). Thus, there exists no partition
$\mu$ such that $\lambda/\mu$ is a vertical $i$-strip. Hence, $\sum
_{\substack{\mu\text{ is a partition;}\\\lambda/\mu\text{ is a vertical
}i\text{-strip}}}\mathbf{s}_{\mu}=\left(  \text{empty sum}\right)  =0$. But
Corollary \ref{cor.pieri.eskew} yields $\left(  \mathbf{e}_{i}\right)
^{\perp}\mathbf{s}_{\lambda}=\sum_{\substack{\mu\text{ is a partition;}%
\\\lambda/\mu\text{ is a vertical }i\text{-strip}}}\mathbf{s}_{\mu}=0$. This
proves (\ref{pf.prop.pieri-hlast.a.1}).]

Recall that $\mathbf{e}_{0}=1$ and thus $\left(  \mathbf{e}_{0}\right)
^{\perp}=1^{\perp}=\operatorname*{id}$. Hence, $\left(  \mathbf{e}_{0}\right)
^{\perp}\mathbf{s}_{\lambda}=\operatorname*{id}\mathbf{s}_{\lambda}%
=\mathbf{s}_{\lambda}$.

But $n-k\geq\lambda_{1}$. Hence, Proposition \ref{prop.bernstein.1} (applied
to $m=n-k$) yields%
\[
\sum_{i\in\mathbb{N}}\left(  -1\right)  ^{i}\mathbf{h}_{n-k+i}\left(
\mathbf{e}_{i}\right)  ^{\perp}\mathbf{s}_{\lambda}=\mathbf{s}_{\left(
n-k,\lambda_{1},\lambda_{2},\lambda_{3},\ldots\right)  }.
\]
Hence,%
\begin{align*}
\mathbf{s}_{\left(  n-k,\lambda_{1},\lambda_{2},\lambda_{3},\ldots\right)  }
&  =\sum_{i\in\mathbb{N}}\left(  -1\right)  ^{i}\mathbf{h}_{n-k+i}\left(
\mathbf{e}_{i}\right)  ^{\perp}\mathbf{s}_{\lambda}\\
&  =\sum_{i=0}^{k}\left(  -1\right)  ^{i}\mathbf{h}_{n-k+i}\left(
\mathbf{e}_{i}\right)  ^{\perp}\mathbf{s}_{\lambda}+\sum_{i=k+1}^{\infty
}\left(  -1\right)  ^{i}\mathbf{h}_{n-k+i}\underbrace{\left(  \mathbf{e}%
_{i}\right)  ^{\perp}\mathbf{s}_{\lambda}}_{\substack{=0\\\text{(by
(\ref{pf.prop.pieri-hlast.a.1}))}}}\\
&  =\sum_{i=0}^{k}\left(  -1\right)  ^{i}\mathbf{h}_{n-k+i}\left(
\mathbf{e}_{i}\right)  ^{\perp}\mathbf{s}_{\lambda}\\
&  =\underbrace{\left(  -1\right)  ^{0}}_{=1}\underbrace{\mathbf{h}_{n-k+0}%
}_{=\mathbf{h}_{n-k}}\underbrace{\left(  \mathbf{e}_{0}\right)  ^{\perp
}\mathbf{s}_{\lambda}}_{=\mathbf{s}_{\lambda}}+\sum_{i=1}^{k}\left(
-1\right)  ^{i}\mathbf{h}_{n-k+i}\underbrace{\left(  \mathbf{e}_{i}\right)
^{\perp}\mathbf{s}_{\lambda}}_{\substack{=\sum_{\substack{\mu\text{ is a
partition;}\\\lambda/\mu\text{ is a vertical }i\text{-strip}}}\mathbf{s}_{\mu
}\\\text{(by Corollary \ref{cor.pieri.eskew})}}}\\
&  =\underbrace{\mathbf{h}_{n-k}\mathbf{s}_{\lambda}}_{=\mathbf{s}_{\lambda
}\mathbf{h}_{n-k}}+\sum_{i=1}^{k}\left(  -1\right)  ^{i}\mathbf{h}_{n-k+i}%
\sum_{\substack{\mu\text{ is a partition;}\\\lambda/\mu\text{ is a vertical
}i\text{-strip}}}\mathbf{s}_{\mu}\\
&  =\mathbf{s}_{\lambda}\mathbf{h}_{n-k}+\sum_{i=1}^{k}\left(  -1\right)
^{i}\mathbf{h}_{n-k+i}\sum_{\substack{\mu\text{ is a partition;}\\\lambda
/\mu\text{ is a vertical }i\text{-strip}}}\mathbf{s}_{\mu},
\end{align*}
so that%
\[
\mathbf{s}_{\lambda}\mathbf{h}_{n-k}=\mathbf{s}_{\left(  n-k,\lambda
_{1},\lambda_{2},\lambda_{3},\ldots\right)  }-\sum_{i=1}^{k}\left(  -1\right)
^{i}\mathbf{h}_{n-k+i}\sum_{\substack{\mu\text{ is a partition;}\\\lambda
/\mu\text{ is a vertical }i\text{-strip}}}\mathbf{s}_{\mu}.
\]
This is an equality in $\Lambda$. If we evaluate both of its sides at
$x_{1},x_{2},\ldots,x_{k}$, then we obtain%
\begin{align*}
s_{\lambda}h_{n-k}  &  =s_{\left(  n-k,\lambda_{1},\lambda_{2},\lambda
_{3},\ldots\right)  }-\sum_{i=1}^{k}\left(  -1\right)  ^{i}%
\underbrace{h_{n-k+i}}_{\substack{\equiv a_{i}\operatorname{mod}I\\\text{(by
(\ref{eq.h=amodI}))}}}\underbrace{\sum_{\substack{\mu\text{ is a
partition;}\\\lambda/\mu\text{ is a vertical }i\text{-strip}}}}%
_{\substack{=\sum_{\substack{\mu\in P_{k,n};\\\lambda/\mu\text{ is a vertical
}i\text{-strip}}}\\\text{(because if }\mu\text{ is a partition such}%
\\\text{that }\lambda/\mu\text{ is a vertical }i\text{-strip, then }\mu\in
P_{k,n}\\\text{(since }\mu\subseteq\lambda\text{ and }\lambda\in
P_{k,n}\text{))}}}s_{\mu}\\
&  \equiv s_{\left(  n-k,\lambda_{1},\lambda_{2},\lambda_{3},\ldots\right)
}-\sum_{i=1}^{k}\left(  -1\right)  ^{i}a_{i}\sum_{\substack{\mu\in
P_{k,n};\\\lambda/\mu\text{ is a vertical }i\text{-strip}}}s_{\mu
}\operatorname{mod}I.
\end{align*}
In other words,%
\begin{align*}
\overline{s_{\lambda}h_{n-k}}  &  =\overline{s_{\left(  n-k,\lambda
_{1},\lambda_{2},\lambda_{3},\ldots\right)  }-\sum_{i=1}^{k}\left(  -1\right)
^{i}a_{i}\sum_{\substack{\mu\in P_{k,n};\\\lambda/\mu\text{ is a vertical
}i\text{-strip}}}s_{\mu}}\\
&  =\overline{s_{\left(  n-k,\lambda_{1},\lambda_{2},\lambda_{3}%
,\ldots\right)  }}-\sum_{i=1}^{k}\left(  -1\right)  ^{i}a_{i}\sum
_{\substack{\mu\in P_{k,n};\\\lambda/\mu\text{ is a vertical }i\text{-strip}%
}}\overline{s_{\mu}}.
\end{align*}
This proves Proposition \ref{prop.pieri-hlast} \textbf{(a)}.

\textbf{(b)} Assume that $\lambda_{k}>0$. Hence, the partition $\left(
n-k,\lambda_{1},\lambda_{2},\lambda_{3},\ldots\right)  $ has more than $k$
parts (since its $\left(  k+1\right)  $-st entry is $\lambda_{k}>0$). Thus,
(\ref{eq.slam=0-too-long}) (applied to $\left(  n-k,\lambda_{1},\lambda
_{2},\lambda_{3},\ldots\right)  $ instead of $\lambda$) yields $s_{\left(
n-k,\lambda_{1},\lambda_{2},\lambda_{3},\ldots\right)  }=0$. Hence,
$\overline{s_{\left(  n-k,\lambda_{1},\lambda_{2},\lambda_{3},\ldots\right)
}}=\overline{0}=0$. Now, Proposition \ref{prop.pieri-hlast} \textbf{(a)}
yields%
\begin{align*}
\overline{s_{\lambda}h_{n-k}}  &  =\underbrace{\overline{s_{\left(
n-k,\lambda_{1},\lambda_{2},\lambda_{3},\ldots\right)  }}}_{=0}-\sum_{i=1}%
^{k}\left(  -1\right)  ^{i}a_{i}\sum_{\substack{\mu\in P_{k,n};\\\lambda
/\mu\text{ is a vertical }i\text{-strip}}}\overline{s_{\mu}}\\
&  =-\sum_{i=1}^{k}\left(  -1\right)  ^{i}a_{i}\sum_{\substack{\mu\in
P_{k,n};\\\lambda/\mu\text{ is a vertical }i\text{-strip}}}\overline{s_{\mu}}.
\end{align*}
This proves Proposition \ref{prop.pieri-hlast} \textbf{(b)}.
\end{proof}

\subsection{Multiplying by $\overline{h_{j}}$}

At last, let us give an explicit expansion for $\overline{s_{\lambda}h_{j}}$
in the basis $\left(  \overline{s_{\mu}}\right)  _{\mu\in P_{k,n}}$ that holds
for all $j\in\left\{  0,1,\ldots,n-k\right\}  $. Before we state it, we need a notation:

\begin{definition}
Let $\mathbf{f}\in\Lambda$ be any symmetric function. Then, $\overline
{\mathbf{f}}\in\mathcal{S}/I$ is defined to be $\overline{f}$, where
$f\in\mathcal{S}$ is the result of evaluating the symmetric function
$\mathbf{f}\in\Lambda$ at the $k$ variables $x_{1},x_{2},\ldots,x_{k}$. Thus,
for every partition $\lambda$, we have $\overline{\mathbf{s}_{\lambda}%
}=\overline{s_{\lambda}}$. Likewise, for any $m\in\mathbb{N}$, we have
$\overline{\mathbf{h}_{m}}=h_{m}$ and $\overline{\mathbf{e}_{m}}=e_{m}$.
\end{definition}

\begin{theorem}
\label{thm.pieri-hall}Let $\lambda\in P_{k,n}$. Let $j\in\left\{
0,1,\ldots,n-k\right\}  $. Then,%
\[
\overline{s_{\lambda}h_{j}}=\sum_{\substack{\mu\in P_{k,n};\\\mu/\lambda\text{
is a horizontal }j\text{-strip}}}\overline{s_{\mu}}-\sum_{i=1}^{k}\left(
-1\right)  ^{i}a_{i}\overline{\left(  \mathbf{s}_{\left(  n-k-j+1,1^{i-1}%
\right)  }\right)  ^{\perp}\mathbf{s}_{\lambda}}.
\]

\end{theorem}

\begin{example}
\label{exa.pieri-hall.73}If $n=7$ and $k=3$, then%
\begin{align*}
&  \overline{s_{\left(  4,3,2\right)  }h_{2}}\\
&  =\overline{s_{\left(  4,4,3\right)  }}+a_{1}\left(  \overline{s_{\left(
4,2\right)  }}+\overline{s_{\left(  3,2,1\right)  }}+\overline{s_{\left(
3,3\right)  }}\right)  -a_{2}\left(  \overline{s_{\left(  4,1\right)  }%
}+\overline{s_{\left(  2,2,1\right)  }}+\overline{s_{\left(  3,1,1\right)  }%
}+2\overline{s_{\left(  3,2\right)  }}\right) \\
&  \ \ \ \ \ \ \ \ \ \ +a_{3}\left(  \overline{s_{\left(  2,2\right)  }%
}+\overline{s_{\left(  2,1,1\right)  }}+\overline{s_{\left(  3,1\right)  }%
}\right)  .
\end{align*}

\end{example}

It is not hard to reveal Propositions \ref{prop.pieri-h1} and
\ref{prop.pieri-hlast} as particular cases of Theorem \ref{thm.pieri-hall} (by
setting $j=1$ or $j=n-k$, respectively). Likewise, one can see that Theorem
\ref{thm.pieri-hall} generalizes \cite[(22)]{BeCiFu99}. Indeed, \cite[(22)]%
{BeCiFu99} says that if $a_{1}=a_{2}=\cdots=a_{k-1}=0$, then every $\lambda\in
P_{k,n}$ and $j\in\left\{  0,1,\ldots,n-k\right\}  $ satisfy%
\[
\overline{s_{\lambda}h_{j}}=\sum_{\substack{\mu\in P_{k,n};\\\mu/\lambda\text{
is a horizontal }j\text{-strip}}}\overline{s_{\mu}}-\left(  -1\right)
^{k}a_{k}\sum_{\nu}\overline{s_{\nu}},
\]
where the second sum runs over all $\nu\in P_{k,n}$ satisfying
\begin{align*}
&  \left(  \lambda_{i}-1\geq\nu_{i}\text{ for all }i\in\left\{  1,2,\ldots
,k\right\}  \right)  \ \ \ \ \ \ \ \ \ \ \text{and}\\
&  \left(  \nu_{i}\geq\lambda_{i+1}-1\text{ for all }i\in\left\{
1,2,\ldots,k-1\right\}  \right)  \ \ \ \ \ \ \ \ \ \ \text{and}\\
&  \left\vert \nu\right\vert =\left\vert \lambda\right\vert +j-n.
\end{align*}

Note, however, that the sums in Theorem \ref{thm.pieri-hall} contain
multiplicities (see the \textquotedblleft$2\overline{s_{\left(  3,2\right)  }%
}$\textquotedblright\ in Example \ref{exa.pieri-hall.73}), unlike those in
\cite[(22)]{BeCiFu99}.

We shall prove Theorem \ref{thm.pieri-hall} by deriving it from an identity
between genuine symmetric functions (in $\Lambda$, not in $\mathcal{S}$ or
$\mathcal{S}/I$):

\begin{theorem}
\label{thm.pieri-hall.Lam}Let $\lambda\in P_{k,n}$. Let $j\in\left\{
0,1,\ldots,n-k\right\}  $. Then,%
\[
\mathbf{s}_{\lambda}\mathbf{h}_{j}=\sum_{\substack{\mu\text{ is a
partition;}\\\mu_{1}\leq n-k;\\\mu/\lambda\text{ is a horizontal
}j\text{-strip}}}\mathbf{s}_{\mu}-\sum_{i=1}^{k}\left(  -1\right)
^{i}\mathbf{h}_{n-k+i}\left(  \mathbf{s}_{\left(  n-k-j+1,1^{i-1}\right)
}\right)  ^{\perp}\mathbf{s}_{\lambda}.
\]

\end{theorem}

Before we prove this theorem, we need several auxiliary results. First, we
recall one of the Pieri rules (\cite[(2.7.1)]{GriRei18}):

\begin{proposition}
\label{prop.pieri.h}Let $\lambda$ be a partition, and let $i\in\mathbb{N}$.
Then,%
\[
\mathbf{s}_{\lambda}\mathbf{h}_{i}=\sum_{\substack{\mu\text{ is a
partition;}\\\mu/\lambda\text{ is a horizontal }i\text{-strip}}}\mathbf{s}%
_{\mu}.
\]

\end{proposition}

From this, we can easily derive the following:

\begin{corollary}
\label{cor.pieri.hskew}Let $\lambda$ be a partition, and let $i\in\mathbb{N}$.
Then,%
\[
\left(  \mathbf{h}_{i}\right)  ^{\perp}\mathbf{s}_{\lambda}=\sum
_{\substack{\mu\text{ is a partition;}\\\lambda/\mu\text{ is a horizontal
}i\text{-strip}}}\mathbf{s}_{\mu}.
\]

\end{corollary}

Corollary \ref{cor.pieri.hskew} is also proven in \cite[(2.8.3)]{GriRei18}.

Next, let us show some further lemmas:

\begin{lemma}
\label{lem.pieri-hall.Lam1}Let $\lambda\in P_{k,n}$. Let $j\in\left\{
0,1,\ldots,n-k\right\}  $. Let $g$ be a positive integer. Then,%
\[
\sum_{\substack{\mu\text{ is a partition;}\\\mu_{1}=n-k+g;\\\mu/\lambda\text{
is a horizontal }j\text{-strip}}}\mathbf{s}_{\mu}=\sum_{w\geq1}\left(
-1\right)  ^{w-g}\mathbf{h}_{n-k+w}\left(  \mathbf{h}_{n-k+g-j}\mathbf{e}%
_{w-g}\right)  ^{\perp}\mathbf{s}_{\lambda}.
\]

\end{lemma}

\begin{proof}
[Proof of Lemma \ref{lem.pieri-hall.Lam1} (sketched).]First, we observe that
$\lambda_{1}\leq n-k$ (since $\lambda\in P_{k,n}$). Now, every partition $\mu$
satisfying $\mu_{1}=n-k+g$ must automatically satisfy $\mu_{1}\geq\lambda_{1}$
(because $\mu_{1}=n-k+\underbrace{g}_{\geq0}\geq n-k\geq\lambda_{1}$).

Let $A$ be the set of all partitions $\mu$ such that $\mu_{1}=n-k+g$ and such
that $\mu/\lambda$ is a horizontal $j$-strip. Let $B$ be the set of all
partitions $\nu$ such that $\lambda/\nu$ is a horizontal $\left(
n-k+g-j\right)  $-strip. Then,\footnote{We are using Definition
\ref{def.omega-and-complement} \textbf{(c)} here.}%
\begin{align*}
A  &  =\left\{  \mu\text{ is a partition}\ \mid\ \mu_{1}=n-k+g\text{ and
}\left\vert \mu\right\vert -\left\vert \lambda\right\vert =j\right. \\
&  \ \ \ \ \ \ \ \ \ \ \ \ \ \ \ \ \ \ \ \ \left.  \text{and }\mu_{1}%
\geq\lambda_{1}\geq\mu_{2}\geq\lambda_{2}\geq\mu_{3}\geq\lambda_{3}\geq
\cdots\right\} \\
&  =\left\{  \mu\text{ is a partition}\ \mid\ \mu_{1}=n-k+g\text{ and
}\left\vert \mu\right\vert -\left\vert \lambda\right\vert =j\right. \\
&  \ \ \ \ \ \ \ \ \ \ \ \ \ \ \ \ \ \ \ \ \left.  \text{and }\lambda_{1}%
\geq\mu_{2}\geq\lambda_{2}\geq\mu_{3}\geq\lambda_{3}\geq\mu_{4}\geq
\cdots\right\}
\end{align*}
(since every partition $\mu$ satisfying $\mu_{1}=n-k+g$ must automatically
satisfy $\mu_{1}\geq\lambda_{1}$) and%
\begin{align*}
B  &  =\left\{  \nu\text{ is a partition}\ \mid\ \left\vert \lambda\right\vert
-\left\vert \nu\right\vert =n-k+g-j\right. \\
&  \ \ \ \ \ \ \ \ \ \ \ \ \ \ \ \ \ \ \ \ \left.  \text{and }\lambda_{1}%
\geq\nu_{1}\geq\lambda_{2}\geq\nu_{2}\geq\lambda_{3}\geq\nu_{3}\geq
\cdots\right\}  .
\end{align*}
Hence, it is easy to check that the map%
\begin{align*}
B  &  \rightarrow A,\\
\nu &  \mapsto\left(  n-k+g,\nu_{1},\nu_{2},\nu_{3},\ldots\right)
\end{align*}
is well-defined (because every $\nu\in B$ satisfies $\lambda_{1}\geq\nu_{1}$
and thus $n-k+\underbrace{g}_{\geq0}\geq n-k\geq\lambda_{1}\geq\nu_{1}$) and
is a bijection (its inverse map just sends each $\mu\in A$ to $\left(  \mu
_{2},\mu_{3},\mu_{4},\ldots\right)  \in B$). Thus, we can substitute $\left(
n-k+g,\nu_{1},\nu_{2},\nu_{3},\ldots\right)  $ for $\mu$ in the sum $\sum
_{\mu\in A}\mathbf{s}_{\mu}$. We thus obtain%
\begin{equation}
\sum_{\mu\in A}\mathbf{s}_{\mu}=\sum_{\nu\in B}\mathbf{s}_{\left(
n-k+g,\nu_{1},\nu_{2},\nu_{3},\ldots\right)  }.
\label{pf.lem.pieri-hall.Lam1.1}%
\end{equation}

But each $\nu\in B$ satisfies $n-k+\underbrace{g}_{\geq0}\geq n-k\geq
\lambda_{1}\geq\nu_{1}$ and thus%
\begin{equation}
\sum_{i\in\mathbb{N}}\left(  -1\right)  ^{i}\mathbf{h}_{n-k+g+i}\left(
\mathbf{e}_{i}\right)  ^{\perp}\mathbf{s}_{\nu}=\mathbf{s}_{\left(
n-k+g,\nu_{1},\nu_{2},\nu_{3},\ldots\right)  }
\label{pf.lem.pieri-hall.Lam1.2}%
\end{equation}
(by Proposition \ref{prop.bernstein.1}, applied to $\nu$ and $n-k+g$ instead
of $\lambda$ and $m$). Hence, (\ref{pf.lem.pieri-hall.Lam1.1}) becomes%
\begin{align}
\sum_{\mu\in A}\mathbf{s}_{\mu}  &  =\sum_{\nu\in B}\underbrace{\mathbf{s}%
_{\left(  n-k+g,\nu_{1},\nu_{2},\nu_{3},\ldots\right)  }}_{\substack{=\sum
_{i\in\mathbb{N}}\left(  -1\right)  ^{i}\mathbf{h}_{n-k+g+i}\left(
\mathbf{e}_{i}\right)  ^{\perp}\mathbf{s}_{\nu}\\\text{(by
(\ref{pf.lem.pieri-hall.Lam1.2}))}}}=\sum_{\nu\in B}\sum_{i\in\mathbb{N}%
}\left(  -1\right)  ^{i}\mathbf{h}_{n-k+g+i}\left(  \mathbf{e}_{i}\right)
^{\perp}\mathbf{s}_{\nu}\nonumber\\
&  =\sum_{i\in\mathbb{N}}\left(  -1\right)  ^{i}\mathbf{h}_{n-k+g+i}\left(
\mathbf{e}_{i}\right)  ^{\perp}\left(  \sum_{\nu\in B}\mathbf{s}_{\nu}\right)
. \label{pf.lem.pieri-hall.Lam1.3}%
\end{align}

But Corollary \ref{cor.pieri.hskew} (applied to $i=n-k+g-j$)
yields\footnote{More precisely: This follows from Corollary
\ref{cor.pieri.hskew} (applied to $i=n-k+g-j$) when $n-k+g-j\in\mathbb{N}$.
But otherwise, it is obvious for trivial reasons ($0=0$).}%
\begin{align}
\left(  \mathbf{h}_{n-k+g-j}\right)  ^{\perp}\mathbf{s}_{\lambda}  &
=\underbrace{\sum_{\substack{\mu\text{ is a partition;}\\\lambda/\mu\text{ is
a horizontal }\left(  n-k+g-j\right)  \text{-strip}}}}_{\substack{=\sum
_{\mu\in B}\\\text{(by the definition of }B\text{)}}}\mathbf{s}_{\mu}%
=\sum_{\mu\in B}\mathbf{s}_{\mu}\nonumber\\
&  =\sum_{\nu\in B}\mathbf{s}_{\nu}. \label{pf.lem.pieri-hall.Lam1.4}%
\end{align}
Hence, (\ref{pf.lem.pieri-hall.Lam1.3}) becomes%
\begin{align*}
\sum_{\mu\in A}\mathbf{s}_{\mu}  &  =\sum_{i\in\mathbb{N}}\left(  -1\right)
^{i}\mathbf{h}_{n-k+g+i}\left(  \mathbf{e}_{i}\right)  ^{\perp}\left(
\underbrace{\sum_{\nu\in B}\mathbf{s}_{\nu}}_{\substack{=\left(
\mathbf{h}_{n-k+g-j}\right)  ^{\perp}\mathbf{s}_{\lambda}\\\text{(by
(\ref{pf.lem.pieri-hall.Lam1.4}))}}}\right) \\
&  =\sum_{i\in\mathbb{N}}\left(  -1\right)  ^{i}\mathbf{h}_{n-k+g+i}%
\underbrace{\left(  \mathbf{e}_{i}\right)  ^{\perp}\left(  \left(
\mathbf{h}_{n-k+g-j}\right)  ^{\perp}\mathbf{s}_{\lambda}\right)  }_{=\left(
\left(  \mathbf{e}_{i}\right)  ^{\perp}\circ\left(  \mathbf{h}_{n-k+g-j}%
\right)  ^{\perp}\right)  \mathbf{s}_{\lambda}}\\
&  =\sum_{i\in\mathbb{N}}\left(  -1\right)  ^{i}\mathbf{h}_{n-k+g+i}%
\underbrace{\left(  \left(  \mathbf{e}_{i}\right)  ^{\perp}\circ\left(
\mathbf{h}_{n-k+g-j}\right)  ^{\perp}\right)  }_{\substack{=\left(
\mathbf{h}_{n-k+g-j}\mathbf{e}_{i}\right)  ^{\perp}\\\text{(since
(\ref{eq.skewing.fg})}\\\text{yields }\left(  \mathbf{h}_{n-k+g-j}%
\mathbf{e}_{i}\right)  ^{\perp}=\left(  \mathbf{e}_{i}\right)  ^{\perp}%
\circ\left(  \mathbf{h}_{n-k+g-j}\right)  ^{\perp}\text{)}}}\mathbf{s}%
_{\lambda}\\
&  =\sum_{i\in\mathbb{N}}\left(  -1\right)  ^{i}\mathbf{h}_{n-k+g+i}\left(
\mathbf{h}_{n-k+g-j}\mathbf{e}_{i}\right)  ^{\perp}\mathbf{s}_{\lambda}\\
&  =\sum_{w\geq g}\left(  -1\right)  ^{w-g}\mathbf{h}_{n-k+w}\left(
\mathbf{h}_{n-k+g-j}\mathbf{e}_{w-g}\right)  ^{\perp}\mathbf{s}_{\lambda}\\
&  \ \ \ \ \ \ \ \ \ \ \left(  \text{here, we have substituted }w-g\text{ for
}i\text{ in the sum}\right)  .
\end{align*}
Comparing this with%
\begin{align*}
&  \sum_{w\geq1}\left(  -1\right)  ^{w-g}\mathbf{h}_{n-k+w}\left(
\mathbf{h}_{n-k+g-j}\mathbf{e}_{w-g}\right)  ^{\perp}\mathbf{s}_{\lambda}\\
&  =\sum_{w=1}^{g-1}\left(  -1\right)  ^{w-g}\mathbf{h}_{n-k+w}\left(
\mathbf{h}_{n-k+g-j}\underbrace{\mathbf{e}_{w-g}}_{\substack{=0\\\text{(since
}w-g<0\\\text{(since }w\leq g-1<g\text{))}}}\right)  ^{\perp}\mathbf{s}%
_{\lambda}\\
&  \ \ \ \ \ \ \ \ \ \ +\sum_{w\geq g}\left(  -1\right)  ^{w-g}\mathbf{h}%
_{n-k+w}\left(  \mathbf{h}_{n-k+g-j}\mathbf{e}_{w-g}\right)  ^{\perp
}\mathbf{s}_{\lambda}\\
&  \ \ \ \ \ \ \ \ \ \ \left(  \text{since }g\text{ is a positive
integer}\right) \\
&  =\underbrace{\sum_{w=1}^{g-1}\left(  -1\right)  ^{w-g}\mathbf{h}%
_{n-k+w}\left(  \mathbf{h}_{n-k+g-j}0\right)  ^{\perp}\mathbf{s}_{\lambda}%
}_{=0}+\sum_{w\geq g}\left(  -1\right)  ^{w-g}\mathbf{h}_{n-k+w}\left(
\mathbf{h}_{n-k+g-j}\mathbf{e}_{w-g}\right)  ^{\perp}\mathbf{s}_{\lambda}\\
&  =\sum_{w\geq g}\left(  -1\right)  ^{w-g}\mathbf{h}_{n-k+w}\left(
\mathbf{h}_{n-k+g-j}\mathbf{e}_{w-g}\right)  ^{\perp}\mathbf{s}_{\lambda},
\end{align*}
we obtain%
\begin{equation}
\sum_{\mu\in A}\mathbf{s}_{\mu}=\sum_{w\geq1}\left(  -1\right)  ^{w-g}%
\mathbf{h}_{n-k+w}\left(  \mathbf{h}_{n-k+g-j}\mathbf{e}_{w-g}\right)
^{\perp}\mathbf{s}_{\lambda}. \label{pf.lem.pieri-hall.Lam1.6}%
\end{equation}
In view of%
\[
\sum_{\mu\in A}=\sum_{\substack{\mu\text{ is a partition;}\\\mu_{1}%
=n-k+g;\\\mu/\lambda\text{ is a horizontal }j\text{-strip}}%
}\ \ \ \ \ \ \ \ \ \ \left(  \text{by the definition of }A\right)  ,
\]
this rewrites as%
\[
\sum_{\substack{\mu\text{ is a partition;}\\\mu_{1}=n-k+g;\\\mu/\lambda\text{
is a horizontal }j\text{-strip}}}\mathbf{s}_{\mu}=\sum_{w\geq1}\left(
-1\right)  ^{w-g}\mathbf{h}_{n-k+w}\left(  \mathbf{h}_{n-k+g-j}\mathbf{e}%
_{w-g}\right)  ^{\perp}\mathbf{s}_{\lambda}.
\]
This proves Lemma \ref{lem.pieri-hall.Lam1}.
\end{proof}

Our next lemma will be a slight generalization of Lemma \ref{lem.sm1j.1}; but
first we extend our definition of $\mathbf{s}_{\left(  m,1^{j}\right)  }$:

\begin{convention}
\label{conv.sm1j.jneg}Let $m\in\mathbb{N}$, and let $j$ be a negative integer.
Then, we shall understand the (otherwise undefined) expression $\mathbf{s}%
_{\left(  m,1^{j}\right)  }$ to mean $0\in\Lambda$.
\end{convention}

We can now generalize Lemma \ref{lem.sm1j.1} as follows:

\begin{lemma}
\label{lem.sm1j.1gen}Let $m$ be a positive integer. Let $j\in\mathbb{Z}$ be
such that $m+j>0$. Then,%
\[
\mathbf{s}_{\left(  m,1^{j}\right)  }=\sum_{i=1}^{m}\left(  -1\right)
^{i-1}\mathbf{h}_{m-i}\mathbf{e}_{j+i}.
\]

\end{lemma}

\begin{proof}
[Proof of Lemma \ref{lem.sm1j.1gen}.]If $j\in\mathbb{N}$, then this follows
directly from Lemma \ref{lem.sm1j.1}. Hence, for the rest of this proof, we
WLOG assume that $j\notin\mathbb{N}$. Hence, $j<0$. Now, the proof of Lemma
\ref{lem.sm1j.1gen} is the same as our above proof of Lemma \ref{lem.sm1j.1},
with two changes:

\begin{itemize}
\item The inequality $m+j>0$ no longer follows from $m>0$ and $j\geq0$, but
rather comes straight from the assumptions.

\item The equality $\sum_{i=0}^{j}\left(  -1\right)  ^{i}\mathbf{h}%
_{m+i}\mathbf{e}_{j-i}=\mathbf{s}_{\left(  m,1^{j}\right)  }$ no longer
follows from (\ref{pf.prop.redh.Lam.1}), but rather comes from comparing
$\sum_{i=0}^{j}\left(  -1\right)  ^{i}\mathbf{h}_{m+i}\mathbf{e}_{j-i}=\left(
\text{empty sum}\right)  =0$ with $\mathbf{s}_{\left(  m,1^{j}\right)  }=0$.
\end{itemize}

Thus, Lemma \ref{lem.sm1j.1gen} is proven.
\end{proof}

\begin{lemma}
\label{lem.pieri-hall.Lam2}Let $j\in\left\{  0,1,\ldots,n-k\right\}  $, and
let $w$ be a positive integer. Then,%
\[
\sum_{g=1}^{j}\left(  -1\right)  ^{w-g}\mathbf{h}_{n-k+g-j}\mathbf{e}%
_{w-g}=\left(  -1\right)  ^{w-j}\mathbf{s}_{\left(  n-k+1,1^{w-j-1}\right)
}-\left(  -1\right)  ^{w}\mathbf{s}_{\left(  n-k-j+1,1^{w-1}\right)  }.
\]

\end{lemma}

\begin{proof}
[Proof of Lemma \ref{lem.pieri-hall.Lam2}.]From $j\in\left\{  0,1,\ldots
,n-k\right\}  $, we obtain $0\leq j\leq n-k\leq n-k+1$.

We have $\underbrace{n}_{\geq k}-k+1\geq k-k+1=1$; thus, $n-k+1$ is a positive
integer. Also, $\left(  n-k+1\right)  +\left(  w-\underbrace{j}_{\leq
n-k}-1\right)  \geq\left(  n-k+1\right)  +\left(  w-\left(  n-k\right)
-1\right)  =w>0$ (since $w$ is a positive integer). Hence, Lemma
\ref{lem.sm1j.1gen} (applied to $n-k+1$ and $w-j-1$ instead of $m$ and $j$)
yields%
\begin{align}
\mathbf{s}_{\left(  n-k+1,1^{w-j-1}\right)  } &  =\sum_{i=1}^{n-k+1}\left(
-1\right)  ^{i-1}\mathbf{h}_{n-k+1-i}\mathbf{e}_{w-j-1+i}\nonumber\\
&  =\sum_{i=1}^{j}\left(  -1\right)  ^{i-1}\mathbf{h}_{n-k+1-i}\mathbf{e}%
_{w-j-1+i}+\sum_{i=j+1}^{n-k+1}\left(  -1\right)  ^{i-1}\mathbf{h}%
_{n-k+1-i}\mathbf{e}_{w-j-1+i}\label{pf.lem.pieri-hall.Lam2.1}%
\end{align}
(since $0\leq j\leq n-k+1$). Also, $n-k-\underbrace{j}_{\leq n-k}+1\geq
n-k-\left(  n-k\right)  +1=1$; thus, $n-k-j+1$ is a positive integer.
Furthermore, $w-1\in\mathbb{N}$ (since $w$ is a positive integer). Hence,
Lemma \ref{lem.sm1j.1} (applied to $n-k-j+1$ and $w-1$ instead of $m$ and $j$)
yields%
\begin{align*}
\mathbf{s}_{\left(  n-k-j+1,1^{w-1}\right)  } &  =\sum_{i=1}^{n-k-j+1}\left(
-1\right)  ^{i-1}\mathbf{h}_{n-k-j+1-i}\mathbf{e}_{w-1+i}\\
&  =\sum_{i=j+1}^{n-k+1}\underbrace{\left(  -1\right)  ^{i-j-1}}_{=\left(
-1\right)  ^{j}\left(  -1\right)  ^{i-1}}\underbrace{\mathbf{h}%
_{n-k-j+1-\left(  i-j\right)  }}_{=\mathbf{h}_{n-k+1-i}}\underbrace{\mathbf{e}%
_{w-1+i-j}}_{=\mathbf{e}_{w-j-1+i}}\\
&  \ \ \ \ \ \ \ \ \ \ \left(  \text{here, we have substituted }i-j\text{ for
}i\text{ in the sum}\right)  \\
&  =\left(  -1\right)  ^{j}\sum_{i=j+1}^{n-k+1}\left(  -1\right)
^{i-1}\mathbf{h}_{n-k+1-i}\mathbf{e}_{w-j-1+i}.
\end{align*}
Multiplying this equality by $\left(  -1\right)  ^{j}$, we find%
\[
\left(  -1\right)  ^{j}\mathbf{s}_{\left(  n-k-j+1,1^{w-1}\right)  }%
=\sum_{i=j+1}^{n-k+1}\left(  -1\right)  ^{i-1}\mathbf{h}_{n-k+1-i}%
\mathbf{e}_{w-j-1+i}.
\]
Subtracting this equality from (\ref{pf.lem.pieri-hall.Lam2.1}), we obtain%
\begin{align}
&  \mathbf{s}_{\left(  n-k+1,1^{w-j-1}\right)  }-\left(  -1\right)
^{j}\mathbf{s}_{\left(  n-k-j+1,1^{w-1}\right)  }\nonumber\\
&  =\left(  \sum_{i=1}^{j}\left(  -1\right)  ^{i-1}\mathbf{h}_{n-k+1-i}%
\mathbf{e}_{w-j-1+i}+\sum_{i=j+1}^{n-k+1}\left(  -1\right)  ^{i-1}%
\mathbf{h}_{n-k+1-i}\mathbf{e}_{w-j-1+i}\right)  \nonumber\\
&  \ \ \ \ \ \ \ \ \ \ -\sum_{i=j+1}^{n-k+1}\left(  -1\right)  ^{i-1}%
\mathbf{h}_{n-k+1-i}\mathbf{e}_{w-j-1+i}\nonumber\\
&  =\sum_{i=1}^{j}\left(  -1\right)  ^{i-1}\mathbf{h}_{n-k+1-i}\mathbf{e}%
_{w-j-1+i}.\label{pf.lem.pieri-hall.Lam2.4}%
\end{align}

On the other hand,%
\begin{align*}
&  \sum_{g=1}^{j}\left(  -1\right)  ^{w-g}\mathbf{h}_{n-k+g-j}\mathbf{e}%
_{w-g}\\
&  =\sum_{i=1}^{j}\underbrace{\left(  -1\right)  ^{w-\left(  j+1-i\right)  }%
}_{=\left(  -1\right)  ^{w-j}\left(  -1\right)  ^{i-1}}\underbrace{\mathbf{h}%
_{n-k+\left(  j+1-i\right)  -j}}_{=\mathbf{h}_{n-k+1-i}}\underbrace{\mathbf{e}%
_{w-\left(  j+1-i\right)  }}_{=\mathbf{e}_{w-j-1+i}}\\
&  \ \ \ \ \ \ \ \ \ \ \left(  \text{here, we have substituted }j+1-i\text{
for }g\text{ in the sum}\right) \\
&  =\left(  -1\right)  ^{w-j}\underbrace{\sum_{i=1}^{j}\left(  -1\right)
^{i-1}\mathbf{h}_{n-k+1-i}\mathbf{e}_{w-j-1+i}}_{\substack{=\mathbf{s}%
_{\left(  n-k+1,1^{w-j-1}\right)  }-\left(  -1\right)  ^{j}\mathbf{s}_{\left(
n-k-j+1,1^{w-1}\right)  }\\\text{(by (\ref{pf.lem.pieri-hall.Lam2.4}))}}}\\
&  =\left(  -1\right)  ^{w-j}\left(  \mathbf{s}_{\left(  n-k+1,1^{w-j-1}%
\right)  }-\left(  -1\right)  ^{j}\mathbf{s}_{\left(  n-k-j+1,1^{w-1}\right)
}\right) \\
&  =\left(  -1\right)  ^{w-j}\mathbf{s}_{\left(  n-k+1,1^{w-j-1}\right)
}-\underbrace{\left(  -1\right)  ^{w-j}\left(  -1\right)  ^{j}}_{=\left(
-1\right)  ^{w}}\mathbf{s}_{\left(  n-k-j+1,1^{w-1}\right)  }\\
&  =\left(  -1\right)  ^{w-j}\mathbf{s}_{\left(  n-k+1,1^{w-j-1}\right)
}-\left(  -1\right)  ^{w}\mathbf{s}_{\left(  n-k-j+1,1^{w-1}\right)  }.
\end{align*}
This proves Lemma \ref{lem.pieri-hall.Lam2}.
\end{proof}

\begin{proof}
[Proof of Theorem \ref{thm.pieri-hall.Lam}.]We have $j\in\left\{
0,1,\ldots,n-k\right\}  $, thus $0\leq j\leq n-k$. Also, we have $\lambda\in
P_{k,n}$; thus, the partition $\lambda$ has at most $k$ parts and satisfies
$\lambda_{1}\leq n-k$.

Let $g$ be an integer such that $g\geq j+1$. If $\mu$ is a partition such that
$\mu/\lambda$ is a horizontal $j$-strip, then $\mu_{1}\leq\underbrace{\lambda
_{1}}_{\leq n-k}+j\leq n-k+\underbrace{j}_{<j+1\leq g}<n-k+g$ and thus
$\mu_{1}\neq n-k+g$. Thus, there exists no partition $\mu$ such that $\mu
_{1}=n-k+g$ and such that $\mu/\lambda$ is a horizontal $j$-strip. Hence,
\begin{equation}
\sum_{\substack{\mu\text{ is a partition;}\\\mu_{1}=n-k+g;\\\mu/\lambda\text{
is a horizontal }j\text{-strip}}}\mathbf{s}_{\mu}=\left(  \text{empty
sum}\right)  =0. \label{pf.thm.pieri-hall.Lam.1}%
\end{equation}

Now, forget that we fixed $g$. We thus have proven the equality
(\ref{pf.thm.pieri-hall.Lam.1}) for every integer $g$ satisfying $g\geq j+1$.

On the other hand, let $g\in\left\{  1,2,\ldots,j\right\}  $. Thus, $g\leq
j\leq n-k$. If $w$ is an integer satisfying $w\geq n+1$, then $\underbrace{w}%
_{\geq n+1}-\underbrace{g}_{\leq n-k}\geq\left(  n+1\right)  -\left(
n-k\right)  =k+1>k$, and thus the partition $\left(  1^{w-g}\right)  $ does
\textbf{not} satisfy $\left(  1^{w-g}\right)  \subseteq\lambda$ (because the
partition $\lambda$ has at most $k$ parts, whereas the partition $\left(
1^{w-g}\right)  $ has $w-g>k$ parts), and therefore we have
\begin{align}
\left(  \underbrace{\mathbf{e}_{w-g}}_{=\mathbf{s}_{\left(  1^{w-g}\right)  }%
}\right)  ^{\perp}\left(  \mathbf{s}_{\lambda}\right)   &  =\left(
\mathbf{s}_{\left(  1^{w-g}\right)  }\right)  ^{\perp}\left(  \mathbf{s}%
_{\lambda}\right)  =\mathbf{s}_{\lambda/\left(  1^{w-g}\right)  }%
\ \ \ \ \ \ \ \ \ \ \left(  \text{by (\ref{eq.skewing.ss})}\right) \nonumber\\
&  =0\ \ \ \ \ \ \ \ \ \ \left(  \text{since we don't have }\left(
1^{w-g}\right)  \subseteq\lambda\right)  . \label{pf.thm.pieri-hall.Lam.2a}%
\end{align}
Hence, if $w$ is an integer satisfying $w\geq n+1$, then%
\begin{align}
\left(  \underbrace{\mathbf{h}_{n-k+g-j}\mathbf{e}_{w-g}}_{=\mathbf{e}%
_{w-g}\mathbf{h}_{n-k+g-j}}\right)  ^{\perp}\mathbf{s}_{\lambda}  &
=\underbrace{\left(  \mathbf{e}_{w-g}\mathbf{h}_{n-k+g-j}\right)  ^{\perp}%
}_{\substack{=\left(  \mathbf{h}_{n-k+g-j}\right)  ^{\perp}\circ\left(
\mathbf{e}_{w-g}\right)  ^{\perp}\\\text{(by (\ref{eq.skewing.fg}))}%
}}\mathbf{s}_{\lambda}=\left(  \left(  \mathbf{h}_{n-k+g-j}\right)  ^{\perp
}\circ\left(  \mathbf{e}_{w-g}\right)  ^{\perp}\right)  \left(  \mathbf{s}%
_{\lambda}\right) \nonumber\\
&  =\left(  \mathbf{h}_{n-k+g-j}\right)  ^{\perp}\underbrace{\left(  \left(
\mathbf{e}_{w-g}\right)  ^{\perp}\left(  \mathbf{s}_{\lambda}\right)  \right)
}_{\substack{=0\\\text{(by (\ref{pf.thm.pieri-hall.Lam.2a}))}}}=0.
\label{pf.thm.pieri-hall.Lam.2b}%
\end{align}

Now,%
\begin{align}
&  \sum_{\substack{\mu\text{ is a partition;}\\\mu_{1}=n-k+g;\\\mu
/\lambda\text{ is a horizontal }j\text{-strip}}}\mathbf{s}_{\mu}\nonumber\\
&  =\sum_{w\geq1}\left(  -1\right)  ^{w-g}\mathbf{h}_{n-k+w}\left(
\mathbf{h}_{n-k+g-j}\mathbf{e}_{w-g}\right)  ^{\perp}\mathbf{s}_{\lambda
}\ \ \ \ \ \ \ \ \ \ \left(  \text{by Lemma \ref{lem.pieri-hall.Lam1}}\right)
\nonumber\\
&  =\sum_{w=1}^{n}\left(  -1\right)  ^{w-g}\mathbf{h}_{n-k+w}\left(
\mathbf{h}_{n-k+g-j}\mathbf{e}_{w-g}\right)  ^{\perp}\mathbf{s}_{\lambda
}\nonumber\\
&  \ \ \ \ \ \ \ \ \ \ +\sum_{w\geq n+1}\left(  -1\right)  ^{w-g}%
\mathbf{h}_{n-k+w}\underbrace{\left(  \mathbf{h}_{n-k+g-j}\mathbf{e}%
_{w-g}\right)  ^{\perp}\mathbf{s}_{\lambda}}_{\substack{=0\\\text{(by
(\ref{pf.thm.pieri-hall.Lam.2b}))}}}\nonumber\\
&  =\sum_{w=1}^{n}\left(  -1\right)  ^{w-g}\mathbf{h}_{n-k+w}\left(
\mathbf{h}_{n-k+g-j}\mathbf{e}_{w-g}\right)  ^{\perp}\mathbf{s}_{\lambda}.
\label{pf.thm.pieri-hall.Lam.3}%
\end{align}

Now, forget that we fixed $g$. We thus have proven the equality
(\ref{pf.thm.pieri-hall.Lam.3}) for each $g\in\left\{  1,2,\ldots,j\right\}  $.

Proposition \ref{prop.pieri.h} (applied to $i=j$) yields%
\[
\mathbf{s}_{\lambda}\mathbf{h}_{j}=\sum_{\substack{\mu\text{ is a
partition;}\\\mu/\lambda\text{ is a horizontal }j\text{-strip}}}\mathbf{s}%
_{\mu}=\sum_{\substack{\mu\text{ is a partition;}\\\mu_{1}\leq n-k;\\\mu
/\lambda\text{ is a horizontal }j\text{-strip}}}\mathbf{s}_{\mu}%
+\sum_{\substack{\mu\text{ is a partition;}\\\mu_{1}>n-k;\\\mu/\lambda\text{
is a horizontal }j\text{-strip}}}\mathbf{s}_{\mu}%
\]
(since each partition $\mu$ satisfies either $\mu_{1}\leq n-k$ or $\mu
_{1}>n-k$). Hence,%
\begin{align}
&  \mathbf{s}_{\lambda}\mathbf{h}_{j}-\sum_{\substack{\mu\text{ is a
partition;}\\\mu_{1}\leq n-k;\\\mu/\lambda\text{ is a horizontal
}j\text{-strip}}}\mathbf{s}_{\mu}\nonumber\\
&  =\sum_{\substack{\mu\text{ is a partition;}\\\mu_{1}>n-k;\\\mu
/\lambda\text{ is a horizontal }j\text{-strip}}}\mathbf{s}_{\mu}=\sum_{g\geq
1}\sum_{\substack{\mu\text{ is a partition;}\\\mu_{1}=n-k+g;\\\mu
/\lambda\text{ is a horizontal }j\text{-strip}}}\mathbf{s}_{\mu}\nonumber\\
&  \ \ \ \ \ \ \ \ \ \ \left(
\begin{array}
[c]{c}%
\text{because the partitions }\mu\text{ satisfying }\mu_{1}>n-k\text{ are
precisely}\\
\text{the partitions }\mu\text{ satisfying }\mu_{1}=n-k+g\text{ for some
}g\geq1\text{,}\\
\text{and moreover the }g\text{ is uniquely determined by the partition}%
\end{array}
\right) \nonumber\\
&  =\sum_{g=1}^{j}\underbrace{\sum_{\substack{\mu\text{ is a partition;}%
\\\mu_{1}=n-k+g;\\\mu/\lambda\text{ is a horizontal }j\text{-strip}%
}}\mathbf{s}_{\mu}}_{\substack{=\sum_{w=1}^{n}\left(  -1\right)
^{w-g}\mathbf{h}_{n-k+w}\left(  \mathbf{h}_{n-k+g-j}\mathbf{e}_{w-g}\right)
^{\perp}\mathbf{s}_{\lambda}\\\text{(by (\ref{pf.thm.pieri-hall.Lam.3}))}%
}}+\sum_{g=j+1}^{\infty}\underbrace{\sum_{\substack{\mu\text{ is a
partition;}\\\mu_{1}=n-k+g;\\\mu/\lambda\text{ is a horizontal }%
j\text{-strip}}}\mathbf{s}_{\mu}}_{\substack{=0\\\text{(by
(\ref{pf.thm.pieri-hall.Lam.1}))}}}\nonumber\\
&  =\underbrace{\sum_{g=1}^{j}\sum_{w=1}^{n}}_{=\sum_{w=1}^{n}\sum_{g=1}^{j}%
}\left(  -1\right)  ^{w-g}\mathbf{h}_{n-k+w}\left(  \mathbf{h}_{n-k+g-j}%
\mathbf{e}_{w-g}\right)  ^{\perp}\mathbf{s}_{\lambda}\nonumber\\
&  =\sum_{w=1}^{n}\sum_{g=1}^{j}\left(  -1\right)  ^{w-g}\mathbf{h}%
_{n-k+w}\left(  \mathbf{h}_{n-k+g-j}\mathbf{e}_{w-g}\right)  ^{\perp
}\mathbf{s}_{\lambda}\nonumber\\
&  =\sum_{w=1}^{n}\mathbf{h}_{n-k+w}\left(  \underbrace{\sum_{g=1}^{j}\left(
-1\right)  ^{w-g}\mathbf{h}_{n-k+g-j}\mathbf{e}_{w-g}}_{\substack{=\left(
-1\right)  ^{w-j}\mathbf{s}_{\left(  n-k+1,1^{w-j-1}\right)  }-\left(
-1\right)  ^{w}\mathbf{s}_{\left(  n-k-j+1,1^{w-1}\right)  }\\\text{(by Lemma
\ref{lem.pieri-hall.Lam2})}}}\right)  ^{\perp}\mathbf{s}_{\lambda}\nonumber\\
&  =\sum_{w=1}^{n}\mathbf{h}_{n-k+w}\left(  \left(  -1\right)  ^{w-j}%
\mathbf{s}_{\left(  n-k+1,1^{w-j-1}\right)  }-\left(  -1\right)
^{w}\mathbf{s}_{\left(  n-k-j+1,1^{w-1}\right)  }\right)  ^{\perp}%
\mathbf{s}_{\lambda}\nonumber\\
&  =\sum_{w=1}^{n}\mathbf{h}_{n-k+w}\left(  -1\right)  ^{w-j}\left(
\mathbf{s}_{\left(  n-k+1,1^{w-j-1}\right)  }\right)  ^{\perp}\mathbf{s}%
_{\lambda}\nonumber\\
&  \ \ \ \ \ \ \ \ \ \ -\sum_{w=1}^{n}\mathbf{h}_{n-k+w}\left(  -1\right)
^{w}\left(  \mathbf{s}_{\left(  n-k-j+1,1^{w-1}\right)  }\right)  ^{\perp
}\mathbf{s}_{\lambda}. \label{pf.thm.pieri-hall.Lam.5}%
\end{align}

Next, we claim that%
\begin{equation}
\left(  \mathbf{s}_{\left(  n-k+1,1^{w-j-1}\right)  }\right)  ^{\perp
}\mathbf{s}_{\lambda}=0\ \ \ \ \ \ \ \ \ \ \text{for each }w\in\left\{
1,2,\ldots,n\right\}  . \label{pf.thm.pieri-hall.Lam.6}%
\end{equation}

[\textit{Proof of (\ref{pf.thm.pieri-hall.Lam.6}):} Let $w\in\left\{
1,2,\ldots,n\right\}  $. If $w-j-1$ is a negative integer, then $\mathbf{s}%
_{\left(  n-k+1,1^{w-j-1}\right)  }=0$ (by Convention \ref{conv.sm1j.jneg}),
and thus (\ref{pf.thm.pieri-hall.Lam.6}) holds in this case. Hence, for the
rest of this proof of (\ref{pf.thm.pieri-hall.Lam.6}), we WLOG assume that
$w-j-1$ is not a negative integer. Thus, $w-j-1\in\mathbb{N}$. Now, the
partition $\left(  n-k+1,1^{w-j-1}\right)  $ has a bigger first entry than the
partition $\lambda$ (since its first entry is $n-k+1>n-k\geq\lambda_{1}$).
Thus, we do not have $\left(  n-k+1,1^{w-j-1}\right)  \subseteq\lambda$.
Hence, $\mathbf{s}_{\lambda/\left(  n-k+1,1^{w-j-1}\right)  }=0$. But
(\ref{eq.skewing.ss}) yields $\left(  \mathbf{s}_{\left(  n-k+1,1^{w-j-1}%
\right)  }\right)  ^{\perp}\mathbf{s}_{\lambda}=\mathbf{s}_{\lambda/\left(
n-k+1,1^{w-j-1}\right)  }=0$. This proves (\ref{pf.thm.pieri-hall.Lam.6}).]

Next, we claim that%
\begin{equation}
\left(  \mathbf{s}_{\left(  n-k-j+1,1^{w-1}\right)  }\right)  ^{\perp
}\mathbf{s}_{\lambda}=0\ \ \ \ \ \ \ \ \ \ \text{for each }w\in\left\{
k+1,k+2,\ldots,n\right\}  . \label{pf.thm.pieri-hall.Lam.7}%
\end{equation}

[\textit{Proof of (\ref{pf.thm.pieri-hall.Lam.7}):} Let $w\in\left\{
k+1,k+2,\ldots,n\right\}  $. Then, $w\geq k+1$. Now, the number of parts of
the partition $\left(  n-k-j+1,1^{w-1}\right)  $ is $1+\left(  w-1\right)
=w\geq k+1>k$, which is bigger than the number of parts of $\lambda$ (since
$\lambda$ has at most $k$ parts). Hence, we don't have $\left(
n-k-j+1,1^{w-1}\right)  \subseteq\lambda$. Thus, $\mathbf{s}_{\lambda/\left(
n-k-j+1,1^{w-1}\right)  }=0$. But (\ref{eq.skewing.ss}) yields $\left(
\mathbf{s}_{\left(  n-k-j+1,1^{w-1}\right)  }\right)  ^{\perp}\mathbf{s}%
_{\lambda}=\mathbf{s}_{\lambda/\left(  n-k-j+1,1^{w-1}\right)  }=0$. This
proves (\ref{pf.thm.pieri-hall.Lam.7}).]

Now, (\ref{pf.thm.pieri-hall.Lam.5}) becomes%
\begin{align*}
&  \mathbf{s}_{\lambda}\mathbf{h}_{j}-\sum_{\substack{\mu\text{ is a
partition;}\\\mu_{1}\leq n-k;\\\mu/\lambda\text{ is a horizontal
}j\text{-strip}}}\mathbf{s}_{\mu}\\
&  =\sum_{w=1}^{n}\mathbf{h}_{n-k+w}\left(  -1\right)  ^{w-j}%
\underbrace{\left(  \mathbf{s}_{\left(  n-k+1,1^{w-j-1}\right)  }\right)
^{\perp}\mathbf{s}_{\lambda}}_{\substack{=0\\\text{(by
(\ref{pf.thm.pieri-hall.Lam.6}))}}}\\
&  \ \ \ \ \ \ \ \ \ \ -\sum_{w=1}^{n}\mathbf{h}_{n-k+w}\left(  -1\right)
^{w}\left(  \mathbf{s}_{\left(  n-k-j+1,1^{w-1}\right)  }\right)  ^{\perp
}\mathbf{s}_{\lambda}\\
&  =-\sum_{w=1}^{n}\mathbf{h}_{n-k+w}\left(  -1\right)  ^{w}\left(
\mathbf{s}_{\left(  n-k-j+1,1^{w-1}\right)  }\right)  ^{\perp}\mathbf{s}%
_{\lambda}\\
&  =-\left(  \sum_{w=1}^{k}\mathbf{h}_{n-k+w}\left(  -1\right)  ^{w}\left(
\mathbf{s}_{\left(  n-k-j+1,1^{w-1}\right)  }\right)  ^{\perp}\mathbf{s}%
_{\lambda}\right. \\
&  \ \ \ \ \ \ \ \ \ \ \ \ \ \ \ \ \ \ \ \ \left.  +\sum_{w=k+1}^{n}%
\mathbf{h}_{n-k+w}\left(  -1\right)  ^{w}\underbrace{\left(  \mathbf{s}%
_{\left(  n-k-j+1,1^{w-1}\right)  }\right)  ^{\perp}\mathbf{s}_{\lambda}%
}_{\substack{=0\\\text{(by (\ref{pf.thm.pieri-hall.Lam.7}))}}}\right) \\
&  \ \ \ \ \ \ \ \ \ \ \left(  \text{since }0\leq k\leq n\right) \\
&  =-\sum_{w=1}^{k}\mathbf{h}_{n-k+w}\left(  -1\right)  ^{w}\left(
\mathbf{s}_{\left(  n-k-j+1,1^{w-1}\right)  }\right)  ^{\perp}\mathbf{s}%
_{\lambda}\\
&  =-\sum_{w=1}^{k}\left(  -1\right)  ^{w}\mathbf{h}_{n-k+w}\left(
\mathbf{s}_{\left(  n-k-j+1,1^{w-1}\right)  }\right)  ^{\perp}\mathbf{s}%
_{\lambda}\\
&  =-\sum_{i=1}^{k}\left(  -1\right)  ^{i}\mathbf{h}_{n-k+i}\left(
\mathbf{s}_{\left(  n-k-j+1,1^{i-1}\right)  }\right)  ^{\perp}\mathbf{s}%
_{\lambda}%
\end{align*}
(here, we have renamed the summation index $w$ as $i$). Hence,%
\[
\mathbf{s}_{\lambda}\mathbf{h}_{j}=\sum_{\substack{\mu\text{ is a
partition;}\\\mu_{1}\leq n-k;\\\mu/\lambda\text{ is a horizontal
}j\text{-strip}}}\mathbf{s}_{\mu}-\sum_{i=1}^{k}\left(  -1\right)
^{i}\mathbf{h}_{n-k+i}\left(  \mathbf{s}_{\left(  n-k-j+1,1^{i-1}\right)
}\right)  ^{\perp}\mathbf{s}_{\lambda}.
\]
This proves Theorem \ref{thm.pieri-hall.Lam}.
\end{proof}

\begin{proof}
[Proof of Theorem \ref{thm.pieri-hall}.]Theorem \ref{thm.pieri-hall.Lam}
yields%
\[
\mathbf{s}_{\lambda}\mathbf{h}_{j}=\sum_{\substack{\mu\text{ is a
partition;}\\\mu_{1}\leq n-k;\\\mu/\lambda\text{ is a horizontal
}j\text{-strip}}}\mathbf{s}_{\mu}-\sum_{i=1}^{k}\left(  -1\right)
^{i}\mathbf{h}_{n-k+i}\left(  \mathbf{s}_{\left(  n-k-j+1,1^{i-1}\right)
}\right)  ^{\perp}\mathbf{s}_{\lambda}.
\]
Both sides of this equality are symmetric functions in $\Lambda$. If we
evaluate them at $x_{1},x_{2},\ldots,x_{k}$ and project the resulting
symmetric polynomials onto $\mathcal{S}/I$, then we obtain%
\begin{align}
\overline{s_{\lambda}h_{j}}  &  =\sum_{\substack{\mu\text{ is a partition;}%
\\\mu_{1}\leq n-k;\\\mu/\lambda\text{ is a horizontal }j\text{-strip}%
}}\overline{s_{\mu}}-\sum_{i=1}^{k}\left(  -1\right)  ^{i}%
\underbrace{\overline{h_{n-k+i}}}_{\substack{=a_{i}\\\text{(since
(\ref{eq.h=amodI})}\\\text{yields }h_{n-k+i}\equiv a_{i}\operatorname{mod}%
I\text{)}}}\overline{\left(  \mathbf{s}_{\left(  n-k-j+1,1^{i-1}\right)
}\right)  ^{\perp}\mathbf{s}_{\lambda}}\nonumber\\
&  =\sum_{\substack{\mu\text{ is a partition;}\\\mu_{1}\leq n-k;\\\mu
/\lambda\text{ is a horizontal }j\text{-strip}}}\overline{s_{\mu}}-\sum
_{i=1}^{k}\left(  -1\right)  ^{i}a_{i}\overline{\left(  \mathbf{s}_{\left(
n-k-j+1,1^{i-1}\right)  }\right)  ^{\perp}\mathbf{s}_{\lambda}}.
\label{pf.thm.pieri-hall.2}%
\end{align}

But every partition $\mu$ has either at most $k$ parts or more than $k$ parts.
Hence,
\begin{align*}
&  \sum_{\substack{\mu\text{ is a partition;}\\\mu_{1}\leq n-k;\\\mu
/\lambda\text{ is a horizontal }j\text{-strip}}}\overline{s_{\mu}}\\
&  =\sum_{\substack{\mu\text{ is a partition;}\\\mu_{1}\leq n-k;\\\mu\text{
has at most }k\text{ parts;}\\\mu/\lambda\text{ is a horizontal }%
j\text{-strip}}}\overline{s_{\mu}}+\sum_{\substack{\mu\text{ is a
partition;}\\\mu_{1}\leq n-k;\\\mu\text{ has more than }k\text{ parts;}%
\\\mu/\lambda\text{ is a horizontal }j\text{-strip}}}\underbrace{\overline
{s_{\mu}}}_{\substack{=0\\\text{(because (\ref{eq.slam=0-too-long}) (applied
to }\mu\\\text{instead of }\lambda\text{) yields }s_{\mu}=0\text{)}}}\\
&  =\underbrace{\sum_{\substack{\mu\text{ is a partition;}\\\mu_{1}\leq
n-k;\\\mu\text{ has at most }k\text{ parts;}\\\mu/\lambda\text{ is a
horizontal }j\text{-strip}}}}_{\substack{=\sum_{\substack{\mu\in P_{k,n}%
;\\\mu/\lambda\text{ is a horizontal }j\text{-strip}}}\\\text{(because the
partitions }\mu\text{ such that }\mu_{1}\leq n-k\\\text{and such that }%
\mu\text{ has at most }k\text{ parts}\\\text{are precisely the partitions }%
\mu\in P_{k,n}\text{)}}}\overline{s_{\mu}}=\sum_{\substack{\mu\in
P_{k,n};\\\mu/\lambda\text{ is a horizontal }j\text{-strip}}}\overline{s_{\mu
}}.
\end{align*}
Hence, (\ref{pf.thm.pieri-hall.2}) becomes%
\begin{align*}
\overline{s_{\lambda}h_{j}}  &  =\underbrace{\sum_{\substack{\mu\text{ is a
partition;}\\\mu_{1}\leq n-k;\\\mu/\lambda\text{ is a horizontal
}j\text{-strip}}}\overline{s_{\mu}}}_{=\sum_{\substack{\mu\in P_{k,n}%
;\\\mu/\lambda\text{ is a horizontal }j\text{-strip}}}\overline{s_{\mu}}}%
-\sum_{i=1}^{k}\left(  -1\right)  ^{i}a_{i}\overline{\left(  \mathbf{s}%
_{\left(  n-k-j+1,1^{i-1}\right)  }\right)  ^{\perp}\mathbf{s}_{\lambda}}\\
&  =\sum_{\substack{\mu\in P_{k,n};\\\mu/\lambda\text{ is a horizontal
}j\text{-strip}}}\overline{s_{\mu}}-\sum_{i=1}^{k}\left(  -1\right)  ^{i}%
a_{i}\overline{\left(  \mathbf{s}_{\left(  n-k-j+1,1^{i-1}\right)  }\right)
^{\perp}\mathbf{s}_{\lambda}}.
\end{align*}
This proves Theorem \ref{thm.pieri-hall}.
\end{proof}

Let us again use the notation $c_{\alpha,\beta}^{\gamma}$ for a
Littlewood--Richardson coefficient (defined as in \cite[Definition
2.5.8]{GriRei18}, for example). Then, we can restate Theorem
\ref{thm.pieri-hall} as follows:

\begin{theorem}
\label{thm.pieri-hall-LR}Let $\lambda\in P_{k,n}$. Let $j\in\left\{
0,1,\ldots,n-k\right\}  $. Then,%
\[
\overline{s_{\lambda}h_{j}}=\sum_{\substack{\mu\in P_{k,n};\\\mu/\lambda\text{
is a horizontal }j\text{-strip}}}\overline{s_{\mu}}-\sum_{i=1}^{k}\left(
-1\right)  ^{i}a_{i}\sum_{\nu\subseteq\lambda}c_{\left(  n-k-j+1,1^{i-1}%
\right)  ,\nu}^{\lambda}\overline{s_{\nu}},
\]
where the last sum ranges over all partitions $\nu$ satisfying $\nu
\subseteq\lambda$.
\end{theorem}

\begin{proof}
[Proof of Theorem \ref{thm.pieri-hall-LR}.]Let $\mu$ be a partition. Then,
(\ref{eq.skewing.ss}) yields%
\begin{align}
\left(  \mathbf{s}_{\mu}\right)  ^{\perp}\mathbf{s}_{\lambda}  &
=\mathbf{s}_{\lambda/\mu}=\sum_{\nu\text{ is a partition}}c_{\mu,\nu}%
^{\lambda}\mathbf{s}_{\nu}\nonumber\\
&  =\sum_{\substack{\nu\text{ is a partition;}\\\nu\subseteq\lambda}%
}c_{\mu,\nu}^{\lambda}\mathbf{s}_{\nu}+\sum_{\substack{\nu\text{ is a
partition;}\\\text{we don't have }\nu\subseteq\lambda}}\underbrace{c_{\mu,\nu
}^{\lambda}}_{\substack{=0\\\text{(by Proposition \ref{prop.LR.props}
\textbf{(b)})}}}\mathbf{s}_{\nu}\nonumber\\
&  =\underbrace{\sum_{\substack{\nu\text{ is a partition;}\\\nu\subseteq
\lambda}}}_{=\sum_{\nu\subseteq\lambda}}c_{\mu,\nu}^{\lambda}\mathbf{s}_{\nu
}=\sum_{\nu\subseteq\lambda}c_{\mu,\nu}^{\lambda}\mathbf{s}_{\nu}.
\label{pf.thm.pieri-hall-LR.1}%
\end{align}
Both sides of this equality are symmetric functions in $\Lambda$. If we
evaluate them at $x_{1},x_{2},\ldots,x_{k}$ and project the resulting
symmetric polynomials onto $\mathcal{S}/I$, then we obtain%
\begin{equation}
\overline{\left(  \mathbf{s}_{\mu}\right)  ^{\perp}\mathbf{s}_{\lambda}}%
=\sum_{\nu\subseteq\lambda}c_{\mu,\nu}^{\lambda}\overline{s_{\nu}}.
\label{pf.thm.pieri-hall-LR.2}%
\end{equation}

Now, forget that we fixed $\mu$. We thus have proven
(\ref{pf.thm.pieri-hall-LR.2}) for each partition $\mu$. Theorem
\ref{thm.pieri-hall} yields%
\begin{align*}
\overline{s_{\lambda}h_{j}}  &  =\sum_{\substack{\mu\in P_{k,n};\\\mu
/\lambda\text{ is a horizontal }j\text{-strip}}}\overline{s_{\mu}}-\sum
_{i=1}^{k}\left(  -1\right)  ^{i}a_{i}\underbrace{\overline{\left(
\mathbf{s}_{\left(  n-k-j+1,1^{i-1}\right)  }\right)  ^{\perp}\mathbf{s}%
_{\lambda}}}_{\substack{=\sum_{\nu\subseteq\lambda}c_{\left(  n-k-j+1,1^{i-1}%
\right)  ,\nu}^{\lambda}\overline{s_{\nu}}\\\text{(by
(\ref{pf.thm.pieri-hall-LR.2}), applied to }\mu=\left(  n-k-j+1,1^{i-1}%
\right)  \text{)}}}\\
&  =\sum_{\substack{\mu\in P_{k,n};\\\mu/\lambda\text{ is a horizontal
}j\text{-strip}}}\overline{s_{\mu}}-\sum_{i=1}^{k}\left(  -1\right)  ^{i}%
a_{i}\sum_{\nu\subseteq\lambda}c_{\left(  n-k-j+1,1^{i-1}\right)  ,\nu
}^{\lambda}\overline{s_{\nu}}.
\end{align*}
This proves Theorem \ref{thm.pieri-hall-LR}.
\end{proof}

Note that Theorem \ref{thm.pieri-hall-LR} can also be used to prove Theorem
\ref{thm.s-h-triangularity}.


\begin{verlong}
The Pieri rule has several consequences. One is the following
\textquotedblleft reduction rule\textquotedblright\ conjectured by Josh Swanson:

\begin{definition}
Let $Z$ be the $\mathbb{Z}$-submodule of $\mathcal{S}/I$ spanned by all
products of the form $a_{i}\overline{s_{\mu}}$ with $i\in\left\{
1,2,\ldots,k\right\}  $ and $\mu\in P_{k,n}$. (This is not a $\mathbf{k}$-submodule.)
\end{definition}

\begin{proposition}
\label{prop.swanson1}Let $\lambda$ be a partition such that $2\left(
n-k\right)  \geq\lambda_{1}>n-k\geq\lambda_{2}$. Then, $\overline{s_{\lambda}%
}\in Z$.
\end{proposition}

Visually speaking, this proposition says that if $\lambda$ is a partition
which \textquotedblleft protrudes\textquotedblright\ from the rectangle
$\omega$ in its first row but not in any other row, and if the first row of
$\lambda$ is not more than twice the width of this rectangle, then
$\overline{s_{\lambda}}$ can be reduced to a $\mathbb{Z}$-linear combination
of products of the form $a_{i}\overline{s_{\mu}}$ with $i\in\left\{
1,2,\ldots,k\right\}  $ and $\mu\in P_{k,n}$. For example, for $n=6$ and
$k=3$, we have%
\[
\overline{s_{\left(  5,3,2\right)  }}=a_{3}\overline{s_{\left(  2,1,1\right)
}}-a_{2}\overline{s_{\left(  2,2,1\right)  }}-a_{2}\overline{s_{\left(
3,1,1\right)  }}+a_{1}\overline{s_{\left(  3,2,1\right)  }}.
\]

Note that in the case when $\lambda_{1}=n-k+1$, Proposition
\ref{prop.swanson1} follows easily from Lemma \ref{lem.coeffw.0} (and one
obtains an explicit representation for $\overline{s_{\lambda}}$ as a
$\mathbb{Z}$-linear combination of $a_{i}\overline{s_{\mu}}$ in this case).

One way to prove Proposition \ref{prop.swanson1} is by induction on
$\lambda_{1}$ (using Theorem \ref{thm.pieri-hall-LR} to prove that
$s_{\overline{\lambda}}h_{\lambda_{1}}=s_{\lambda}+f$, where $f$ is a
symmetric polynomial of degree $<\left\vert \lambda\right\vert $).
Alternatively, Proposition \ref{prop.swanson1} can also be derived from
Theorem \ref{thm.rimhook} further below; with some extra work it can be shown
that (in the situation of Proposition \ref{prop.swanson1}) we can write
$\overline{s_{\lambda}}$ as a $\mathbf{k}$-linear combination of $\pm
a_{i}\overline{s_{\mu}}$ with $\mu\in P_{k,n}$ in which no single
$\overline{s_{\mu}}$ occurs more than once.

[This little section has been removed, since Theorem \ref{thm.rimhook} makes
it uninteresting.]
\end{verlong}

\subsection{Positivity?}

Let us recall some background about the quantum cohomology ring
$\operatorname*{QH}\nolimits^{\ast}\left(  \operatorname*{Gr}\nolimits_{kn}%
\right)  $ discussed in \cite{Postni05}. The structure constants of the
$\mathbb{Z}\left[  q\right]  $-algebra $\operatorname*{QH}\nolimits^{\ast
}\left(  \operatorname*{Gr}\nolimits_{kn}\right)  $ are polynomials in the
indeterminate $q$, whose coefficients are the famous Gromov-Witten invariants
$C_{\lambda\mu\nu}^{d}$. These Gromov-Witten invariants $C_{\lambda\mu\nu}%
^{d}$ are nonnegative integers (as follows from their geometric
interpretation, but also from the \textquotedblleft Quantum
Littlewood-Richardson Rule\textquotedblright\ \cite[Theorem 2]{BKPT16}). This
appears to generalize to the general case of $\mathcal{S}/I$:

\begin{conjecture}
Let $b_{i}=\left(  -1\right)  ^{n-k-1}a_{i}$ for each $i\in\left\{
1,2,\ldots,k\right\}  $. Let $\lambda$, $\mu$ and $\nu$ be three partitions in
$P_{k,n}$. Then, $\left(  -1\right)  ^{\left\vert \lambda\right\vert
+\left\vert \mu\right\vert -\left\vert \nu\right\vert }\operatorname*{coeff}%
\nolimits_{\nu}\left(  \overline{s_{\lambda}s_{\mu}}\right)  $ is a polynomial
in $b_{1},b_{2},\ldots,b_{k}$ with nonnegative integer coefficients. (See
Definition \ref{def.omega-and-complement} \textbf{(b)} for the meaning of
$\operatorname*{coeff}\nolimits_{\nu}$.)
\end{conjecture}

We have verified this conjecture for all $n\leq8$ using SageMath.

\section{The \textquotedblleft rim hook algorithm\textquotedblright}

We shall next take aim at a recursive formula for \textquotedblleft
straightening\textquotedblright\ a Schur polynomial -- i.e., representing an
$\overline{s_{\mu}}$, where $\mu$ is a partition that does not belong to
$P_{k,n}$, as a $\mathbf{k}$-linear combination of \textquotedblleft
smaller\textquotedblright\ $\overline{s_{\lambda}}$'s. However, before we can
state this formula, we will have to introduce several new notations.

\subsection{Schur polynomials for non-partitions}

Recall Definition \ref{def.Pk}. Thus, the elements of $P_{k}$ are weakly
decreasing $k$-tuples in $\mathbb{N}^{k}$. For each $\lambda\in P_{k}$, a
Schur polynomial $s_{\lambda}\in\mathcal{S}$ is defined. Let us extend this
definition by defining $s_{\lambda}$ for each $\lambda\in\mathbb{Z}^{k}$:

\begin{definition}
\label{def.non-part.schur}Let $\lambda=\left(  \lambda_{1},\lambda_{2}%
,\ldots,\lambda_{k}\right)  \in\mathbb{Z}^{k}$. Then, we define a symmetric
polynomial $s_{\lambda}\in\mathcal{S}$ by%
\begin{equation}
s_{\lambda}=\det\left(  \left(  h_{\lambda_{u}-u+v}\right)  _{1\leq u\leq
k,\ 1\leq v\leq k}\right)  . \label{eq.def.non-part.schur.sla=}%
\end{equation}
This new definition does not clash with the previous use of the notation
$s_{\lambda}$, because when $\lambda\in P_{k}$, both definitions yield the
same result (because of Proposition \ref{prop.jacobi-trudi.Sh} \textbf{(a)}).
\end{definition}

This definition is similar to the definition of $\overline{s}_{\left(
\alpha_{1},\alpha_{2},\ldots,\alpha_{n}\right)  }$ in \cite[Exercise 2.9.1
(c)]{GriRei18}, but we are working with symmetric polynomials rather than
symmetric functions here.

Definition \ref{def.non-part.schur} does not really open the gates to a new
world of symmetric polynomials; indeed, each $s_{\alpha}$ (with $\alpha
\in\mathbb{Z}^{k}$) defined in Definition \ref{def.non-part.schur} is either
$0$ or can be rewritten in the form $\pm s_{\lambda}$ for some $\lambda\in
P_{k}$. Here is a more precise statement of this:

\begin{proposition}
\label{prop.non-part.schur-straighten}Let $\alpha\in\mathbb{Z}^{k}$. Define a
$k$-tuple $\beta=\left(  \beta_{1},\beta_{2},\ldots,\beta_{k}\right)  $ by%
\[
\left(  \beta_{i}=\alpha_{i}+k-i\ \ \ \ \ \ \ \ \ \ \text{for each }%
i\in\left\{  1,2,\ldots,k\right\}  \right)  .
\]

\textbf{(a)} If $\beta$ has at least one negative entry, then $s_{\alpha}=0$.

\textbf{(b)} If $\beta$ has two equal entries, then $s_{\alpha}=0$.

\textbf{(c)} Assume that $\beta$ has no negative entries and no two equal
entries. Let $\sigma\in S_{k}$ be the permutation such that $\beta
_{\sigma\left(  1\right)  }>\beta_{\sigma\left(  2\right)  }>\cdots
>\beta_{\sigma\left(  k\right)  }$. (Such a permutation $\sigma$ exists and is
unique, since $\beta$ has no two equal entries.) Define a $k$-tuple
$\lambda=\left(  \lambda_{1},\lambda_{2},\ldots,\lambda_{k}\right)
\in\mathbb{Z}^{k}$ by%
\[
\left(  \lambda_{i}=\beta_{\sigma\left(  i\right)  }%
-k+i\ \ \ \ \ \ \ \ \ \ \text{for each }i\in\left\{  1,2,\ldots,k\right\}
\right)  .
\]
Then, $\lambda\in P_{k}$ and $s_{\alpha}=\left(  -1\right)  ^{\sigma
}s_{\lambda}$.
\end{proposition}

\begin{proof}
[Proof of Proposition \ref{prop.non-part.schur-straighten}.]For each
$u\in\left\{  1,2,\ldots,k\right\}  $, we have $\beta_{u}=\alpha_{u}+k-u$ (by
the definition of $\beta_{u}$) and thus%
\begin{equation}
\underbrace{\beta_{u}}_{=\alpha_{u}+k-u}-k=\left(  \alpha_{u}+k-u\right)
-k=\alpha_{u}-u. \label{pf.prop.non-part.schur-straighten.1}%
\end{equation}
The definition of $s_{\alpha}$ yields%
\begin{align}
s_{\alpha}  &  =\det\left(  \left(  \underbrace{h_{\alpha_{u}-u+v}%
}_{\substack{=h_{\beta_{u}-k+v}\\\text{(since
(\ref{pf.prop.non-part.schur-straighten.1}) yields }\alpha_{u}-u=\beta
_{u}-k\text{)}}}\right)  _{1\leq u\leq k,\ 1\leq v\leq k}\right) \nonumber\\
&  =\det\left(  \left(  h_{\beta_{u}-k+v}\right)  _{1\leq u\leq k,\ 1\leq
v\leq k}\right)  . \label{pf.prop.non-part.schur-straighten.2}%
\end{align}

\textbf{(b)} Assume that $\beta$ has two equal entries. In other words, there
are two distinct elements $i$ and $j$ of $\left\{  1,2,\ldots,k\right\}  $
such that $\beta_{i}=\beta_{j}$. Consider these $i$ and $j$. The $i$-th and
$j$-th rows of the matrix $\left(  h_{\beta_{u}-k+v}\right)  _{1\leq u\leq
k,\ 1\leq v\leq k}$ are equal (since $\beta_{i}=\beta_{j}$). Hence, this
matrix has two equal rows. Thus, its determinant is $0$. In other words,
$\det\left(  \left(  h_{\beta_{u}-k+v}\right)  _{1\leq u\leq k,\ 1\leq v\leq
k}\right)  =0$. Now, (\ref{pf.prop.non-part.schur-straighten.2}) becomes
$s_{\alpha}=\det\left(  \left(  h_{\beta_{u}-k+v}\right)  _{1\leq u\leq
k,\ 1\leq v\leq k}\right)  =0$. This proves Proposition
\ref{prop.non-part.schur-straighten} \textbf{(b)}.

\textbf{(a)} Assume that $\beta$ has at least one negative entry. In other
words, there exists some $i\in\left\{  1,2,\ldots,k\right\}  $ such that
$\beta_{i}<0$. Consider this $i$. For each $v\in\left\{  1,2,\ldots,k\right\}
$, we have $\beta_{i}-k+\underbrace{v}_{\leq k}\leq\beta_{i}-k+k=\beta_{i}<0$
and thus $h_{\beta_{i}-k+v}=0$. Hence, all entries of the $i$-th row of the
matrix $\left(  h_{\beta_{u}-k+v}\right)  _{1\leq u\leq k,\ 1\leq v\leq k}$
are $0$. Hence, this matrix has a zero row. Thus, its determinant is $0$. In
other words, $\det\left(  \left(  h_{\beta_{u}-k+v}\right)  _{1\leq u\leq
k,\ 1\leq v\leq k}\right)  =0$. Now, the equality
(\ref{pf.prop.non-part.schur-straighten.2}) becomes $s_{\alpha}=\det\left(
\left(  h_{\beta_{u}-k+v}\right)  _{1\leq u\leq k,\ 1\leq v\leq k}\right)
=0$. This proves Proposition \ref{prop.non-part.schur-straighten} \textbf{(a)}.

\textbf{(c)} It is well-known that if we permute the rows of a $k\times
k$-matrix using a permutation $\tau$, then the determinant of the matrix gets
multiplied by $\left(  -1\right)  ^{\tau}$. In other words, every $k\times
k$-matrix $\left(  b_{u,v}\right)  _{1\leq u\leq k,\ 1\leq v\leq k}$ and every
$\tau\in S_{k}$ satisfy $\det\left(  \left(  b_{\tau\left(  u\right)
,v}\right)  _{1\leq u\leq k,\ 1\leq v\leq k}\right)  =\left(  -1\right)
^{\tau}\det\left(  \left(  b_{u,v}\right)  _{1\leq u\leq k,\ 1\leq v\leq
k}\right)  $. Applying this to $\left(  b_{u,v}\right)  _{1\leq u\leq
k,\ 1\leq v\leq k}=\left(  h_{\beta_{u}-k+v}\right)  _{1\leq u\leq k,\ 1\leq
v\leq k}$ and $\tau=\sigma$, we obtain%
\[
\det\left(  \left(  h_{\beta_{\sigma\left(  u\right)  }-k+v}\right)  _{1\leq
u\leq k,\ 1\leq v\leq k}\right)  =\left(  -1\right)  ^{\sigma}\det\left(
\left(  h_{\beta_{u}-k+v}\right)  _{1\leq u\leq k,\ 1\leq v\leq k}\right)  .
\]
Multiplying both sides of this equality by $\left(  -1\right)  ^{\sigma}$, we
find%
\begin{align}
\left(  -1\right)  ^{\sigma}\det\left(  \left(  h_{\beta_{\sigma\left(
u\right)  }-k+v}\right)  _{1\leq u\leq k,\ 1\leq v\leq k}\right)   &
=\underbrace{\left(  -1\right)  ^{\sigma}\left(  -1\right)  ^{\sigma}%
}_{=\left(  \left(  -1\right)  ^{\sigma}\right)  ^{2}=1}\det\left(  \left(
h_{\beta_{u}-k+v}\right)  _{1\leq u\leq k,\ 1\leq v\leq k}\right) \nonumber\\
&  =\det\left(  \left(  h_{\beta_{u}-k+v}\right)  _{1\leq u\leq k,\ 1\leq
v\leq k}\right)  . \label{pf.prop.non-part.schur-straighten.c.0}%
\end{align}

For each $u\in\left\{  1,2,\ldots,k\right\}  $, we have $\lambda_{u}%
=\beta_{\sigma\left(  u\right)  }-k+u$ (by the definition of $\lambda_{u}$)
and thus
\begin{equation}
\lambda_{u}-u=\beta_{\sigma\left(  u\right)  }-k.
\label{pf.prop.non-part.schur-straighten.c.1}%
\end{equation}
Now, (\ref{pf.prop.non-part.schur-straighten.2}) becomes%
\begin{align*}
s_{\alpha}  &  =\det\left(  \left(  h_{\beta_{u}-k+v}\right)  _{1\leq u\leq
k,\ 1\leq v\leq k}\right) \\
&  =\left(  -1\right)  ^{\sigma}\det\left(  \left(  \underbrace{h_{\beta
_{\sigma\left(  u\right)  }-k+v}}_{\substack{=h_{\lambda_{u}-u+v}%
\\\text{(since (\ref{pf.prop.non-part.schur-straighten.c.1}) yields }%
\beta_{\sigma\left(  u\right)  }-k=\lambda_{u}-u\text{)}}}\right)  _{1\leq
u\leq k,\ 1\leq v\leq k}\right)  \ \ \ \ \ \ \ \ \ \ \left(  \text{by
(\ref{pf.prop.non-part.schur-straighten.c.0})}\right) \\
&  =\left(  -1\right)  ^{\sigma}\underbrace{\det\left(  \left(  h_{\lambda
_{u}-u+v}\right)  _{1\leq u\leq k,\ 1\leq v\leq k}\right)  }%
_{\substack{=s_{\lambda}\\\text{(by (\ref{eq.def.non-part.schur.sla=}))}%
}}=\left(  -1\right)  ^{\sigma}s_{\lambda}.
\end{align*}
It remains to prove that $\lambda\in P_{k}$.

Let $i\in\left\{  1,2,\ldots,k-1\right\}  $. Then, $\beta_{\sigma\left(
i\right)  }>\beta_{\sigma\left(  i+1\right)  }$ (since $\beta_{\sigma\left(
1\right)  }>\beta_{\sigma\left(  2\right)  }>\cdots>\beta_{\sigma\left(
k\right)  }$) and thus $\beta_{\sigma\left(  i\right)  }\geq\beta
_{\sigma\left(  i+1\right)  }+1$ (since $\beta_{\sigma\left(  i\right)  }$ and
$\beta_{\sigma\left(  i+1\right)  }$ are integers). The definition of
$\lambda_{i+1}$ yields $\lambda_{i+1}=\beta_{\sigma\left(  i+1\right)
}-k+\left(  i+1\right)  $. The definition of $\lambda_{i}$ yields%
\[
\lambda_{i}=\underbrace{\beta_{\sigma\left(  i\right)  }}_{\geq\beta
_{\sigma\left(  i+1\right)  }+1}-k+i\geq\beta_{\sigma\left(  i+1\right)
}+1-k+i=\beta_{\sigma\left(  i+1\right)  }-k+\left(  i+1\right)
=\lambda_{i+1}.
\]
Now, forget that we fixed $i$. We thus have proven that $\lambda_{i}%
\geq\lambda_{i+1}$ for each $i\in\left\{  1,2,\ldots,k-1\right\}  $. In other
words, $\lambda_{1}\geq\lambda_{2}\geq\cdots\geq\lambda_{k}$.

Let $i\in\left\{  1,2,\ldots,k\right\}  $. Then, $1\leq i\leq k$ and thus
$k\geq1$, so that $\lambda_{k}$ is well-defined. Furthermore, from $i\leq k$,
we obtain $\lambda_{i}\geq\lambda_{k}$ (since $\lambda_{1}\geq\lambda_{2}%
\geq\cdots\geq\lambda_{k}$). But the definition of $\lambda_{k}$ yields
$\lambda_{k}=\beta_{\sigma\left(  k\right)  }-k+k=\beta_{\sigma\left(
k\right)  }\geq0$ (since all entries of $\beta$ are nonnegative (since $\beta$
has no negative entries)). Thus, $\lambda_{i}\geq\lambda_{k}\geq0$.

Now, forget that we fixed $i$. We thus have proven that $\lambda_{i}\geq0$ for
each $i\in\left\{  1,2,\ldots,k\right\}  $. In other words, $\lambda
_{1},\lambda_{2},\ldots,\lambda_{k}$ are nonnegative integers (since they are
clearly integers). Hence, $\left(  \lambda_{1},\lambda_{2},\ldots,\lambda
_{k}\right)  \in\mathbb{N}^{k}$. Combining this with $\lambda_{1}\geq
\lambda_{2}\geq\cdots\geq\lambda_{k}$, we obtain $\left(  \lambda_{1}%
,\lambda_{2},\ldots,\lambda_{k}\right)  \in P_{k}$. Hence, $\lambda=\left(
\lambda_{1},\lambda_{2},\ldots,\lambda_{k}\right)  \in P_{k}$. This completes
the proof of Proposition \ref{prop.non-part.schur-straighten} \textbf{(c)}.
\end{proof}

Let us next recall the bialternant formula for Schur polynomials. We need a
few definitions first:

\begin{definition}
\textbf{(a)} Let $\rho$ denote the $k$-tuple $\left(  k-1,k-2,\ldots,0\right)
\in\mathbb{N}^{k}$.

\textbf{(b)} We regard $\mathbb{Z}^{k}$ as a $\mathbb{Z}$-module in the
obvious way: Addition is define entrywise (i.e., we set $\alpha+\beta=\left(
\alpha_{1}+\beta_{1},\alpha_{2}+\beta_{2},\ldots,\alpha_{k}+\beta_{k}\right)
$ for any $\alpha=\left(  \alpha_{1},\alpha_{2},\ldots,\alpha_{k}\right)
\in\mathbb{Z}^{k}$ and any $\beta=\left(  \beta_{1},\beta_{2},\ldots,\beta
_{k}\right)  \in\mathbb{Z}^{k}$). This also defines subtraction on
$\mathbb{Z}^{k}$ (which, too, works entrywise). We let $\mathbf{0}$ denote the
$k$-tuple $\left(  \underbrace{0,0,\ldots,0}_{k\text{ entries}}\right)
\in\mathbb{N}^{k}\subseteq\mathbb{Z}^{k}$; this is the zero vector of
$\mathbb{Z}^{k}$.
\end{definition}

\begin{definition}
Let $\alpha=\left(  \alpha_{1},\alpha_{2},\ldots,\alpha_{k}\right)
\in\mathbb{N}^{k}$. Then, we define the \textit{alternant }$a_{\alpha}%
\in\mathcal{P}$ by%
\[
a_{\alpha}=\det\left(  \left(  x_{i}^{\alpha_{j}}\right)  _{1\leq i\leq
k,\ 1\leq j\leq k}\right)  .
\]

\end{definition}

The two definitions we have just made match the notations in \cite[\S 2.6]%
{GriRei18}, except that we are using $k$ instead of $n$ for the number of indeterminates.

Note that the element $a_{\rho}$ of $\mathcal{P}$ is the Vandermonde
determinant \newline$\det\left(  \left(  x_{i}^{k-j}\right)  _{1\leq i\leq
k,\ 1\leq j\leq k}\right)  =\prod_{1\leq i<j\leq k}\left(  x_{i}-x_{j}\right)
$; it is a regular element of $\mathcal{P}$ (that is, a non-zero-divisor).

We recall the \textit{bialternant formula} for Schur polynomials
(\cite[Corollary 2.6.7]{GriRei18}):

\begin{proposition}
\label{prop.bialternant}For any $\lambda\in P_{k}$, we have $s_{\lambda
}=a_{\lambda+\rho}/a_{\rho}$ in $\mathcal{P}$.
\end{proposition}

Let us extend this fact to arbitrary $\lambda\in\mathbb{Z}^{k}$ satisfying
$\lambda+\rho\in\mathbb{N}^{k}$ (and rename $\lambda$ as $\alpha$):

\begin{proposition}
\label{prop.non-part.bialternant}Let $\alpha\in\mathbb{Z}^{k}$ be such that
$\alpha+\rho\in\mathbb{N}^{k}$. Then, $s_{\alpha}=a_{\alpha+\rho}/a_{\rho}$ in
$\mathcal{P}$.
\end{proposition}

\begin{proof}
[Proof of Proposition \ref{prop.non-part.bialternant}.]We have $\rho=\left(
k-1,k-2,\ldots,0\right)  $. Thus,%
\begin{equation}
\rho_{i}=k-i\ \ \ \ \ \ \ \ \ \ \text{for each }i\in\left\{  1,2,\ldots
,k\right\}  . \label{pf.prop.non-part.bialternant.rhoi}%
\end{equation}
Define a $k$-tuple $\beta=\left(  \beta_{1},\beta_{2},\ldots,\beta_{k}\right)
$ as in Proposition \ref{prop.non-part.schur-straighten}. Thus, for each
$i\in\left\{  1,2,\ldots,k\right\}  $, we have%
\[
\beta_{i}=\alpha_{i}+\underbrace{k-i}_{\substack{=\rho_{i}\\\text{(by
(\ref{pf.prop.non-part.bialternant.rhoi}))}}}=\alpha_{i}+\rho_{i}=\left(
\alpha+\rho\right)  _{i}.
\]
In other words, $\beta=\alpha+\rho$. Hence, $\beta=\alpha+\rho\in
\mathbb{N}^{k}$. Thus, the $k$-tuple $\beta$ has no negative entries.

Moreover, from $\alpha+\rho=\beta$, we obtain%
\begin{align}
a_{\alpha+\rho}  &  =a_{\beta}=\det\left(  \left(  x_{i}^{\beta_{j}}\right)
_{1\leq i\leq k,\ 1\leq j\leq k}\right)  \ \ \ \ \ \ \ \ \ \ \left(  \text{by
the definition of }a_{\beta}\right) \nonumber\\
&  =\det\left(  \left(  x_{u}^{\beta_{v}}\right)  _{1\leq u\leq k,\ 1\leq
v\leq k}\right)  \label{pf.prop.non-part.bialternant.a=}%
\end{align}
(here, we have renamed the indices $i$ and $j$ as $u$ and $v$). Now, we are in
one of the following two cases:

\textit{Case 1:} The $k$-tuple $\beta$ has two equal entries.

\textit{Case 2:} The $k$-tuple $\beta$ has no two equal entries.

Let us first consider Case 1. In this case, the $k$-tuple $\beta$ has two
equal entries. In other words, there are two distinct elements $i$ and $j$ of
$\left\{  1,2,\ldots,k\right\}  $ such that $\beta_{i}=\beta_{j}$. Consider
these $i$ and $j$. The $i$-th and $j$-th columns of the matrix $\left(
x_{u}^{\beta_{v}}\right)  _{1\leq u\leq k,\ 1\leq v\leq k}$ are equal (since
$\beta_{i}=\beta_{j}$). Hence, this matrix has two equal columns. Thus, its
determinant is $0$. In other words, $\det\left(  \left(  x_{u}^{\beta_{v}%
}\right)  _{1\leq u\leq k,\ 1\leq v\leq k}\right)  =0$. Now,
(\ref{pf.prop.non-part.bialternant.a=}) becomes $a_{\alpha+\rho}=\det\left(
\left(  x_{u}^{\beta_{v}}\right)  _{1\leq u\leq k,\ 1\leq v\leq k}\right)
=0$. Hence, $a_{\alpha+\rho}/a_{\rho}=0/a_{\rho}=0$. Comparing this with
$s_{\alpha}=0$ (which follows from Proposition
\ref{prop.non-part.schur-straighten} \textbf{(b)}), we obtain $s_{\alpha
}=a_{\alpha+\rho}/a_{\rho}$. Thus, Proposition \ref{prop.non-part.bialternant}
is proven in Case 1.

Let us next consider Case 2. In this case, the $k$-tuple $\beta$ has no two
equal entries. Thus, there is a unique permutation $\sigma\in S_{k}$ that
sorts this $k$-tuple into strictly decreasing order. In other words, there is
a unique permutation $\sigma\in S_{k}$ such that $\beta_{\sigma\left(
1\right)  }>\beta_{\sigma\left(  2\right)  }>\cdots>\beta_{\sigma\left(
k\right)  }$. Consider this $\sigma$. Define a $k$-tuple $\lambda=\left(
\lambda_{1},\lambda_{2},\ldots,\lambda_{k}\right)  \in\mathbb{Z}^{k}$ by%
\[
\left(  \lambda_{i}=\beta_{\sigma\left(  i\right)  }%
-k+i\ \ \ \ \ \ \ \ \ \ \text{for each }i\in\left\{  1,2,\ldots,k\right\}
\right)  .
\]
Then, Proposition \ref{prop.non-part.schur-straighten} \textbf{(c)} yields
$\lambda\in P_{k}$ and $s_{\alpha}=\left(  -1\right)  ^{\sigma}s_{\lambda}$.

It is well-known that if we permute the columns of a $k\times k$-matrix using
a permutation $\tau$, then the determinant of the matrix gets multiplied by
$\left(  -1\right)  ^{\tau}$. In other words, every $k\times k$-matrix
$\left(  b_{u,v}\right)  _{1\leq u\leq k,\ 1\leq v\leq k}$ and every $\tau\in
S_{k}$ satisfy $\det\left(  \left(  b_{u,\tau\left(  v\right)  }\right)
_{1\leq u\leq k,\ 1\leq v\leq k}\right)  =\left(  -1\right)  ^{\tau}%
\det\left(  \left(  b_{u,v}\right)  _{1\leq u\leq k,\ 1\leq v\leq k}\right)
$. Applying this to $\left(  b_{u,v}\right)  _{1\leq u\leq k,\ 1\leq v\leq
k}=\left(  x_{u}^{\beta_{v}}\right)  _{1\leq u\leq k,\ 1\leq v\leq k}$ and
$\tau=\sigma$, we obtain%
\begin{equation}
\det\left(  \left(  x_{u}^{\beta_{\sigma\left(  v\right)  }}\right)  _{1\leq
u\leq k,\ 1\leq v\leq k}\right)  =\left(  -1\right)  ^{\sigma}\det\left(
\left(  x_{u}^{\beta_{v}}\right)  _{1\leq u\leq k,\ 1\leq v\leq k}\right)  .
\label{pf.prop.non-part.bialternant.c2.3}%
\end{equation}

But each $v\in\left\{  1,2,\ldots,k\right\}  $ satisfies
\begin{align}
\left(  \lambda+\rho\right)  _{v}  &  =\underbrace{\lambda_{v}}%
_{\substack{=\beta_{\sigma\left(  v\right)  }-k+v\\\text{(by the definition of
}\lambda_{v}\text{)}}}+\underbrace{\rho_{v}}_{\substack{=k-v\\\text{(by
(\ref{pf.prop.non-part.bialternant.rhoi}))}}}=\left(  \beta_{\sigma\left(
v\right)  }-k+v\right)  +\left(  k-v\right) \nonumber\\
&  =\beta_{\sigma\left(  v\right)  }.
\label{pf.prop.non-part.bialternant.c2.4}%
\end{align}
Now, the definition of $a_{\lambda+\rho}$ yields%
\begin{align*}
a_{\lambda+\rho}  &  =\det\left(  \left(  x_{i}^{\left(  \lambda+\rho\right)
_{j}}\right)  _{1\leq i\leq k,\ 1\leq j\leq k}\right)  =\det\left(  \left(
\underbrace{x_{u}^{\left(  \lambda+\rho\right)  _{v}}}_{\substack{=x_{u}%
^{\beta_{\sigma\left(  v\right)  }}\\\text{(by
(\ref{pf.prop.non-part.bialternant.c2.4}))}}}\right)  _{1\leq u\leq k,\ 1\leq
v\leq k}\right) \\
&  \ \ \ \ \ \ \ \ \ \ \left(  \text{here, we have renamed the indices
}i\text{ and }j\text{ as }u\text{ and }v\right) \\
&  =\det\left(  \left(  x_{u}^{\beta_{\sigma\left(  v\right)  }}\right)
_{1\leq u\leq k,\ 1\leq v\leq k}\right) \\
&  =\left(  -1\right)  ^{\sigma}\underbrace{\det\left(  \left(  x_{u}%
^{\beta_{v}}\right)  _{1\leq u\leq k,\ 1\leq v\leq k}\right)  }%
_{\substack{=a_{\alpha+\rho}\\\text{(by (\ref{pf.prop.non-part.bialternant.a=}%
))}}}\ \ \ \ \ \ \ \ \ \ \left(  \text{by
(\ref{pf.prop.non-part.bialternant.c2.3})}\right) \\
&  =\left(  -1\right)  ^{\sigma}a_{\alpha+\rho}.
\end{align*}

But $\lambda\in P_{k}$. Hence, Proposition \ref{prop.bialternant} yields
\[
s_{\lambda}=\underbrace{a_{\lambda+\rho}}_{=\left(  -1\right)  ^{\sigma
}a_{\alpha+\rho}}/a_{\rho}=\left(  -1\right)  ^{\sigma}a_{\alpha+\rho}%
/a_{\rho}.
\]
Hence,
\[
s_{\alpha}=\left(  -1\right)  ^{\sigma}\underbrace{s_{\lambda}}_{=\left(
-1\right)  ^{\sigma}a_{\alpha+\rho}/a_{\rho}}=\underbrace{\left(  -1\right)
^{\sigma}\left(  -1\right)  ^{\sigma}}_{=\left(  \left(  -1\right)  ^{\sigma
}\right)  ^{2}=1}a_{\alpha+\rho}/a_{\rho}=a_{\alpha+\rho}/a_{\rho}.
\]
Thus, Proposition \ref{prop.non-part.bialternant} is proven in Case 2.

We have now proven Proposition \ref{prop.non-part.bialternant} in both Cases 1
and 2. Thus, Proposition \ref{prop.non-part.bialternant} is proven.
\end{proof}

\subsection{The uncancelled Pieri rule}

Having defined $s_{\lambda}$ for all $\lambda\in\mathbb{Z}^{k}$ (rather than
merely for partitions), we can state a nonstandard version of the Pieri rule
for products of the form $s_{\lambda}h_{i}$, which will turn out rather$_{{}}$ useful:

\begin{theorem}
\label{thm.non-part.pieri-h}Let $\lambda\in\mathbb{Z}^{k}$ be such that
$\lambda+\rho\in\mathbb{N}^{k}$. Let $m\in\mathbb{N}$. Then,%
\[
s_{\lambda}h_{m}=\sum_{\substack{\nu\in\mathbb{N}^{k};\\\left\vert
\nu\right\vert =m}}s_{\lambda+\nu}.
\]

\end{theorem}

\begin{example}
For this example, let $k=3$ and $\lambda=\left(  -2,2,1\right)  $. Then,
$\lambda+\rho=\left(  -2,2,1\right)  +\left(  2,1,0\right)  =\left(
0,3,1\right)  $. It is easy to see (using Proposition
\ref{prop.non-part.schur-straighten} \textbf{(c)}) that $s_{\lambda
}=s_{\left(  1\right)  }$.

Furthermore, set $m=2$. Then, the $\nu\in\mathbb{N}^{k}$ satisfying
$\left\vert \nu\right\vert =m$ are the six $3$-tuples%
\[
\left(  2,0,0\right)  ,\ \ \ \ \ \ \ \ \ \ \left(  0,2,0\right)
,\ \ \ \ \ \ \ \ \ \ \left(  0,0,2\right)  ,\ \ \ \ \ \ \ \ \ \ \left(
1,1,0\right)  ,\ \ \ \ \ \ \ \ \ \ \left(  1,0,1\right)
,\ \ \ \ \ \ \ \ \ \ \left(  0,1,1\right)  .
\]
Hence, Theorem \ref{thm.non-part.pieri-h} yields%
\begin{align*}
s_{\left(  -2,2,1\right)  }h_{2}  &  =\sum_{\substack{\nu\in\mathbb{N}%
^{k};\\\left\vert \nu\right\vert =m}}s_{\left(  -2,2,1\right)  +\nu}\\
&  =\underbrace{s_{\left(  -2,2,1\right)  +\left(  2,0,0\right)  }%
}_{\substack{=s_{\left(  0,2,1\right)  }=-s_{\left(  1,1,1\right)
}\\\text{(by Proposition \ref{prop.non-part.schur-straighten} \textbf{(c)})}%
}}+\underbrace{s_{\left(  -2,2,1\right)  +\left(  0,2,0\right)  }%
}_{\substack{=s_{\left(  -2,4,1\right)  }=s_{\left(  3\right)  }\\\text{(by
Proposition \ref{prop.non-part.schur-straighten} \textbf{(c)})}}%
}+\underbrace{s_{\left(  -2,2,1\right)  +\left(  0,0,2\right)  }%
}_{\substack{=s_{\left(  -2,2,3\right)  }=0\\\text{(by Proposition
\ref{prop.non-part.schur-straighten} \textbf{(b)})}}}\\
&  \ \ \ \ \ \ \ \ \ \ +\underbrace{s_{\left(  -2,2,1\right)  +\left(
1,1,0\right)  }}_{\substack{=s_{\left(  -1,3,1\right)  }=0\\\text{(by
Proposition \ref{prop.non-part.schur-straighten} \textbf{(b)})}}%
}+\underbrace{s_{\left(  -2,2,1\right)  +\left(  1,0,1\right)  }%
}_{\substack{=s_{\left(  -1,2,2\right)  }=s_{\left(  1,1,1\right)
}\\\text{(by Proposition \ref{prop.non-part.schur-straighten} \textbf{(c)})}%
}}+\underbrace{s_{\left(  -2,2,1\right)  +\left(  0,1,1\right)  }%
}_{\substack{=s_{\left(  -2,3,2\right)  }=s_{\left(  2,1\right)  }\\\text{(by
Proposition \ref{prop.non-part.schur-straighten} \textbf{(c)})}}}\\
&  =-s_{\left(  1,1,1\right)  }+s_{\left(  3\right)  }+0+0+s_{\left(
1,1,1\right)  }+s_{\left(  2,1\right)  }=s_{\left(  2,1\right)  }+s_{\left(
3\right)  }.
\end{align*}
In view of $s_{\left(  -2,2,1\right)  }=s_{\left(  1\right)  }$, this rewrites
as $s_{\left(  1\right)  }h_{2}=s_{\left(  2,1\right)  }+s_{\left(  3\right)
}$, which is exactly what the usual Pieri rule would yield. Note that the
expression we obtained from Theorem \ref{thm.non-part.pieri-h} involves both
vanishing addends (here, $s_{\left(  -2,2,1\right)  +\left(  0,0,2\right)  }$
and $s_{\left(  -2,2,1\right)  +\left(  1,1,0\right)  }$) and mutually
cancelling addends (here, $s_{\left(  -2,2,1\right)  +\left(  2,0,0\right)  }$
and $s_{\left(  -2,2,1\right)  +\left(  1,0,1\right)  }$); this is why I call
it the \textquotedblleft uncancelled Pieri rule\textquotedblright.
\end{example}

We note that the idea of such an \textquotedblleft uncancelled Pieri
rule\textquotedblright\ as our Theorem \ref{thm.non-part.pieri-h} is not new
(similar things appeared in \cite[\S 2]{LakTho07} and \cite{Tamvak13}), but we
have not seen it stated in this exact form anywhere in the literature. Thus,
let us give a proof:

\begin{proof}
[Proof of Theorem \ref{thm.non-part.pieri-h}.]Define $\beta\in\mathbb{N}^{k}$
by $\beta=\lambda+\rho$. (This is well-defined, since $\lambda+\rho
\in\mathbb{N}^{k}$.)

From (\ref{eq.hm}), we obtain%
\begin{equation}
h_{m}=\sum_{\substack{\alpha\in\mathbb{N}^{k};\\\left\vert \alpha\right\vert
=m}}\underbrace{x^{\alpha}}_{\substack{=x_{1}^{\alpha_{1}}x_{2}^{\alpha_{2}%
}\cdots x_{k}^{\alpha_{k}}\\=\prod_{i=1}^{k}x_{i}^{\alpha_{i}}}}=\sum
_{\substack{\alpha\in\mathbb{N}^{k};\\\left\vert \alpha\right\vert =m}%
}\prod_{i=1}^{k}x_{i}^{\alpha_{i}}. \label{pf.thm.non-part.pieri-h.hm=}%
\end{equation}
For each permutation $\sigma\in S_{k}$, we have%
\begin{align}
h_{m}  &  =h_{m}\left(  x_{\sigma\left(  1\right)  },x_{\sigma\left(
2\right)  },\ldots,x_{\sigma\left(  k\right)  }\right)
\ \ \ \ \ \ \ \ \ \ \left(  \text{since the polynomial }h_{m}\text{ is
symmetric}\right) \nonumber\\
&  =\sum_{\substack{\alpha\in\mathbb{N}^{k};\\\left\vert \alpha\right\vert
=m}}\prod_{i=1}^{k}x_{\sigma\left(  i\right)  }^{\alpha_{i}}
\label{pf.thm.non-part.pieri-h.hm=2}%
\end{align}
(here, we have substituted $x_{\sigma\left(  1\right)  },x_{\sigma\left(
2\right)  },\ldots,x_{\sigma\left(  k\right)  }$ for $x_{1},x_{2},\ldots
,x_{k}$ in the equality (\ref{pf.thm.non-part.pieri-h.hm=})).

But Proposition \ref{prop.non-part.bialternant} (applied to $\alpha=\lambda$)
yields $s_{\lambda}=a_{\lambda+\rho}/a_{\rho}$ in $\mathcal{P}$. Thus,
\begin{align*}
a_{\rho}s_{\lambda}  &  =a_{\lambda+\rho}=a_{\beta}\ \ \ \ \ \ \ \ \ \ \left(
\text{since }\lambda+\rho=\beta\right) \\
&  =\det\left(  \left(  x_{i}^{\beta_{j}}\right)  _{1\leq i\leq k,\ 1\leq
j\leq k}\right)  \ \ \ \ \ \ \ \ \ \ \left(  \text{by the definition of
}a_{\beta}\right) \\
&  =\det\left(  \left(  x_{j}^{\beta_{i}}\right)  _{1\leq i\leq k,\ 1\leq
j\leq k}\right) \\
&  \ \ \ \ \ \ \ \ \ \ \left(
\begin{array}
[c]{c}%
\text{since the determinant of a matrix equals}\\
\text{the determinant of its transpose}%
\end{array}
\right) \\
&  =\sum_{\sigma\in S_{k}}\left(  -1\right)  ^{\sigma}\prod_{i=1}^{k}%
x_{\sigma\left(  i\right)  }^{\beta_{i}}\ \ \ \ \ \ \ \ \ \ \left(  \text{by
the definition of a determinant}\right)  .
\end{align*}
Multiplying both sides of this equality with $h_{m}$, we find%
\begin{align}
a_{\rho}s_{\lambda}h_{m}  &  =\left(  \sum_{\sigma\in S_{k}}\left(  -1\right)
^{\sigma}\prod_{i=1}^{k}x_{\sigma\left(  i\right)  }^{\beta_{i}}\right)
h_{m}=\sum_{\sigma\in S_{k}}\left(  -1\right)  ^{\sigma}\left(  \prod
_{i=1}^{k}x_{\sigma\left(  i\right)  }^{\beta_{i}}\right)  \underbrace{h_{m}%
}_{\substack{=\sum_{\substack{\alpha\in\mathbb{N}^{k};\\\left\vert
\alpha\right\vert =m}}\prod_{i=1}^{k}x_{\sigma\left(  i\right)  }^{\alpha_{i}%
}\\\text{(by (\ref{pf.thm.non-part.pieri-h.hm=2}))}}}\nonumber\\
&  =\sum_{\sigma\in S_{k}}\left(  -1\right)  ^{\sigma}\left(  \prod_{i=1}%
^{k}x_{\sigma\left(  i\right)  }^{\beta_{i}}\right)  \sum_{\substack{\alpha
\in\mathbb{N}^{k};\\\left\vert \alpha\right\vert =m}}\prod_{i=1}^{k}%
x_{\sigma\left(  i\right)  }^{\alpha_{i}}\nonumber\\
&  =\sum_{\substack{\alpha\in\mathbb{N}^{k};\\\left\vert \alpha\right\vert
=m}}\sum_{\sigma\in S_{k}}\left(  -1\right)  ^{\sigma}\underbrace{\left(
\prod_{i=1}^{k}x_{\sigma\left(  i\right)  }^{\beta_{i}}\right)  \prod
_{i=1}^{k}x_{\sigma\left(  i\right)  }^{\alpha_{i}}}_{=\prod_{i=1}^{k}\left(
x_{\sigma\left(  i\right)  }^{\beta_{i}}x_{\sigma\left(  i\right)  }%
^{\alpha_{i}}\right)  }\nonumber\\
&  =\sum_{\substack{\alpha\in\mathbb{N}^{k};\\\left\vert \alpha\right\vert
=m}}\sum_{\sigma\in S_{k}}\left(  -1\right)  ^{\sigma}\prod_{i=1}%
^{k}\underbrace{\left(  x_{\sigma\left(  i\right)  }^{\beta_{i}}%
x_{\sigma\left(  i\right)  }^{\alpha_{i}}\right)  }_{\substack{=x_{\sigma
\left(  i\right)  }^{\beta_{i}+\alpha_{i}}=x_{\sigma\left(  i\right)
}^{\left(  \beta+\alpha\right)  _{i}}\\\text{(since }\beta_{i}+\alpha
_{i}=\left(  \beta+\alpha\right)  _{i}\text{)}}}\nonumber\\
&  =\sum_{\substack{\alpha\in\mathbb{N}^{k};\\\left\vert \alpha\right\vert
=m}}\sum_{\sigma\in S_{k}}\left(  -1\right)  ^{\sigma}\prod_{i=1}^{k}%
x_{\sigma\left(  i\right)  }^{\left(  \beta+\alpha\right)  _{i}}\nonumber\\
&  =\sum_{\substack{\nu\in\mathbb{N}^{k};\\\left\vert \nu\right\vert =m}%
}\sum_{\sigma\in S_{k}}\left(  -1\right)  ^{\sigma}\prod_{i=1}^{k}%
x_{\sigma\left(  i\right)  }^{\left(  \beta+\nu\right)  _{i}}
\label{pf.thm.non-part.pieri-h.3}%
\end{align}
(here, we have renamed the summation index $\alpha$ as $\nu$).

On the other hand, let $\nu\in\mathbb{N}^{k}$. Then, $\left(  \lambda
+\nu\right)  +\rho=\underbrace{\left(  \lambda+\rho\right)  }_{\in
\mathbb{N}^{k}}+\underbrace{\nu}_{\in\mathbb{N}^{k}}\in\mathbb{N}^{k}$. Thus,
Proposition \ref{prop.non-part.bialternant} (applied to $\alpha=\lambda+\nu$)
yields $s_{\lambda+\nu}=a_{\left(  \lambda+\nu\right)  +\rho}/a_{\rho}$ in
$\mathcal{P}$. Thus,
\begin{align}
a_{\rho}s_{\lambda+\nu}  &  =a_{\left(  \lambda+\nu\right)  +\rho}%
=a_{\beta+\nu}\ \ \ \ \ \ \ \ \ \ \left(  \text{since }\left(  \lambda
+\nu\right)  +\rho=\underbrace{\lambda+\rho}_{=\beta}+\nu=\beta+\nu\right)
\nonumber\\
&  =\det\left(  \left(  x_{i}^{\left(  \beta+\nu\right)  _{j}}\right)  _{1\leq
i\leq k,\ 1\leq j\leq k}\right)  \ \ \ \ \ \ \ \ \ \ \left(  \text{by the
definition of }a_{\beta+\nu}\right) \nonumber\\
&  =\det\left(  \left(  x_{j}^{\left(  \beta+\nu\right)  _{i}}\right)  _{1\leq
i\leq k,\ 1\leq j\leq k}\right) \nonumber\\
&  \ \ \ \ \ \ \ \ \ \ \left(
\begin{array}
[c]{c}%
\text{since the determinant of a matrix equals}\\
\text{the determinant of its transpose}%
\end{array}
\right) \nonumber\\
&  =\sum_{\sigma\in S_{k}}\left(  -1\right)  ^{\sigma}\prod_{i=1}^{k}%
x_{\sigma\left(  i\right)  }^{\left(  \beta+\nu\right)  _{i}}
\label{pf.thm.non-part.pieri-h.4}%
\end{align}
(by the definition of a determinant).

Now, forget that we fixed $\nu$. We thus have proven
(\ref{pf.thm.non-part.pieri-h.4}) for each $\nu\in\mathbb{N}^{k}$.

Now, (\ref{pf.thm.non-part.pieri-h.3}) becomes%
\[
a_{\rho}s_{\lambda}h_{m}=\sum_{\substack{\nu\in\mathbb{N}^{k};\\\left\vert
\nu\right\vert =m}}\underbrace{\sum_{\sigma\in S_{k}}\left(  -1\right)
^{\sigma}\prod_{i=1}^{k}x_{\sigma\left(  i\right)  }^{\left(  \beta
+\nu\right)  _{i}}}_{\substack{=a_{\rho}s_{\lambda+\nu}\\\text{(by
(\ref{pf.thm.non-part.pieri-h.4}))}}}=\sum_{\substack{\nu\in\mathbb{N}%
^{k};\\\left\vert \nu\right\vert =m}}a_{\rho}s_{\lambda+\nu}=a_{\rho}%
\sum_{\substack{\nu\in\mathbb{N}^{k};\\\left\vert \nu\right\vert
=m}}s_{\lambda+\nu}.
\]
We can cancel $a_{\rho}$ from this equality (since $a_{\rho}$ is a regular
element of $\mathcal{P}$), and thus obtain%
\[
s_{\lambda}h_{m}=\sum_{\substack{\nu\in\mathbb{N}^{k};\\\left\vert
\nu\right\vert =m}}s_{\lambda+\nu}.
\]
This proves Theorem \ref{thm.non-part.pieri-h}.
\end{proof}

\subsection{The \textquotedblleft rim hook algorithm\textquotedblright}

For the rest of this section, we assume that $k>0$.

We need one more weird definition:

\begin{definition}
\label{def.rimhook.V}Let $V$ be the set of all $k$-tuples $\left(  -n,\tau
_{2},\tau_{3},\ldots,\tau_{k}\right)  \in\mathbb{Z}^{k}$ satisfying%
\begin{equation}
\left(  \tau_{i}\in\left\{  0,1\right\}  \ \ \ \ \ \ \ \ \ \ \text{for each
}i\in\left\{  2,3,\ldots,k\right\}  \right)  . \label{eq.def.rimhook.V.cond}%
\end{equation}

\end{definition}

\begin{example}
\label{exa.rimhook.V}If $n=6$ and $k=3$, then%
\begin{equation}
V=\left\{  \left(  -6,0,0\right)  ,\left(  -6,0,1\right)  ,\left(
-6,1,0\right)  ,\left(  -6,1,1\right)  \right\}  . \label{eq.exa.rimhook.V.63}%
\end{equation}

\end{example}

\begin{proposition}
\label{prop.rimhook.V.size}Let $\tau\in V$. Then, $-\left\vert \tau\right\vert
\in\left\{  n-k+1,n-k+2,\ldots,n\right\}  $.
\end{proposition}

\begin{proof}
[Proof of Proposition \ref{prop.rimhook.V.size}.]We have $\tau\in V$. Thus,
$\tau$ has the form
\[
\tau=\left(  -n,\tau_{2},\tau_{3},\ldots,\tau_{k}\right)  \in\mathbb{Z}^{k}%
\]
for some $\tau_{2},\tau_{3},\ldots,\tau_{k}$ satisfying
(\ref{eq.def.rimhook.V.cond}) (by the definition of $V$). Consider these
$\tau_{2},\tau_{3},\ldots,\tau_{k}$. We have $\tau=\left(  -n,\tau_{2}%
,\tau_{3},\ldots,\tau_{k}\right)  $ and thus%
\[
\left\vert \tau\right\vert =\left(  -n\right)  +\underbrace{\tau_{2}+\tau
_{3}+\cdots+\tau_{k}}_{=\sum_{i=2}^{k}\tau_{i}}=\left(  -n\right)  +\sum
_{i=2}^{k}\tau_{i}.
\]
Therefore,%
\begin{align*}
-\left\vert \tau\right\vert  &  =-\left(  \left(  -n\right)  +\sum_{i=2}%
^{k}\tau_{i}\right)  =n-\sum_{i=2}^{k}\underbrace{\tau_{i}}_{\substack{\leq
1\\\text{(since (\ref{eq.def.rimhook.V.cond})}\\\text{yields }\tau_{i}%
\in\left\{  0,1\right\}  \text{)}}}\geq n-\underbrace{\sum_{i=2}^{k}1}%
_{=k-1}\\
&  =n-\left(  k-1\right)  =n-k+1.
\end{align*}
Combining this with%
\[
-\left\vert \tau\right\vert =n-\sum_{i=2}^{k}\underbrace{\tau_{i}%
}_{\substack{\geq0\\\text{(since (\ref{eq.def.rimhook.V.cond})}\\\text{yields
}\tau_{i}\in\left\{  0,1\right\}  \text{)}}}\leq n-\underbrace{\sum_{i=2}%
^{k}0}_{=0}=n,
\]
we obtain $n-k+1\leq-\left\vert \tau\right\vert \leq n$. Thus, $-\left\vert
\tau\right\vert \in\left\{  n-k+1,n-k+2,\ldots,n\right\}  $ (since
$-\left\vert \tau\right\vert $ is an integer). This proves Proposition
\ref{prop.rimhook.V.size}.
\end{proof}

We are now ready to state the main theorem of this section: a generalization
of the \textquotedblleft rim hook algorithm\textquotedblright\ from
\cite[\S 2, Main Lemma]{BeCiFu99}:

\begin{theorem}
\label{thm.rimhook}Assume that $a_{1},a_{2},\ldots,a_{k}$ belong to
$\mathbf{k}$.

Let $\mu\in P_{k}$ be such that $\mu_{1}>n-k$. Then,%
\[
\overline{s_{\mu}}=\sum_{j=1}^{k}\left(  -1\right)  ^{k-j}a_{j}\sum
_{\substack{\tau\in V;\\-\left\vert \tau\right\vert =n-k+j}}\overline
{s_{\mu+\tau}}.
\]

\end{theorem}

\begin{example}
For this example, set $n=6$ and $k=3$ and $\mu=\left(  5,4,1\right)  $. Then,
Theorem \ref{thm.rimhook} yields%
\begin{align*}
&  \overline{s_{\left(  5,4,1\right)  }}\\
&  =\sum_{j=1}^{k}\left(  -1\right)  ^{k-j}a_{j}\sum_{\substack{\tau\in
V;\\-\left\vert \tau\right\vert =n-k+j}}\overline{s_{\left(  5,4,1\right)
+\tau}}\\
&  =\left(  -1\right)  ^{3-1}a_{1}\overline{s_{\left(  5,4,1\right)  +\left(
-6,1,1\right)  }}+\left(  -1\right)  ^{3-2}a_{2}\left(  \overline{s_{\left(
5,4,1\right)  +\left(  -6,0,1\right)  }}+\overline{s_{\left(  5,4,1\right)
+\left(  -6,1,0\right)  }}\right) \\
&  \ \ \ \ \ \ \ \ \ \ +\left(  -1\right)  ^{3-3}a_{3}\overline{s_{\left(
5,4,1\right)  +\left(  -6,0,0\right)  }}\ \ \ \ \ \ \ \ \ \ \left(  \text{by
(\ref{eq.exa.rimhook.V.63})}\right) \\
&  =a_{1}\underbrace{\overline{s_{\left(  5,4,1\right)  +\left(
-6,1,1\right)  }}}_{\substack{=\overline{s_{\left(  -1,5,2\right)  }%
}\\=\overline{s_{\left(  4,1,1\right)  }}\\\text{(by Proposition
\ref{prop.non-part.schur-straighten} \textbf{(c)})}}}-a_{2}\left(
\underbrace{\overline{s_{\left(  5,4,1\right)  +\left(  -6,0,1\right)  }}%
}_{\substack{=\overline{s_{\left(  -1,4,2\right)  }}\\=\overline{s_{\left(
3,1,1\right)  }}\\\text{(by Proposition \ref{prop.non-part.schur-straighten}
\textbf{(c)})}}}+\underbrace{\overline{s_{\left(  5,4,1\right)  +\left(
-6,1,0\right)  }}}_{\substack{=\overline{s_{\left(  -1,5,1\right)  }%
}\\=0\\\text{(by Proposition \ref{prop.non-part.schur-straighten}
\textbf{(b)})}}}\right) \\
&  \ \ \ \ \ \ \ \ \ \ +a_{3}\underbrace{\overline{s_{\left(  5,4,1\right)
+\left(  -6,0,0\right)  }}}_{\substack{=\overline{s_{\left(  -1,4,1\right)  }%
}\\=0\\\text{(by Proposition \ref{prop.non-part.schur-straighten}
\textbf{(b)})}}}\\
&  =a_{1}\overline{s_{\left(  4,1,1\right)  }}-a_{2}\overline{s_{\left(
3,1,1\right)  }}.
\end{align*}
Note that this is \textbf{not} yet an expansion of $\overline{s_{\mu}}$ in the
basis $\left(  \overline{s_{\lambda}}\right)  _{\lambda\in P_{k,n}}$. Indeed,
we still have a term $\overline{s_{\left(  4,1,1\right)  }}$ on the right hand
side which has $\left(  4,1,1\right)  \notin P_{k,n}$. But this term can, in
turn, be rewritten using Theorem \ref{thm.rimhook}, and so on until we end up
with an expansion of $\overline{s_{\mu}}$ in the basis $\left(  \overline
{s_{\lambda}}\right)  _{\lambda\in P_{k,n}}$, namely%
\[
\overline{s_{\left(  5,4,1\right)  }}=-a_{2}\overline{s_{\left(  3,1,1\right)
}}+a_{1}^{2}\overline{s_{\left(  1,1\right)  }}-a_{1}a_{2}\overline{s_{\left(
1\right)  }}+a_{1}a_{3}\overline{s_{\left(  {}\right)  }}.
\]

\end{example}

As we saw in this example, when we apply Theorem \ref{thm.rimhook}, some of
the $\overline{s_{\mu+\tau}}$ addends on the right hand side may be $0$ (by
Proposition \ref{prop.non-part.schur-straighten} \textbf{(b)}). Once these
addends are removed, the remaining addends can be rewritten in the form
$\pm\overline{s_{\lambda}}$ for some $\lambda\in P_{k}$ satisfying $\left\vert
\lambda\right\vert <\left\vert \mu\right\vert $ (using Proposition
\ref{prop.non-part.schur-straighten} \textbf{(c)}). The resulting sum is
multiplicity-free -- in the sense that no $\overline{s_{\lambda}}$ occurs more
than once in it. (This is not difficult to check, but would take us too far
afield.) However, this sum is (in general) not an expansion of $\overline
{s_{\mu}}$ in the basis $\left(  \overline{s_{\lambda}}\right)  _{\lambda\in
P_{k,n}}$ yet, because it often contains terms $\overline{s_{\lambda}}$ with
$\lambda\notin P_{k,n}$. If we keep applying Theorem \ref{thm.rimhook}
multiple times until we reach an expansion of $\overline{s_{\mu}}$ in the
basis $\left(  \overline{s_{\lambda}}\right)  _{\lambda\in P_{k,n}}$, then
this latter expansion may contain multiplicities: For example, for $n=6$ and
$k=3$, we have%
\[
\overline{s_{\left(  4,4,3\right)  }}=-a_{2}\overline{s_{\left(  3,3\right)
}}+a_{3}\overline{s_{\left(  3,2\right)  }}+a_{1}^{2}\overline{s_{\left(
3\right)  }}-2a_{1}a_{2}\overline{s_{\left(  2\right)  }}+a_{2}^{2}%
\overline{s_{\left(  1\right)  }}.
\]

We owe the reader an explanation of why we call Theorem \ref{thm.rimhook} a
\textquotedblleft rim hook algorithm\textquotedblright. It owes this name to
the fact that it generalizes the \textquotedblleft rim hook
algorithm\textquotedblright\ for quantum cohomology \cite[\S 2, Main
Lemma]{BeCiFu99} (which can be obtained from it with some work by setting
$a_{i}=0$ for all $i<k$). Nevertheless, it does not visibly involve any rim
hooks itself. I am, in fact, unaware of a way to restate it in the language of
Young diagrams; the operation $\mu\mapsto\mu+\tau$ for $\tau\in V$ resembles
both the removal of an $n$-rim hook (since it lowers the first entry by $n$)
and the addition of a vertical strip (since it increases each of the remaining
entries by $0$ or $1$), but it cannot be directly stated as one of these
operations followed by the other.

We shall prove Theorem \ref{thm.rimhook} by deriving it from an identity in
$\mathcal{S}$:

\begin{theorem}
\label{thm.rimhook-S}Let $\mu\in P_{k}$ be such that $\mu_{1}>n-k$. Then,%
\[
s_{\mu}=\sum_{j=1}^{k}\left(  -1\right)  ^{k-j}h_{n-k+j}\sum_{\substack{\tau
\in V;\\-\left\vert \tau\right\vert =n-k+j}}s_{\mu+\tau}.
\]

\end{theorem}

Our proof of this identity, in turn, will rely on the following combinatorial lemmas:

\begin{lemma}
\label{lem.rimhook-S.sum0lem}Let $j\in\left\{  2,3,\ldots,k\right\}  $. Let
$\Delta$ be the vector $\left(  0,0,\ldots,0,1,0,0,\ldots,0\right)
\in\mathbb{Z}^{k}$, where $1$ is the $j$-th entry.

\textbf{(a)} If $\tau\in V$ satisfies $\tau_{j}=0$, then $\tau+\Delta\in V$
and $\left(  \tau+\Delta\right)  _{j}=1$.

\textbf{(b)} If $\tau\in V$ satisfies $\tau_{j}=1$, then $\tau-\Delta\in V$
and $\left(  \tau-\Delta\right)  _{j}=0$.

\textbf{(c)} If $\nu\in\mathbb{N}^{k}$ satisfies $\nu_{j}\neq0$, then
$\nu-\Delta\in\mathbb{N}^{k}$.

\textbf{(d)} If $\nu\in\mathbb{N}^{k}$, then $\nu+\Delta\in\mathbb{N}^{k}$.
\end{lemma}

\begin{proof}
[Proof of Lemma \ref{lem.rimhook-S.sum0lem}.]We have $j\in\left\{
2,3,\ldots,k\right\}  $, thus $j\neq1$.

We have $\Delta=\left(  0,0,\ldots,0,1,0,0,\ldots,0\right)  \in\mathbb{N}^{k}%
$. Thus, $\Delta_{j}=1$ and%
\begin{equation}
\left(  \Delta_{i}=0\ \ \ \ \ \ \ \ \ \ \text{for each }i\in\left\{
1,2,\ldots,k\right\}  \text{ satisfying }i\neq j\right)  .
\label{pf.lem.rimhook-S.sum0lem.1}%
\end{equation}

Applying (\ref{pf.lem.rimhook-S.sum0lem.1}) to $i=1$, we obtain $\Delta_{1}=0$
(since $1\neq j$).

\textbf{(a)} Let $\tau\in V$ be such that $\tau_{j}=0$.

We have $\tau\in V$. According to the definition of $V$, this means that
$\tau$ is a $k$-tuple $\left(  -n,\tau_{2},\tau_{3},\ldots,\tau_{k}\right)
\in\mathbb{Z}^{k}$ satisfying (\ref{eq.def.rimhook.V.cond}). In other words,
$\tau\in\mathbb{Z}^{k}$ and $\tau_{1}=-n$ and%
\begin{equation}
\left(  \tau_{i}\in\left\{  0,1\right\}  \ \ \ \ \ \ \ \ \ \ \text{for each
}i\in\left\{  2,3,\ldots,k\right\}  \right)  .
\label{pf.lem.rimhook-S.sum0lem.a.1}%
\end{equation}

Define a $k$-tuple $\sigma\in\mathbb{Z}^{k}$ by $\sigma=\tau+\Delta$. Thus,
$\sigma_{1}=\left(  \tau+\Delta\right)  _{1}=\tau_{1}+\underbrace{\Delta_{1}%
}_{=0}=\tau_{1}=-n$.

Furthermore, from $\sigma=\tau+\Delta$, we obtain $\sigma_{j}=\left(
\tau+\Delta\right)  _{j}=\underbrace{\tau_{j}}_{=0}+\underbrace{\Delta_{j}%
}_{=1}=1\in\left\{  0,1\right\}  $.

Next, we have $\sigma_{i}\in\left\{  0,1\right\}  $ for each $i\in\left\{
2,3,\ldots,k\right\}  $\ \ \ \ \footnote{\textit{Proof.} Let $i\in\left\{
2,3,\ldots,k\right\}  $. We must prove $\sigma_{i}\in\left\{  0,1\right\}  $.
\par
If $i=j$, then this follows from $\sigma_{j}\in\left\{  0,1\right\}  $. Hence,
for the rest of this proof, we WLOG assume that $i\neq j$. Thus,
(\ref{pf.lem.rimhook-S.sum0lem.1}) yields $\Delta_{i}=0$. Now, from
$\sigma=\tau+\Delta$, we obtain $\sigma_{i}=\left(  \tau+\Delta\right)
_{i}=\tau_{i}+\underbrace{\Delta_{i}}_{=0}=\tau_{i}\in\left\{  0,1\right\}  $
(by (\ref{pf.lem.rimhook-S.sum0lem.a.1})). Qed.}. Altogether, we thus have
shown that $\sigma\in\mathbb{Z}^{k}$ and $\sigma_{1}=-n$ and
\begin{equation}
\left(  \sigma_{i}\in\left\{  0,1\right\}  \ \ \ \ \ \ \ \ \ \ \text{for each
}i\in\left\{  2,3,\ldots,k\right\}  \right)  .
\label{pf.lem.rimhook-S.sum0lem.a.2}%
\end{equation}
In other words, $\sigma$ is a $k$-tuple $\left(  -n,\sigma_{2},\sigma
_{3},\ldots,\sigma_{k}\right)  \in\mathbb{Z}^{k}$ satisfying
(\ref{pf.lem.rimhook-S.sum0lem.a.2}). In other words, $\sigma\in V$ (by the
definition of $V$). Thus, $\tau+\Delta=\sigma\in V$. So we have proven that
$\tau+\Delta\in V$ and $\left(  \tau+\Delta\right)  _{j}=1$. Thus, Lemma
\ref{lem.rimhook-S.sum0lem} \textbf{(a)} is proven.

\textbf{(b)} The proof of Lemma \ref{lem.rimhook-S.sum0lem} \textbf{(b)} is
analogous to the above proof of Lemma \ref{lem.rimhook-S.sum0lem}
\textbf{(a)}, and is left to the reader.

\textbf{(c)} Let $\nu\in\mathbb{N}^{k}$ be such that $\nu_{j}\neq0$. We must
prove that $\nu-\Delta\in\mathbb{N}^{k}$.

We have $\nu_{j}\in\mathbb{N}$ (since $\nu\in\mathbb{N}^{k}$). Hence, from
$\nu_{j}\neq0$, we conclude that $\nu_{j}\geq1$. Thus, $\nu_{j}-1\in
\mathbb{N}$. Also, the entries $\nu_{1},\nu_{2},\ldots,\nu_{j-1},\nu_{j+1}%
,\nu_{j+2},\ldots,\nu_{k}$ of $\nu$ belong to $\mathbb{N}$ (since $\nu
\in\mathbb{N}^{k}$).

Recall that $\Delta$ is the vector $\left(  0,0,\ldots,0,1,0,0,\ldots
,0\right)  \in\mathbb{Z}^{k}$, where $1$ is the $j$-th entry. Hence,%
\begin{align*}
\nu-\Delta &  =\nu-\left(  0,0,\ldots,0,1,0,0,\ldots,0\right) \\
&  =\left(  \nu_{1},\nu_{2},\ldots,\nu_{j-1},\nu_{j}-1,\nu_{j+1},\nu
_{j+2},\ldots,\nu_{k}\right)  \in\mathbb{N}^{k}%
\end{align*}
(since $\nu_{j}-1\in\mathbb{N}$ and since the entries $\nu_{1},\nu_{2}%
,\ldots,\nu_{j-1},\nu_{j+1},\nu_{j+2},\ldots,\nu_{k}$ of $\nu$ belong to
$\mathbb{N}$). This proves Lemma \ref{lem.rimhook-S.sum0lem} \textbf{(c)}.

\textbf{(d)} Let $\nu\in\mathbb{N}^{k}$. Also, $\Delta\in\mathbb{N}^{k}$.
Thus, $\underbrace{\nu}_{\in\mathbb{N}^{k}}+\underbrace{\Delta}_{\in
\mathbb{N}^{k}}\in\mathbb{N}^{k}$. This proves Lemma
\ref{lem.rimhook-S.sum0lem} \textbf{(d)}.
\end{proof}

\begin{lemma}
\label{lem.rimhook-S.sum0}Let $\gamma\in\mathbb{Z}^{k}$. Then,%
\[
\sum_{\tau\in V}\sum_{\substack{\nu\in\mathbb{N}^{k};\\\left\vert
\nu\right\vert =-\left\vert \tau\right\vert ;\\\nu+\tau=\gamma}}\left(
-1\right)  ^{n+\left\vert \tau\right\vert }=%
\begin{cases}
1, & \text{if }\gamma=\mathbf{0};\\
0, & \text{if }\gamma\neq\mathbf{0}%
\end{cases}
.
\]
(Recall that $\mathbf{0}$ denotes the vector $\left(  \underbrace{0,0,\ldots
,0}_{k\text{ zeroes}}\right)  \in\mathbb{Z}^{k}$.)
\end{lemma}

\begin{proof}
[Proof of Lemma \ref{lem.rimhook-S.sum0}.]Let $Q$ be the set of all pairs
$\left(  \tau,\nu\right)  \in V\times\mathbb{N}^{k}$ satisfying $\left\vert
\nu\right\vert =-\left\vert \tau\right\vert $ and $\nu+\tau=\gamma$. We have
the following equality of summation signs:%
\begin{equation}
\sum_{\tau\in V}\sum_{\substack{\nu\in\mathbb{N}^{k};\\\left\vert
\nu\right\vert =-\left\vert \tau\right\vert ;\\\nu+\tau=\gamma}}=\sum
_{\substack{\left(  \tau,\nu\right)  \in V\times\mathbb{N}^{k};\\\left\vert
\nu\right\vert =-\left\vert \tau\right\vert ;\\\nu+\tau=\gamma}}=\sum_{\left(
\tau,\nu\right)  \in Q} \label{pf.lem.rimhook-S.sum0.sumeq}%
\end{equation}
(since $Q$ is the set of all pairs $\left(  \tau,\nu\right)  \in
V\times\mathbb{N}^{k}$ satisfying $\left\vert \nu\right\vert =-\left\vert
\tau\right\vert $ and $\nu+\tau=\gamma$).

We are in one of the following three cases:

\textit{Case 1:} We have $\left(  \gamma_{2},\gamma_{3},\ldots,\gamma
_{k}\right)  \neq\left(  0,0,\ldots,0\right)  $.

\textit{Case 2:} We have $\left(  \gamma_{2},\gamma_{3},\ldots,\gamma
_{k}\right)  =\left(  0,0,\ldots,0\right)  $ and $\gamma_{1}\neq0$.

\textit{Case 3:} We have $\left(  \gamma_{2},\gamma_{3},\ldots,\gamma
_{k}\right)  =\left(  0,0,\ldots,0\right)  $ and $\gamma_{1}=0$.

Let us first consider Case 1. In this case, we have $\left(  \gamma_{2}%
,\gamma_{3},\ldots,\gamma_{k}\right)  \neq\left(  0,0,\ldots,0\right)  $. In
other words, there exists a $j\in\left\{  2,3,\ldots,k\right\}  $ such that
$\gamma_{j}\neq0$. Consider such a $j$. Clearly, $\gamma\neq\mathbf{0}$ (since
$\gamma_{j}\neq0$). Hence, $%
\begin{cases}
1, & \text{if }\gamma=\mathbf{0};\\
0, & \text{if }\gamma\neq\mathbf{0}%
\end{cases}
=0$.

Let $\Delta$ be the vector $\left(  0,0,\ldots,0,1,0,0,\ldots,0\right)
\in\mathbb{Z}^{k}$, where $1$ is the $j$-th entry. Clearly, $\Delta
\in\mathbb{N}^{k}$ and $\left\vert \Delta\right\vert =1$.

Let $Q_{0}$ be the set of all $\left(  \tau,\nu\right)  \in Q$ satisfying
$\tau_{j}=0$. (Recall that $\tau_{j}$ denotes the $j$-th entry of the
$k$-tuple $\tau\in V\subseteq\mathbb{Z}^{k}$.) Let $Q_{1}$ be the set of all
$\left(  \tau,\nu\right)  \in Q$ satisfying $\tau_{j}=1$. Each $\left(
\tau,\nu\right)  \in Q$ satisfies $\left(  \tau,\nu\right)  \in V\times
\mathbb{N}^{k}$ (by the definition of $Q$) and thus $\tau\in V$ and thus
$\tau_{j}\in\left\{  0,1\right\}  $ (by (\ref{eq.def.rimhook.V.cond}), applied
to $i=j$). In other words, each $\left(  \tau,\nu\right)  \in Q$ satisfies
either $\tau_{j}=0$ or $\tau_{j}=1$ (but not both at the same time). In other
words, each $\left(  \tau,\nu\right)  \in Q$ belongs to either $Q_{0}$ or
$Q_{1}$ (but not both at the same time).

For each $\left(  \tau,\nu\right)  \in Q_{0}$, we have $\left(  \tau
+\Delta,\nu-\Delta\right)  \in Q_{1}$\ \ \ \ \footnote{\textit{Proof.} Let
$\left(  \tau,\nu\right)  \in Q_{0}$. According to the definition of $Q_{0}$,
this means that $\left(  \tau,\nu\right)  \in Q$ and $\tau_{j}=0$.
\par
We have $\left(  \tau,\nu\right)  \in Q$. According to the definition of $Q$,
this means that $\left(  \tau,\nu\right)  \in V\times\mathbb{N}^{k}$ and
$\left\vert \nu\right\vert =-\left\vert \tau\right\vert $ and $\nu+\tau
=\gamma$.
\par
From $\left(  \tau,\nu\right)  \in V\times\mathbb{N}^{k}$, we obtain $\tau\in
V$ and $\nu\in\mathbb{N}^{k}$.
\par
From $\nu+\tau=\gamma$, we obtain $\left(  \nu+\tau\right)  _{j}=\gamma_{j}$.
Hence, $\gamma_{j}=\left(  \nu+\tau\right)  _{j}=\nu_{j}+\underbrace{\tau_{j}%
}_{=0}=\nu_{j}$. Thus, $\nu_{j}=\gamma_{j}\neq0$. Thus, Lemma
\ref{lem.rimhook-S.sum0lem} \textbf{(c)} yields $\nu-\Delta\in\mathbb{N}^{k}$.
Also, Lemma \ref{lem.rimhook-S.sum0lem} \textbf{(a)} yields that $\tau
+\Delta\in V$ and $\left(  \tau+\Delta\right)  _{j}=1$. Also, any two
$k$-tuples $\alpha\in\mathbb{N}^{k}$ and $\beta\in\mathbb{N}^{k}$ satisfy
$\left\vert \alpha+\beta\right\vert =\left\vert \alpha\right\vert +\left\vert
\beta\right\vert $ and $\left\vert \alpha-\beta\right\vert =\left\vert
\alpha\right\vert -\left\vert \beta\right\vert $. Thus, $\left\vert
\tau+\Delta\right\vert =\left\vert \tau\right\vert +\left\vert \Delta
\right\vert $ and $\left\vert \nu-\Delta\right\vert =\underbrace{\left\vert
\nu\right\vert }_{=-\left\vert \tau\right\vert }-\left\vert \Delta\right\vert
=-\left\vert \tau\right\vert -\left\vert \Delta\right\vert
=-\underbrace{\left(  \left\vert \tau\right\vert +\left\vert \Delta\right\vert
\right)  }_{=\left\vert \tau+\Delta\right\vert }=-\left\vert \tau
+\Delta\right\vert $. Also, $\left(  \nu-\Delta\right)  +\left(  \tau
+\Delta\right)  =\nu+\tau=\gamma$.
\par
From $\tau+\Delta\in V$ and $\nu-\Delta\in\mathbb{N}^{k}$ and $\left\vert
\nu-\Delta\right\vert =-\left\vert \tau+\Delta\right\vert $ and $\left(
\nu-\Delta\right)  +\left(  \tau+\Delta\right)  =\gamma$, we obtain $\left(
\tau+\Delta,\nu-\Delta\right)  \in Q$ (by the definition of $Q$). Combining
this with $\left(  \tau+\Delta\right)  _{j}=1$, we obtain $\left(  \tau
+\Delta,\nu-\Delta\right)  \in Q_{1}$ (by the definition of $Q_{1}$), qed.}.
Thus, the map%
\begin{equation}
Q_{0}\rightarrow Q_{1},\ \ \ \ \ \ \ \ \ \ \left(  \tau,\nu\right)
\mapsto\left(  \tau+\Delta,\nu-\Delta\right)
\label{pf.lem.rimhook-S.sum0.map1}%
\end{equation}
is well-defined.

For each $\left(  \tau,\nu\right)  \in Q_{1}$, we have $\left(  \tau
-\Delta,\nu+\Delta\right)  \in Q_{0}$\ \ \ \ \footnote{\textit{Proof.} Let
$\left(  \tau,\nu\right)  \in Q_{1}$. According to the definition of $Q_{1}$,
this means that $\left(  \tau,\nu\right)  \in Q$ and $\tau_{j}=1$.
\par
We have $\left(  \tau,\nu\right)  \in Q$. According to the definition of $Q$,
this means that $\left(  \tau,\nu\right)  \in V\times\mathbb{N}^{k}$ and
$\left\vert \nu\right\vert =-\left\vert \tau\right\vert $ and $\nu+\tau
=\gamma$.
\par
From $\left(  \tau,\nu\right)  \in V\times\mathbb{N}^{k}$, we obtain $\tau\in
V$ and $\nu\in\mathbb{N}^{k}$.
\par
Lemma \ref{lem.rimhook-S.sum0lem} \textbf{(d)} yields $\nu+\Delta\in
\mathbb{N}^{k}$. Also, Lemma \ref{lem.rimhook-S.sum0lem} \textbf{(b)} yields
that $\tau-\Delta\in V$ and $\left(  \tau-\Delta\right)  _{j}=0$. Also, any
two $k$-tuples $\alpha\in\mathbb{N}^{k}$ and $\beta\in\mathbb{N}^{k}$ satisfy
$\left\vert \alpha+\beta\right\vert =\left\vert \alpha\right\vert +\left\vert
\beta\right\vert $ and $\left\vert \alpha-\beta\right\vert =\left\vert
\alpha\right\vert -\left\vert \beta\right\vert $. Thus, $\left\vert
\tau-\Delta\right\vert =\left\vert \tau\right\vert -\left\vert \Delta
\right\vert $ and $\left\vert \nu+\Delta\right\vert =\underbrace{\left\vert
\nu\right\vert }_{=-\left\vert \tau\right\vert }+\left\vert \Delta\right\vert
=-\left\vert \tau\right\vert +\left\vert \Delta\right\vert
=-\underbrace{\left(  \left\vert \tau\right\vert -\left\vert \Delta\right\vert
\right)  }_{=\left\vert \tau-\Delta\right\vert }=-\left\vert \tau
-\Delta\right\vert $. Also, $\left(  \nu+\Delta\right)  +\left(  \tau
-\Delta\right)  =\nu+\tau=\gamma$.
\par
From $\tau-\Delta\in V$ and $\nu+\Delta\in\mathbb{N}^{k}$ and $\left\vert
\nu+\Delta\right\vert =-\left\vert \tau-\Delta\right\vert $ and $\left(
\nu+\Delta\right)  +\left(  \tau-\Delta\right)  =\gamma$, we obtain $\left(
\tau-\Delta,\nu+\Delta\right)  \in Q$ (by the definition of $Q$). Combining
this with $\left(  \tau-\Delta\right)  _{j}=0$, we obtain $\left(  \tau
-\Delta,\nu+\Delta\right)  \in Q_{0}$ (by the definition of $Q_{0}$), qed.}.
Thus, the map%
\begin{equation}
Q_{1}\rightarrow Q_{0},\ \ \ \ \ \ \ \ \ \ \left(  \tau,\nu\right)
\mapsto\left(  \tau-\Delta,\nu+\Delta\right)
\label{pf.lem.rimhook-S.sum0.map2}%
\end{equation}
is well-defined.

The two maps (\ref{pf.lem.rimhook-S.sum0.map1}) and
(\ref{pf.lem.rimhook-S.sum0.map2}) are mutually inverse (this is clear from
their definitions), and thus are bijections. Hence, in particular, the map
(\ref{pf.lem.rimhook-S.sum0.map1}) is a bijection.

Also, each $\tau\in\mathbb{Z}^{k}$ satisfies
\begin{align*}
\left\vert \tau+\Delta\right\vert  &  =\left\vert \tau\right\vert
+\underbrace{\left\vert \Delta\right\vert }_{=1}\ \ \ \ \ \ \ \ \ \ \left(
\text{since }\left\vert \alpha+\beta\right\vert =\left\vert \alpha\right\vert
+\left\vert \beta\right\vert \text{ for all }\alpha\in\mathbb{Z}^{k}\text{ and
}\beta\in\mathbb{Z}^{k}\right) \\
&  =\left\vert \tau\right\vert +1
\end{align*}
and thus%
\begin{equation}
\left(  -1\right)  ^{n+\left\vert \tau+\Delta\right\vert }=\left(  -1\right)
^{n+\left\vert \tau\right\vert +1}=-\left(  -1\right)  ^{n+\left\vert
\tau\right\vert }. \label{pf.lem.rimhook-S.sum0.signs}%
\end{equation}

Now, recall that $Q_{0}$ and $Q_{1}$ are two subsets of $Q$ such that each
$\left(  \tau,\nu\right)  \in Q$ belongs to either $Q_{0}$ or $Q_{1}$ (but not
both at the same time). In other words, $Q_{0}$ and $Q_{1}$ are two disjoint
subsets of $Q$ whose union is the whole set $Q$. Hence, we can split the sum
$\sum_{\left(  \tau,\nu\right)  \in Q}\left(  -1\right)  ^{n+\left\vert
\tau\right\vert }$ as follows:%
\begin{align*}
\sum_{\left(  \tau,\nu\right)  \in Q}\left(  -1\right)  ^{n+\left\vert
\tau\right\vert }  &  =\sum_{\left(  \tau,\nu\right)  \in Q_{0}}\left(
-1\right)  ^{n+\left\vert \tau\right\vert }+\underbrace{\sum_{\left(  \tau
,\nu\right)  \in Q_{1}}\left(  -1\right)  ^{n+\left\vert \tau\right\vert }%
}_{\substack{=\sum_{\left(  \tau,\nu\right)  \in Q_{0}}\left(  -1\right)
^{n+\left\vert \tau+\Delta\right\vert }\\\text{(here, we have substituted
}\left(  \tau+\Delta,\nu-\Delta\right)  \text{ for }\left(  \tau,\nu\right)
\\\text{in the sum, since the map (\ref{pf.lem.rimhook-S.sum0.map1}) is a
bijection)}}}\\
&  =\sum_{\left(  \tau,\nu\right)  \in Q_{0}}\left(  -1\right)  ^{n+\left\vert
\tau\right\vert }+\sum_{\left(  \tau,\nu\right)  \in Q_{0}}\underbrace{\left(
-1\right)  ^{n+\left\vert \tau+\Delta\right\vert }}_{\substack{=-\left(
-1\right)  ^{n+\left\vert \tau\right\vert }\\\text{(by
(\ref{pf.lem.rimhook-S.sum0.signs}))}}}\\
&  =\sum_{\left(  \tau,\nu\right)  \in Q_{0}}\left(  -1\right)  ^{n+\left\vert
\tau\right\vert }+\sum_{\left(  \tau,\nu\right)  \in Q_{0}}\left(  -\left(
-1\right)  ^{n+\left\vert \tau\right\vert }\right) \\
&  =\sum_{\left(  \tau,\nu\right)  \in Q_{0}}\left(  -1\right)  ^{n+\left\vert
\tau\right\vert }-\sum_{\left(  \tau,\nu\right)  \in Q_{0}}\left(  -1\right)
^{n+\left\vert \tau\right\vert }=0.
\end{align*}

Now, (\ref{pf.lem.rimhook-S.sum0.sumeq}) yields%
\[
\sum_{\tau\in V}\sum_{\substack{\nu\in\mathbb{N}^{k};\\\left\vert
\nu\right\vert =-\left\vert \tau\right\vert ;\\\nu+\tau=\gamma}}\left(
-1\right)  ^{n+\left\vert \tau\right\vert }=\sum_{\left(  \tau,\nu\right)  \in
Q}\left(  -1\right)  ^{n+\left\vert \tau\right\vert }=0=%
\begin{cases}
1, & \text{if }\gamma=\mathbf{0};\\
0, & \text{if }\gamma\neq\mathbf{0}%
\end{cases}
\]
(since $%
\begin{cases}
1, & \text{if }\gamma=\mathbf{0};\\
0, & \text{if }\gamma\neq\mathbf{0}%
\end{cases}
=0$). Thus, Lemma \ref{lem.rimhook-S.sum0} is proven in Case 1.

Let us now consider Case 2. In this case, we have $\left(  \gamma_{2}%
,\gamma_{3},\ldots,\gamma_{k}\right)  =\left(  0,0,\ldots,0\right)  $ and
$\gamma_{1}\neq0$. From $\gamma_{1}\neq0$, we obtain $\gamma\neq\mathbf{0}$
and thus $%
\begin{cases}
1, & \text{if }\gamma=\mathbf{0};\\
0, & \text{if }\gamma\neq\mathbf{0}%
\end{cases}
=0$.

Now, $Q=\varnothing$\ \ \ \ \footnote{\textit{Proof.} Let $\left(  \tau
,\nu\right)  \in Q$. We shall derive a contradiction.
\par
Indeed, we have $\left(  \tau,\nu\right)  \in Q$. According to the definition
of $Q$, this means that $\left(  \tau,\nu\right)  \in V\times\mathbb{N}^{k}$
and $\left\vert \nu\right\vert =-\left\vert \tau\right\vert $ and $\nu
+\tau=\gamma$.
\par
From $\left(  \tau,\nu\right)  \in V\times\mathbb{N}^{k}$, we obtain $\tau\in
V$ and $\nu\in\mathbb{N}^{k}$.
\par
We have $\tau\in V$. According to the definition of $V$, this means that
$\tau$ is a $k$-tuple $\left(  -n,\tau_{2},\tau_{3},\ldots,\tau_{k}\right)
\in\mathbb{Z}^{k}$ satisfying (\ref{eq.def.rimhook.V.cond}). In other words,
$\tau\in\mathbb{Z}^{k}$ and $\tau_{1}=-n$ and the condition
(\ref{eq.def.rimhook.V.cond}) holds.
\par
Now, fix $j\in\left\{  2,3,\ldots,k\right\}  $. Then, $\tau_{j}\in\left\{
0,1\right\}  $ (by (\ref{eq.def.rimhook.V.cond}), applied to $i=j$). Hence,
$\tau_{j}\geq0$. Also, $\nu_{j}\in\mathbb{N}$ (since $\nu\in\mathbb{N}^{k}$),
so that $\nu_{j}\geq0$. But $\left(  \gamma_{2},\gamma_{3},\ldots,\gamma
_{k}\right)  =\left(  0,0,\ldots,0\right)  $, and thus $\gamma_{j}=0$ (since
$j\in\left\{  2,3,\ldots,k\right\}  $). But $\gamma=\nu+\tau$, and thus
$\gamma_{j}=\left(  \nu+\tau\right)  _{j}=\nu_{j}+\tau_{j}$. Hence, $\nu
_{j}+\tau_{j}=\gamma_{j}=0$, so that $\nu_{j}=-\underbrace{\tau_{j}}_{\geq
0}\leq0$. Combining this with $\nu_{j}\geq0$, we obtain $\nu_{j}=0$. Hence,
$\nu_{j}=-\tau_{j}$ rewrites as $0=-\tau_{j}$, so that $\tau_{j}=0$.
\par
Now, forget that we fixed $j$. Thus, we have shown that each $j\in\left\{
2,3,\ldots,k\right\}  $ satisfies
\begin{equation}
\nu_{j}=0 \label{pf.lem.rimhook-S.sum0.c2.fn1.1}%
\end{equation}
and
\begin{equation}
\tau_{j}=0. \label{pf.lem.rimhook-S.sum0.c2.fn1.2}%
\end{equation}
\par
Now,%
\[
\left\vert \tau\right\vert =\tau_{1}+\tau_{2}+\cdots+\tau_{k}=\sum_{j=1}%
^{k}\tau_{j}=\tau_{1}+\sum_{j=2}^{k}\underbrace{\tau_{j}}%
_{\substack{=0\\\text{(by (\ref{pf.lem.rimhook-S.sum0.c2.fn1.2}))}}}=\tau
_{1}=-n,
\]
so that $-\left\vert \tau\right\vert =n$. Furthermore,%
\[
\left\vert \nu\right\vert =\nu_{1}+\nu_{2}+\cdots+\nu_{k}=\sum_{j=1}^{k}%
\nu_{j}=\nu_{1}+\sum_{j=2}^{k}\underbrace{\nu_{j}}_{\substack{=0\\\text{(by
(\ref{pf.lem.rimhook-S.sum0.c2.fn1.1}))}}}=\nu_{1},
\]
so that $\nu_{1}=\left\vert \nu\right\vert =-\left\vert \tau\right\vert =n$.
\par
Now, from $\gamma=\nu+\tau$, we obtain $\gamma_{1}=\left(  \nu+\tau\right)
_{1}=\underbrace{\nu_{1}}_{=n}+\underbrace{\tau_{1}}_{=-n}=n+\left(
-n\right)  =0$. This contradicts $\gamma_{1}\neq0$.
\par
Now, forget that we fixed $\left(  \tau,\nu\right)  $. We thus have found a
contradiction for each $\left(  \tau,\nu\right)  \in Q$. Thus, there exists no
$\left(  \tau,\nu\right)  \in Q$. In other words, $Q=\varnothing$.}. But
(\ref{pf.lem.rimhook-S.sum0.sumeq}) yields%
\begin{align*}
\sum_{\tau\in V}\sum_{\substack{\nu\in\mathbb{N}^{k};\\\left\vert
\nu\right\vert =-\left\vert \tau\right\vert ;\\\nu+\tau=\gamma}}\left(
-1\right)  ^{n+\left\vert \tau\right\vert }  &  =\sum_{\left(  \tau
,\nu\right)  \in Q}\left(  -1\right)  ^{n+\left\vert \tau\right\vert }=\left(
\text{empty sum}\right)  \ \ \ \ \ \ \ \ \ \ \left(  \text{since
}Q=\varnothing\right) \\
&  =0=%
\begin{cases}
1, & \text{if }\gamma=\mathbf{0};\\
0, & \text{if }\gamma\neq\mathbf{0}%
\end{cases}
\end{align*}
(since $%
\begin{cases}
1, & \text{if }\gamma=\mathbf{0};\\
0, & \text{if }\gamma\neq\mathbf{0}%
\end{cases}
=0$). Thus, Lemma \ref{lem.rimhook-S.sum0} is proven in Case 2.

Let us finally consider Case 3. In this case, we have $\left(  \gamma
_{2},\gamma_{3},\ldots,\gamma_{k}\right)  =\left(  0,0,\ldots,0\right)  $ and
$\gamma_{1}=0$. Combining these two equalities, we obtain $\gamma_{i}=0$ for
all $i\in\left\{  1,2,\ldots,k\right\}  $. In other words, $\gamma=\mathbf{0}%
$. Hence, $%
\begin{cases}
1, & \text{if }\gamma=\mathbf{0};\\
0, & \text{if }\gamma\neq\mathbf{0}%
\end{cases}
=1$.

Now, define two $k$-tuples $\tau_{0}\in\mathbb{Z}^{k}$ and $\nu_{0}%
\in\mathbb{Z}^{k}$ by%
\[
\tau_{0}=\left(  -n,0,0,\ldots,0\right)  \ \ \ \ \ \ \ \ \ \ \text{and}%
\ \ \ \ \ \ \ \ \ \ \nu_{0}=\left(  n,0,0,\ldots,0\right)  .
\]
Clearly, $\tau_{0}\in V$ (by the definition of $V$) and $\nu_{0}\in
\mathbb{N}^{k}$ and $\left\vert \tau_{0}\right\vert =-n$ and $\left\vert
\nu_{0}\right\vert =n$ and $\nu_{0}+\tau_{0}=\mathbf{0}$.

From $\tau_{0}\in V$ and $\nu_{0}\in\mathbb{N}^{k}$, we obtain $\left(
\tau_{0},\nu_{0}\right)  \in V\times\mathbb{N}^{k}$. Also, $\left\vert \nu
_{0}\right\vert =-\left\vert \tau_{0}\right\vert $ (since
$\underbrace{\left\vert \nu_{0}\right\vert }_{=n}+\underbrace{\left\vert
\tau_{0}\right\vert }_{=-n}=n+\left(  -n\right)  =0$) and $\nu_{0}+\tau
_{0}=\mathbf{0}=\gamma$. Thus, we have shown that $\left(  \tau_{0},\nu
_{0}\right)  \in V\times\mathbb{N}^{k}$ and $\left\vert \nu_{0}\right\vert
=-\left\vert \tau_{0}\right\vert $ and $\nu_{0}+\tau_{0}=\gamma$. In other
words, $\left(  \tau_{0},\nu_{0}\right)  \in Q$ (by the definition of $Q$). In
other words, $\left\{  \left(  \tau_{0},\nu_{0}\right)  \right\}  \subseteq Q$.

On the other hand, $Q\subseteq\left\{  \left(  \tau_{0},\nu_{0}\right)
\right\}  $\ \ \ \ \footnote{\textit{Proof.} Let $\left(  \tau,\nu\right)  \in
Q$. We shall prove that $\left(  \tau,\nu\right)  =\left(  \tau_{0},\nu
_{0}\right)  $.
\par
Most of the following argument is copypasted from the previous footnote.
\par
We have $\left(  \tau,\nu\right)  \in Q$. According to the definition of $Q$,
this means that $\left(  \tau,\nu\right)  \in V\times\mathbb{N}^{k}$ and
$\left\vert \nu\right\vert =-\left\vert \tau\right\vert $ and $\nu+\tau
=\gamma$.
\par
From $\left(  \tau,\nu\right)  \in V\times\mathbb{N}^{k}$, we obtain $\tau\in
V$ and $\nu\in\mathbb{N}^{k}$.
\par
We have $\tau\in V$. According to the definition of $V$, this means that
$\tau$ is a $k$-tuple $\left(  -n,\tau_{2},\tau_{3},\ldots,\tau_{k}\right)
\in\mathbb{Z}^{k}$ satisfying (\ref{eq.def.rimhook.V.cond}). In other words,
$\tau\in\mathbb{Z}^{k}$ and $\tau_{1}=-n$ and the condition
(\ref{eq.def.rimhook.V.cond}) holds.
\par
Now, fix $j\in\left\{  2,3,\ldots,k\right\}  $. Then, $\tau_{j}\in\left\{
0,1\right\}  $ (by (\ref{eq.def.rimhook.V.cond}), applied to $i=j$). Hence,
$\tau_{j}\geq0$. Also, $\nu_{j}\in\mathbb{N}$ (since $\nu\in\mathbb{N}^{k}$),
so that $\nu_{j}\geq0$. But $\left(  \gamma_{2},\gamma_{3},\ldots,\gamma
_{k}\right)  =\left(  0,0,\ldots,0\right)  $, and thus $\gamma_{j}=0$ (since
$j\in\left\{  2,3,\ldots,k\right\}  $). But $\gamma=\nu+\tau$, and thus
$\gamma_{j}=\left(  \nu+\tau\right)  _{j}=\nu_{j}+\tau_{j}$. Hence, $\nu
_{j}+\tau_{j}=\gamma_{j}=0$, so that $\nu_{j}=-\underbrace{\tau_{j}}_{\geq
0}\leq0$. Combining this with $\nu_{j}\geq0$, we obtain $\nu_{j}=0$. Hence,
$\nu_{j}=-\tau_{j}$ rewrites as $0=-\tau_{j}$, so that $\tau_{j}=0$.
\par
Now, forget that we fixed $j$. Thus, we have shown that each $j\in\left\{
2,3,\ldots,k\right\}  $ satisfies $\tau_{j}=0$. In other words, $\left(
\tau_{2},\tau_{3},\ldots,\tau_{k}\right)  =\left(  0,0,\ldots,0\right)  $.
Combining this with $\tau_{1}=-n$, we obtain $\tau=\left(  -n,0,0,\ldots
,0\right)  =\tau_{0}$.
\par
From $\nu+\tau=\gamma$, we obtain $\nu=\gamma-\underbrace{\tau}_{=\tau_{0}%
}=\gamma-\tau_{0}=\nu_{0}$ (since $\nu_{0}+\tau_{0}=\gamma$). Combining this
with $\tau=\tau_{0}$, we obtain $\left(  \tau,\nu\right)  =\left(  \tau
_{0},\nu_{0}\right)  \in\left\{  \left(  \tau_{0},\nu_{0}\right)  \right\}  $.
\par
Now, forget that we fixed $\left(  \tau,\nu\right)  $. We thus have proven
that $\left(  \tau,\nu\right)  \in\left\{  \left(  \tau_{0},\nu_{0}\right)
\right\}  $ for each $\left(  \tau,\nu\right)  \in Q$. In other words,
$Q\subseteq\left\{  \left(  \tau_{0},\nu_{0}\right)  \right\}  $.}. Combining
this with $\left\{  \left(  \tau_{0},\nu_{0}\right)  \right\}  \subseteq Q$,
we obtain $Q=\left\{  \left(  \tau_{0},\nu_{0}\right)  \right\}  $.

But (\ref{pf.lem.rimhook-S.sum0.sumeq}) yields%
\begin{align*}
\sum_{\tau\in V}\sum_{\substack{\nu\in\mathbb{N}^{k};\\\left\vert
\nu\right\vert =-\left\vert \tau\right\vert ;\\\nu+\tau=\gamma}}\left(
-1\right)  ^{n+\left\vert \tau\right\vert }  &  =\sum_{\left(  \tau
,\nu\right)  \in Q}\left(  -1\right)  ^{n+\left\vert \tau\right\vert }=\left(
-1\right)  ^{n+\left\vert \tau_{0}\right\vert }\ \ \ \ \ \ \ \ \ \ \left(
\text{since }Q=\left\{  \left(  \tau_{0},\nu_{0}\right)  \right\}  \right) \\
&  =\left(  -1\right)  ^{0}\ \ \ \ \ \ \ \ \ \ \left(  \text{since
}n+\underbrace{\left\vert \tau_{0}\right\vert }_{=-n}=n+\left(  -n\right)
=0\right) \\
&  =1=%
\begin{cases}
1, & \text{if }\gamma=\mathbf{0};\\
0, & \text{if }\gamma\neq\mathbf{0}%
\end{cases}
\end{align*}
(since $%
\begin{cases}
1, & \text{if }\gamma=\mathbf{0};\\
0, & \text{if }\gamma\neq\mathbf{0}%
\end{cases}
=1$). Thus, Lemma \ref{lem.rimhook-S.sum0} is proven in Case 3.

We have now proven Lemma \ref{lem.rimhook-S.sum0} in each of the three Cases
1, 2 and 3. Hence, Lemma \ref{lem.rimhook-S.sum0} always holds.
\end{proof}

\begin{proof}
[Proof of Theorem \ref{thm.rimhook-S}.]Each $\tau\in V$ satisfies $-\left\vert
\tau\right\vert \in\left\{  n-k+1,n-k+2,\ldots,n\right\}  $ (by Proposition
\ref{prop.rimhook.V.size}). Thus, we have the following equality of summation
signs:%
\begin{equation}
\sum_{\tau\in V}=\sum_{i=n-k+1}^{n}\sum_{\substack{\tau\in V;\\-\left\vert
\tau\right\vert =i}}=\sum_{j=1}^{k}\sum_{\substack{\tau\in V;\\-\left\vert
\tau\right\vert =n-k+j}} \label{pf.thm.rimhook-S.sum=sumsum}%
\end{equation}
(here, we have substituted $n-k+j$ for $i$ in the outer sum). Now,%
\begin{align}
&  \sum_{j=1}^{k}\left(  -1\right)  ^{k-j}h_{n-k+j}\sum_{\substack{\tau\in
V;\\-\left\vert \tau\right\vert =n-k+j}}s_{\mu+\tau}\nonumber\\
&  =\underbrace{\sum_{j=1}^{k}\sum_{\substack{\tau\in V;\\-\left\vert
\tau\right\vert =n-k+j}}}_{\substack{=\sum_{\tau\in V}\\\text{(by
(\ref{pf.thm.rimhook-S.sum=sumsum}))}}}\underbrace{\left(  -1\right)  ^{k-j}%
}_{\substack{=\left(  -1\right)  ^{n+\left\vert \tau\right\vert }%
\\\text{(since }k-j=n+\left\vert \tau\right\vert \\\text{(because }-\left\vert
\tau\right\vert =n-k+j\text{))}}}\underbrace{h_{n-k+j}}%
_{\substack{=h_{-\left\vert \tau\right\vert }\\\text{(since }n-k+j=-\left\vert
\tau\right\vert \\\text{(because }-\left\vert \tau\right\vert =n-k+j\text{))}%
}}s_{\mu+\tau}\nonumber\\
&  =\sum_{\tau\in V}\left(  -1\right)  ^{n+\left\vert \tau\right\vert
}h_{-\left\vert \tau\right\vert }s_{\mu+\tau}. \label{pf.thm.rimhook-S.2}%
\end{align}

But each $\tau\in V$ satisfies%
\begin{equation}
h_{-\left\vert \tau\right\vert }s_{\mu+\tau}=\sum_{\substack{\nu\in
\mathbb{N}^{k};\\\left\vert \nu\right\vert =-\left\vert \tau\right\vert
}}s_{\mu+\left(  \nu+\tau\right)  }. \label{pf.thm.rimhook-S.pieri}%
\end{equation}

[\textit{Proof of (\ref{pf.thm.rimhook-S.pieri}):} Let $\tau\in V$. According
to the definition of $V$, this means that $\tau$ is a $k$-tuple $\left(
-n,\tau_{2},\tau_{3},\ldots,\tau_{k}\right)  \in\mathbb{Z}^{k}$ satisfying
(\ref{eq.def.rimhook.V.cond}). In other words, $\tau\in\mathbb{Z}^{k}$ and
$\tau_{1}=-n$ and the relation (\ref{eq.def.rimhook.V.cond}) holds.

Proposition \ref{prop.rimhook.V.size} yields $-\left\vert \tau\right\vert
\in\left\{  n-k+1,n-k+2,\ldots,n\right\}  \subseteq\mathbb{N}$.

Also, $\mu\in P_{k}\subseteq\mathbb{N}^{k}$; hence,
\begin{equation}
\mu_{i}\geq0\ \ \ \ \ \ \ \ \ \ \text{for each }i\in\left\{  1,2,\ldots
,k\right\}  . \label{pf.thm.rimhook-S.pieri.pf.1}%
\end{equation}

Also, $\rho_{1}=k-1$ (by the definition of $\rho$) and $\rho\in\mathbb{N}^{k}$
(likewise). Now,%
\[
\left(  \mu+\tau+\rho\right)  _{1}=\underbrace{\mu_{1}}_{>n-k}%
+\underbrace{\tau_{1}}_{=-n}+\underbrace{\rho_{1}}_{=k-1}>\left(  n-k\right)
+\left(  -n\right)  +\left(  k-1\right)  =-1.
\]
Thus, $\left(  \mu+\tau+\rho\right)  _{1}\geq0$ (since $\left(  \mu+\tau
+\rho\right)  _{1}$ is an integer). In other words, $\left(  \mu+\tau
+\rho\right)  _{1}\in\mathbb{N}$. Furthermore, for each $i\in\left\{
2,3,\ldots,k\right\}  $, we have%
\[
\left(  \mu+\tau+\rho\right)  _{i}=\underbrace{\mu_{i}}_{\substack{\in
\mathbb{N}\\\text{(since }\mu\in\mathbb{N}^{k}\text{)}}}+\underbrace{\tau_{i}%
}_{\substack{\in\mathbb{N}\\\text{(since (\ref{eq.def.rimhook.V.cond}%
)}\\\text{yields }\tau_{i}\in\left\{  0,1\right\}  \subseteq\mathbb{N}%
\text{)}}}+\underbrace{\rho_{i}}_{\substack{\in\mathbb{N}\\\text{(since }%
\rho\in\mathbb{N}^{k}\text{)}}}\in\mathbb{N}.
\]
This also holds for $i=1$ (since $\left(  \mu+\tau+\rho\right)  _{1}%
\in\mathbb{N}$). Thus, we have $\left(  \mu+\tau+\rho\right)  _{i}%
\in\mathbb{N}$ for each $i\in\left\{  1,2,\ldots,k\right\}  $. In other words,
$\mu+\tau+\rho\in\mathbb{N}^{k}$. Hence, Theorem \ref{thm.non-part.pieri-h}
(applied to $\lambda=\mu+\tau$ and $m=-\left\vert \tau\right\vert $) yields%
\[
s_{\mu+\tau}h_{-\left\vert \tau\right\vert }=\sum_{\substack{\nu\in
\mathbb{N}^{k};\\\left\vert \nu\right\vert =-\left\vert \tau\right\vert
}}\underbrace{s_{\mu+\tau+\nu}}_{=s_{\mu+\left(  \nu+\tau\right)  }}%
=\sum_{\substack{\nu\in\mathbb{N}^{k};\\\left\vert \nu\right\vert =-\left\vert
\tau\right\vert }}s_{\mu+\left(  \nu+\tau\right)  }.
\]
Thus,
\[
h_{-\left\vert \tau\right\vert }s_{\mu+\tau}=s_{\mu+\tau}h_{-\left\vert
\tau\right\vert }=\sum_{\substack{\nu\in\mathbb{N}^{k};\\\left\vert
\nu\right\vert =-\left\vert \tau\right\vert }}s_{\mu+\left(  \nu+\tau\right)
}.
\]
This proves (\ref{pf.thm.rimhook-S.pieri}).]

Now, (\ref{pf.thm.rimhook-S.2}) becomes%
\begin{align*}
&  \sum_{j=1}^{k}\left(  -1\right)  ^{k-j}h_{n-k+j}\sum_{\substack{\tau\in
V;\\-\left\vert \tau\right\vert =n-k+j}}s_{\mu+\tau}\\
&  =\sum_{\tau\in V}\left(  -1\right)  ^{n+\left\vert \tau\right\vert
}\underbrace{h_{-\left\vert \tau\right\vert }s_{\mu+\tau}}_{\substack{=\sum
_{\substack{\nu\in\mathbb{N}^{k};\\\left\vert \nu\right\vert =-\left\vert
\tau\right\vert }}s_{\mu+\left(  \nu+\tau\right)  }\\\text{(by
(\ref{pf.thm.rimhook-S.pieri}))}}}\\
&  =\sum_{\tau\in V}\left(  -1\right)  ^{n+\left\vert \tau\right\vert
}\underbrace{\sum_{\substack{\nu\in\mathbb{N}^{k};\\\left\vert \nu\right\vert
=-\left\vert \tau\right\vert }}}_{=\sum_{\gamma\in\mathbb{Z}^{k}}%
\sum_{\substack{\nu\in\mathbb{N}^{k};\\\left\vert \nu\right\vert =-\left\vert
\tau\right\vert ;\\\nu+\tau=\gamma}}}s_{\mu+\left(  \nu+\tau\right)  }%
=\sum_{\tau\in V}\left(  -1\right)  ^{n+\left\vert \tau\right\vert }%
\sum_{\gamma\in\mathbb{Z}^{k}}\sum_{\substack{\nu\in\mathbb{N}^{k}%
;\\\left\vert \nu\right\vert =-\left\vert \tau\right\vert ;\\\nu+\tau=\gamma
}}\underbrace{s_{\mu+\left(  \nu+\tau\right)  }}_{\substack{=s_{\mu+\gamma
}\\\text{(since }\nu+\tau=\gamma\text{)}}}\\
&  =\sum_{\tau\in V}\left(  -1\right)  ^{n+\left\vert \tau\right\vert }%
\sum_{\gamma\in\mathbb{Z}^{k}}\sum_{\substack{\nu\in\mathbb{N}^{k}%
;\\\left\vert \nu\right\vert =-\left\vert \tau\right\vert ;\\\nu+\tau=\gamma
}}s_{\mu+\gamma}=\sum_{\gamma\in\mathbb{Z}^{k}}\underbrace{\left(  \sum
_{\tau\in V}\sum_{\substack{\nu\in\mathbb{N}^{k};\\\left\vert \nu\right\vert
=-\left\vert \tau\right\vert ;\\\nu+\tau=\gamma}}\left(  -1\right)
^{n+\left\vert \tau\right\vert }\right)  }_{\substack{=%
\begin{cases}
1, & \text{if }\gamma=\mathbf{0};\\
0, & \text{if }\gamma\neq\mathbf{0}%
\end{cases}
\\\text{(by Lemma \ref{lem.rimhook-S.sum0})}}}s_{\mu+\gamma}\\
&  =\sum_{\gamma\in\mathbb{Z}^{k}}%
\begin{cases}
1, & \text{if }\gamma=\mathbf{0};\\
0, & \text{if }\gamma\neq\mathbf{0}%
\end{cases}
s_{\mu+\gamma}=\underbrace{%
\begin{cases}
1, & \text{if }\mathbf{0}=\mathbf{0};\\
0, & \text{if }\mathbf{0}\neq\mathbf{0}%
\end{cases}
}_{\substack{=1\\\text{(since }\mathbf{0}=\mathbf{0}\text{)}}}s_{\mu
+\mathbf{0}}+\sum_{\substack{\gamma\in\mathbb{Z}^{k};\\\gamma\neq\mathbf{0}%
}}\underbrace{%
\begin{cases}
1, & \text{if }\gamma=\mathbf{0};\\
0, & \text{if }\gamma\neq\mathbf{0}%
\end{cases}
}_{\substack{=0\\\text{(since }\gamma\neq\mathbf{0}\text{)}}}s_{\mu+\gamma}\\
&  \ \ \ \ \ \ \ \ \ \ \left(  \text{here, we have split off the addend for
}\gamma=\mathbf{0}\text{ from the sum}\right) \\
&  =s_{\mu+\mathbf{0}}=s_{\mu}.
\end{align*}
This proves Theorem \ref{thm.rimhook-S}.
\end{proof}

\begin{proof}
[Proof of Theorem \ref{thm.rimhook}.]Theorem \ref{thm.rimhook-S} yields%
\[
s_{\mu}=\sum_{j=1}^{k}\left(  -1\right)  ^{k-j}\underbrace{h_{n-k+j}%
}_{\substack{\equiv a_{j}\operatorname{mod}I\\\text{(by (\ref{eq.h=amodI}))}%
}}\sum_{\substack{\tau\in V;\\-\left\vert \tau\right\vert =n-k+j}}s_{\mu+\tau
}\equiv\sum_{j=1}^{k}\left(  -1\right)  ^{k-j}a_{j}\sum_{\substack{\tau\in
V;\\-\left\vert \tau\right\vert =n-k+j}}s_{\mu+\tau}\operatorname{mod}I.
\]
Thus, in $\mathcal{S}/I$, we have%
\[
\overline{s_{\mu}}=\overline{\sum_{j=1}^{k}\left(  -1\right)  ^{k-j}a_{j}%
\sum_{\substack{\tau\in V;\\-\left\vert \tau\right\vert =n-k+j}}s_{\mu+\tau}%
}=\sum_{j=1}^{k}\left(  -1\right)  ^{k-j}a_{j}\sum_{\substack{\tau\in
V;\\-\left\vert \tau\right\vert =n-k+j}}\overline{s_{\mu+\tau}}.
\]
This proves Theorem \ref{thm.rimhook}.
\end{proof}

\section{\label{sect.bigquot}Deforming symmetric functions}

\subsection{The basis theorem}

\begin{convention}
Let $R$ be any commutative ring. Let $\left(  a_{1},a_{2},\ldots,a_{p}\right)
$ be any list of elements of $R$. Then, $\left\langle a_{1},a_{2},\ldots
,a_{p}\right\rangle _{R}$ shall denote the ideal of $R$ generated by these
elements $a_{1},a_{2},\ldots,a_{p}$. When it is clear from the context what
$R$ is, we will simply write $\left\langle a_{1},a_{2},\ldots,a_{p}%
\right\rangle $ for this ideal (thus omitting the mention of $R$); for
example, when we write \textquotedblleft$R/\left\langle a_{1},a_{2}%
,\ldots,a_{p}\right\rangle $\textquotedblright, we will always mean
$R/\left\langle a_{1},a_{2},\ldots,a_{p}\right\rangle _{R}$.
\end{convention}

We have so far studied a quotient $\mathcal{S}/I$ of the ring $\mathcal{S}$ of
symmetric polynomials in $k$ variables $x_{1},x_{2},\ldots,x_{k}$. But
$\mathcal{S}$ itself is a quotient of a larger ring -- the ring $\Lambda$ of
symmetric functions in infinitely many variables. More precisely,%
\[
\mathcal{S}\cong\Lambda/\left\langle \mathbf{e}_{k+1},\mathbf{e}%
_{k+2},\mathbf{e}_{k+3},\ldots\right\rangle
\]
(and the canonical $\mathbf{k}$-algebra isomorphism $\mathcal{S}%
\rightarrow\Lambda/\left\langle \mathbf{e}_{k+1},\mathbf{e}_{k+2}%
,\mathbf{e}_{k+3},\ldots\right\rangle $ sends $\mathbf{e}_{1},\mathbf{e}%
_{2},\mathbf{e}_{3},\ldots$ to $e_{1},e_{2},\ldots,e_{k},0,0,0,\ldots$).
Hence, at least when $a_{1},a_{2},\ldots,a_{k}\in\mathbf{k}$, we have%
\[
\mathcal{S}/I\cong\Lambda/\left(  \left\langle \mathbf{h}_{n-k+1}%
-a_{1},\mathbf{h}_{n-k+2}-a_{2},\ldots,\mathbf{h}_{n}-a_{k}\right\rangle
+\left\langle \mathbf{e}_{k+1},\mathbf{e}_{k+2},\mathbf{e}_{k+3}%
,\ldots\right\rangle \right)  .
\]
If $a_{1},a_{2},\ldots,a_{k}$ are themselves elements of $\mathcal{S}$, then
we need to lift them to elements $\mathbf{a}_{1},\mathbf{a}_{2},\ldots
,\mathbf{a}_{k}$ of $\Lambda$ in order for such an isomorphism to hold.

This suggests a further generalization: What if we replace $\mathbf{e}%
_{k+1},\mathbf{e}_{k+2},\mathbf{e}_{k+3},\ldots$ by $\mathbf{e}_{k+1}%
-\mathbf{b}_{1},\mathbf{e}_{k+2}-\mathbf{b}_{2},\mathbf{e}_{k+3}%
-\mathbf{b}_{3},\ldots$ for some $\mathbf{b}_{1},\mathbf{b}_{2},\mathbf{b}%
_{3},\ldots\in\Lambda$ ? Let us take a look at this generalization:

\begin{definition}
\label{def.bigquot.setup}Throughout Section \ref{sect.bigquot}, we shall use
the following notations:

Let $\Lambda$ be the ring of symmetric functions in infinitely many
indeterminates $x_{1},x_{2},x_{3},\ldots$ over $\mathbf{k}$. (See
\cite[Chapter 2]{GriRei18} for more about this ring $\Lambda$.) Let
$\mathbf{e}_{m}$ and $\mathbf{h}_{m}$ be the elementary symmetric functions
and the complete homogeneous symmetric functions in $\Lambda$. For each
partition $\lambda$, let $\mathbf{s}_{\lambda}$ be the Schur function in
$\Lambda$ corresponding to $\lambda$.

For each $i\in\left\{  1,2,\ldots,k\right\}  $, let $\mathbf{a}_{i}$ be an
element of $\Lambda$ with degree $<n-k+i$.

For each $i\in\left\{  1,2,3,\ldots\right\}  $, let $\mathbf{b}_{i}$ be an
element of $\Lambda$ with degree $<k+i$.

Let $K$ be the ideal%
\[
\left\langle \mathbf{h}_{n-k+1}-\mathbf{a}_{1},\mathbf{h}_{n-k+2}%
-\mathbf{a}_{2},\ldots,\mathbf{h}_{n}-\mathbf{a}_{k}\right\rangle
+\left\langle \mathbf{e}_{k+1}-\mathbf{b}_{1},\mathbf{e}_{k+2}-\mathbf{b}%
_{2},\mathbf{e}_{k+3}-\mathbf{b}_{3},\ldots\right\rangle
\]
of $\Lambda$. For each $f\in\Lambda$, we let $\overline{f}$ denote the
projection of $f$ onto the quotient $\Lambda/K$.
\end{definition}

\begin{theorem}
\label{thm.bigquot}The $\mathbf{k}$-module $\Lambda/K$ is a free $\mathbf{k}%
$-module with basis $\left(  \overline{\mathbf{s}_{\lambda}}\right)
_{\lambda\in P_{k,n}}$.
\end{theorem}

\subsection{Spanning}

Proving Theorem \ref{thm.bigquot} will take us a while. We start with some
easy observations:

\begin{itemize}
\item For each $i\in\left\{  1,2,\ldots,k\right\}  $, we have%
\begin{equation}
\mathbf{a}_{i}=\left(  \text{some symmetric function of degree }<n-k+i\right)
. \label{eq.thm.bigquot.ai-deg}%
\end{equation}
(This follows from the definition of $\mathbf{a}_{i}$.)

\item For each $i\in\left\{  1,2,3,\ldots\right\}  $, we have%
\begin{equation}
\mathbf{b}_{i}=\left(  \text{some symmetric function of degree }<k+i\right)  .
\label{eq.thm.bigquot.bi-deg}%
\end{equation}
(This follows from the definition of $\mathbf{b}_{i}$.)

\item For each $i\in\left\{  1,2,3,\ldots\right\}  $, we have%
\begin{equation}
\mathbf{e}_{k+i}-\mathbf{b}_{i}\in K \label{eq.thm.bigquot.ek+iinK}%
\end{equation}
(because of how $K$ was defined). In other words, for each $i\in\left\{
1,2,3,\ldots\right\}  $, we have%
\begin{equation}
\mathbf{e}_{k+i}\equiv\mathbf{b}_{i}\operatorname{mod}K.
\label{eq.thm.bigquot.ek+imodK}%
\end{equation}
Substituting $j-k$ for $i$ in this statement, we obtain the following: For
each $j\in\left\{  k+1,k+2,k+3,\ldots\right\}  $, we have%
\begin{equation}
\mathbf{e}_{j}\equiv\mathbf{b}_{j-k}\operatorname{mod}K.
\label{eq.thm.bigquot.ejmodK}%
\end{equation}

\item For each $i\in\left\{  1,2,\ldots,k\right\}  $, we have%
\begin{equation}
\mathbf{h}_{n-k+i}-\mathbf{a}_{i}\in K \label{eq.thm.bigquot.hn-k+iinK}%
\end{equation}
(because of how $K$ was defined). In other words, for each $i\in\left\{
1,2,\ldots,k\right\}  $, we have%
\begin{equation}
\mathbf{h}_{n-k+i}\equiv\mathbf{a}_{i}\operatorname{mod}K.
\label{eq.thm.bigquot.hn-k+imodK}%
\end{equation}
Substituting $j-\left(  n-k\right)  $ for $i$ in this statement, we obtain the
following: For each $j\in\left\{  n-k+1,n-k+2,\ldots,n\right\}  $, we have%
\begin{equation}
\mathbf{h}_{j}\equiv\mathbf{a}_{j-\left(  n-k\right)  }\operatorname{mod}K.
\label{eq.thm.bigquot.hjmodK}%
\end{equation}

\end{itemize}

Let $\operatorname*{Par}$ denote the set of all partitions.

For each $m\in\mathbb{Z}$, we let $\Lambda_{\deg\leq m}$ denote the
$\mathbf{k}$-submodule of $\Lambda$ that consists of all symmetric functions
$f\in\Lambda$ of degree $\leq m$. Thus, $\left(  \Lambda_{\deg\leq m}\right)
_{m\in\mathbb{N}}$ is a filtration of the $\mathbf{k}$-algebra $\Lambda$. In
particular, $1\in\Lambda_{\deg\leq0}$ and
\begin{equation}
\Lambda_{\deg\leq i}\Lambda_{\deg\leq j}\subseteq\Lambda_{\deg\leq
i+j}\ \ \ \ \ \ \ \ \ \ \text{for all }i,j\in\mathbb{N}.
\label{eq.thm.bigquot.Lamfil}%
\end{equation}
Also, $\Lambda_{\deg\leq m}$ is the $\mathbf{k}$-submodule $0$ of $\Lambda$
whenever $m\in\mathbb{Z}$ is negative; thus, in particular, $\Lambda_{\deg
\leq-1}=0$.

We state an analogue of Lemma \ref{lem.I.cofactor1}:

\begin{lemma}
\label{lem.K.cofactor1}Let $\lambda=\left(  \lambda_{1},\lambda_{2}%
,\ldots,\lambda_{\ell}\right)  $ be any partition. Let $i\in\left\{
1,2,\ldots,\ell\right\}  $ and $j\in\left\{  1,2,\ldots,\ell\right\}  $. Then:

\textbf{(a)} The $\left(  i,j\right)  $-th cofactor of the matrix $\left(
\mathbf{h}_{\lambda_{u}-u+v}\right)  _{1\leq u\leq\ell,\ 1\leq v\leq\ell}$ is
a homogeneous element of $\Lambda$ of degree $\left\vert \lambda\right\vert
-\left(  \lambda_{i}-i+j\right)  $.

\textbf{(b)} The $\left(  i,j\right)  $-th cofactor of the matrix $\left(
\mathbf{e}_{\lambda_{u}-u+v}\right)  _{1\leq u\leq\ell,\ 1\leq v\leq\ell}$ is
a homogeneous element of $\Lambda$ of degree $\left\vert \lambda\right\vert
-\left(  \lambda_{i}-i+j\right)  $.
\end{lemma}

\begin{proof}
[Proof of Lemma \ref{lem.K.cofactor1}.]Each of the two parts of Lemma
\ref{lem.K.cofactor1} is proven in the same way as Lemma \ref{lem.I.cofactor1}%
, with the obvious modifications to the argument (viz., replacing
$\mathcal{S}$ by $\Lambda$, and replacing $h_{m}$ by $\mathbf{h}_{m}$ or by
$\mathbf{e}_{m}$).
\end{proof}

Next, we claim a lemma that will yield one half of Theorem \ref{thm.bigquot}
(namely, that the family $\left(  \overline{\mathbf{s}_{\lambda}}\right)
_{\lambda\in P_{k,n}}$ spans the $\mathbf{k}$-module $\Lambda/K$):

\begin{lemma}
\label{lem.bigquot.sl-red}Let $\lambda$ be a partition such that
$\lambda\notin P_{k,n}$. Then,%
\[
\mathbf{s}_{\lambda}\equiv\left(  \text{some symmetric function of degree
}<\left\vert \lambda\right\vert \right)  \operatorname{mod}K.
\]

\end{lemma}

We will not prove Lemma \ref{lem.bigquot.sl-red} immediately; instead, let us
show a weakening of it first:

\begin{lemma}
\label{lem.bigquot.sl-red-1}Let $\lambda$ be a partition such that
$\lambda\notin P_{k}$. Then,%
\[
\mathbf{s}_{\lambda}\equiv\left(  \text{some symmetric function of degree
}<\left\vert \lambda\right\vert \right)  \operatorname{mod}K.
\]

\end{lemma}

\begin{proof}
[Proof of Lemma \ref{lem.bigquot.sl-red-1} (sketched).]We have $\lambda\notin
P_{k}$. Hence, the partition $\lambda$ has more than $k$ parts.

Define a partition $\mu$ by $\mu=\lambda^{t}$. Hence, $\mu_{1}$ is the number
of parts of $\lambda$. Thus, we have $\mu_{1}>k$ (since $\lambda$ has more
than $k$ parts), so that $\mu_{1}\geq k+1$. Moreover, $\left\vert
\mu\right\vert =\left\vert \lambda\right\vert $ (since $\mu=\lambda^{t}$).

Write the partition $\mu$ in the form $\mu=\left(  \mu_{1},\mu_{2},\ldots
,\mu_{\ell}\right)  $. For each $j\in\left\{  1,2,\ldots,\ell\right\}  $, we
have $\mu_{1}-1+\underbrace{j}_{\geq1}\geq\mu_{1}-1+1=\mu_{1}\geq k+1$ and
thus $\mu_{1}-1+j\in\left\{  k+1,k+2,k+3,\ldots\right\}  $ and therefore%
\begin{align}
&  \mathbf{e}_{\mu_{1}-1+j}\nonumber\\
&  \equiv\mathbf{b}_{\mu_{1}-1+j-k}\ \ \ \ \ \ \ \ \ \ \left(  \text{by
(\ref{eq.thm.bigquot.ejmodK}), applied to }\mu_{1}-1+j\text{ instead of
}j\right) \nonumber\\
&  =\left(  \text{some symmetric function of degree }<\underbrace{k+\left(
\mu_{1}-1+j-k\right)  }_{=\mu_{1}-1+j}\right) \nonumber\\
&  \ \ \ \ \ \ \ \ \ \ \left(  \text{by (\ref{eq.thm.bigquot.bi-deg}), applied
to }i=\mu_{1}-1+j-k\right) \nonumber\\
&  =\left(  \text{some symmetric function of degree }<\mu_{1}-1+j\right)
\operatorname{mod}K. \label{pf.lem.bigquot.sl-red-1.e-red}%
\end{align}

From $\mu=\lambda^{t}$, we obtain $\mu^{t}=\left(  \lambda^{t}\right)
^{t}=\lambda$. But Corollary \ref{cor.jt.el} (applied to $\mu$ and $\mu_{i}$
instead of $\lambda$ and $\lambda_{i}$) yields%
\begin{align}
\mathbf{s}_{\mu^{t}}  &  =\det\left(  \left(  \mathbf{e}_{\mu_{i}-i+j}\right)
_{1\leq i\leq\ell,\ 1\leq j\leq\ell}\right)  =\det\left(  \left(
\mathbf{e}_{\mu_{u}-u+v}\right)  _{1\leq u\leq\ell,\ 1\leq v\leq\ell}\right)
\nonumber\\
&  =\sum_{j=1}^{\ell}\mathbf{e}_{\mu_{1}-1+j}\cdot C_{j},
\label{pf.lem.bigquot.sl-red.JT1}%
\end{align}
where $C_{j}$ denotes the $\left(  1,j\right)  $-th cofactor of the
$\ell\times\ell$-matrix $\left(  \mathbf{e}_{\mu_{u}-u+v}\right)  _{1\leq
u\leq\ell,\ 1\leq v\leq\ell}$. (Here, the last equality sign follows from
(\ref{eq.det.lap-exp}), applied to $R=\Lambda$ and $A=\left(  \mathbf{e}%
_{\mu_{u}-u+v}\right)  _{1\leq u\leq\ell,\ 1\leq v\leq\ell}$ and
$a_{u,v}=\mathbf{e}_{\mu_{u}-u+v}$ and $i=1$.)

For each $j\in\left\{  1,2,\ldots,\ell\right\}  $, the element $C_{j}$ is the
$\left(  1,j\right)  $-th cofactor of the matrix $\left(  \mathbf{e}_{\mu
_{u}-u+v}\right)  _{1\leq u\leq\ell,\ 1\leq v\leq\ell}$ (by its definition),
and thus is a homogeneous element of $\Lambda$ of degree $\left\vert
\mu\right\vert -\left(  \mu_{1}-1+j\right)  $ (by Lemma \ref{lem.K.cofactor1}
\textbf{(b)}, applied to $1$ and $\mu$ instead of $i$ and $\lambda$). Hence,%
\begin{equation}
C_{j}=\left(  \text{some symmetric function of degree }\leq\left\vert
\mu\right\vert -\left(  \mu_{1}-1+j\right)  \right)
\label{pf.lem.bigquot.sl-red.deg1}%
\end{equation}
for each $j\in\left\{  1,2,\ldots,\ell\right\}  $. Therefore,
(\ref{pf.lem.bigquot.sl-red.JT1}) becomes%
\begin{align*}
\mathbf{s}_{\mu^{t}}  &  =\sum_{j=1}^{\ell}\underbrace{\mathbf{e}_{\mu
_{1}-1+j}}_{\substack{\equiv\left(  \text{some symmetric function of degree
}<\mu_{1}-1+j\right)  \operatorname{mod}K\\\text{(by
(\ref{pf.lem.bigquot.sl-red-1.e-red})})}}\\
&  \ \ \ \ \ \ \ \ \ \ \cdot\underbrace{C_{j}}_{\substack{=\left(  \text{some
symmetric function of degree }\leq\left\vert \mu\right\vert -\left(  \mu
_{1}-1+j\right)  \right)  \\\text{(by (\ref{pf.lem.bigquot.sl-red.deg1}))}}}\\
&  \equiv\sum_{j=1}^{k}\left(  \text{some symmetric function of degree }%
<\mu_{1}-1+j\right) \\
&  \ \ \ \ \ \ \ \ \ \ \cdot\left(  \text{some symmetric function of degree
}\leq\left\vert \mu\right\vert -\left(  \mu_{1}-1+j\right)  \right) \\
&  =\left(  \text{some symmetric function of degree }<\left\vert
\mu\right\vert \right)  \operatorname{mod}K.
\end{align*}
In view of $\mu^{t}=\lambda$ and $\left\vert \mu\right\vert =\left\vert
\lambda\right\vert $, this rewrites as%
\[
\mathbf{s}_{\lambda}\equiv\left(  \text{some symmetric function of degree
}<\left\vert \lambda\right\vert \right)  \operatorname{mod}K.
\]
This proves Lemma \ref{lem.bigquot.sl-red-1}.
\end{proof}

Our next lemma is an analogue of Lemma \ref{lem.I.hi-red}:

\begin{lemma}
\label{lem.bigquot.hi-red}Let $i$ be an integer such that $i>n-k$. Then,%
\[
\mathbf{h}_{i}\equiv\left(  \text{some symmetric function of degree
}<i\right)  \operatorname{mod}K.
\]

\end{lemma}

\begin{proof}
[Proof of Lemma \ref{lem.bigquot.hi-red} (sketched).]We shall prove Lemma
\ref{lem.bigquot.hi-red} by strong induction on $i$. Thus, we assume (as the
induction hypothesis) that%
\begin{equation}
\mathbf{h}_{j}\equiv\left(  \text{some symmetric function of degree
}<j\right)  \operatorname{mod}K \label{pf.lem.bigquot.hi-red.IH}%
\end{equation}
for every $j\in\left\{  n-k+1,n-k+2,\ldots,i-1\right\}  $.

If $i\leq n$, then (\ref{eq.thm.bigquot.hjmodK}) (applied to $j=i$) yields
$\mathbf{h}_{i}\equiv\mathbf{a}_{i-\left(  n-k\right)  }\operatorname{mod}K$
(since $i\in\left\{  n-k+1,n-k+2,\ldots,n\right\}  $), which clearly proves
Lemma \ref{lem.bigquot.hi-red} (since $\mathbf{a}_{i-\left(  n-k\right)  }$ is
a symmetric function of degree $<i$\ \ \ \ \footnote{by
(\ref{eq.thm.bigquot.ai-deg})}). Thus, for the rest of this proof, we WLOG
assume that $i>n$. Hence, each $t\in\left\{  1,2,\ldots,k\right\}  $ satisfies
\newline$i-t\in\left\{  n-k+1,n-k+2,\ldots,i-1\right\}  $ (since
$\underbrace{i}_{>n}-\underbrace{t}_{\leq k}>n-k$ and $i-\underbrace{t}%
_{\geq1}\leq i-1$) and therefore%
\begin{equation}
\mathbf{h}_{i-t}\equiv\left(  \text{some symmetric function of degree
}<i-t\right)  \operatorname{mod}K \label{pf.lem.bigquot.hi-red.2}%
\end{equation}
(by (\ref{pf.lem.bigquot.hi-red.IH}), applied to $j=i-t$).

On the other hand, each $t\in\left\{  k+1,k+2,k+3,\ldots\right\}  $ satisfies%
\begin{equation}
\mathbf{e}_{t}\equiv\left(  \text{some symmetric function of degree
}<t\right)  \operatorname{mod}K. \label{pf.lem.bigquot.hi-red.3}%
\end{equation}

[\textit{Proof of (\ref{pf.lem.bigquot.hi-red.3}):} Let $t\in\left\{
k+1,k+2,k+3,\ldots\right\}  $. Thus, $t>k$. Hence, the partition $\left(
1^{t}\right)  $ has more than $k$ parts (since it has $t$ parts), and
therefore we have $\left(  1^{t}\right)  \notin P_{k}$. Hence, Lemma
\ref{lem.bigquot.sl-red-1} (applied to $\lambda=\left(  1^{t}\right)  $)
yields%
\[
\mathbf{s}_{\left(  1^{t}\right)  }\equiv\left(  \text{some symmetric function
of degree }<\left\vert 1^{t}\right\vert \right)  \operatorname{mod}K.
\]
In view of $\mathbf{s}_{\left(  1^{t}\right)  }=\mathbf{e}_{t}$ and
$\left\vert 1^{t}\right\vert =t$, this rewrites as
\[
\mathbf{e}_{t}\equiv\left(  \text{some symmetric function of degree
}<t\right)  \operatorname{mod}K.
\]
This proves (\ref{pf.lem.bigquot.hi-red.3}).]

Now, we claim that each $t\in\left\{  1,2,\ldots,i\right\}  $ satisfies%
\begin{equation}
\mathbf{h}_{i-t}\mathbf{e}_{t}\equiv\left(  \text{some symmetric function of
degree }<i\right)  \operatorname{mod}K. \label{pf.lem.bigquot.hi-red.4}%
\end{equation}

[\textit{Proof of (\ref{pf.lem.bigquot.hi-red.4}):} Let $t\in\left\{
1,2,\ldots,i\right\}  $. We are in one of the following two cases:

\textit{Case 1:} We have $t\leq k$.

\textit{Case 2:} We have $t>k$.

Let us first consider Case 1. In this case, we have $t\leq k$. Hence,
$t\in\left\{  1,2,\ldots,k\right\}  $. Thus,
\begin{align*}
&  \underbrace{\mathbf{h}_{i-t}}_{\substack{\equiv\left(  \text{some symmetric
function of degree }<i-t\right)  \operatorname{mod}K\\\text{(by
(\ref{pf.lem.bigquot.hi-red.2}))}}}\mathbf{e}_{t}\\
&  \equiv\left(  \text{some symmetric function of degree }<i-t\right)
\cdot\mathbf{e}_{t}\\
&  =\left(  \text{some symmetric function of degree }<i\right)
\operatorname{mod}K
\end{align*}
(since $\mathbf{e}_{t}$ is a symmetric function of degree $t$). Hence,
(\ref{pf.lem.bigquot.hi-red.4}) is proven in Case 1.

Let us next consider Case 2. In this case, we have $t>k$. Hence, $t\in\left\{
k+1,k+2,k+3,\ldots\right\}  $. Thus,%
\begin{align*}
&  \mathbf{h}_{i-t}\underbrace{\mathbf{e}_{t}}_{\substack{\equiv\left(
\text{some symmetric function of degree }<t\right)  \operatorname{mod}%
K\\\text{(by (\ref{pf.lem.bigquot.hi-red.3}))}}}\\
&  \equiv\mathbf{h}_{i-t}\cdot\left(  \text{some symmetric function of degree
}<t\right) \\
&  =\left(  \text{some symmetric function of degree }<i\right)
\operatorname{mod}K
\end{align*}
(since $\mathbf{h}_{i-t}$ is a symmetric function of degree $i-t$). Thus,
(\ref{pf.lem.bigquot.hi-red.4}) is proven in Case 2.

We have now proven (\ref{pf.lem.bigquot.hi-red.4}) in both Cases 1 and 2.
Thus, (\ref{pf.lem.bigquot.hi-red.4}) always holds.]

On the other hand, $i>n-k\geq0$ (since $n\geq k$), so that $i\neq0$. Now,
(\ref{pf.prop.redh.Lam.3}) (applied to $N=i$) yields%
\[
\sum_{j=0}^{i}\left(  -1\right)  ^{j}\mathbf{h}_{i-j}\mathbf{e}_{j}%
=\delta_{0,i}=0\ \ \ \ \ \ \ \ \ \ \left(  \text{since }i\neq0\right)  .
\]
Hence,%
\begin{align*}
0  &  =\sum_{j=0}^{i}\left(  -1\right)  ^{j}\mathbf{h}_{i-j}\mathbf{e}%
_{j}=\sum_{t=0}^{i}\left(  -1\right)  ^{t}\mathbf{h}_{i-t}\mathbf{e}_{t}\\
&  \ \ \ \ \ \ \ \ \ \ \left(  \text{here, we have renamed the summation index
}j\text{ as }t\right) \\
&  =\underbrace{\left(  -1\right)  ^{0}}_{=1}\underbrace{\mathbf{h}_{i-0}%
}_{=\mathbf{h}_{i}}\underbrace{\mathbf{e}_{0}}_{=1}+\sum_{t=1}^{i}\left(
-1\right)  ^{t}\mathbf{h}_{i-t}\mathbf{e}_{t}\\
&  \ \ \ \ \ \ \ \ \ \ \left(  \text{here, we have split off the addend for
}t=0\text{ from the sum}\right) \\
&  =\mathbf{h}_{i}+\sum_{t=1}^{i}\left(  -1\right)  ^{t}\mathbf{h}%
_{i-t}\mathbf{e}_{t}.
\end{align*}
Hence,%
\begin{align*}
\mathbf{h}_{i}  &  =-\sum_{t=1}^{i}\left(  -1\right)  ^{t}%
\underbrace{\mathbf{h}_{i-t}\mathbf{e}_{t}}_{\substack{\equiv\left(
\text{some symmetric function of degree }<i\right)  \operatorname{mod}%
K\\\text{(by (\ref{pf.lem.bigquot.hi-red.4}))}}}\\
&  \equiv-\sum_{t=1}^{i}\left(  -1\right)  ^{t}\left(  \text{some symmetric
function of degree }<i\right) \\
&  =\left(  \text{some symmetric function of degree }<i\right)
\operatorname{mod}K.
\end{align*}
This completes the induction step. Thus, Lemma \ref{lem.bigquot.hi-red} is proven.
\end{proof}

Recall the first Jacobi-Trudi identity (\cite[(2.4.16)]{GriRei18}):

\begin{proposition}
\label{prop.jt.h}Let $\lambda=\left(  \lambda_{1},\lambda_{2},\ldots
,\lambda_{\ell}\right)  $ and $\mu=\left(  \mu_{1},\mu_{2},\ldots,\mu_{\ell
}\right)  $ be two partitions. Then,%
\[
\mathbf{s}_{\lambda/\mu}=\det\left(  \left(  \mathbf{h}_{\lambda_{i}-\mu
_{j}-i+j}\right)  _{1\leq i\leq\ell,\ 1\leq j\leq\ell}\right)  .
\]

\end{proposition}

Next, we are ready to prove Lemma \ref{lem.bigquot.sl-red}:

\begin{proof}
[Proof of Lemma \ref{lem.bigquot.sl-red} (sketched).]We must prove that
\[
\mathbf{s}_{\lambda}\equiv\left(  \text{some symmetric function of degree
}<\left\vert \lambda\right\vert \right)  \operatorname{mod}K.
\]
If $\lambda\notin P_{k}$, then this follows from Lemma
\ref{lem.bigquot.sl-red-1}. Thus, for the rest of this proof, we WLOG assume
that $\lambda\in P_{k}$.

From $\lambda\in P_{k}$ and $\lambda\notin P_{k,n}$, we conclude that not all
parts of the partition $\lambda$ are $\leq n-k$. Thus, the first entry
$\lambda_{1}$ of $\lambda$ is $>n-k$ (since $\lambda_{1}\geq\lambda_{2}%
\geq\lambda_{3}\geq\cdots$). But $\lambda=\left(  \lambda_{1},\lambda
_{2},\ldots,\lambda_{k}\right)  $ (since $\lambda\in P_{k}$). Thus,
Proposition \ref{prop.jt.h} (applied to $\ell=k$, $\mu=\varnothing$ and
$\mu_{i}=0$) yields%
\begin{align}
\mathbf{s}_{\lambda/\varnothing}  &  =\det\left(  \left(  \mathbf{h}%
_{\lambda_{i}-0-i+j}\right)  _{1\leq i\leq k,\ 1\leq j\leq k}\right)
=\det\left(  \left(  \mathbf{h}_{\lambda_{i}-i+j}\right)  _{1\leq i\leq
k,\ 1\leq j\leq k}\right) \nonumber\\
&  =\det\left(  \left(  \mathbf{h}_{\lambda_{u}-u+v}\right)  _{1\leq u\leq
k,\ 1\leq v\leq k}\right) \nonumber\\
&  \ \ \ \ \ \ \ \ \ \ \left(
\begin{array}
[c]{c}%
\text{here, we have renamed the indices }i\text{ and }j\\
\text{as }u\text{ and }v\text{ in the matrix}%
\end{array}
\right) \nonumber\\
&  =\sum_{j=1}^{k}\mathbf{h}_{\lambda_{1}-1+j}\cdot C_{j},
\label{pf.lem.bigquot.sl-red.JT}%
\end{align}
where $C_{j}$ denotes the $\left(  1,j\right)  $-th cofactor of the $k\times
k$-matrix $\left(  \mathbf{h}_{\lambda_{u}-u+v}\right)  _{1\leq u\leq
k,\ 1\leq v\leq k}$. (Here, the last equality sign follows from
(\ref{eq.det.lap-exp}), applied to $\ell=k$ and $R=\Lambda$ and $A=\left(
\mathbf{h}_{\lambda_{u}-u+v}\right)  _{1\leq u\leq k,\ 1\leq v\leq k}$ and
$a_{u,v}=\mathbf{h}_{\lambda_{u}-u+v}$ and $i=1$.)

For each $j\in\left\{  1,2,\ldots,k\right\}  $, we have $\lambda_{1}%
-1+j\geq\lambda_{1}-1+1=\lambda_{1}>n-k$ and therefore%
\begin{align}
&  \mathbf{h}_{\lambda_{1}-1+j}\nonumber\\
&  \equiv\left(  \text{some symmetric function of degree }<\lambda
_{1}-1+j\right)  \operatorname{mod}K \label{pf.lem.bigquot.sl-red.deg2}%
\end{align}
(by Lemma \ref{lem.bigquot.hi-red}, applied to $i=\lambda_{1}-1+j$).

For each $j\in\left\{  1,2,\ldots,k\right\}  $, the polynomial $C_{j}$ is the
$\left(  1,j\right)  $-th cofactor of the matrix $\left(  \mathbf{h}%
_{\lambda_{u}-u+v}\right)  _{1\leq u\leq k,\ 1\leq v\leq k}$ (by its
definition), and thus is a homogeneous element of $\Lambda$ of degree
$\left\vert \lambda\right\vert -\left(  \lambda_{1}-1+j\right)  $ (by Lemma
\ref{lem.K.cofactor1} \textbf{(a)}, applied to $\ell=k$ and $i=1$). Hence,%
\begin{equation}
C_{j}=\left(  \text{some symmetric function of degree }\leq\left\vert
\lambda\right\vert -\left(  \lambda_{1}-1+j\right)  \right)
\label{pf.lem.bigquot.sl-red.deg}%
\end{equation}
for each $j\in\left\{  1,2,\ldots,k\right\}  $.

Therefore, (\ref{pf.lem.bigquot.sl-red.JT}) becomes%
\begin{align*}
\mathbf{s}_{\lambda/\varnothing}  &  =\sum_{j=1}^{k}\underbrace{\mathbf{h}%
_{\lambda_{1}-1+j}}_{\substack{\equiv\left(  \text{some symmetric function of
degree }<\lambda_{1}-1+j\right)  \operatorname{mod}K\\\text{(by
(\ref{pf.lem.bigquot.sl-red.deg2}))}}}\\
&  \ \ \ \ \ \ \ \ \ \ \cdot\underbrace{C_{j}}_{\substack{=\left(  \text{some
symmetric function of degree }\leq\left\vert \lambda\right\vert -\left(
\lambda_{1}-1+j\right)  \right)  \\\text{(by (\ref{pf.lem.bigquot.sl-red.deg}%
))}}}\\
&  \equiv\sum_{j=1}^{k}\left(  \text{some symmetric function of degree
}<\lambda_{1}-1+j\right) \\
&  \ \ \ \ \ \ \ \ \ \ \cdot\left(  \text{some symmetric function of degree
}\leq\left\vert \lambda\right\vert -\left(  \lambda_{1}-1+j\right)  \right) \\
&  =\left(  \text{some symmetric function of degree }<\left\vert
\lambda\right\vert \right)  \operatorname{mod}K.
\end{align*}
In view of $\mathbf{s}_{\lambda/\varnothing}=\mathbf{s}_{\lambda}$, this
rewrites as
\[
\mathbf{s}_{\lambda}\equiv\left(  \text{some symmetric function of degree
}<\left\vert \lambda\right\vert \right)  \operatorname{mod}K.
\]
This proves Lemma \ref{lem.bigquot.sl-red}.
\end{proof}

\begin{lemma}
\label{lem.bigquot.sm-as-smaller}Let $N\in\mathbb{N}$. Let $f\in\Lambda$ be a
symmetric function of degree $<N$. Then, there exists a family $\left(
c_{\kappa}\right)  _{\kappa\in\operatorname*{Par};\ \left\vert \kappa
\right\vert <N}$ of elements of $\mathbf{k}$ such that $f=\sum
_{\substack{\kappa\in\operatorname*{Par};\\\left\vert \kappa\right\vert
<N}}c_{\kappa}\mathbf{s}_{\kappa}$.
\end{lemma}

\begin{proof}
[Proof of Lemma \ref{lem.bigquot.sm-as-smaller}.]For each $d\in\mathbb{N}$, we
let $\Lambda_{\deg=d}$ be the $d$-th graded part of the graded $\mathbf{k}%
$-module $\Lambda$. This is the $\mathbf{k}$-submodule of $\Lambda$ consisting
of all homogeneous elements of $\Lambda$ of degree $d$ (including the zero
vector $0$, which is homogeneous of every degree).

Recall that the family $\left(  \mathbf{s}_{\lambda}\right)  _{\lambda
\in\operatorname*{Par}}$ is a graded basis of the graded $\mathbf{k}$-module
$\Lambda$. In other words, for each $d\in\mathbb{N}$, the family $\left(
\mathbf{s}_{\lambda}\right)  _{\lambda\in\operatorname*{Par};\ \left\vert
\lambda\right\vert =d}$ is a basis of the $\mathbf{k}$-submodule
$\Lambda_{\deg=d}$ of $\Lambda$. Hence, for each $d\in\mathbb{N}$, we have%
\begin{align}
\Lambda_{\deg=d}  &  =\left(  \text{the }\mathbf{k}\text{-linear span of the
family }\left(  \mathbf{s}_{\lambda}\right)  _{\lambda\in\operatorname*{Par}%
;\ \left\vert \lambda\right\vert =d}\right) \nonumber\\
&  =\sum_{\substack{\lambda\in\operatorname*{Par};\\\left\vert \lambda
\right\vert =d}}\mathbf{ks}_{\lambda}. \label{pf.lem.bigquot.sm-as-smaller.2}%
\end{align}

The symmetric function $f$ has degree $<N$. Hence, we can write $f$ in the
form $f=\sum_{d=0}^{N-1}f_{d}$ for some $f_{0},f_{1},\ldots,f_{N-1}\in\Lambda
$, where each $f_{d}$ is a homogeneous symmetric function of degree $d$.
Consider these $f_{0},f_{1},\ldots,f_{N-1}$. For each $d\in\left\{
0,1,\ldots,N-1\right\}  $, the symmetric function $f_{d}$ is an element of
$\Lambda$ and is homogeneous of degree $d$ (as we already know). In other
words, for each $d\in\left\{  0,1,\ldots,N-1\right\}  $, we have%
\begin{equation}
f_{d}\in\Lambda_{\deg=d}. \label{pf.lem.bigquot.sm-as-smaller.4}%
\end{equation}
Now,
\[
f=\sum_{d=0}^{N-1}\underbrace{f_{d}}_{\substack{\in\Lambda_{\deg=d}\\\text{(by
(\ref{pf.lem.bigquot.sm-as-smaller.4}))}}}\in\sum_{d=0}^{N-1}%
\underbrace{\Lambda_{\deg=d}}_{\substack{=\sum_{\substack{\lambda
\in\operatorname*{Par};\\\left\vert \lambda\right\vert =d}}\mathbf{ks}%
_{\lambda}\\\text{(by (\ref{pf.lem.bigquot.sm-as-smaller.2}))}}%
}=\underbrace{\sum_{d=0}^{N-1}\sum_{\substack{\lambda\in\operatorname*{Par}%
;\\\left\vert \lambda\right\vert =d}}}_{=\sum_{\substack{\lambda
\in\operatorname*{Par};\\\left\vert \lambda\right\vert <N}}}\mathbf{ks}%
_{\lambda}=\sum_{\substack{\lambda\in\operatorname*{Par};\\\left\vert
\lambda\right\vert <N}}\mathbf{ks}_{\lambda}=\sum_{\substack{\kappa
\in\operatorname*{Par};\\\left\vert \kappa\right\vert <N}}\mathbf{ks}_{\kappa}%
\]
(here, we have renamed the summation index $\lambda$ as $\kappa$ in the sum).
In other words, there exists a family $\left(  c_{\kappa}\right)  _{\kappa
\in\operatorname*{Par};\ \left\vert \kappa\right\vert <N}$ of elements of
$\mathbf{k}$ such that $f=\sum_{\substack{\kappa\in\operatorname*{Par}%
;\\\left\vert \kappa\right\vert <N}}c_{\kappa}\mathbf{s}_{\kappa}$. This
proves Lemma \ref{lem.bigquot.sm-as-smaller}.
\end{proof}

\begin{lemma}
\label{lem.bigquot.sl-red2}For each $\mu\in\operatorname*{Par}$, the element
$\overline{\mathbf{s}_{\mu}}\in\Lambda/K$ belongs to the $\mathbf{k}%
$-submodule of $\Lambda/K$ spanned by the family $\left(  \overline
{\mathbf{s}_{\lambda}}\right)  _{\lambda\in P_{k,n}}$.
\end{lemma}

\begin{proof}
[Proof of Lemma \ref{lem.bigquot.sl-red2}.]Let $M$ be the $\mathbf{k}%
$-submodule of $\Lambda/K$ spanned by the family $\left(  \overline
{\mathbf{s}_{\lambda}}\right)  _{\lambda\in P_{k,n}}$. We thus must prove that
$\overline{\mathbf{s}_{\mu}}\in M$ for each $\mu\in\operatorname*{Par}$.

We shall prove this by strong induction on $\left\vert \mu\right\vert $. Thus,
we fix some $N\in\mathbb{N}$, and we assume (as induction hypothesis) that
\begin{equation}
\overline{\mathbf{s}_{\kappa}}\in M\ \ \ \ \ \ \ \ \ \ \text{for each }%
\kappa\in\operatorname*{Par}\text{ satisfying }\left\vert \kappa\right\vert
<N. \label{pf.lem.bigquot.sl-red2.IH}%
\end{equation}
Now, let $\mu\in\operatorname*{Par}$ be such that $\left\vert \mu\right\vert
=N$. We then must show that $\overline{\mathbf{s}_{\mu}}\in M$.

If $\mu\in P_{k,n}$, then this is obvious (since $\overline{\mathbf{s}_{\mu}}$
then belongs to the family that spans $M$). Thus, for the rest of this proof,
we WLOG assume that $\mu\notin P_{k,n}$. Hence, Lemma \ref{lem.bigquot.sl-red}
(applied to $\lambda=\mu$) yields%
\[
\mathbf{s}_{\mu}\equiv\left(  \text{some symmetric function of degree
}<\left\vert \mu\right\vert \right)  \operatorname{mod}K.
\]
In other words, there exists some symmetric function $f\in\Lambda$ of degree
$<\left\vert \mu\right\vert $ such that $\mathbf{s}_{\mu}\equiv
f\operatorname{mod}K$. Consider this $f$.

But $f$ is a symmetric function of degree $<\left\vert \mu\right\vert $. In
other words, $f$ is a symmetric function of degree $<N$ (since $\left\vert
\mu\right\vert =N$). Hence, Lemma \ref{lem.bigquot.sm-as-smaller} shows that
there exists a family $\left(  c_{\kappa}\right)  _{\kappa\in
\operatorname*{Par};\ \left\vert \kappa\right\vert <N}$ of elements of
$\mathbf{k}$ such that $f=\sum_{\substack{\kappa\in\operatorname*{Par}%
;\\\left\vert \kappa\right\vert <N}}c_{\kappa}\mathbf{s}_{\kappa}$. Consider
this family. From $f=\sum_{\substack{\kappa\in\operatorname*{Par};\\\left\vert
\kappa\right\vert <N}}c_{\kappa}\mathbf{s}_{\kappa}$, we obtain%
\[
\overline{f}=\overline{\sum_{\substack{\kappa\in\operatorname*{Par}%
;\\\left\vert \kappa\right\vert <N}}c_{\kappa}\mathbf{s}_{\kappa}}%
=\sum_{\substack{\kappa\in\operatorname*{Par};\\\left\vert \kappa\right\vert
<N}}c_{\kappa}\underbrace{\overline{\mathbf{s}_{\kappa}}}_{\substack{\in
M\\\text{(by (\ref{pf.lem.bigquot.sl-red2.IH}))}}}\in\sum_{\substack{\kappa
\in\operatorname*{Par};\\\left\vert \kappa\right\vert <N}}c_{\kappa}M\subseteq
M\ \ \ \ \ \ \ \ \ \ \left(  \text{since }M\text{ is a }\mathbf{k}%
\text{-module}\right)  .
\]
But from $\mathbf{s}_{\mu}\equiv f\operatorname{mod}K$, we obtain
$\overline{\mathbf{s}_{\mu}}=\overline{f}\in M$. This completes our induction
step. Thus, we have proven by strong induction that $\overline{\mathbf{s}%
_{\mu}}\in M$ for each $\mu\in\operatorname*{Par}$. This proves Lemma
\ref{lem.bigquot.sl-red2}.
\end{proof}

\begin{corollary}
\label{cor.bigquot.span}The family $\left(  \overline{\mathbf{s}_{\lambda}%
}\right)  _{\lambda\in P_{k,n}}$ spans the $\mathbf{k}$-module $\Lambda/K$.
\end{corollary}

\begin{proof}
[Proof of Corollary \ref{cor.bigquot.span}.]It is well-known that $\left(
\mathbf{s}_{\lambda}\right)  _{\lambda\in\operatorname*{Par}}$ is a basis of
the $\mathbf{k}$-module $\Lambda$. Hence, $\left(  \overline{\mathbf{s}%
_{\lambda}}\right)  _{\lambda\in\operatorname*{Par}}$ is a spanning set of the
$\mathbf{k}$-module $\Lambda/K$. Thus, $\left(  \overline{\mathbf{s}_{\lambda
}}\right)  _{\lambda\in P_{k,n}}$ is also a spanning set of the $\mathbf{k}%
$-module $\Lambda/K$ (because Lemma \ref{lem.bigquot.sl-red2} shows that every
element of the first spanning set belongs to the span of the second). This
proves Corollary \ref{cor.bigquot.span}.
\end{proof}

With Corollary \ref{cor.bigquot.span}, we have proven \textquotedblleft one
half\textquotedblright\ of Theorem \ref{thm.bigquot}.

\subsection{A lemma on filtrations}

Next, we recall the definition of a filtration of a \textbf{$k$}-module:

\begin{definition}
\label{def.filtr.modfiltr}Let $V$ be a $\mathbf{k}$-module. A $\mathbf{k}%
$\textit{-module filtration} of $V$ means a sequence $\left(  V_{m}\right)
_{m\in\mathbb{N}}$ of $\mathbf{k}$-submodules of $V$ such that $\bigcup
\limits_{m\in\mathbb{N}}V_{m}=V$ and $V_{0}\subseteq V_{1}\subseteq
V_{2}\subseteq\cdots$.
\end{definition}

For example, $\left(  \Lambda_{\deg\leq m}\right)  _{m\in\mathbb{N}}$ is a
$\mathbf{k}$-module filtration of $\Lambda$.

The filtered $\mathbf{k}$-modules are the objects of a category, whose
morphisms are $\mathbf{k}$-linear maps respecting the filtration. Here is how
they are defined:

\begin{definition}
\label{def.filtr.respect}Let $V$ and $W$ be two $\mathbf{k}$-modules. Let
$f:V\rightarrow W$ be a $\mathbf{k}$-module homomorphism. Let $\left(
V_{m}\right)  _{m\in\mathbb{N}}$ be a $\mathbf{k}$-module filtration of $V$,
and let $\left(  W_{m}\right)  _{m\in\mathbb{N}}$ be a $\mathbf{k}$-module
filtration of $W$.

We say that \textit{the map }$f$ \textit{respects the filtrations }$\left(
V_{m}\right)  _{m\in\mathbb{N}}$ \textit{and }$\left(  W_{m}\right)
_{m\in\mathbb{N}}$ if it satisfies $\left(  f\left(  V_{m}\right)  \subseteq
W_{m}\text{ for every }m\in\mathbb{N}\right)  $. Sometimes we abbreviate
\textquotedblleft the map $f$ respects the filtrations $\left(  V_{m}\right)
_{m\geq0}$ and $\left(  W_{m}\right)  _{m\geq0}$\textquotedblright\ to
\textquotedblleft\textit{the map }$f$\textit{ respects the filtration}%
\textquotedblright, as long as the filtrations $\left(  V_{m}\right)
_{m\in\mathbb{N}}$ and $\left(  W_{m}\right)  _{m\in\mathbb{N}}$ are clear
from the context.
\end{definition}

The following elementary fact about filtrations of $\mathbf{k}$-modules will
be crucial to us:

\begin{proposition}
\label{prop.filtr.iso}Let $V$ be a $\mathbf{k}$-module. Let $\left(
V_{m}\right)  _{m\in\mathbb{N}}$ be a $\mathbf{k}$-module filtration of $V$.
Let $f:V\rightarrow V$ be a $\mathbf{k}$-module homomorphism which satisfies%
\[
\left(  f\left(  V_{m}\right)  \subseteq V_{m-1}\ \ \ \ \ \ \ \ \ \ \text{for
every }m\in\mathbb{N}\right)  ,
\]
where $V_{-1}$ denotes the $\mathbf{k}$-submodule $0$ of $V$. Then:

\textbf{(a)} The $\mathbf{k}$-module homomorphism $\operatorname{id}-f$ is an isomorphism.

\textbf{(b)} Each of the maps $\operatorname{id}-f$ and $\left(
\operatorname{id}-f\right)  ^{-1}$ respects the filtration.
\end{proposition}

Proposition \ref{prop.filtr.iso} is classical; a proof can be found in
\cite[Proposition 1.99]{Grinbe11} (see the detailed version of \cite{Grinbe11}
for a detailed proof). Let us restate this proposition in a form adapted for
our use:

\begin{proposition}
\label{prop.filtr.iso-g}Let $V$ be a $\mathbf{k}$-module. Let $\left(
V_{m}\right)  _{m\in\mathbb{N}}$ be a $\mathbf{k}$-module filtration of $V$.
Let $g:V\rightarrow V$ be a $\mathbf{k}$-module homomorphism which satisfies%
\begin{equation}
\left(  g\left(  v\right)  \in v+V_{m-1}\ \ \ \ \ \ \ \ \ \ \text{for every
}m\in\mathbb{N}\text{ and each }v\in V_{m}\right)  ,
\label{eq.prop.filtr.iso-g.ass}%
\end{equation}
where $V_{-1}$ denotes the $\mathbf{k}$-submodule $0$ of $V$. Then:

\textbf{(a)} The $\mathbf{k}$-module homomorphism $g$ is an isomorphism.

\textbf{(b)} Each of the maps $g$ and $g^{-1}$ respects the filtration.
\end{proposition}

\begin{proof}
[Proof of Proposition \ref{prop.filtr.iso-g}.]Let $f:V\rightarrow V$ be the
$\mathbf{k}$-module homomorphism $\operatorname*{id}-g$. Then,
$f=\operatorname*{id}-g$, so that $g=\operatorname*{id}-f$. Now, for each
$m\in\mathbb{N}$, we have $f\left(  V_{m}\right)  \subseteq V_{m-1}$ (since
each $v\in V_{m}$ satisfies%
\begin{align*}
\underbrace{f}_{=\operatorname*{id}-g}\left(  v\right)   &  =\left(
\operatorname*{id}-g\right)  \left(  v\right)  =\underbrace{\operatorname*{id}%
\left(  v\right)  }_{=v}-g\left(  v\right)  =v-g\left(  v\right) \\
&  =-\underbrace{\left(  g\left(  v\right)  -v\right)  }_{\substack{\in
V_{m-1}\\\text{(since }g\left(  v\right)  \in v+V_{m-1}\\\text{(by
(\ref{eq.prop.filtr.iso-g.ass})))}}}\in-V_{m-1}\\
&  \subseteq V_{m-1}\ \ \ \ \ \ \ \ \ \ \left(  \text{since }V_{m-1}\text{ is
a }\mathbf{k}\text{-module}\right)
\end{align*}
). Hence, Proposition \ref{prop.filtr.iso} \textbf{(a)} yields that the
$\mathbf{k}$-module homomorphism $\operatorname{id}-f$ is an isomorphism. In
other words, the $\mathbf{k}$-module homomorphism $g$ is an isomorphism (since
$g=\operatorname*{id}-f$). This proves Proposition \ref{prop.filtr.iso-g}
\textbf{(a)}.

\textbf{(b)} Proposition \ref{prop.filtr.iso} \textbf{(b)} yields that each of
the maps $\operatorname{id}-f$ and $\left(  \operatorname{id}-f\right)  ^{-1}$
respects the filtration. In other words, each of the maps $g$ and $g^{-1}$
respects the filtration (since $g=\operatorname*{id}-f$). This proves
Proposition \ref{prop.filtr.iso-g} \textbf{(b)}.
\end{proof}

We next move back to symmetric functions. Recall that $\left(  \Lambda
_{\deg\leq m}\right)  _{m\in\mathbb{N}}$ is a $\mathbf{k}$-module filtration
of $\Lambda$. Whenever we say that a map $\varphi:\Lambda\rightarrow\Lambda$
\textquotedblleft respects the filtration\textquotedblright, we shall be
referring to this filtration.

\begin{lemma}
\label{lem.bigquot.e-as-smaller}Let $N\in\mathbb{N}$. Let $f\in\Lambda$ be a
symmetric function of degree $<N$. Then, there exists a family $\left(
c_{\kappa}\right)  _{\kappa\in\operatorname*{Par};\ \left\vert \kappa
\right\vert <N}$ of elements of $\mathbf{k}$ such that $f=\sum
_{\substack{\kappa\in\operatorname*{Par};\\\left\vert \kappa\right\vert
<N}}c_{\kappa}\mathbf{e}_{\kappa}$.
\end{lemma}

\begin{proof}
[Proof of Lemma \ref{lem.bigquot.e-as-smaller}.]This can be proved using the
same argument that we used to prove Lemma \ref{lem.bigquot.sm-as-smaller}, as
long as we replace every Schur function $\mathbf{s}_{\mu}$ by the
corresponding $\mathbf{e}_{\mu}$.
\end{proof}

Recall one of our notations defined long time ago: For any partition $\lambda
$, we let $\mathbf{e}_{\lambda}$ be the corresponding elementary symmetric
function in $\Lambda$. (This is called $e_{\lambda}$ in \cite[Definition
2.2.1]{GriRei18}.)

\begin{lemma}
\label{lem.bigquot.filtr-hom}Let $\varphi:\Lambda\rightarrow\Lambda$ be a
$\mathbf{k}$-algebra homomorphism. Assume that%
\begin{equation}
\varphi\left(  \mathbf{e}_{i}\right)  \in\mathbf{e}_{i}+\Lambda_{\deg\leq
i-1}\ \ \ \ \ \ \ \ \ \ \text{for each }i\in\left\{  1,2,3,\ldots\right\}  .
\label{eq.lem.bigquot.filtr-hom.ass}%
\end{equation}
Then:

\textbf{(a)} We have $\varphi\left(  v\right)  \in v+\Lambda_{\deg\leq m-1}$
for each $m\in\mathbb{N}$ and $v\in\Lambda_{\deg\leq m}$. (Here,
$\Lambda_{\deg\leq-1}$ denotes the $\mathbf{k}$-submodule $0$ of $\Lambda$.)

\textbf{(b)} The map $\varphi:\Lambda\rightarrow\Lambda$ is a $\mathbf{k}%
$-algebra isomorphism.

\textbf{(c)} Each of the maps $\varphi$ and $\varphi^{-1}$ respects the filtration.
\end{lemma}

\begin{proof}
[Proof of Lemma \ref{lem.bigquot.filtr-hom}.]We shall use the notation
$\ell\left(  \lambda\right)  $ defined in Definition \ref{def.LpHp}
\textbf{(a)}.

Let us first prove a few auxiliary claims:

\begin{statement}
\textit{Claim 1:} Let $i,j\in\left\{  -1,0,1,\ldots\right\}  $. Then,
$\Lambda_{\deg\leq i}\Lambda_{\deg\leq j}\subseteq\Lambda_{\deg\leq i+j}$.
\end{statement}

[\textit{Proof of Claim 1:} If one of $i$ and $j$ is $-1$, then Claim 1 holds
for obvious reasons (since $\Lambda_{\deg\leq-1}=0$ and thus $\Lambda
_{\deg\leq i}\Lambda_{\deg\leq j}=0$ in this case). Hence, for the rest of
this proof, we WLOG assume that none of $i$ and $j$ is $-1$. Hence, $i$ and
$j$ belong to $\mathbb{N}$ (since $i,j\in\left\{  -1,0,1,\ldots\right\}  $).
Thus, (\ref{eq.thm.bigquot.Lamfil}) yields $\Lambda_{\deg\leq i}\Lambda
_{\deg\leq j}\subseteq\Lambda_{\deg\leq i+j}$. This proves Claim 1.]

\begin{statement}
\textit{Claim 2:} Let $\alpha,\beta\in\mathbb{N}$. Let $a\in\Lambda_{\deg
\leq\alpha}$ and $b\in\Lambda_{\deg\leq\beta}$. Let $u\in a+\Lambda_{\deg
\leq\alpha-1}$ and $v\in b+\Lambda_{\deg\leq\beta-1}$. Then, $uv\in
ab+\Lambda_{\deg\leq\alpha+\beta-1}$.
\end{statement}

[\textit{Proof of Claim 2:} For every $m\in\mathbb{N}$, we have $\Lambda
_{\deg\leq m-1}\subseteq\Lambda_{\deg\leq m}$ (indeed, this is clear from the
definitions of $\Lambda_{\deg\leq m-1}$ and $\Lambda_{\deg\leq m}$). Thus,
$\Lambda_{\deg\leq\alpha-1}\subseteq\Lambda_{\deg\leq\alpha}$ and
$\Lambda_{\deg\leq\beta-1}\subseteq\Lambda_{\deg\leq\beta}$.

We have $u\in a+\Lambda_{\deg\leq\alpha-1}$. In other words, $u=a+x$ for some
$x\in\Lambda_{\deg\leq\alpha-1}$. Consider this $x$.

We have $v\in b+\Lambda_{\deg\leq\beta-1}$. In other words, $v=b+y$ for some
$y\in\Lambda_{\deg\leq\beta-1}$. Consider this $y$.

We have $v\in\underbrace{b}_{\in\Lambda_{\deg\leq\beta}}+\underbrace{\Lambda
_{\deg\leq\beta-1}}_{\subseteq\Lambda_{\deg\leq\beta}}\subseteq\Lambda
_{\deg\leq\beta}+\Lambda_{\deg\leq\beta}\subseteq\Lambda_{\deg\leq\beta}$
(since $\Lambda_{\deg\leq\beta}$ is a $\mathbf{k}$-module).

Now,
\begin{align*}
\underbrace{x}_{\in\Lambda_{\deg\leq\alpha-1}}\underbrace{v}_{\in\Lambda
_{\deg\leq\beta}}  &  \in\Lambda_{\deg\leq\alpha-1}\Lambda_{\deg\leq\beta
}\subseteq\Lambda_{\deg\leq\left(  \alpha-1\right)  +\beta}\\
&  \ \ \ \ \ \ \ \ \ \ \left(  \text{by Claim 1, applied to }i=\alpha-1\text{
and }j=\beta\right) \\
&  =\Lambda_{\deg\leq\alpha+\beta-1}\ \ \ \ \ \ \ \ \ \ \left(  \text{since
}\left(  \alpha-1\right)  +\beta=\alpha+\beta-1\right)  .
\end{align*}
Furthermore,%
\begin{align*}
\underbrace{a}_{\in\Lambda_{\deg\leq\alpha}}\underbrace{y}_{\in\Lambda
_{\deg\leq\beta-1}}  &  \in\Lambda_{\deg\leq\alpha}\Lambda_{\deg\leq\beta
-1}\subseteq\Lambda_{\deg\leq\alpha+\left(  \beta-1\right)  }\\
&  \ \ \ \ \ \ \ \ \ \ \left(  \text{by Claim 1, applied to }i=\alpha\text{
and }j=\beta-1\right) \\
&  =\Lambda_{\deg\leq\alpha+\beta-1}\ \ \ \ \ \ \ \ \ \ \left(  \text{since
}\alpha+\left(  \beta-1\right)  =\alpha+\beta-1\right)  .
\end{align*}
Now,%
\begin{align*}
\underbrace{u}_{=a+x}v  &  =\left(  a+x\right)  v=a\underbrace{v}%
_{=b+y}+xv=a\left(  b+y\right)  +xv\\
&  =ab+\underbrace{ay}_{\in\Lambda_{\deg\leq\alpha+\beta-1}}+\underbrace{xv}%
_{\in\Lambda_{\deg\leq\alpha+\beta-1}}\in ab+\underbrace{\Lambda_{\deg
\leq\alpha+\beta-1}+\Lambda_{\deg\leq\alpha+\beta-1}}_{\substack{\subseteq
\Lambda_{\deg\leq\alpha+\beta-1}\\\text{(since }\Lambda_{\deg\leq\alpha
+\beta-1}\text{ is a }\mathbf{k}\text{-module)}}}\\
&  \subseteq ab+\Lambda_{\deg\leq\alpha+\beta-1}.
\end{align*}
This proves Claim 2.]

\begin{statement}
\textit{Claim 3:} We have $\varphi\left(  \mathbf{e}_{\lambda}\right)
\in\mathbf{e}_{\lambda}+\Lambda_{\deg\leq\left\vert \lambda\right\vert -1}$
for each partition $\lambda$.
\end{statement}

[\textit{Proof of Claim 3:} We shall prove Claim 3 by induction on
$\ell\left(  \lambda\right)  $.

\textit{Induction base:} Claim 3 is clearly true when $\ell\left(
\lambda\right)  =0$\ \ \ \ \footnote{\textit{Proof.} Let $\lambda$ be a
partition such that $\ell\left(  \lambda\right)  =0$. We must show that
$\varphi\left(  \mathbf{e}_{\lambda}\right)  \in\mathbf{e}_{\lambda}%
+\Lambda_{\deg\leq\left\vert \lambda\right\vert -1}$.
\par
We have $\lambda=\varnothing$ (since $\ell\left(  \lambda\right)  =0$) and
thus $\mathbf{e}_{\lambda}=\mathbf{e}_{\varnothing}=1$. Hence, $\varphi\left(
\mathbf{e}_{\lambda}\right)  =\varphi\left(  1\right)  =1$ (since $\varphi$ is
a $\mathbf{k}$-algebra homomorphism). Thus, $\underbrace{\varphi\left(
\mathbf{e}_{\lambda}\right)  }_{=1}-\underbrace{\mathbf{e}_{\lambda}}%
_{=1}=1-1=0\in\Lambda_{\deg\leq\left\vert \lambda\right\vert -1}$ (since
$\Lambda_{\deg\leq\left\vert \lambda\right\vert -1}$ is a $\mathbf{k}%
$-module), so that $\varphi\left(  \mathbf{e}_{\lambda}\right)  \in
\mathbf{e}_{\lambda}+\Lambda_{\deg\leq\left\vert \lambda\right\vert -1}$. This
is precisely what we needed to show; qed.}. This completes the induction base.

\textit{Induction step:} Let $r$ be a positive integer. Assume (as the
induction hypothesis) that Claim 3 is true whenever $\ell\left(
\lambda\right)  =r-1$. We must prove that Claim 3 is true whenever
$\ell\left(  \lambda\right)  =r$.

So let $\lambda$ be a partition such that $\ell\left(  \lambda\right)  =r$. We
must prove that $\varphi\left(  \mathbf{e}_{\lambda}\right)  \in
\mathbf{e}_{\lambda}+\Lambda_{\deg\leq\left\vert \lambda\right\vert -1}$.

We have $\ell\left(  \lambda\right)  =r$. Thus, the entries $\lambda
_{1},\lambda_{2},\ldots,\lambda_{r}$ of $\lambda$ are positive, while
$\lambda_{r+1}=\lambda_{r+2}=\lambda_{r+3}=\cdots=0$. Hence, $\lambda=\left(
\lambda_{1},\lambda_{2},\ldots,\lambda_{r}\right)  $.

We have $1\in\left\{  1,2,\ldots,r\right\}  $ (since $r$ is positive). Hence,
$\lambda_{1}$ is positive (since $\lambda_{1},\lambda_{2},\ldots,\lambda_{r}$
are positive).

Let $\overline{\lambda}$ be the partition $\left(  \lambda_{2},\lambda
_{3},\lambda_{4},\ldots\right)  $. Then, $\overline{\lambda}=\left(
\lambda_{2},\lambda_{3},\lambda_{4},\ldots\right)  =\left(  \lambda
_{2},\lambda_{3},\ldots,\lambda_{r}\right)  $ (since $\lambda_{r+1}%
=\lambda_{r+2}=\lambda_{r+3}=\cdots=0$), so that $\ell\left(  \overline
{\lambda}\right)  =r-1$ (since $\lambda_{1},\lambda_{2},\ldots,\lambda_{r}$
are positive). Hence, our induction hypothesis shows that Claim 3 holds for
$\overline{\lambda}$ instead of $\lambda$. In other words, we have
$\varphi\left(  \mathbf{e}_{\overline{\lambda}}\right)  \in\mathbf{e}%
_{\overline{\lambda}}+\Lambda_{\deg\leq\left\vert \overline{\lambda
}\right\vert -1}$.

But from $\lambda=\left(  \lambda_{1},\lambda_{2},\ldots,\lambda_{r}\right)  $
and $\overline{\lambda}=\left(  \lambda_{2},\lambda_{3},\ldots,\lambda
_{r}\right)  $, we see easily that $\left\vert \lambda\right\vert =\lambda
_{1}+\left\vert \overline{\lambda}\right\vert $. Furthermore, $\lambda_{1}%
\in\left\{  1,2,3,\ldots\right\}  $ (since $\lambda_{1}$ is positive). Hence,
(\ref{eq.lem.bigquot.filtr-hom.ass}) (applied to $i=\lambda_{1}$) yields
$\varphi\left(  \mathbf{e}_{\lambda_{1}}\right)  \in\mathbf{e}_{\lambda_{1}%
}+\Lambda_{\deg\leq\lambda_{1}-1}$.

The symmetric function $\mathbf{e}_{\lambda_{1}}$ is homogeneous of degree
$\lambda_{1}$. Thus, $\mathbf{e}_{\lambda_{1}}\in\Lambda_{\deg\leq\lambda_{1}%
}$.

The symmetric function $\mathbf{e}_{\overline{\lambda}}$ is homogeneous of
degree $\left\vert \overline{\lambda}\right\vert $. Thus, $\mathbf{e}%
_{\overline{\lambda}}\in\Lambda_{\deg\leq\left\vert \overline{\lambda
}\right\vert }$.

We have now shown that $\mathbf{e}_{\lambda_{1}}\in\Lambda_{\deg\leq
\lambda_{1}}$ and $\mathbf{e}_{\overline{\lambda}}\in\Lambda_{\deg
\leq\left\vert \overline{\lambda}\right\vert }$ and $\varphi\left(
\mathbf{e}_{\lambda_{1}}\right)  \in\mathbf{e}_{\lambda_{1}}+\Lambda_{\deg
\leq\lambda_{1}-1}$ and $\varphi\left(  \mathbf{e}_{\overline{\lambda}%
}\right)  \in\mathbf{e}_{\overline{\lambda}}+\Lambda_{\deg\leq\left\vert
\overline{\lambda}\right\vert -1}$. Thus, Claim 2 (applied to $\alpha
=\lambda_{1}$, $\beta=\left\vert \overline{\lambda}\right\vert $,
$a=\mathbf{e}_{\lambda_{1}}$, $b=\mathbf{e}_{\overline{\lambda}}$,
$u=\varphi\left(  \mathbf{e}_{\lambda_{1}}\right)  $ and $v=\varphi\left(
\mathbf{e}_{\overline{\lambda}}\right)  $) yields that%
\begin{align}
\varphi\left(  \mathbf{e}_{\lambda_{1}}\right)  \varphi\left(  \mathbf{e}%
_{\overline{\lambda}}\right)   &  \in\mathbf{e}_{\lambda_{1}}\mathbf{e}%
_{\overline{\lambda}}+\Lambda_{\deg\leq\lambda_{1}+\left\vert \overline
{\lambda}\right\vert -1}\nonumber\\
&  =\mathbf{e}_{\lambda_{1}}\mathbf{e}_{\overline{\lambda}}+\Lambda_{\deg
\leq\left\vert \lambda\right\vert -1} \label{pf.lem.bigquot.filtr-hom.2}%
\end{align}
(since $\lambda_{1}+\left\vert \overline{\lambda}\right\vert =\left\vert
\lambda\right\vert $).

But $\overline{\lambda}=\left(  \lambda_{2},\lambda_{3},\lambda_{4}%
,\ldots\right)  $; thus, the definition of $\mathbf{e}_{\overline{\lambda}}$
yields%
\begin{equation}
\mathbf{e}_{\overline{\lambda}}=\mathbf{e}_{\lambda_{2}}\mathbf{e}%
_{\lambda_{3}}\mathbf{e}_{\lambda_{4}}\cdots.
\label{pf.lem.bigquot.filtr-hom.3}%
\end{equation}

But the definition of $\mathbf{e}_{\lambda}$ yields
\begin{equation}
\mathbf{e}_{\lambda}=\mathbf{e}_{\lambda_{1}}\mathbf{e}_{\lambda_{2}%
}\mathbf{e}_{\lambda_{3}}\cdots=\mathbf{e}_{\lambda_{1}}\underbrace{\left(
\mathbf{e}_{\lambda_{2}}\mathbf{e}_{\lambda_{3}}\mathbf{e}_{\lambda_{4}}%
\cdots\right)  }_{\substack{=\mathbf{e}_{\overline{\lambda}}\\\text{(by
(\ref{pf.lem.bigquot.filtr-hom.3}))}}}=\mathbf{e}_{\lambda_{1}}\mathbf{e}%
_{\overline{\lambda}}. \label{pf.lem.bigquot.filtr-hom.5}%
\end{equation}
Applying the map $\varphi$ to both sides of this equality, we obtain%
\begin{align*}
\varphi\left(  \mathbf{e}_{\lambda}\right)   &  =\varphi\left(  \mathbf{e}%
_{\lambda_{1}}\mathbf{e}_{\overline{\lambda}}\right)  =\varphi\left(
\mathbf{e}_{\lambda_{1}}\right)  \varphi\left(  \mathbf{e}_{\overline{\lambda
}}\right)  \ \ \ \ \ \ \ \ \ \ \left(  \text{since }\varphi\text{ is a
}\mathbf{k}\text{-algebra homomorphism}\right) \\
&  \in\underbrace{\mathbf{e}_{\lambda_{1}}\mathbf{e}_{\overline{\lambda}}%
}_{\substack{=\mathbf{e}_{\lambda}\\\text{(by
(\ref{pf.lem.bigquot.filtr-hom.5}))}}}+\Lambda_{\deg\leq\left\vert
\lambda\right\vert -1}\ \ \ \ \ \ \ \ \ \ \left(  \text{by
(\ref{pf.lem.bigquot.filtr-hom.2})}\right) \\
&  =\mathbf{e}_{\lambda}+\Lambda_{\deg\leq\left\vert \lambda\right\vert -1}.
\end{align*}

Now, forget that we fixed $\lambda$. We thus have proven that $\varphi\left(
\mathbf{e}_{\lambda}\right)  \in\mathbf{e}_{\lambda}+\Lambda_{\deg
\leq\left\vert \lambda\right\vert -1}$ for each partition $\lambda$ satisfying
$\ell\left(  \lambda\right)  =r$. In other words, Claim 3 is true whenever
$\ell\left(  \lambda\right)  =r$. This completes the induction step. Thus,
Claim 3 is proven.]

We also notice that%
\begin{equation}
\Lambda_{\deg\leq-1}\subseteq\Lambda_{\deg\leq0}\subseteq\Lambda_{\deg\leq
1}\subseteq\Lambda_{\deg\leq2}\subseteq\cdots
\label{pf.lem.bigquot.filtr-hom.filt}%
\end{equation}
(by the definition of the $\Lambda_{\deg\leq m}$).

\textbf{(a)} Let $m\in\mathbb{N}$. Let $v\in\Lambda_{\deg\leq m}$. We must
prove that $\varphi\left(  v\right)  \in v+\Lambda_{\deg\leq m-1}$.

We know that $v$ is a symmetric function of degree $\leq m$ (since
$v\in\Lambda_{\deg\leq m}$). Thus, $v$ is a symmetric function of degree
$<m+1$. Hence, Lemma \ref{lem.bigquot.e-as-smaller} (applied to $N=m+1$ and
$f=v$) yields that there exists a family $\left(  c_{\kappa}\right)
_{\kappa\in\operatorname*{Par};\ \left\vert \kappa\right\vert <m+1}$ of
elements of $\mathbf{k}$ such that
\begin{equation}
v=\sum_{\substack{\kappa\in\operatorname*{Par};\\\left\vert \kappa\right\vert
<m+1}}c_{\kappa}\mathbf{e}_{\kappa}. \label{pf.lem.bigquot.filtr-hom.a.v=sum}%
\end{equation}
Consider this $\left(  c_{\kappa}\right)  _{\kappa\in\operatorname*{Par}%
;\ \left\vert \kappa\right\vert <m+1}$.

For every $\kappa\in\operatorname*{Par}$ satisfying $\left\vert \kappa
\right\vert <m+1$, we have%
\begin{equation}
\varphi\left(  \mathbf{e}_{\kappa}\right)  \in\mathbf{e}_{\kappa}%
+\Lambda_{\deg\leq m-1}. \label{pf.lem.bigquot.filtr-hom.a.1}%
\end{equation}

[\textit{Proof of (\ref{pf.lem.bigquot.filtr-hom.a.1}):} Let $\kappa
\in\operatorname*{Par}$ be such that $\left\vert \kappa\right\vert <m+1$. From
$\left\vert \kappa\right\vert <m+1$, we obtain $\left\vert \kappa\right\vert
-1<m$ and thus $\left\vert \kappa\right\vert -1\leq m-1$ (since $\left\vert
\kappa\right\vert -1$ and $m$ are integers). Hence, $\Lambda_{\deg
\leq\left\vert \kappa\right\vert -1}\subseteq\Lambda_{\deg\leq m-1}$ (by
(\ref{pf.lem.bigquot.filtr-hom.filt})).

But Claim 3 (applied to $\lambda=\kappa$) yields $\varphi\left(
\mathbf{e}_{\kappa}\right)  \in\mathbf{e}_{\kappa}+\underbrace{\Lambda
_{\deg\leq\left\vert \kappa\right\vert -1}}_{\subseteq\Lambda_{\deg\leq m-1}%
}\subseteq\mathbf{e}_{\kappa}+\Lambda_{\deg\leq m-1}$. This proves
(\ref{pf.lem.bigquot.filtr-hom.a.1}).]

Now, applying the map $\varphi$ to both sides of the equality
(\ref{pf.lem.bigquot.filtr-hom.a.v=sum}), we obtain%
\begin{align*}
\varphi\left(  v\right)   &  =\varphi\left(  \sum_{\substack{\kappa
\in\operatorname*{Par};\\\left\vert \kappa\right\vert <m+1}}c_{\kappa
}\mathbf{e}_{\kappa}\right)  =\sum_{\substack{\kappa\in\operatorname*{Par}%
;\\\left\vert \kappa\right\vert <m+1}}c_{\kappa}\underbrace{\varphi\left(
\mathbf{e}_{\kappa}\right)  }_{\substack{\in\mathbf{e}_{\kappa}+\Lambda
_{\deg\leq m-1}\\\text{(by (\ref{pf.lem.bigquot.filtr-hom.a.1}))}%
}}\ \ \ \ \ \ \ \ \ \ \left(  \text{since the map }\varphi\text{ is
}\mathbf{k}\text{-linear}\right) \\
&  \in\sum_{\substack{\kappa\in\operatorname*{Par};\\\left\vert \kappa
\right\vert <m+1}}\underbrace{c_{\kappa}\left(  \mathbf{e}_{\kappa}%
+\Lambda_{\deg\leq m-1}\right)  }_{=c_{\kappa}\mathbf{e}_{\kappa}+c_{\kappa
}\Lambda_{\deg\leq m-1}}=\sum_{\substack{\kappa\in\operatorname*{Par}%
;\\\left\vert \kappa\right\vert <m+1}}\left(  c_{\kappa}\mathbf{e}_{\kappa
}+c_{\kappa}\Lambda_{\deg\leq m-1}\right) \\
&  =\underbrace{\sum_{\substack{\kappa\in\operatorname*{Par};\\\left\vert
\kappa\right\vert <m+1}}c_{\kappa}\mathbf{e}_{\kappa}}%
_{\substack{=v\\\text{(by (\ref{pf.lem.bigquot.filtr-hom.a.v=sum}))}%
}}+\underbrace{\sum_{\substack{\kappa\in\operatorname*{Par};\\\left\vert
\kappa\right\vert <m+1}}c_{\kappa}\Lambda_{\deg\leq m-1}}_{\substack{\subseteq
\Lambda_{\deg\leq m-1}\\\text{(since }\Lambda_{\deg\leq m-1}\text{ is a
}\mathbf{k}\text{-module)}}}\subseteq v+\Lambda_{\deg\leq m-1}.
\end{align*}
This proves Lemma \ref{lem.bigquot.filtr-hom} \textbf{(a)}.

\textbf{(b)} The map $\varphi:\Lambda\rightarrow\Lambda$ is a $\mathbf{k}%
$-algebra homomorphism, thus a $\mathbf{k}$-module homomorphism. Lemma
\ref{lem.bigquot.filtr-hom} \textbf{(a)} shows that $\varphi\left(  v\right)
\in v+\Lambda_{\deg\leq m-1}$ for each $m\in\mathbb{N}$ and $v\in\Lambda
_{\deg\leq m}$, where $\Lambda_{\deg\leq-1}$ denotes the $\mathbf{k}%
$-submodule $0$ of $\Lambda$. Hence, Proposition \ref{prop.filtr.iso-g}
\textbf{(a)} (applied to $V=\Lambda$, $V_{m}=\Lambda_{\deg\leq m}$ and
$g=\varphi$) yields that the $\mathbf{k}$-module homomorphism $\varphi$ is an
isomorphism. Hence, this homomorphism $\varphi$ is bijective and thus a
$\mathbf{k}$-algebra isomorphism (since it is a $\mathbf{k}$-algebra
homomorphism). This proves Lemma \ref{lem.bigquot.filtr-hom} \textbf{(b)}.

\textbf{(c)} The map $\varphi:\Lambda\rightarrow\Lambda$ is a $\mathbf{k}%
$-algebra homomorphism, thus a $\mathbf{k}$-module homomorphism. Lemma
\ref{lem.bigquot.filtr-hom} \textbf{(a)} shows that $\varphi\left(  v\right)
\in v+\Lambda_{\deg\leq m-1}$ for each $m\in\mathbb{N}$ and $v\in\Lambda
_{\deg\leq m}$, where $\Lambda_{\deg\leq-1}$ denotes the $\mathbf{k}%
$-submodule $0$ of $\Lambda$. Hence, Proposition \ref{prop.filtr.iso-g}
\textbf{(b)} (applied to $V=\Lambda$, $V_{m}=\Lambda_{\deg\leq m}$ and
$g=\varphi$) yields that each of the maps $\varphi$ and $\varphi^{-1}$
respects the filtration. This proves Lemma \ref{lem.bigquot.filtr-hom}
\textbf{(c)}.
\end{proof}

\subsection{Linear independence}

\begin{proof}
[Proof of Theorem \ref{thm.bigquot}.]Corollary \ref{cor.bigquot.span} shows
that the family $\left(  \overline{\mathbf{s}_{\lambda}}\right)  _{\lambda\in
P_{k,n}}$ spans the $\mathbf{k}$-module $\Lambda/K$. We need to prove that it
is a basis of $\Lambda/K$.

Let us first recall that $\Lambda/\left\langle \mathbf{e}_{k+1},\mathbf{e}%
_{k+2},\mathbf{e}_{k+3},\ldots\right\rangle \cong\mathcal{S}$. More precisely,
there is a canonical surjective $\mathbf{k}$-algebra homomorphism $\pi
:\Lambda\rightarrow\mathcal{S}$ which is given by substituting $0$ for each of
the variables $x_{k+1},x_{k+2},x_{k+3},\ldots$; the kernel of this
homomorphism is precisely the ideal $\left\langle \mathbf{e}_{k+1}%
,\mathbf{e}_{k+2},\mathbf{e}_{k+3},\ldots\right\rangle $ of $\Lambda$. This
homomorphism sends each $\mathbf{h}_{m}\in\Lambda$ to the polynomial $h_{m}%
\in\mathcal{S}$ defined in (\ref{eq.hm}).

It is well-known that the commutative $\mathbf{k}$-algebra $\Lambda$ is freely
generated by its elements $\mathbf{e}_{1},\mathbf{e}_{2},\mathbf{e}_{3}%
,\ldots$. Hence, we can define an $\mathbf{k}$-algebra homomorphism
$\varphi:\Lambda\rightarrow\Lambda$ by letting%
\begin{align}
\varphi\left(  \mathbf{e}_{i}\right)   &  =\mathbf{e}_{i}%
\ \ \ \ \ \ \ \ \ \ \text{for each }i\in\left\{  1,2,\ldots,k\right\}
;\label{pf.thm.bigquot.phi1}\\
\varphi\left(  \mathbf{e}_{i}\right)   &  =\mathbf{e}_{i}-\mathbf{b}%
_{i-k}\ \ \ \ \ \ \ \ \ \ \text{for each }i\in\left\{  k+1,k+2,k+3,\ldots
\right\}  . \label{pf.thm.bigquot.phi2}%
\end{align}
Consider this $\varphi$. Then, we have $\varphi\left(  \mathbf{e}_{i}\right)
\in\mathbf{e}_{i}+\Lambda_{\deg\leq i-1}$ for each $i\in\left\{
1,2,3,\ldots\right\}  $\ \ \ \ \footnote{\textit{Proof.} Let $i\in\left\{
1,2,3,\ldots\right\}  $. We must prove that $\varphi\left(  \mathbf{e}%
_{i}\right)  \in\mathbf{e}_{i}+\Lambda_{\deg\leq i-1}$.
\par
If $i\in\left\{  1,2,\ldots,k\right\}  $, then this is obvious (since the
definition of $\varphi$ yields $\varphi\left(  \mathbf{e}_{i}\right)
=\mathbf{e}_{i}=\mathbf{e}_{i}+\underbrace{0}_{\in\Lambda_{\deg\leq i-1}}%
\in\mathbf{e}_{i}+\Lambda_{\deg\leq i-1}$ in this case). Hence, for the rest
of this proof, we WLOG assume that we don't have $i\in\left\{  1,2,\ldots
,k\right\}  $. Hence,%
\[
i\in\left\{  1,2,3,\ldots\right\}  \setminus\left\{  1,2,\ldots,k\right\}
=\left\{  k+1,k+2,k+3,\ldots\right\}  .
\]
Thus, the definition of $\varphi$ yields $\varphi\left(  \mathbf{e}%
_{i}\right)  =\mathbf{e}_{i}-\mathbf{b}_{i-k}$. But
(\ref{eq.thm.bigquot.bi-deg}) (applied to $i-k$ instead of $i$) yields that
\begin{align*}
\mathbf{b}_{i-k}  &  =\left(  \text{some symmetric function of degree
}<\underbrace{k+\left(  i-k\right)  }_{=i}\right) \\
&  =\left(  \text{some symmetric function of degree }<i\right) \\
&  =\left(  \text{some symmetric function of degree }\leq i-1\right)
\in\Lambda_{\deg\leq i-1}.
\end{align*}
Thus, $\varphi\left(  \mathbf{e}_{i}\right)  =\mathbf{e}_{i}%
-\underbrace{\mathbf{b}_{i-k}}_{\in\Lambda_{\deg\leq i-1}}\in\mathbf{e}%
_{i}-\Lambda_{\deg\leq i-1}=\mathbf{e}_{i}+\Lambda_{\deg\leq i-1}$ (since
$\Lambda_{\deg\leq i-1}$ is a $\mathbf{k}$-module). Qed.}. Hence, Lemma
\ref{lem.bigquot.filtr-hom} \textbf{(a)} shows that we have%
\begin{equation}
\varphi\left(  v\right)  \in v+\Lambda_{\deg\leq m-1}%
\ \ \ \ \ \ \ \ \ \ \text{for each }m\in\mathbb{N}\text{ and }v\in
\Lambda_{\deg\leq m}. \label{pf.thm.bigquot.phiv}%
\end{equation}
(Here, $\Lambda_{\deg\leq-1}$ denotes the $\mathbf{k}$-submodule $0$ of
$\Lambda$.) Furthermore, Lemma \ref{lem.bigquot.filtr-hom} \textbf{(b)} shows
that the map $\varphi:\Lambda\rightarrow\Lambda$ is a $\mathbf{k}$-algebra
isomorphism. In other words, $\varphi$ is an automorphism of the $\mathbf{k}%
$-algebra $\Lambda$. Finally, Lemma \ref{lem.bigquot.filtr-hom} \textbf{(c)}
shows that each of the maps $\varphi$ and $\varphi^{-1}$ respects the filtration.

The map $\varphi$ is a $\mathbf{k}$-algebra automorphism of $\Lambda$ and
sends the elements
\[
\mathbf{e}_{k+1},\mathbf{e}_{k+2},\mathbf{e}_{k+3},\ldots
\ \ \ \ \ \ \ \ \ \ \text{to}\ \ \ \ \ \ \ \ \ \ \mathbf{e}_{k+1}%
-\mathbf{b}_{1},\mathbf{e}_{k+2}-\mathbf{b}_{2},\mathbf{e}_{k+3}%
-\mathbf{b}_{3},\ldots,
\]
respectively (according to (\ref{pf.thm.bigquot.phi2})). Hence, it sends the
ideal $\left\langle \mathbf{e}_{k+1},\mathbf{e}_{k+2},\mathbf{e}_{k+3}%
,\ldots\right\rangle $ of $\Lambda$ to the ideal $\left\langle \mathbf{e}%
_{k+1}-\mathbf{b}_{1},\mathbf{e}_{k+2}-\mathbf{b}_{2},\mathbf{e}%
_{k+3}-\mathbf{b}_{3},\ldots\right\rangle $ of $\Lambda$. In other words,%
\begin{equation}
\varphi\left(  \left\langle \mathbf{e}_{k+1},\mathbf{e}_{k+2},\mathbf{e}%
_{k+3},\ldots\right\rangle \right)  =\left\langle \mathbf{e}_{k+1}%
-\mathbf{b}_{1},\mathbf{e}_{k+2}-\mathbf{b}_{2},\mathbf{e}_{k+3}%
-\mathbf{b}_{3},\ldots\right\rangle . \label{pf.thm.bigquot.phiideal2}%
\end{equation}

For each $i\in\left\{  1,2,\ldots,k\right\}  $, define $\mathbf{c}_{i}%
\in\Lambda$ by
\begin{equation}
\mathbf{c}_{i}=\varphi^{-1}\left(  \varphi\left(  \mathbf{h}_{n-k+i}\right)
-\mathbf{h}_{n-k+i}+\mathbf{a}_{i}\right)  . \label{pf.thm.bigquot.ci=}%
\end{equation}
This is well-defined, since $\varphi$ is an isomorphism. For each
$i\in\left\{  1,2,\ldots,k\right\}  $, we have
\begin{align*}
\varphi\left(  \mathbf{h}_{n-k+i}-\mathbf{c}_{i}\right)   &  =\varphi\left(
\mathbf{h}_{n-k+i}\right)  -\underbrace{\varphi\left(  \mathbf{c}_{i}\right)
}_{\substack{=\varphi\left(  \mathbf{h}_{n-k+i}\right)  -\mathbf{h}%
_{n-k+i}+\mathbf{a}_{i}\\\text{(by (\ref{pf.thm.bigquot.ci=}))}}}\\
&  \ \ \ \ \ \ \ \ \ \ \left(  \text{since }\varphi\text{ is a }%
\mathbf{k}\text{-algebra homomorphism}\right) \\
&  =\varphi\left(  \mathbf{h}_{n-k+i}\right)  -\left(  \varphi\left(
\mathbf{h}_{n-k+i}\right)  -\mathbf{h}_{n-k+i}+\mathbf{a}_{i}\right)
=\mathbf{h}_{n-k+i}-\mathbf{a}_{i}.
\end{align*}
In other words, the map $\varphi$ sends the elements $\mathbf{h}%
_{n-k+1}-\mathbf{c}_{1},\mathbf{h}_{n-k+2}-\mathbf{c}_{2},\ldots
,\mathbf{h}_{n}-\mathbf{c}_{k}$ to the elements $\mathbf{h}_{n-k+1}%
-\mathbf{a}_{1},\mathbf{h}_{n-k+2}-\mathbf{a}_{2},\ldots,\mathbf{h}%
_{n}-\mathbf{a}_{k}$, respectively. Thus, it sends the ideal $\left\langle
\mathbf{h}_{n-k+1}-\mathbf{c}_{1},\mathbf{h}_{n-k+2}-\mathbf{c}_{2}%
,\ldots,\mathbf{h}_{n}-\mathbf{c}_{k}\right\rangle $ of $\Lambda$ to the ideal
\newline$\left\langle \mathbf{h}_{n-k+1}-\mathbf{a}_{1},\mathbf{h}%
_{n-k+2}-\mathbf{a}_{2},\ldots,\mathbf{h}_{n}-\mathbf{a}_{k}\right\rangle $ of
$\Lambda$ (since $\varphi$ is a $\mathbf{k}$-algebra automorphism). In other
words,%
\begin{align}
&  \varphi\left(  \left\langle \mathbf{h}_{n-k+1}-\mathbf{c}_{1}%
,\mathbf{h}_{n-k+2}-\mathbf{c}_{2},\ldots,\mathbf{h}_{n}-\mathbf{c}%
_{k}\right\rangle \right) \nonumber\\
&  =\left\langle \mathbf{h}_{n-k+1}-\mathbf{a}_{1},\mathbf{h}_{n-k+2}%
-\mathbf{a}_{2},\ldots,\mathbf{h}_{n}-\mathbf{a}_{k}\right\rangle .
\label{pf.thm.bigquot.phiideal1}%
\end{align}

Recall that $\varphi$ is a $\mathbf{k}$-algebra homomorphism; thus,%
\begin{align}
&  \varphi\left(  \left\langle \mathbf{h}_{n-k+1}-\mathbf{c}_{1}%
,\mathbf{h}_{n-k+2}-\mathbf{c}_{2},\ldots,\mathbf{h}_{n}-\mathbf{c}%
_{k}\right\rangle +\left\langle \mathbf{e}_{k+1},\mathbf{e}_{k+2}%
,\mathbf{e}_{k+3},\ldots\right\rangle \right) \nonumber\\
&  =\underbrace{\varphi\left(  \left\langle \mathbf{h}_{n-k+1}-\mathbf{c}%
_{1},\mathbf{h}_{n-k+2}-\mathbf{c}_{2},\ldots,\mathbf{h}_{n}-\mathbf{c}%
_{k}\right\rangle \right)  }_{\substack{=\left\langle \mathbf{h}%
_{n-k+1}-\mathbf{a}_{1},\mathbf{h}_{n-k+2}-\mathbf{a}_{2},\ldots
,\mathbf{h}_{n}-\mathbf{a}_{k}\right\rangle \\\text{(by
(\ref{pf.thm.bigquot.phiideal1}))}}}+\underbrace{\varphi\left(  \left\langle
\mathbf{e}_{k+1},\mathbf{e}_{k+2},\mathbf{e}_{k+3},\ldots\right\rangle
\right)  }_{\substack{=\left\langle \mathbf{e}_{k+1}-\mathbf{b}_{1}%
,\mathbf{e}_{k+2}-\mathbf{b}_{2},\mathbf{e}_{k+3}-\mathbf{b}_{3}%
,\ldots\right\rangle \\\text{(by (\ref{pf.thm.bigquot.phiideal2}))}%
}}\nonumber\\
&  =\left\langle \mathbf{h}_{n-k+1}-\mathbf{a}_{1},\mathbf{h}_{n-k+2}%
-\mathbf{a}_{2},\ldots,\mathbf{h}_{n}-\mathbf{a}_{k}\right\rangle
+\left\langle \mathbf{e}_{k+1}-\mathbf{b}_{1},\mathbf{e}_{k+2}-\mathbf{b}%
_{2},\mathbf{e}_{k+3}-\mathbf{b}_{3},\ldots\right\rangle \nonumber\\
&  =K \label{pf.thm.bigquot.phiidealK}%
\end{align}
(by the definition of $K$).

For each $i\in\left\{  1,2,\ldots,k\right\}  $, let us consider the projection
$\overline{\mathbf{c}_{i}}$ of $\mathbf{c}_{i}\in\Lambda$ onto $\mathcal{S}$.
Let $I_{\mathbf{c}}$ denote the ideal of $\mathcal{S}$ generated by the $k$
differences%
\[
h_{n-k+1}-\overline{\mathbf{c}_{1}},h_{n-k+2}-\overline{\mathbf{c}_{2}}%
,\ldots,h_{n}-\overline{\mathbf{c}_{k}}.
\]

Moreover, for each $i\in\left\{  1,2,\ldots,k\right\}  $, the element
$\mathbf{c}_{i}$ is a symmetric function of degree $<n-k+i$%
\ \ \ \ \footnote{\textit{Proof.} Let $i\in\left\{  1,2,\ldots,k\right\}  $.
Thus, $\mathbf{h}_{n-k+i}$ is a homogeneous symmetric function of degree
$n-k+i$. Hence, $\mathbf{h}_{n-k+i}\in\Lambda_{\deg\leq n-k+i}$. Thus,
(\ref{pf.thm.bigquot.phiv}) (applied to $m=n-k+i$ and $v=\mathbf{h}_{n-k+i}$)
yields $\varphi\left(  \mathbf{h}_{n-k+i}\right)  \in\mathbf{h}_{n-k+i}%
+\Lambda_{\deg\leq n-k+i-1}$. In other words, $\varphi\left(  \mathbf{h}%
_{n-k+i}\right)  -\mathbf{h}_{n-k+i}\in\Lambda_{\deg\leq n-k+i-1}$.
\par
Also, (\ref{eq.thm.bigquot.ai-deg}) yields%
\begin{align*}
\mathbf{a}_{i}  &  =\left(  \text{some symmetric function of degree
}<n-k+i\right) \\
&  =\left(  \text{some symmetric function of degree }\leq n-k+i-1\right)
\in\Lambda_{\deg\leq n-k+i-1}.
\end{align*}
Hence,
\[
\underbrace{\varphi\left(  \mathbf{h}_{n-k+i}\right)  -\mathbf{h}_{n-k+i}%
}_{\in\Lambda_{\deg\leq n-k+i-1}}+\underbrace{\mathbf{a}_{i}}_{\in
\Lambda_{\deg\leq n-k+i-1}}\in\Lambda_{\deg\leq n-k+i-1}+\Lambda_{\deg\leq
n-k+i-1}\subseteq\Lambda_{\deg\leq n-k+i-1}%
\]
(since $\Lambda_{\deg\leq n-k+i-1}$ is a $\mathbf{k}$-module). But the map
$\varphi^{-1}$ respects the filtration; in other words, we have $\varphi
^{-1}\left(  \Lambda_{\deg\leq m}\right)  \subseteq\Lambda_{\deg\leq m}$ for
each $m\in\mathbb{N}$. Applying this to $m=n-k+i-1$, we obtain $\varphi
^{-1}\left(  \Lambda_{\deg\leq n-k+i-1}\right)  \subseteq\Lambda_{\deg\leq
n-k+i-1}$. Now, (\ref{pf.thm.bigquot.ci=}) becomes%
\[
\mathbf{c}_{i}=\varphi^{-1}\left(  \underbrace{\varphi\left(  \mathbf{h}%
_{n-k+i}\right)  -\mathbf{h}_{n-k+i}+\mathbf{a}_{i}}_{\in\Lambda_{\deg\leq
n-k+i-1}}\right)  \in\varphi^{-1}\left(  \Lambda_{\deg\leq n-k+i-1}\right)
\subseteq\Lambda_{\deg\leq n-k+i-1}.
\]
In other words, $\mathbf{c}_{i}$ is a symmetric function of degree $\leq
n-k+i-1$. Hence, $\mathbf{c}_{i}$ is a symmetric function of degree $<n-k+i$.
Qed.}. Hence, for each $i\in\left\{  1,2,\ldots,k\right\}  $, the projection
$\overline{\mathbf{c}_{i}}$ of $\mathbf{c}_{i}\in\Lambda$ onto $\mathcal{S}$
is a symmetric polynomial of degree $<n-k+i$ (because projecting a symmetric
function from $\Lambda$ onto $\mathcal{S}$ cannot raise the degree). Thus,
Theorem \ref{thm.S/J} (applied to $\overline{\mathbf{c}_{i}}$ and
$I_{\mathbf{c}}$ instead of $a_{i}$ and $I$) yields that the $\mathbf{k}%
$-module $\mathcal{S}/I_{\mathbf{c}}$ is free with basis $\left(
\overline{s_{\lambda}}\right)  _{\lambda\in P_{k,n}}$. Hence, this
$\mathbf{k}$-module $S/I_{\mathbf{c}}$ is free and has a basis of size
$\left\vert P_{k,n}\right\vert $.

But $\varphi$ is a $\mathbf{k}$-algebra automorphism of $\Lambda$. Thus, we
have a $\mathbf{k}$-module isomorphism%
\begin{align*}
&  \Lambda/\left(  \left\langle \mathbf{h}_{n-k+1}-\mathbf{c}_{1}%
,\mathbf{h}_{n-k+2}-\mathbf{c}_{2},\ldots,\mathbf{h}_{n}-\mathbf{c}%
_{k}\right\rangle +\left\langle \mathbf{e}_{k+1},\mathbf{e}_{k+2}%
,\mathbf{e}_{k+3},\ldots\right\rangle \right) \\
&  \cong\Lambda/\underbrace{\varphi\left(  \left\langle \mathbf{h}%
_{n-k+1}-\mathbf{c}_{1},\mathbf{h}_{n-k+2}-\mathbf{c}_{2},\ldots
,\mathbf{h}_{n}-\mathbf{c}_{k}\right\rangle +\left\langle \mathbf{e}%
_{k+1},\mathbf{e}_{k+2},\mathbf{e}_{k+3},\ldots\right\rangle \right)
}_{\substack{=K\\\text{(by (\ref{pf.thm.bigquot.phiidealK}))}}}\\
&  =\Lambda/K.
\end{align*}
Hence, we have the following chain of $\mathbf{k}$-module isomorphisms:%
\begin{align*}
\Lambda/K  &  \cong\Lambda/\left(  \left\langle \mathbf{h}_{n-k+1}%
-\mathbf{c}_{1},\mathbf{h}_{n-k+2}-\mathbf{c}_{2},\ldots,\mathbf{h}%
_{n}-\mathbf{c}_{k}\right\rangle +\left\langle \mathbf{e}_{k+1},\mathbf{e}%
_{k+2},\mathbf{e}_{k+3},\ldots\right\rangle \right) \\
&  \cong\underbrace{\left(  \Lambda/\left\langle \mathbf{e}_{k+1}%
,\mathbf{e}_{k+2},\mathbf{e}_{k+3},\ldots\right\rangle \right)  }%
_{\cong\mathcal{S}}/\left\langle \overline{\mathbf{h}_{n-k+1}-\mathbf{c}_{1}%
},\overline{\mathbf{h}_{n-k+2}-\mathbf{c}_{2}},\ldots,\overline{\mathbf{h}%
_{n}-\mathbf{c}_{k}}\right\rangle \\
&  \ \ \ \ \ \ \ \ \ \ \left(
\begin{array}
[c]{c}%
\text{where }\overline{\mathbf{h}_{n-k+1}-\mathbf{c}_{1}},\overline
{\mathbf{h}_{n-k+2}-\mathbf{c}_{2}},\ldots,\overline{\mathbf{h}_{n}%
-\mathbf{c}_{k}}\text{ denote}\\
\text{the projections of }\mathbf{h}_{n-k+1}-\mathbf{c}_{1},\mathbf{h}%
_{n-k+2}-\mathbf{c}_{2},\ldots,\mathbf{h}_{n}-\mathbf{c}_{k}\\
\text{onto }\Lambda/\left\langle \mathbf{e}_{k+1},\mathbf{e}_{k+2}%
,\mathbf{e}_{k+3},\ldots\right\rangle
\end{array}
\right) \\
&  \cong\mathcal{S}/\underbrace{\left\langle h_{n-k+1}-\overline
{\mathbf{c}_{1}},h_{n-k+2}-\overline{\mathbf{c}_{2}},\ldots,h_{n}%
-\overline{\mathbf{c}_{k}}\right\rangle }_{\substack{=I_{\mathbf{c}%
}\\\text{(by the definition of }I_{\mathbf{c}}\text{)}}}=\mathcal{S}%
/I_{\mathbf{c}}.
\end{align*}
Hence, the $\mathbf{k}$-module $\Lambda/K$ is free and has a basis of size
$\left\vert P_{k,n}\right\vert $ (since the $\mathbf{k}$-module
$S/I_{\mathbf{c}}$ is free and has a basis of size $\left\vert P_{k,n}%
\right\vert $).

Now, recall that the family $\left(  \overline{\mathbf{s}_{\lambda}}\right)
_{\lambda\in P_{k,n}}$ spans the $\mathbf{k}$-module $\Lambda/K$. Hence, Lemma
\ref{lem.freemod-span-basis} shows that this family must be a basis of
$\Lambda/K$ (since it has the same size as a basis of $\Lambda/K$). This
proves Theorem \ref{thm.bigquot}.
\end{proof}

\end{document}